\documentclass{article}
\usepackage{geometry}
\usepackage{latexsym}
\usepackage{indentfirst}
\usepackage{graphicx}
\usepackage{amsmath,amsfonts}
\usepackage{amssymb}
\usepackage{amsthm}
\usepackage{bm}
\usepackage{accents}
\usepackage{verbatim}
\usepackage{changepage}
\usepackage{color}
\usepackage{bigints}
\usepackage{enumerate}
\usepackage{amsmath, nccmath}
\usepackage{yhmath}
\usepackage{dirtytalk}
\usepackage{hyperref}
\usepackage{cleveref}
\usepackage{upgreek}
\usepackage{abraces}

\newcommand{\ca}{{\mathcal A}}
\newcommand{\cu}{{\mathcal U}}
\newcommand{\cb}{{\mathcal B}}
\newcommand{\crr}{{\mathcal R}}
\newcommand{\cf}{{\mathcal F}}
\newcommand{\ce}{{\mathcal E}}

\newcommand{\cg}{{\mathcal G}}

\newtheorem{prop}{Proposition}[section]
\newtheorem*{defi*}{Definition}
\newtheorem{defi}[prop]{Definition}

\newtheorem{lem}[prop]{Lemma}
\newtheorem{rem}[prop]{Remark}
\newtheorem{thm}[prop]{Theorem}
\newtheorem{coro}[prop]{Corollary}

\numberwithin{equation}{section}

\title{Regularization of the superposition principle: Potential theory meets Fokker-Planck equations}

\author{Lucian Beznea\footnote{Simion Stoilow Institute of Mathematics  of the Romanian Academy,
 Research unit No. 2, 
P.O. Box \mbox{1-764,} RO-014700 Bucharest, Romania, and 
University POLITEHNICA Bucharest, CAMPUS Institute
(e-mail: lucian.beznea@imar.ro)},
Iulian C\^{i}mpean\footnote{Department of Mathematics, Faculty of Mathematics and Computer Science, University of Bucharest, 14 Academiei, 010014~Bucharest, Romania and Simion Stoilow Institute of Mathematics of the Romanian Academy, 21~Calea Grivi\c{t}ei, 010702~Bucharest, Romania
(email: iulian.cimpean@unibuc.ro; iulian.cimpean@imar.ro)},
Michael R\"ockner\footnote{Fakult\"at f\"ur Mathematik, Universit\"at Bielefeld,
Postfach 100 131, D-33501 Bielefeld, Germany, Academy for Mathematics and Systems Science, CAS, Beijing, and School of Data Science, The Chinese University of Hong Kong, Shenzhen
(CUHK-Shenzhen), China 
(e-mail: roeckner@mathematik.uni-bielefeld.de)}}

\date{}

\begin{document}

\maketitle

\tableofcontents

\begin{abstract}
For a (probability measure valued) solution to a (possibly nonlinear) Fokker-Planck equation (FPE) the powerful superposition principle renders a probability measure on path space with one dimensional time marginals equal to this solution, and additionally solving the martingale problem for the (possibly nonlinear) Kolmogorov operator given by the FPE. The superposition principle thus reveals that such parabolic PDEs have a probabilistic counter part. The aim of this work is to go a substantial further step and, by exploiting the superposition principle, construct a full fledged Markov process, i.e. a family of path space measures for a large set of space time starting points connected by the Markov property, associated to the (linearized) FPE in the above way. In fact, under very general (merely measurability) conditions on the coefficients of the FPE
this is achieved in this paper in such a way that the resulting process is a right process, which is a particularly useful class of Markov processes, enjoying among other regularity properties the strong Markov property, which is fundamental for the analysis of the underlying FPE as a (nonlinear) parabolic PDE by probabilistic tools. For example, as two main applications we construct fundamental flow solutions for the FPE and we prove a well-posedeness result for the parabolic Dirichlet problem through 
probabilistic means for more general coefficients than could be treated in the existing literature. 
Furthermore, we introduce a Choquet capacity for such FPEs using the corresponding right process. 
The validity of the strong Markov property in the context of the superposition principle was an open problem even in the linear case. In this paper we solve this also in the nonlinear case, i.e. for path laws of solutions to McKean-Vlasov SDEs with Nemytskii type coefficients. A main application here is the FPE given by the generalized porous media equation and its corresponding McKean-Vlasov SDE. For the reader's convenience, 
an introduction to the theory of right process and its potential theory is given in the last section of the paper.
\end{abstract}

\\


\section{Introduction}

Let $(\mu_t)_{t\in [0,T]}\subset \mathcal{P}(\mathbb{R}^d)$ be a weakly continuous solution to the time-dependent (non)linear Fokker-Planck equation
\begin{equation}\label{eq:FP_intro}
    \frac{d}{dt} \mu_t = {\sf L}^\ast_{t,\mu_t} \mu_t, \quad t\in (0,T),
\end{equation}
with $L_{t,\mu_t}$ given by
\begin{equation*}
    {\sf L}_{t,\mu_t}f(t,x):=\sum\limits_{i=1}^d b_i(t,x,\mu_t)\frac{\partial}{\partial x_i} f(t,x) + \frac{1}{2}\sum\limits_{i,j=1}^d a_{i,j}(t,x,\mu_t)\frac{\partial^2}{\partial x_i \partial x_j} f(t,x),  \quad t\in (0,T), x\in \mathbb{R}^d,
\end{equation*}
where $f\in C_c^\infty((0,\infty)\times \mathbb{R}^d)$.
Then, by \cite{BaRo20}, under mild assumptions on the coefficients $a$ and $b$, the well-known superposition principle ensures the existence of a probability measure $\eta$ on the paths-space $C([0,T];\mathbb{R}^d)$ which is a solution to the (non)linear martingale problem associated to $L_{t,\mu_t}$ with one-dimensional time marginals given by $\mu_t, t\geq 0$, or equivalently, which renders a weak solution to the McKean-Vlasov SDE
\begin{equation}\label{eq:MV_intro}
    dX(t)=b\left(t,x,\mathcal{L}_{X(t)}\right)dt+\sigma\left(t,x,\mathcal{L}_{X(t)}\right)dW(t), \quad \mathcal{L}_{X(t)}= \mu_t, t\in [0,T],
\end{equation}
where $W$ is a $d$-dimensional standard Brownian motion whilst $\mathcal{L}_{X(t)}$ denotes the law of $X(t)$, $t\geq 0$.
The converse implication, namely that the one-dimensional time marginals of a weak solution to the McKean-Vlasov SDE \eqref{eq:MV_intro} solves the (non)linear Fokker-Planck equation \eqref{eq:FP_intro}, is just a simple consequence of Ito's formula.
Thus, under mild assumptions on the coefficients, there is a well-understood correspondence between \eqref{eq:FP_intro} and \eqref{eq:MV_intro}.
The hard implication is, of course, the one that comes from the superposition principle, which was first obtained in the linear case (i.e. the coefficients do not depend on $\mu_t$) by \cite{Am04} and \cite{Fi08}, followed up by \cite{Tr16}, \cite{BoRoSh21}, as well as by other contributions.
The extension of the superposition principle to the nonlinear setting and the corresponding McKean-Vlasov SDEs was realized in \cite{BaRo18}, \cite{BaRo20}.
We refer to \cite{BoKrRoSh22} for a comprehensive study of the theory of linear Fokker-Planck-Kolmogorov equations, to \cite{CaDe18} for a detailed presentation of McKean-Vlasov SDEs and their applications, and to \cite{BaRo24} for the analysis of nonlinear Fokker-Planck equations and their probabilistic counterpart which are nonlinear Markov processes in the sense of McKean; see also the seminal paper \cite{McK66}, as well as the very recent work \cite{ReRo25} in which a general recipe is developed how to prove that the family of laws of solutions to a McKean-Vlasov stochastic differential equation with merely measurable coefficients (also only measurable in the measure variable) indexed by a large enough set of initial distributions form a nonlinear Markov process in the sense of McKean. This applies to a large number of examples (see \cite{BaRo24} and the references therein), in particular, to also those in \cite{BaReRo24}, \cite{BaGrReRo25}, where the underlying nonlinear Fokker-Planck equation is the parabolic $p$-Laplace equation and the Leibenson equation respectively, with corresponding non-linear Markov process being the $p$-Brownian motion and the Leibenson process.
See also \cite{AmNiGi05} for a gradient flow theory in metric spaces and the superposition principle for the continuity equation.

Another fundamental feature of the superposition principle is that, under a weak uniqueness assumption, the martingale solution denoted above by $\eta$ enjoys the "simple" Markov property; see \Cref{ss:superposition_simple_markov}, more precisely \eqref{eq:simpleM}.
Thus, under very mild assumptions on the coefficients, the correspondence of (non)linear Fokker-Planck equations and McKean Vlasov SDEs could be analyzed through the lens of Markov processes.
However, it is well known that the "simple" Markov property \eqref{eq:simpleM} is typically not enough in order to benefit from the rich theory of Markov processes, since additional regularity properties of the process, like the strong Markov property, are necessary.
This leads us to the formulation of the first main goal which we achieve in this paper:

\medskip
$\bullet\quad$ {\it We show that, under reasonable conditions (see \Cref{thm:main0}, \Cref{thm:main1}), one can construct a homogenized {\it right (hence strong Markov) process} $X$ on a subset $E\subset [0,\infty)\times \mathbb{R}^d$ with $\mu_t(\{x:(t,x)\in E\})=1, t\geq 0$, which is associated to a given weakly continuous solution $(\mu_t)_{t\geq 0}$ to the linear Fokker-Planck equation \eqref{eq:FP_intro} in such a way that the family of one-dimensional time marginals of the spatial component of $X$ coincides with the given solution $(\mu_t)_{t\geq 0}$.
}

\medskip
As indicated above, a right process automatically has the strong Markov property, in particular we answer the question on the validity of this property for the superposed measure $\eta$ mentioned above, and which was left open in \cite[Remark 2.10]{Tr16}.
Moreover, we deduce the strong Markov property for the nonlinear Markov processes associated with nonlinear Fokker-Planck equation from \cite{BaRo24}.

Furthermore, a right process has associated a transition semigroup with regularity properties that substantially generalize the usual ones valid for Feller transition semigroups. 
However, the class of right processes is a much larger class than the well studied Feller processes.
For a concise introduction to the theory of right processes, we refer to \Cref{Appendix}.

We here recall that right processes can be immediately constructed starting from a Feller transition semigroup on a locally compact space with countable base, that is a Markov transition semigroup which is also a $C_0$-semigroup on the space of continuous functions vanishing at infinity; see e.g. \cite[Theorem 4.4]{BlGe68}.
However, the right processes we shall construct will not have a Feller transition semigroup.

As already mentioned, the interest in constructing right processes is motivated by the fact that they enjoy plenty of useful properties which are otherwise not ensured if a process satisfies merely the simple Markov property \eqref{eq:simpleM}.
Some of these key properties that a right process enjoys are briefly listed below (for details we refer to \Cref{Appendix} and to the classical work \cite{BlGe68} and \cite{Sh88}):

\medskip
$\circ\quad$ The underlying filtration satisfies the usual hypotheses.

\medskip
$\circ\quad$ The hitting times of Borel sets are stopping times, and they can be approximated from above by hitting times of compact sets.

\medskip
$\circ\quad$ The strong Markov property holds.

\medskip
$\circ\quad$ The Blumenthal zero-one law is satisfied.

\medskip
$\circ\quad$ The Dirichlet problem can be  analyzed by means of hitting distributions.

\medskip
$\circ\quad$ Analytic concepts such as polar sets and Choquet capacities have a full probabilistic interpretations.

\medskip
$\circ\quad$ Many useful transformations such as killing, time change, subordination, Girsanov transforms, concatenation, perturbation by kernels, are possible for right processes and leave this class invariant.

\medskip
Concerning the methods and techniques of proof we would like to indicate the following:

\medskip 
$\bullet\quad$ {\it In order to construct the desired right process we first treat linear time-dependent Fokker-Planck equations, and the approach we adopt mixes fundamental tools from the theory of generalized Dirichlet forms, the theory of Fokker-Planck-Kolmogorov equations, and potential theory. 
It consists of three main parts which are systematically treated in \Cref{ss:more regularity}, \Cref{ss:main results}, and \Cref{ss:proof}, and which can be summarized as follows:}
\begin{enumerate}
    \item[I] (see \Cref{ss:more regularity}) Starting from a solution $(\mu_t)_{t\geq 0}$ to the linear Fokker-Planck equation \eqref{eq:LFPshort}, for each $N>0$ we use the theory of generalized Dirichlet forms for time-dependent operators on $L^1$, as developed in \cite{St99a} (see also \cite{St99b} and \cite{HaStTr22}, as well as the classical work \cite{FuOsTa11} and \cite{MaRo92}), in order to construct a time-homogenized right process on a subset of the product space $(0,N)\times \mathbb{R}^d$, which is of full $\mu_t(dx)dt$-measure.
    \item[II] (see \Cref{ss:proof of main0}) We use a convenient restricted uniqueness assumption, the standard superposition principle from \Cref{rem:Trevisan_superposition}, and potential theoretic tools, in order to construct a right process on a subset of $(0,\infty)\times \mathbb{R}^d$ of full $\mu_t(dx)dt$-measure, by taking the projective limit of the processes constructed in I for each $N>0$ using \cite{St99a}.
    \item[III] (see \Cref{ss: proof of main1}) Finally, we use again the superposition principle together with potential theoretic tools developed in \cite{BeCiRo18} in order to extend the process constructed in II to a conservative right process with continuous paths defined on a larger set $E \subset [0,\infty)\times \mathbb{R}^d$ such that $\mu_t(E)=1, t\in [0,\infty)$, and such that, if we start from $\delta_0\otimes\mu_0 $ then the family of one-dimensional time marginals of the spatial component of the right process coincides with the given $(\mu_t)_{t\geq 0}$. 
\end{enumerate}

Among the main consequences of our results \Cref{thm:main0} and \Cref{thm:main1}, we would like to mention the following:

\medskip
$\bullet\quad$ {\it We deduce several new results in the theory of time-dependent PDEs with only  measurable coefficients, namely:} 

\medskip
$\circ\quad$ We show the existence of a {\it fundamental flow solution} to the time-dependent linear Fokker-Planck equation \eqref{eq:LFP}; see \Cref{coro:fundamental}.

\medskip
$\circ\quad$ We derive maximal tail estimates for the path-space law obtained by the superposition of the fundamental flow solution to the time-dependent linear Fokker-Planck equation \eqref{eq:LFP} by means of Lyapunov functions; see \Cref{prop:LV}.

\medskip
$\circ\quad$ We solve the backward Kolmogorov PDE associated to the time-dependent linear operators \eqref{eq:L_t} by solving the parabolic Dirichlet problem through the hitting distribution of the homogenized right process obtained in \Cref{thm:main0}-\Cref{thm:main1}; see \Cref{coro:BKE}.

\medskip
$\circ\quad$ We solve the time-dependent linear Fokker-Planck equation perturbed by a nonlocal operator, and simultaneously the corresponding martingale problem, through a method of adding bounded jumps to a right process; see \Cref{ss:jumps}.

\medskip
As already mentioned, the above results are valid for solutions to linear Fokker-Planck equations.
By the linearization technique (in the spirit of \cite{BaRo20}), we then lift these results to non-linear Fokker-Planck equations.
More precisely, we achieve the second fundamental goal of this paper:

\medskip
$\bullet\quad$ {\it In \Cref{s:nonlinear_superposition}, see \Cref{thm:MVE}, we construct a right processes
$X$ on a subset $E\subset [0,\infty)\times \mathbb{R}^d$ with $\mu_t(\{x:(t,x)\in E\})=1, t\geq 0$, in such a way that the family of one-dimensional time marginals of the path-law $\eta$ of the spatial component of $X$ started at $\delta_0\otimes\mu_0$, coincides with the given solution $(\mu_t)_{t\geq 0}$ of the nonlinear Fokker-Planck equation  (thus realizing McKean vision from 
\cite{McK66}),  
or equivalently, $\eta$ coincides with the path-law of the weak solution to the McKean-Vlasov SDE \eqref{eq:MV SDE}.
In particular, we derive the strong Markov property for the weak solution to the McKean-Vlasov SDE, which, as mentioned before, was an open problem even for the linear case.}

\medskip
In the same \Cref{s:nonlinear_superposition} we achieve two more goals:

\medskip
$\circ\quad$ We construct a Choquet capacity for the solution to the nonlinear Fokker-Planck equation through the corresponding hitting distribution of the solution to the McKean-Vlasov SDE \eqref{eq:MV SDE}; see \Cref{coro: capacity}.

\medskip
$\circ\quad$ As our main application in the nonlinear case, we investigate sufficient explicit conditions on the coefficients of a class of nonlinear Fokker-Planck equations, more precisely, generalized porous media equations, for which all the hypotheses in \Cref{thm:MVE} are fulfilled and hence our theory fully applies; see \Cref{ss:example_nonlinear}.

\medskip
Many techniques in constructing and manipulating the right processes mentioned above are of potential theoretic nature.
Thus, to make this work as self-contained as possible:

\medskip
$\bullet\quad$ {\it In \Cref{Appendix} we give a concise introduction to the theory of probabilistic potential theory, with special emphasis on those tools that are fundamentally used in order to derive the main results from \Cref{s:reg_superposition_linear} and \Cref{s:consequences}. 
A reader who is interested in the technical details of this work is encouraged to have a preliminary look at \Cref{Appendix} before diving into the details of the main results.}

\medskip
$\circ\quad$ We would like to point out that, besides well known results, in \Cref{Appendix} we also include new results on right processes that are of general nature and which are, as well, needed in the main part concerning Fokker-Planck equations, namely \Cref{s:reg_superposition_linear} and \Cref{s:consequences}; see e.g. \Cref{lem:B_is_finely_feller}, \Cref{prop:B-boundary}, \Cref{prop:continuous_modification_special}, \Cref{Borel-measurab}, \Cref{prop:martingale_absorbing}, as well as those from \Cref{ss:adding jumps}.

\medskip
The paper is structured as follows:

In \Cref{s:reg_superposition_linear} we deal with time-dependent linear Fokker-Planck equations. 
In \Cref{ss:superposition_simple_markov} we briefly recall the well-known superposition principle, a restricted uniqueness assumption and its relation with the "simple" Markov property.
In \Cref{ss:more regularity} we place our problem of constructing more regular Markov processes in the context given by generalized Dirichlet forms, recall some fundamental notions and results, and discuss two problems which have remained unsolved up to now (OP1 and OP2), and which are the main obstacles that we overcome in this work.
In \Cref{ss:main results} we present the main results of \Cref{s:reg_superposition_linear}.
\Cref{ss:proof} is entirely devoted to the proof of these main results, namely \Cref{thm:main0} and \Cref{thm:main0}, for which we allocate separate subsections. 

In \Cref{s:consequences} we use the right processes constructed in \Cref{s:reg_superposition_linear} in order to deduce new results about time-dependent linear PDEs: Construction of fundamental flow solutions, solving parabolic Dirichlet problems by probabilistic means, deriving maximal tail estimates for the path-space law obtained by the superposition of the fundamental flow solution to the time-dependent linear Fokker-Planck equation \eqref{eq:LFP} by means of Lyapunov functions, and solving Fokker-Planck equations perturbed by nonlocal operators.

In \Cref{s:construction} we present several results of independent interest establishing that the square root of the density of solutions to Fokker–Planck equations with non-degenerate diffusion coefficients belongs locally to $H^1$. 
This regularity result plays a central role in the present work, as it shows that one of our main assumptions, namely \eqref{eq:H_sqrt}, is in fact satisfied under very mild conditions; it significantly generalizes \cite[Theorem 7.4.1]{BoKrRoSh22}.

In \Cref{s:nonlinear_superposition} we extend to nonlinear Fokker-Planck equations the results presented above for the linear case. 
We construct Choquet capacities for such nonlinear equations by probabilistic means and we investigate explicit sufficient conditions on the coefficients for a class of generalized porous media equations to which our theory applies.

In \Cref{Appendix} we offer a concise introduction to the theory of right processes and their potential theory, and we also include here some new results with full proofs, which we need in the previous sections.

\section[Regularization of the superposition principle for linear FPE]{Regularization of the superposition principle for linear time-dependent Fokker-Planck equations}
\label{s:reg_superposition_linear}

Let us consider the time-dependent diffusion operator
\begin{equation}\label{eq:L_t}
    {\sf L}_tf(t,x):=\sum\limits_{i=1}^d b_i(t,x)\frac{\partial}{\partial x_i} f(t,x) + \frac{1}{2}\sum\limits_{i,j=1}^d a_{i,j}(t,x)\frac{\partial^2}{\partial x_i \partial x_j} f(t,x),  
\end{equation}
for $(t,x) \in (0,\infty)\times \mathbb{R}^d$, which acts on Borel measurable functions $f:(0,\infty)\times \mathbb{R}^d\rightarrow \mathbb{R}$ which are twice continuously differentiable in the second argument, whilst
\begin{equation}\label{eq:b_sigma}
    b=(b_i)_{1\leq i \leq d}:[0,\infty)\times \mathbb{R}^d \rightarrow \mathbb{R}^d, \quad a=(a_{ij})_{1\leq i,j\leq d}:[0,\infty)\times\mathbb{R}^d\rightarrow \mathbb{R}^{d\times d}
\end{equation}
are Borel measurable coefficients, such that $a(t,x)$ is symmetric and non-negative definite for all $(t,x)\in [0,\infty)\times \mathbb{R}^d$.

Throughout, for a measurable space $(E,\mathcal{B})$ we denote by $\mathcal{M}(E)$ the set of all finite (non-negative) measures on $E$, whilst by $\mathcal{P}(E)$ we denote the subset of $\mathcal{M}(E)$ containing all probability measures on $E$.
Furthermore, by $p\mathcal{B}$ (resp. $b\mathcal{B}$) we denote the space of all non-negative (resp. bounded), real valued, and $\mathcal{B}$-measurable functions defined on $E$; $bp\mathcal{B}:=p\mathcal{B}\cap b\mathcal{B}$.
If $E$ is a topological spaces then by $\mathcal{B}(E)$ we denote the Borel $\sigma$-algebra, whilst by $C_b(E)$ we denote the space of all real-valued, bounded, and continuous functions defined on $E$.
Also, by $C_c(E)$ we denote the subset of $C_b(E)$ consisting of functions which are compactly supported.

For $\mu \in \mathcal{M}(\mathbb{R}^d)$ and $f\in p\mathcal{B}$ we often use the notation
\begin{equation*}
    \mu(f):=\int_E f \;d\mu.
\end{equation*}

Let $(\mu_n)_{n}\subset \mathcal{M}(\mathbb{R}^d)$ and $\mu \in \mathcal{M}(\mathbb{R}^d)$. 
We say that $(\mu_n)_{n\geq 1}$ converges weakly (resp. vaguely) to $\mu$ if
\begin{equation*}
    \lim_n\mu_n(f)=\mu(f) \quad \mbox{for all } f\in C_b(\mathbb{R}^d) \quad (\mbox{resp. } C_c(\mathbb{R}^d)).
\end{equation*}
Throughout, we use $B(a,r)$ to denote the Euclidean ball in $\mathbb{R}^d$ centered at $a\in \mathbb{R}^d$ and with radius $r>0$.

\begin{defi} \label{def:solution}
Let $0\leq s< T<\infty$. 
\begin{enumerate}
    \item[(i)] A family $\left(\mu_t\right)_{t\in (s,T)}\subset \mathcal{M}(\mathbb{R}^d)$ is called a (weak, or distributional) solution on $(s,T)$ to the linear Fokker-Planck equation  
    \begin{equation}\label{eq:LFPshort}
        \frac{d}{dt} \mu_t = {\sf L}^\ast_t \mu_t, \quad t\in (s,T),
    \end{equation}
    if $\left(\mu_t\right)_{ t\in (s,T)}$ is a Borel curve in $\mathcal{M}(\mathbb{R}^d)$ endowed with the Borel $\sigma$-algebra $\mathcal{B}\left(\mathcal{M}(\mathbb{R}^d)\right)$ associated to the topology of weak convergence,
    \begin{equation}\label{eq:a,b}
        \int_s^T\int_{B(0,R)} \|b(t,x)\|+\|a(t,x)\| \; \mu_t(dx) dt <\infty, \quad R>0,
    \end{equation}
    and
    \begin{equation}\label{eq:LFPmu}
        \int_s^T\int_{\mathbb{R}^d} \left[\frac{d}{dt}f(t,x)+ {\sf L}_t f(t,x)  \right] \;\mu_t(dx) dt = 0, \quad f\in C_c^\infty \left((s,T)\times \mathbb{R}^d\right).
    \end{equation} 
\item[(ii)] We say that $\left(\mu_t\right)_{t\in [s,T)}\subset \mathcal{M}(\mathbb{R}^d)$ is a solution to the linear Fokker-Planck equation \eqref{eq:LFPshort} (with initial condition $\mu_s$) if $\left(\mu_t\right)_{t\in (s,T)}$ is a solution on $(s,T)$ as in (i), and $\lim\limits_{t\searrow s} \mu_t=\mu_s$ weakly.
\item[(iii)] We say that $\left(\mu_t\right)_{t\in [s,T)}\subset \mathcal{M}(\mathbb{R}^d)$ is a weakly (resp. vaguely) continuous solution to the linear Fokker-Planck equation \eqref{eq:LFPshort} if it is a solution on $(s,T)$ in the sense of (i), and the curve of measures $[s,T)\ni t\mapsto \mu_t\in \mathcal{M}(\mathbb{R}^d)$ is (resp. vaguely) continuous; this definition is obviously extended when $(\mu_t)_{t\in [s,T)]}$ is merely right-continuous. 
\end{enumerate}
\end{defi}

\begin{rem}\label{rem:equiv}
Let $\left(\mu_t\right)_{t\in [s,T)}\subset \mathcal{M}(\mathbb{R}^d)$ be a vaguely (right-)continuous solution to the linear Fokker-Planck equation \eqref{eq:LFPshort}. 
For $0\leq s<T<\infty$, if we take $f$ of the form $f(r,x)=g(r)h(x)$ with $g\in C_c^\infty((s,T))$ and $h\in C_c^\infty(\mathbb{R^d})$,
then using first Fubini theorem and then integration by parts, \eqref{eq:LFPmu} gives
\begin{equation}\label{eq:integration_by_parts}
-\int_s^T g'(r) \int_{\mathbb{R}^d}h \;d\mu_r dr  =   \int_s^T g(r) \int_{\mathbb{R}^d}{\sf L}_t h \;d\mu_rdr.
\end{equation}
Let $s<s'<t'<T$, $g_n\in C_c^\infty(s,T),n\geq 1$ such that $\lim\limits_n g_n=1_{[s',t']}$ a.e., and $\lim\limits_n g_n'=\delta_{t'}-\delta_{s'}$ in the sense of weak convergence of measures; note that this can be done e.g. by convoluting $1_{[s',t']}$ with smooth and compactly supported mollifiers $\rho_n$, namely taking $g_n:=\rho_n\ast 1_{[s',t']}, n\geq 1$. 
Thus, replacing $g$ with $g_n$ in \eqref{eq:integration_by_parts}, using that $r\mapsto\int_{\mathbb{R}^d} h \;dm_r$ is bounded and continuous, and letting $n$ tend to infinity, we obtain
\begin{equation} \label{eq:cauchy}
\int_{\mathbb{R}^d} h \;d\mu_{t'}-\int_{\mathbb{R}^d} h \;d\mu_{s'}=\int_{s'}^{t'} \int_{\mathbb{R}^d}{\sf L}_rh\; d\mu_rdr, \quad s<s'<t'<T, h\in C_c^\infty(\mathbb{R}^d).    
\end{equation}
It is easy to see that one can also pass from \eqref{eq:cauchy} to \eqref{eq:LFPmu}, hence the two relations are in fact equivalent.    
\end{rem}

\subsection{The superposition principle and the "simple" Markov property}\label{ss:superposition_simple_markov}
For $0\leq s<T<\infty$ we define
\begin{equation*}
    \Pi_s(t,\cdot):C([s,T];\mathbb{R}^d)\rightarrow \mathbb{R}^d, \quad \Pi_s(t,\omega):=\omega(t), \quad \omega\in C([s,T];\mathbb{R}^d), t\in [s,T].
\end{equation*}

As usual, we endow $C([s,T];\mathbb{R}^d)$ with the Borel $\sigma$-algebra rendered by the topology of uniform convergence, which coincides with the $\sigma$-algebra generated by $\left(\Pi_s(t,\cdot)\right)_{t\in[s,T]}$.

If $\eta\in \mathcal{P}(C([s,T];\mathbb{R}^d))$, then by $\eta_t,t\in [s,T]$ we denote the one-dimensional time marginals, namely
\begin{equation*}
    \eta_t:=\eta\circ \left(\Pi_s(t,\cdot)\right)^{-1}\in \mathcal{P}(\mathbb{R}^d), \quad t\in [s,T].
\end{equation*}

\begin{defi}\label{defi:canonicalMP}
    Let $0\leq s<T<\infty$. 
    A probability $\eta\in \mathcal{P}\left(C([s,T];\mathbb{R}^d))\right)$ is called a solution to the canonical martingale problem (canonical MP) associated to $\left({\sf L}_t\right)_{t\in [s,T]}$, with initial distribution $\nu_s\in \mathcal{P}(\mathbb{R}^d)$ if
    \begin{equation*}
        \int_s^T\int_{B(0,R)} \|b(t,x)\|+\|a(t,x)\| \; \eta_t(dx) dt <\infty, \quad R>0,
    \end{equation*}
    and for all $f\in C_c^{1,2}((s,T)\times \mathbb{R}^d)$ the process $(M(t))_{t\in [s,T]}$ given by
    \begin{equation*}
        M(t):=f(t,\Pi_s(t))-f(s,\Pi_s(s))-\int_s^t\overline{{\sf L}}f(r,\Pi_s(r))\;dr, \quad t\in [s,T]
    \end{equation*}
    is an $\left(\mathcal{F}_s(t)\right)$-martingale with respect to $\eta$, where $\mathcal{F}_s(t):=\sigma\left(\Pi_s(r),s\leq r\leq t\right), t\in [s,T]$.
\end{defi}

\begin{rem}\label{rem:mp_LFP} 
Just by taking expectations, it follows immediately that if $\eta$ is a solution to the martingale problem as in \Cref{defi:canonicalMP}, then $(\eta_t)_{t\in [s,T]}$ is a weakly continuous solution to the linear Fokker-Planck equation \eqref{eq:LFP}. 
The following highly nontrivial result, known as the superposition principle, gives the converse implication.
\end{rem}

\begin{thm}[The superposition principle; \cite{BoRoSh21}, \cite{Tr16}]\label[thm]{rem:Trevisan_superposition}
Let $\left(\mu_t\right)_{t\in [s,T]}$ be a weakly continuous solution to \eqref{eq:LFPshort}, and suppose further that
\begin{equation}\label{eq:Bo_stronger}
    \int_s^T\int_{B(0,R)} \frac{|\langle b(t,x),x\rangle|+\|a(t,x)\|}{(1+\|x\|)^2} \; \mu_t(dx) dt <\infty.
\end{equation}
Then, according to \cite[Theorem 1.1]{BoRoSh21}, there exists $\eta\in \mathcal{P}\left(C\left([s, T];\mathbb{R}^d\right)\right)$ a solution to the canonical martingale problem associated to $({\sf L}_t)_{t\in [s,T]}$ on $[s,T]$ such that the one-dimensional time marginals $(\eta_t)_{t\in[s,T]}$ satisfy $\eta_t=\mu_t,t\in[s,T]$.
Moreover, by \cite[Proposition 2.8]{Tr16}, for each $x\in \mathbb{R}^d$ there exists a probability $\eta^{x}\in \mathcal{P}\left(C\left([s, T];\mathbb{R}^d\right)\right)$ such that
\begin{enumerate}[(i)]
    \item For every $A\in \mathcal{B}\left(C\left([s,T];\mathbb{R}^d\right)\right)$
    \begin{equation*}
       \mbox{ the mapping } \mathbb{R}^d\ni x\mapsto \eta^{x}(A) \mbox{ is Borel measurable and } \eta(A)=\int_{\mathbb{R}^d}\eta^{x}(A) \;\mu_s(dx).  
    \end{equation*}
    \item For $\mu_s$-a.e. $x\in \mathbb{R}^d$, $\eta^{x}$ is a solution to the martingale problem associated to $({\sf L}_t)_{t\in [s,T]}$ on $[s,T]$, with initial distribution $\delta_x$. 
\end{enumerate}
As a consequence, if $\nu_s\in \mathcal{P}(\mathbb{R}^d)$ satisfies $\nu_s\leq c\mu_s$ for some constant $c\in (0,\infty)$, then 
\begin{equation*}
    \eta^{\nu_s}(\cdot):=\int_{\mathbb{R}^d}\eta^{x}(\cdot)\;\nu_s(dx)
\end{equation*}
is also a solution to the martingale problem associated to $({\sf L}_t)_{t\in [s,T]}$ on $[s,T]$, with initial distribution $\nu_s$. 
\end{thm}

\begin{rem}
We point out that in \cite{Tr16}, the superposition principle is derived with \eqref{eq:Bo_stronger} replaced by the stronger condition
\begin{equation}\label{eq:Tr_stronger}
    \int_{(s,T)\times \mathbb{R}^d} \|b(t,x)\|+\|a(t,x)\| \; \mu_t(dx) dt <\infty.
\end{equation}
However, there are relevant examples for which \eqref{eq:Tr_stronger} is not fulfilled, but for which the relaxed condition \eqref{eq:Bo_stronger} is valid.
Other key results in \cite{Tr16} remain valid if we assume \eqref{eq:Bo_stronger} instead of \eqref{eq:Tr_stronger}, this being the case of \cite[Lemma 2.12]{Tr16} which we shall use below in the proof of \Cref{prop:uniqueness_MP}.
\end{rem}

Let $N>0$ and $(\mu_t)_{t\in [0,N]}\subset \mathcal{P}(\mathbb{R}^d)$ be a given weakly continuous solution to the linear Fokker-Planck equation \eqref{eq:LFPshort}. 
With respect to such a fixed solution, we consider the following uniqueness hypothesis:

\medskip
\noindent{$\mathbf H_{!
\mu}$}
For every $0\leq s<N<\infty$ and $c>0$ we have that whenever $\left(\nu_t\right)_{t\in [s,N]},\left(\tilde{\nu}_t\right)_{t\in [s,N]}\subset \mathcal{P}(\mathbb{R}^d)$ are two weakly continuous solutions to the linear Fokker-Planck equation \eqref{eq:LFPshort}, 
we have:
\begin{equation*}
     \mbox{If}\quad \nu_t,\tilde{\nu}_t\leq c\mu_t, t\in [s,N],
    \mbox{ and } \nu_s=\tilde{\nu}_s, \quad \mbox{then} \quad \left(\nu_t\right)_{t\in [s,N]}=\left(\tilde{\nu}_t\right)_{t\in [s,N]}.
\end{equation*}

\begin{prop}[Well-posedeness of the canonical MP] \label{prop:uniqueness_MP}
    Let $N>0$ and $(\mu_t)_{t\in [0,N]}\subset \mathcal{P}(\mathbb{R}^d)$ be a given weakly continuous solution to the linear Fokker-Planck equation \eqref{eq:LFPshort}.
    If $\mathbf H_{!\mu}$ is satisfied, $0\leq s<N$, $c\in (0,\infty)$, and $(\nu_t)_{t\in [s,N]}\subset \mathcal{P}(\mathbb{R}^d)$ is a weakly continuous solution to the linear Fokker-Planck equation \eqref{eq:LFPshort} on $[s,N)$ satisfying $\nu_t\leq c\mu_t, t\in [s,N]$, then there exists a unique $\eta\in \mathcal{P}(C([s,N];\mathbb{R}^d))$ which is a solution to the canonical MP associated to $\left({\sf L}_t\right)_{t\in [s,N]}$ and such that $\eta_t=\nu_t$ for every $t\in [s,N]$.
\end{prop}
\begin{proof}
    The statement follows by \cite[Lemma 2.12]{Tr16}, choosing $T=N$ and
    \begin{equation*}
        \begin{split}
            \mathcal{R}_{[s,T]}:=\bigg\{\nu=\left(\nu_t\right)_{t\in [s,T]}\subset \mathcal{P}(\mathbb{R}^d) : \nu \mbox{ is a weakly continuous solution to } \eqref{eq:LFPshort} \mbox{ and }\\
            \mbox{there exists } c<\infty \mbox{ such that } \nu_t\leq c\mu_t, t\in [s,T] \bigg\}, \quad 0\leq s<T.
        \end{split}
    \end{equation*}
\end{proof}

Another result from \cite{Tr16} that remains valid if we assume \eqref{eq:Bo_stronger} instead of \eqref{eq:Tr_stronger} is the following:
\begin{prop}[{\cite[Remark 2.3]{Tr16}}]\label{rem:Trevisan_narrow}
Let $(\mu_t)_{t\in (s,T)}$ be a solution to the linear Fokker-Planck equation as in \Cref{def:solution}, which in addition satisfies 
\begin{equation*}
    (\mu_t)_{t\in (s,T)}\subset \mathcal{P}(\mathbb{R}^d) \quad \mbox{as well as} \quad   \eqref{eq:Bo_stronger}.
\end{equation*}
Then $(\mu_t)_{t\in (s,T)}$ admits a unique $dt$-version $(\tilde{\mu}_t)_{t\in [s,T]}$ which is weakly continuous.
\end{prop}
However, let us point that the proof of \Cref{rem:Trevisan_narrow} requires a light adaptation. 
More precisely, recall first that \cite[Remark 2.3]{Tr16} is fundamentally \cite[Lemma 8.1.2]{AmNiGi05}. 
The latter result, remains nevertheless valid if we replace the assumption \eqref{eq:Tr_stronger} from \cite{Tr16} with the relaxed one \eqref{eq:Bo_stronger} from \cite{BoRoSh21}, with the mention that, under \eqref{eq:Bo_stronger}, the tightness of $(\mu_t)_{t\in [s,T]}$ required in \cite{AmNiGi05} is ensured by \cite[Proposition 2.2 and Proof of Theorem 1.1]{BoRoSh21}; also, note that \cite[Proposition 2.2]{BoRoSh21} is valid without the assumption that $(\mu_t)_t$ is weakly continuous, since the latter was used only for applying the Gronwall's lemma for continuous functions, but the Gronwall's lemma is valid for merely measurable functions as well, see e.g. \cite[Theorem 5.1]{EtKu86}, and which can be employed in \cite[Proposition 2.2]{BoRoSh21}. 

\paragraph{The "simple" Markov property}
Let $(\mu_t)_{t\in [0,T]}$ be a weakly continuous solution to \eqref{eq:LFPshort} such that \eqref{eq:Bo_stronger} and $\mathbf H_{!\mu}$ hold.
Then, one can easily deduce: 

{\it The solution $\eta\in \mathcal{P}(C([0,T];\mathbb{R}^d))$ provided by \Cref{rem:Trevisan_superposition} and \Cref{prop:uniqueness_MP} has the simple Markov property, namely for every $f\in b\mathcal{B}(\mathbb{R}^d)$ we have
\begin{equation}\label{eq:simpleM}
    \mathbb{E}_\eta\left\{f(\Pi_{0}(t+s)) \;|\; \mathcal{F}_0(t)\right\}=\mathbb{E}_\eta\left\{f(\Pi_0(t+s)) \;|\; \Pi_0(t)\right\},\; \eta\mbox{-a.s.} \quad t\in [0,T], s\in[0,T-t],
\end{equation}
where $\mathbb{E}_\eta$ denotes the expectation with respect to the probability $\eta$ on $C([0,T];\mathbb{R}^d)$.}

Recall that the filtration $\left(\mathcal{F}_0(t)\right)_{t\in [0,T]}$ is not right continuous, and there is a priori no guarantee that the simple Markov property \eqref{eq:simpleM} is preserved relative to $\left(\mathcal{F}_0(t+)\right)_{t\in [0,T]}$.

For the reader convenience let us rapidly prove \eqref{eq:simpleM}, following a standard argument from e.g. \cite[Theorem 4.2]{EtKu86}:
Let $t\in[0,T)$ and $F\in \mathcal{F}_0(t)$ with $\eta(F)>0$ be fixed and consider the following probabilities
\begin{equation*}
    \eta^{(1)}(A)=\frac{\mathbb{E}_\eta\left\{1_F \mathbb{E}[A \;|\; \mathcal{F}_0(t)]\right\}}{\eta(F)}, \quad  \eta^{(2)}(A)=\frac{\mathbb{E}_\eta\left\{1_F \mathbb{E}[A \;|\; \Pi_0(t)]\right\}}{\eta(F)},\quad A\in \mathcal{F}_0(T).
\end{equation*}
Then, it follows precisely as in \cite[Theorem 4.2]{EtKu86} that the probability measures 
\begin{equation*}
    \eta^{t,(i)}:=\eta^{(i)}\circ \left(\Pi_0(t+\cdot)\right)^{-1}\in \mathcal{P}(C([t,T];\mathbb{R}^d)), \quad i\in \{1,2\},
\end{equation*}
are both solutions to the martingale problem associated to $({\sf L}_r)_{r\in[t,T]}$, having the distribution at time $t$
\begin{equation*}
    \eta^{t,(1)}_t(B)=\frac{\eta(\Pi_0(t)^{-1}(B) \cap F)}{\eta(F)}=\eta^{t,(2)}_t(B), \quad B\in \mathcal{B}(\mathbb{R}^d).
\end{equation*}
By \Cref{rem:mp_LFP} we get that $(\eta^{t,(i)}_{s})_{s\in [t,T]}$ is weakly continuous solution to the linear Fokker-Planck equation \eqref{eq:LFPshort}, $i\in \{1,2\}$. 
Moreover, we have
\begin{equation*}
    \eta^{t,(i)}_{s}:=\eta^{(i)}\circ \left({\Pi_0(s)}\right)^{-1}\leq \frac{\eta\circ \left({\Pi_0(s)}\right)^{-1}}{\eta(F)}=\frac{1}{\eta(F)}\mu_{s}, \quad r\in [t,T], \quad i\in \{1,2\}.
\end{equation*}
Therefore, by $\mathbf H_{!\mu}$ we deduce that $\eta_{s}^{t,(1)}
    =\eta_{s}^{t,(2)}$, hence for $B\in\mathcal{B}(\mathbb{R}^d)$
\begin{align*}
    \mathbb{E}_\eta\left\{1_F \mathbb{E}[\Pi_{t+r}\in B \;|\; \mathcal{F}_0(t)]\right\}
    &=\eta(F)\eta_{t+r}^{t,(1)}(B)
    =\eta(F)\eta_{t+r}^{t,(2)}(B)\\
    &=\mathbb{E}_\eta\left\{1_F \mathbb{E}[\Pi_{t+r}\in B \;|\; \Pi_t]\right\}, \;r\in [0,T-t],
\end{align*}
from which the simple Markov property follows.

\begin{rem}Note that in \cite[Theorem 4.2]{EtKu86}, the simple Markov property is deduced under the much stronger assumption requiring that the martingale problem has at most one solution for every given initial distribution which even implies that the strong Markov property holds.
We stress that above we proved the simple Markov property \eqref{eq:simpleM} under assumption $\mathbf H_{!\mu}$ which is much weaker than the mentioned uniqueness assumption in \cite[Theorem 4.2]{EtKu86}.
However, the strong Markov property can no longer be deduced  under this weaker uniqueness assumption by the arguments in \cite[Theorem 4.2]{EtKu86}.
In this paper we are fundamentally interested in proving the strong Markov property as well as other useful regularity properties for the solution the martingale problem associated to the linear Fokker-Planck equation \eqref{eq:LFPshort}, not assuming we have uniqueness for the martingale problem under any initial condition. 
Moreover, we extend these results to the non-linear case as well. 
\end{rem} 

\subsection{Towards more regularity through generalized Dirichlet forms}\label{ss:more regularity} 

In this section we approach the superposition principle from the perspective of generalized Dirichlet forms for time-dependent operators on $L^1$, as developed in \cite{St99a}.
As we shall point out, this approach has the key advantage of providing a "full" Hunt (in particular strong) Markov process which is associated to the solution of the linear Fokker-Planck equation. However, two important open problems show up, which we explain below in detail, and systematically solve afterwards.

Throughout we are given a solution $\mu:=(\mu_t)_{t\geq 0}$ to the linear Fokker-Planck equation on $[0, \infty)$, i.e. on each $[0,N), N>0$, which consists of absolutely continuous probability measures on $\mathbb{R}^d$, namely: 
There exists $u:[0,\infty)\times \mathbb{R}^d\rightarrow \mathbb{R}_+$ jointly measurable, such that
\begin{equation}\label{eq:exists_u}
  \mu_t:= u(t,x) dx, \quad \mu_t(\mathbb{R}^d)=1,\quad t\geq 0,  
\end{equation}
\begin{equation}\label{eq:bau}
    \int_0^N\int_{B(0,R)} \left(\|b(t,x)\|+\|a(t,x)\|\right) u(t,x) \; dx dt <\infty, \quad N>0,R>0,
\end{equation}
and
\begin{equation}\label{eq:LFP}
    \int_0^\infty\int_{ \mathbb{R}^d} \left[\frac{d}{dt}f(t,x)+ {\sf L}_t f(t,x)  \right] u(t,x) \;dx dt = 0, \quad f\in C_c^\infty \left((0,\infty)\times \mathbb{R}^d\right).
\end{equation}

Further, we consider the measures
\begin{equation*}
    \overline{\mu}:=u(t,x)\; dt\otimes dx \mbox{ on } [0,\infty)\times\mathbb{R}^d, \quad  \overline{\mu}_N:=u(t,x)\; dt \otimes dx \mbox{ on } [0,N)\times\mathbb{R}^d, N>0,
\end{equation*}
so that, by setting 
\begin{equation*}
\overline{\sf L}f(t,x):=\frac{d}{dt}f(t,x)+ {\sf L}_tf(t,x), \quad f\in C_c^{\infty}((0,\infty)\times \mathbb{R}^d), 
\end{equation*}
\eqref{eq:LFP} rewrites as
\begin{equation*}
    \int_0^\infty\int_{ \mathbb{R}^d} \overline{\sf L}f \; d\overline{\mu}=0, \quad f\in C_c^{\infty}((0,\infty)\times \mathbb{R}^d). 
\end{equation*}
That is, $\overline{\mu}$ is an {\it infinitesimally invariant} $\sigma$-finite measure for $\overline{\sf L}$.

Below, we will fix such a measure $\overline{\mu}$ having certain properties, and call it a {\it reference measure} and its density a {\it reference density}.

Thus, we now perfectly fit in the framework of \cite{St99a}, which we further follow and from which we take over the next  additional assumptions on the coefficients and on $\overline{\mu}$:

\begin{enumerate}
    \item[$\mathbf{H_a}$] For every $N,r>0$ we have
    \begin{equation*}
        \partial_{x_j} a_{i,j}\in L^2\left([0,N)\times B(0,r); \overline{\mu}_N\right) , \quad 1\leq i,j\leq d,
    \end{equation*} 
and there exists a constant $C=C(r,N)>0$ such that
    \begin{equation*}
        C^{-1}\|\xi\|^2\leq \langle a(t,x) \xi, \xi\rangle \leq C \|\xi\|^2, \quad \xi\in \mathbb{R}^d, (t,x)\in [0,N)\times B(0,r).
    \end{equation*}
    \item[$\mathbf{H_b}$] For every $N,r>0$ we have
    \begin{equation*}
        b\in L^2\left([0,N)\times B(0,r); \overline{\mu}_N\right).
    \end{equation*}
    \item[$\mathbf{H^{\sqrt{u}}}$]
    For every $h\in C_c^\infty(\mathbb{R}^d)$ and $N>0$ we have
    \begin{equation}\label{eq:H_sqrt}
        uh\in L^\infty([0,N)\times\mathbb{R}^d), \quad \mbox{and}\quad \sqrt{u}h\in L^2\left([0,N); H^1(\mathbb{R}^d)\right).
    \end{equation}
\end{enumerate}
{\it We say that condition $\mathbf{H_{a,b}^{\sqrt{u}}}$ holds if all the above three conditions hold.}

According to \cite[Theorem 1.7, Remark 1.18, Theorem 1.9, Proposition 1.10, Remark 1.11, and Remark 1.12]{St99a} we have the following collection of fundamental results:
\begin{thm}[cf. \cite{St99a}]\label{thm:Stannat}
    Let $u$ be a reference density (i.e. satisfying \eqref{eq:exists_u}-\eqref{eq:LFP}) such that condition $\mathbf{H_{a,b}^{\sqrt{u}}}$ is fulfilled.
    Then the following assertions hold for each $N>0$:
    \begin{enumerate}
        \item[(i)] There exists a maximal extension $\left(\overline{\sf L}_1^N,{\sf{D}(\overline{\sf{L}}_1^N)}\right)$ of $\left(\overline{\sf L},C_c^{\infty}((0,N)\times \mathbb{R}^d)\right)$ on $L^1\left([0,N)\times \mathbb{R}^d;\overline{\mu}_N\right)$ satisfying the following properties:
        \begin{enumerate}
            \item[(i.1)] It generates a $C_0$-semigroup $\left(\overline{T}_t^N\right)_{t\geq 0}$ on $L^1\left([0,N)\times \mathbb{R}^d;\overline{\mu}_N\right)$, which can be extended (keeping the same notation) as a $C_0$-semigroup of contractions on every $L^p\left([0,N)\times \mathbb{R}^d;\overline{\mu}_N\right)$, $p\in[1,\infty)$, whose corresponding generator is further denoted by $\left(\overline{\sf L}_p^N,{\sf{D}(\overline{\sf{L}}_p^N)}\right)$. 
            \item[(i.2)] $\left(\overline{T}_t^N\right)_{t\geq 0}$ is of evolution type, namely for every $t\geq 0$
            \begin{equation}\label{eq:evol}
                \overline{T}_t^Nf(s,x)=1_{[0,N-t)}(s)\overline{T}_t^N \left(f(s+t,\cdot)\right)(s,x), 
            \end{equation}
            a.e. $(s,x)\in [0,N)\times \mathbb{R}^d$ and $f\in L^p\left([0,N)\times \mathbb{R}^d; \overline{\mu}_N\right)$, $1\leq p\leq \infty$.
        \end{enumerate}
        \item[(ii)] There exists a right (hence strong Markov) process (see \Cref{defi 4.4})
        \begin{equation*}
        X^N=\left(\Omega^N,\mathcal{F}^N,\left(\mathcal{F}^N(t)\right)_{t\geq 0},\left(X^N(t)\right)_{t\geq 0},\mathbb{P}_{s,x}^N, (s,x)\in [0,N)\times \mathbb{R}^d\right),
        \end{equation*} 
        with lifetime $\zeta^N$ and shift operators  $\left(\theta^N(t)\right)_{t}$, such that the following are satisfied: 
        \begin{enumerate}
            \item[(ii.1)] Its transition function $\left(\overline{P}_t^N\right)_{t\geq 0}$ and corresponding resolvent $\overline{\mathcal{U}}^N:=\left(\overline{U}_\alpha^N\right)_{\alpha>0}$ are such that for each $f\in b\mathcal{B}([0,N)\times \mathbb{R}^d)$, $\overline{U}_\alpha^Nf$ is a $\overline{\mu}_N$-version of $\left(\alpha-\overline{L}^N_1\right)^{-1}f$, $\alpha>0$. 
            \item[(ii.2)] For $(s,x)\in [0,N)\times \mathbb{R}^d \quad \overline{\mu}_N\mbox{-a.e.}$ we have
            \begin{equation}\label{eq:diffusion}
                \overline{\mathbb{P}}^N_{s,x}\left( [0,\zeta^N)\ni t\mapsto X^N(t) \in [0,N)\times \mathbb{R}^d \mbox{ is continuous }  \right)=1.
            \end{equation}
            \item[(ii.3)] If we set $X^N(t)=(X^N_1(t),X^N_2(t)), t\geq 0$, then for $(s,x)\in [0,N)\times \mathbb{R}^d \quad \overline{\mu}_N\mbox{-a.e.}$ we have
            \begin{equation}\label{eq:uniform_motion_stannat}
            \overline{\mathbb{P}}^N_{s,x}\left( X^N_1(t)\neq t+s; t<\zeta^N  \right)=0.
            \end{equation}
            \item[(ii.4)] For every $N>0$, and every $\overline{\nu}\in \mathcal{P}([0,N)\times \mathbb{R}^d)$ which satisfies $\overline{\nu}\leq c \overline{\mu}_N$ for some constant $c>0$, the process $X^N$ solves the $\mathbb{P}^N_{\overline{\nu}}$-martingale problem associated with  $\left(\overline{\sf L}, C_c^{\infty}((0,\infty)\times \mathbb{R}^d))\right)$ in the sense that
            \begin{equation}\label{eq:mpStannat}
                f(X^N(t\wedge \zeta^N))-f(X^N(0))-\int_0^{t\wedge \zeta^N} \overline{\sf L}f(X^N(s))\; ds, \quad t\geq 0
            \end{equation}
            is an $\left(\mathcal{F}^N(t)\right)$-martingale under $\mathbb{P}^N_{\overline{\nu}}:=\int_{[0,N)\times \mathbb{R}^d} \mathbb{P}^N_{s,x} \; \overline{\nu}(ds,dx)$, for all $f\in C_c^{\infty}((0,\infty)\times \mathbb{R}^d)$.
        \end{enumerate}
    \end{enumerate} 
\end{thm}

\begin{rem}\label{rem:subprob}
Recall that $\left(\overline{P}_t\right)_{t\geq 0}$ is represented by $X^N$ through 
\begin{equation*}
    \overline{P}_tf(s,x)=\mathbb{E}^N_{s,x}\left\{f(X^N(t));t<\zeta^N\right\}, \quad t\geq 0, f\in b\mathcal{B}([0,N)\times \mathbb{R}^d).
\end{equation*}
In particular, $\left(\overline{P}_t\right)_{t\geq 0}$ is a sub-Markovian transition function which is not Markovian, as $\overline{\mathbb{P}}^N_{\overline{\mu}_N}\left(\zeta^N\leq N\right)=1$ by \eqref{eq:uniform_motion_stannat}.
Throughout the paper we shall often consider the {\it distribution} $\mathbb{P}^N_{\overline{\nu}}\circ \left(X(t)\right)^{-1}:=\overline{\nu}\circ \overline{P}_t$ on $[0,N)\times \mathbb{R}^d$, $t\geq 0$, for some $\overline{\nu}\in \mathcal{P}\left([0,N)\times\mathbb{R}^d\right)$. 
However, since $\zeta^N\leq N$, $\mathbb{P}^N_{\overline{\nu}}\circ \left(X(t)\right)^{-1}$ is merely as a sub-probability on $[0,N)\times \mathbb{R}^d$.
We shall later on prove that under some conditions, we have $\zeta^N=N-s$ $\mathbb{P}_{s,\mu_s}$-a.s., which will be crucial in order to show that the second components $X_2^N,N>0$ admit a projective limit with infinite lifetime.
\end{rem}

There are {\it two unsolved problems} that concern the processes $X^N, N>0$ constructed in \Cref{thm:Stannat}, (ii), which we explain below and solve afterwards (see \Cref{thm:main0} and \Cref{thm:main1} for the solutions to OP1 and OP2, respectively):

\medskip
\noindent{\bf Open problem 1 (OP1):} {\it Is there a global Markov process $X$ on $[0,\infty)\times \mathbb{R}^d$ which is obtained as the projective limit of the processes $X^N,N>0$, constructed in \Cref{thm:Stannat}, (ii)?} And if yes, is it conservative, right, or Hunt? This issue is highly nontrivial since it is  related to the uniqueness property of the martingale problems \eqref{eq:mpStannat}, or equivalently, to the corresponding linear Fokker-Planck equation \eqref{eq:LFP}. 
This is a sensitive issue since, in general, uniqueness is expected to hold only in a {\it restrictive} sense, as it will be discussed later on.

\medskip
\noindent{\bf Open problem 2 (OP2):} The second question is technically much more delicate than the first one, and it concerns both the {\it locally} constructed process $X^N$ in \Cref{thm:Stannat}, (ii), i.e. on each time interval $[0,N)$, and the global one $X$ which is the projective limit of $X^N,N>0$, and which exists if {\bf Q1} has an affirmative answer.
It goes as follows: Fix some $N>0$, take $f(t,x)=h(x), (t,x)\in [0,N)\times \mathbb{R}^d$ with $h\in C_c^\infty(\mathbb{R}^d)$, and $\overline{\nu}\in \mathcal{P}([0,N)\times \mathbb{R}^d)$ with $\overline{\nu}\leq c \overline{\mu}_N$ for some $c>0$.
Then, using \Cref{thm:Stannat}, (ii.3) and \eqref{eq:mpStannat}, it follows that
\begin{equation*}
        h(X^N_2(t\wedge \zeta^N))-h(X^N_2(0))-\int_0^{t\wedge \zeta^N} \overline{\sf L}_{s+X^N_1(0)}h(X^N_2(s))\; ds, \quad t\geq 0
\end{equation*}
is a $\left(\mathcal{F}^N(t)\right)$-martingale under $\mathbb{P}^N_{\overline{\nu}}$ which is continuous on $[0,\zeta^N)$.
Thus, if we are allowed to consider $\mathbb{P}^N_{\overline{\nu}}\circ \left(X^N(0)\right)^{-1}= \delta_0\otimes \nu_0$ for some $\nu_0\in \mathcal{P}(\mathbb{R}^d)$, then the {\it spatial} component $X^N_2$ solves the martingale problem associated with the time-dependent linear operator ${\sf L}_t$ on $0\leq t<\zeta^N$, starting at $t=0$ from $\nu_0$. 
But this is not possible since the initial distribution of $\mathbb{P}^N_{\overline{\nu}}\circ \left(X_1^N(0)\right)^{-1}$ is assumed to be absolutely continuous with respect to $dt$. 
This issue is a deep one and in fact ubiquitous in the theory of (generalized) Dirichlet forms. 
It essentially arises from the fact that (the resolvent of the) Markov process $X^N$ is in fact constructed on the state space $[0,N)\times \mathbb{R}^d$, uniquely except on some abstract set $\mathcal{N}$ which is $\overline{\mu}_N$-{\it inessential} (or $\overline{\mu}_N$-negligible and $\overline{\mu}_N$-polar), and on $\mathcal{N}$ it is trivially constructed so that if it starts from $x\in \mathcal{N}$ then it remains in $x$ forever. 
As a consequence, such a Markov process set to be stuck at any point in $\mathcal{N}$ will not solve the martingale problem starting from an initial distribution that charges $\mathcal{N}$. 
In the case of the process $X^N$ furnished by \Cref{thm:Stannat}, due to the fact that the first component is the uniform motion to the right, one can easily see that the set $\{0\}\times \mathbb{R}^d$ is never hit by the process, namely it is polar; hence, the process $X^N$ given by \Cref{thm:Stannat}, uniquely in law on $[0,T)\times \mathbb{R}^d$ except some abstract set $\mathcal{N}$ as above, is in particular not guaranteed to solve the martingale problem \eqref{eq:mpStannat} with initial distribution $\delta_0\times \mu_0$. 
As a conclusion, {\it the second question is if, and under which conditions, one can construct the process $X$ so that it can solve the martingale problem \eqref{eq:mpStannat} starting at $t=0$ from generic distributions of the type $\delta_0\times \nu_0$.} 
By the same reason, this question concerns also the global process $X$ from {\bf Q1}, once we show it can indeed be constructed.

\subsection{The main results: The corresponding right (in particular strong) Markov process}\label{ss:main results}
This subsection is devoted to answering the two fundamental questions raised above.
Throughout we assume that the basic settings \eqref{eq:exists_u}-\eqref{eq:LFP} are in force.
Further, we introduce the following notation: 
\begin{align}\label{eq:A}
    \mathcal{A}_{\leq u}^{s,N}
    &:=\left\{v:[s,N)\times \mathbb{R}^d\rightarrow [0,\infty) \mbox{ measurable } : v\leq u \; dt\otimes dx\mbox{-a.e.} \right\}\nonumber\\
    \mathcal{A}_{\lesssim u}^{s,N}
    &:=\left\{v:[s,N)\times \mathbb{R}^d\rightarrow [0,\infty) \mbox{ measurable } : \exists\; c>0 \mbox{ such that } v\leq c u \; dt\otimes dx\mbox{-a.e.} \right\}
\end{align}

We introduce the following key hypothesis that concerns the given solution $\left(\mu_t:=u(t,x)dx\right)_{t\geq 0}$ of the linear Fokker-Planck equation \eqref{eq:LFPshort}:

\medskip
\noindent{$\mathbf{H_{\sf \leq u}}$} $\left(\mu_t=u(t,x)dx\right)_{t\in [0,\infty)}$ is weakly continuous and for every $0\leq s<N<\infty$ we have that whenever $v\in\mathcal{A}_{\leq u}^{s,N}$ induces a weakly right-continuous and vaguely left-continuous solution $\left(v(t,x)dx\right)_{t\in [s,N)}$ (of sub-probabilities) to the linear Fokker-Planck equation \eqref{eq:LFPshort} on $[s,N)$ with initial condition $v(s,x)dx=\mu_s$, then we have the identification $\left(v(t,x)dx\right)_{t\in [s,N)}=\left(\mu_t\right)_{t\in [s,N)}$.


\medskip
We introduce now a uniqueness hypothesis which is stronger than $\mathbf{H_{\sf \leq u}}$:

\medskip
\noindent{$\mathbf{H_{\sf \lesssim u}}$} 
We have that $\mathbf{H_{\sf \leq u}}$ holds and for every $0\leq s<N<\infty$ and $c>0$ we have that whenever $\left(v(t,x)dx\right)_{t\in [s,N)}$ and $\left(\tilde{v}(t,x)dx\right)_{t\in [s,N)}$ are two weakly continuous curves with values in $\mathcal{P}(\mathbb{R}^d)$ which are solutions to the linear Fokker-Planck equation \eqref{eq:LFPshort} on $[s,N)$ such that
$\tilde{v},v\in\mathcal{A}_{\lesssim u}^{s,N}$ and $v(s,x)dx=\tilde{v}(s,x)dx$, then the two solutions must coincide, namely $\left(v(t,x)dx\right)_{t\in [s,N)}=\left(\tilde{v}(t,x)dx\right)_{t\in [s,N)}$.

\medskip
We have now all the ingredients so give a positive answer to the open question OP1:

\begin{thm}[Answer to OP1] \label{thm:main0}
Let $u$ be a reference density such that that $\mathbf{H_{a,b}^{\sqrt{u}}}$ and $\mathbf{H_{\sf \lesssim u}}$ hold. 
Then the following assertions hold:
    \begin{enumerate}
        \item[(i)] There exists a maximal extension $\left(\overline{\sf L}_1,{\sf{D}(\overline{\sf{L}}_1)}\right)$ of $\left(\overline{\sf L},C_c^{\infty}((0,\infty)\times \mathbb{R}^d))\right)$ on $L^1\left([0,\infty)\times \mathbb{R}^d;\overline{\mu}\right)$. 
        It generates a $C_0$-resolvent $\overline{\mathcal{U}}=\left(\overline{U}_\alpha\right)_{\alpha> 0}$ on $L^1\left([0,\infty)\times \mathbb{R}^d;\overline{\mu}\right)$, which, keeping the same notation, induces a $C_0$-resolvent of Markov contractions $\overline{\mathcal{U}}$ on every $L^p\left([0,\infty)\times \mathbb{R}^d;\overline{\mu}\right)$, and whose corresponding generator is further denoted by $\left(\overline{\sf L}_p,{\sf{D}(\overline{\sf{L}}_p)}\right)$, $p\in[1,\infty)$.
        Moreover, for every $N>0$ we have 
        \begin{equation*}
        \overline{U}_\alpha \left(f1_{[0,N)\times \mathbb{R}^d}\right)|_{[0,N)\times \mathbb{R}^d}=\overline{U}_\alpha^N \left(f|_{[0,N)\times \mathbb{R}^d}\right), \quad f \in L^p([0,\infty)\times \mathbb{R}^d;\overline{\mu}),
        \end{equation*}
        where $\overline{U}_\alpha^N$ is given in \Cref{thm:Stannat}.
        \item[(ii)] There exists a set $E\in \mathcal{B}([0,\infty)\times \mathbb{R}^d)$ and a conservative right (hence strong Markov) process  on $E$
        \begin{equation*}
        X=\left(\Omega,\mathcal{F},\mathcal{F}_t,X(t),\mathbb{P}_{s,x}, (s,x)\in E, t\geq 0\right) \mbox{ such that:}
        \end{equation*} 
        \begin{enumerate}
            \item[(ii.1)] $\overline{\mu}([0,\infty)\times \mathbb{R}^d\setminus E)=0$ and $\mu_s\left(\left\{x\in \mathbb{R}^d, (s,x)\in E\right\}\right)=1$, $s>0$.
            \item[(ii.2)] The transition function of $X$ denoted by $\left(P_t\right)_{t\geq 0}$, and the corresponding resolvent $\mathcal{U}:=\left(U_\alpha\right)_{\alpha>0}$, are such that 
            \begin{enumerate}
                \item[\scalebox{0.7}{$\bullet$}] $P_t\left(f|_E\right)$ is finely continuous for every $f\in C_b([0,\infty)\times \mathbb{R}^d)$
                \item[\scalebox{0.7}{$\bullet$}] $U_\alpha \left(f|_E\right)$ is a finely continuous $\overline{\mu}$-version of $\left(\alpha-\overline{\sf{L}}_p\right)^{-1}f$, $\alpha>0$, $f\in p\mathcal{B}([0,\infty)\times \mathbb{R}^d)\cap L^p([0,\infty)\times\mathbb{R}^d;\overline{\mu})$. 
            \end{enumerate}
            \item[(ii.3)] For $(s,x)\in E$ we have
            \begin{equation*}
                \mathbb{P}_{s,x}\left( [0,\infty)\ni t\mapsto X(t) \in E\subset [0,\infty)\times \mathbb{R}^d \mbox{ is continuous }  \right)=1.
            \end{equation*}
            \item[(ii.4)] If we set $X(t)=(X_1(t),X_2(t)), t\geq 0$, then for $(s,x)\in E$ we have
            \begin{equation*}
            \mathbb{P}_{s,x}\left( X_1(t)= t+s; 0\leq t<\infty  \right)=1.
            \end{equation*}
            Moreover, we have 
            \begin{equation*}
            \mathbb{P}_{s,\mu_s}\circ \left(X_2(t)\right)^{-1}=\mu_{t+s} \quad \mbox{ for every } s> 0 \mbox{ and } t\geq 0.
            \end{equation*}
            \item[(ii.5)]
            For every $(s,x)\in E$ and every $f\in C_c^{2}([0,\infty)\times \mathbb{R}^d)$ we have that
            \begin{equation*}
                \mathbb{E}_{(s,x)}\left\{\int_0^{t} \left|\overline{\sf L}f(X(r))\right|\; dr
               \right\}<\infty, \quad t\geq 0,
            \end{equation*}
            and
            \begin{equation*}
                f(X(t))-f(X(0))-\int_0^{t} \overline{\sf L}f(X(r))\; dr, \quad t\geq 0
            \end{equation*}
            is a continuous $\left(\mathcal{F}(t)\right)$-martingale under $\mathbb{P}_{(s,x)}$.
            \item[(ii.6)]
            For every $s\in (0,\infty)$, and every $\nu_s\in \mathcal{P}(\mathbb{R}^d)$ which satisfies $\nu_s\leq c \mu_s$ for some constant $c>0$, we have that
            \begin{equation*}
            \nu_{t+s}:=\mathbb{P}_{\delta_s\otimes\nu_s}\circ \left(X_2(t)\right)^{-1}, \quad t\geq 0,
            \end{equation*}
            is the unique weakly continuous solution to the LFPE \eqref{eq:LFPshort} on $[s,N)$ which consists of probability measures, belongs to $\mathcal{A}_{\lesssim u}^{s,N}$, and whose initial condition (at time $s$) is $\nu_s$, for all $N>s$.
            Moreover, the (strong Markov) process $X$ is the unique solution to the martingale problem on $[s,N)$ associated with $\left(\overline{\sf L}, C_c^{\infty}((0,\infty)\times \mathbb{R}^d)\right)$, with initial distribution $\delta_s\otimes\nu_s$ and with $(X_2(t))_{t\geq 0}$ having the one-dimensional marginals in $\mathcal{A}_{\lesssim u}^{s,N}$, for every $N>s$.
        \end{enumerate}
    \end{enumerate} 
\end{thm}
\begin{proof}
    The proof is given in \Cref{ss:proof of main0}.
\end{proof}

\begin{rem}\label{rem:MP-D(L)}
    Note that the martingale problem in \Cref{thm:main0}, (ii.5), may be solved for every $f\in b\mathcal{B}(E)$ such that $f\in {\sf D(\overline{L}_1)}$, in the following sense: Let $\tilde{f}$ be a finely continuous $\overline{\mu}$-version of $f$, and $E'\in \mathcal{B}(E)$ such that $E\setminus E'$ is $\overline{\mu}$-inessential and $\tilde{f}$ is bounded on $E'$; $\tilde{f}$ always exists. By considering the restriction of $X$ to $E'$, taking $g$ to be any $\mathcal{B}(E)$-measurable $\overline{\mu}$-version of $\overline{\sf L}_1f$, and using \Cref{lem:martingale 1} and \Cref{prop:martingale_absorbing}, it follows that there exists a second set $E''\in \mathcal{B}(E)$ such that $E\setminus E''$ is $\overline{\mu}$-inessential and
    \begin{equation*}
                \tilde{f}(X(t))-\tilde{f}(X(0))-\int_0^{t} g(X(r))\; dr, \quad t\geq 0
    \end{equation*}
    is a c\`adl\`ag $\left(\mathcal{F}_t\right)$-martingale under $\mathbb{P}_{(s,x)}$ for every $(s,x)\in E''$.
\end{rem}

\medskip
In order to answer positively to the open question OP2, we need additional regularity properties for the resolvent $\mathcal{U}$ constructed in \Cref{thm:main0}, at points of the type $(0,x)$, namely at $s=0$. 
To this end, the key needed assumption turns out to be the following:

\medskip
\noindent{$\mathbf{H^{\sf TV}_0}$} The following hold:
\begin{enumerate}[(i)]
    \item We have
    \begin{equation}\label{eq:TV}
    \lim\limits_{t\to 0} \mu_t=\mu_0 \quad \mbox{ in total variation, or equivalently, } \lim\limits_{t\to 0} u(t,\cdot)=u(0,\cdot) \mbox{ in } L^1(dx).
    \end{equation}
    \item There exists a family $\mathcal{A}_0$ of probability densities which separates the measures on ${\sf supp}(\mu_0)$, with the additional properties that if $v_0\in \mathcal{A}_0$ then there exists a constant $c$ such that $\nu_0:=v_0dx\leq c u(0,x)dx$, and the corresponding solution $\left(\eta^{1,\nu_0}_t\right)_{0\leq t\leq 1}$, given by \Cref{rem:Trevisan_superposition}, (iii), is continuous at $t=0$ in total variation, or equivalently
    \begin{equation*}
        \lim\limits_{t\to 0}\frac{d\eta^{1,\nu_0}_t}{dx}=v_0 \quad \mbox{in } L^1(dx); 
    \end{equation*}
    see also \Cref{eq:uniqueness eta}, below.
\end{enumerate}

\begin{rem}\label{eq:uniqueness eta}
    Assume that $\mathbf{H_{\sf \lesssim u}}$ holds.
    Then any weakly continuous $(\nu_t)_{0\leq t\leq 1}\subset \mathcal{P}(\mathbb{R}^d)$ which is a solution to the linear Fokker-Planck equation \eqref{eq:LFPshort} such that $\nu_t\leq c\mu_t,0\leq t\leq 1$ for some constant $c$, satisfies
    \begin{equation*}
        \nu_t=\eta_t^{1,\nu_0}, \quad 0\leq t\leq 1.
    \end{equation*}
    Thus, under $\mathbf{H^{\sf TV}_0}$ and if $\nu_0\in \mathcal{A}_0$, we have
    \begin{equation*}
        \lim\limits_{t\to 0}\frac{d\nu_t}{dx}=\frac{d\nu_0}{dx} \quad \mbox{in } L^1(dx). 
    \end{equation*}
\end{rem}

\begin{thm}[Answer to OP2] \label{thm:main1}
Let $u$ be a reference density such that $\mathbf{H_{a,b}^{\sqrt{u}}}$, $\mathbf{H_{\sf \lesssim u}}$ hold.
If, in addition, $\mathbf{H^{\sf TV}_0}$ holds, then in assertions (ii.1), (ii.4), and (ii.6) from \Cref{thm:main0} we can include the case $s=0$ as well; more precisely, the three assertions can be respectively replaced with the following stronger ones:
        \begin{enumerate}
            \item[(ii.1')] $\overline{\mu}([0,\infty)\times \mathbb{R}^d)\setminus E=0$ and $\mu_s\left(\left\{x\in \mathbb{R}^d, (s,x)\in E\right\}\right)=1$, $s\geq 0$.
            \item[(ii.4')] If we set $X(t)=(X_1(t),X_2(t)), t\geq 0$, then for $(s,x)\in E$ we have
            \begin{equation*}
            \mathbb{P}_{s,x}\left( X_1(t)= t+s; 0\leq t<\infty  \right)=1.
            \end{equation*}
            Moreover, we have 
            \begin{equation*}
            \mathbb{P}_{s,\mu_s}\circ \left(X_2(t)\right)^{-1}=\mu_{t+s} \quad \mbox{ for every } s\geq 0 \mbox{ and } t\geq 0.
            \end{equation*}
            \item[(ii.6')]
            For every $s\in [0,\infty)$, and every $\nu_s\in \mathcal{P}(\mathbb{R}^d)$ which satisfies $\nu_s\leq c \mu_s$ for some constant $c>0$, we have that
            \begin{equation*}
            \nu_{t+s}:=\mathbb{P}_{\delta_s\otimes\nu_s}\circ \left(X_2(t)\right)^{-1}, \quad t\geq 0,
            \end{equation*}
            is the unique weakly continuous solution to the LFPE \eqref{eq:LFPshort} on $[s,N)$ which consists of probability measures, belongs to $\mathcal{A}_{\lesssim u}^{s,N}$, and whose initial condition (at time $s$) is $\nu_s$, for all $N>s$.
            Moreover, the (strong Markov) process $X$ is the unique solution to the martingale problem on $[s,N)$ associated with $\left(\overline{\sf L}, C_c^{\infty}((0,\infty)\times \mathbb{R}^d)\right)$, with initial distribution $\delta_s\otimes\nu_s$ and with $(X_2(t))_{t\geq 0}$ having the one-dimensional marginals in $\mathcal{A}_{\lesssim u}^{s,N}$, for every $N>s$.
        \end{enumerate}
\end{thm}
\begin{proof}
    The proof is given in \Cref{ss: proof of main1}.
\end{proof}

\begin{rem}
Let $\tilde{X}$ be given by the trivial extension (see \Cref{defi:trivial}) of $X$ given in \Cref{thm:main0} or in \Cref{thm:main1}, from $E$ to $[0,\infty)\times \mathbb{R}^d$.
Then $\tilde{X}$ is a right process by \Cref{prop-trivial}, and note that $X$ is the restriction (see \Cref{defi:restr-process}) of $X$ from $[0,\infty)\times \mathbb{R}^d$ to $E$, whilst all points in $[0,\infty)\times \mathbb{R}^d\setminus E$ are trap points for $X$.
Consequently, $\tilde{X}$ is a conservative norm-continuous right process on $[0,\infty)\times \mathbb{R}^d$.
\end{rem}

\subsection{Proofs of the main results}\label{ss:proof}
This subsection has three parts. 
In the first part we prepare some technical results that shall be used in the proofs of \Cref{thm:main0} and \Cref{thm:main1}, which are given in great detail in the last two parts, respectively.

\subsubsection{Some technical lemmas needed in the proofs}

\begin{lem}\label{prop:densities}
    Assume that $\mathbf{H_{a,b}^{\sqrt{u}}}$ holds, let $N>0$ and $v^N\in \mathcal{A}_{\lesssim u}^{0,N}$ with $c>0$ as in \eqref{eq:A}, such that $v^N(t,x)dx\in \mathcal{P}(\mathbb{R}^d)$ for every $t\in [0,N)$. 
    Then there exists a measurable set $S\subset [0,N)$ of full Lebesgue measure such that for every $s\in S$ there exists $v_s^N:[0,\infty)\times \mathbb{R}^d\rightarrow [0,\infty)$ measurable satisfying the following:
    \begin{enumerate}
        \item[(i)] For every $t\in [0,\infty)$ we have
        \begin{equation*}
        \mathbb{P}^N_{\delta_s\otimes v^N(s,x)dx}\circ\left(X_2^N(t)\right)^{-1}=v^N_s(t,x) dx,    
        \end{equation*}
        in particular, $v_s^N(t,x)dx$ is a sub-probability on $\mathbb{R}^d$.
        \item[(ii)] We have $dt$-a.e. that
        \begin{equation*}
            v^N_s(t,x)dx\leq c 1_{[0,N)}(s+t)u(s+t,x)dx, \quad \mbox{hence} \quad v^N_s(\cdot-s,\cdot)\in \mathcal{A}_{\lesssim u}^{s,N}.
        \end{equation*}
        \item[(iii)] The family of of sub-probability measures $\left(v^N_s(t-s,x) dx\right)_{t\in [s,N)}$ is a weakly right-continuous solution to the linear Fokker-Planck equation \eqref{eq:LFPshort} on $[s,N)$ with $v_s^N(0,x)dx=v^N(s,x)dx$.
        Moreover, $(s,N)\ni t\mapsto v^N_s(t-s,x) dx$ is also  vaguely left-continuous.
    \end{enumerate}
\end{lem}
\begin{proof}
First of all, recall that $\overline{\mu}_N$ is $\left(\overline{P}_t^N\right)$-sub-invariant, so that the dual (w.r.t. $\overline{\mu}_N$) semigroup of kernels $\overline{P}_t^{N,\ast},t\geq 0$ is sub-Markovian, namely $\overline{P}_t^{N,\ast}1\leq 1, t\geq 0$.
Moreover, property \eqref{eq:evol} transfers by duality as
\begin{equation*}
   \overline{P}^{N,\ast}_t f(s,x)= \overline{P}^{N,\ast}_t \left(f(s-t,\cdot)\right)(s,x) 1_{[t, N)}(s) \quad (s,x)\in [0,N)\times \mathbb{R}^d \mbox{ 
 a.s. }
\end{equation*}
Let $g,h\in C([0,N];\mathbb{R})$ and $f\in C_b(\mathbb{R}^d)$.
Then
\begin{align*}
    \int_0^N g(t)&\int_0^N h(s)\mathbb{E}^N_{\delta_s\otimes v^N(s,x)dx}f\left(X_2^N(t)\right)\;dsdt\\
    &=\int_0^N g(t) \int_{[0,N]\times\mathbb{R}^d}h(s)\overline{P}^N_tf(s,x) v^N(s,x)\;dxdsdt\\
    &=\int_0^N g(t)\int_{[0,N]\times\mathbb{R}^d} h(s)\frac{v^N(s,x)}{u(s,x)}\overline{P}^N_tf(s,x) \overline{\mu}_N(ds,dx)dt\\
    &=\int_0^N g(t)\int_{[0,N]\times\mathbb{R}^d} f(x)\overline{P}^{N,\ast}_t\left(h(\cdot)\frac{v^N(\cdot,\cdot)}{u(\cdot,\cdot)}\right)(s,x) \overline{\mu}_N(ds,dx)dt\\
    &=\int_0^N g(t)\int_{[0,N]\times\mathbb{R}^d}f(x)h(s-t)1_{[t,N)}(s)\overline{P}^{N,\ast}_t\left(\frac{v^N}{u}\right)(s,x) \overline{\mu}_N(ds,dx)dt\\
    &=\int_0^N g(t)\int_0^N h(s)\int_{\mathbb{R}^d}f(x)1_{[0,N-t)}(s)\overline{P}^{N,\ast}_t\left(\frac{v^N}{u}\right)(s+t,x) u(s+t,x)\;dxdsdt.
\end{align*}
Consequently, by setting
\begin{equation*}
    \tilde{v}^N_s(t,x):=1_{[0,N)}(s+t)\overline{P}^{N,\ast}_t\left(\frac{v^N}{u}\right)(s+t,x) u(s+t,x), \quad (s,t,x)\in [0,N)\times [0,\infty)\times \mathbb{R}^d,
\end{equation*}
and since $f,g,h$ are arbitrarily chosen,
it follows that
\begin{equation}\label{eq:density_tilde}
    \mathbb{P}^N_{\delta_s\otimes v^N(s,x)dx}\circ\left(X_2(t)\right)^{-1}=\tilde{v}^N_s(t,x) dx \quad \mbox{ for }(s,t)\in[0,N)\times [0,\infty) \; dsdt\mbox{-a.e.}   
\end{equation}
Also, notice that because $v^N\in \mathcal{A}_{\lesssim u}^{0,N}$ and because $\overline{P}^{N,\ast}_t1\leq 1$, we get that
\begin{equation}\label{eq:v_tilde_u}
   \tilde{v}^N_s(t,x) dx\leq c 1_{[0,N)}(s+t)u(s+t,x) dx \quad dsdt\mbox{-a.e.} 
\end{equation}
Now we shall construct a modification of $\tilde{v}^N_s(t,x) dx$ which is weakly continuous with respect to $t\geq 0$, $ds$-a.e. 
To this end, let $D\subset [0,\infty)$ be dense (from the right) and countable, and $\tilde{S}\subset [0,N)$ be measurable and of full Lebesgue measure such that for every $(s,t)\in \tilde{S}\times D$ we have
\begin{equation}\label{eq:SxD}
    \mathbb{P}^N_{\delta_s\otimes v(s,x)dx}\circ\left(X_2(t)\right)^{-1}=\tilde{v}^N_s(t,x) dx \quad  \mbox{and} \quad \tilde{v}^N_s(t,x) dx\leq c 1_{[0,N)}(s+t)u(s+t,x) dx,
\end{equation}
which is possible due to \eqref{eq:density_tilde} and \eqref{eq:v_tilde_u}.
Let $s\in \tilde{S}$ be fixed.
From the last inequality we deduce that $\left(\tilde{v}^N_s(t,x) dx\right)_{t\in D}$ is a family of sub-probability measures which is tight, hence by Prohorov's theorem, it is sequentially relatively compact with respect to the weak convergence.
Further, let $t\in [0,\infty)$, and $(t_n)_{n\geq 1}\subset D$ such that $\lim_n t_n=t$ decreasingly.
By passing to a subsequence, we may also assume that there exists a sub-probability measure $m_s^N(t)$ on $\mathbb{R}^d$ such that
\begin{equation*}
    \lim\limits_n \tilde{v}^N_s(t_n,x) dx = m_s^N(t),\quad \mbox{weakly}.
\end{equation*}
But with $s\in S$ fixed as above, $t\mapsto\mathbb{P}^N_{\delta_s\otimes v(s,x)dx}\circ\left(X_2(t)\right)^{-1}$ is (right) weakly continuous, hence \eqref{eq:SxD} and the above equality entail
\begin{equation*}
    \mathbb{P}^N_{\delta_s\otimes v^N(s,x)dx}\circ\left(X_2(t)\right)^{-1}=m_s^N(t)=\lim\limits_{D\ni r \searrow t} \tilde{v}^N_s(r,x)dx, \quad \mbox{weakly}.
\end{equation*}
Moreover, by the inequality in \eqref{eq:SxD} we also have that for every $s\in \tilde{S}$ there exists $v_s^N(t,x), (t,x)\in [0,\infty)\times \mathbb{R}^d$ such that
\begin{equation} \label{eq:m<u}
    m_s^N(t)=v_s^N(t,x) dx \leq c 1_{[0,N)}(s+t)u(s+t,x) dx, \quad t\in [0,\infty).
\end{equation}
It is then straightforward to see that for each $s\in \tilde{S}$ we have that $v^N_s(\cdot,\cdot)$ is jointly measurable on $[0,\infty)\times \mathbb{R}^d$.
Thus, we have proved (i) and (ii).

To prove (iii), let us note that by \Cref{rem:subprob} and the above construction we have for all $s\in \tilde{S}$ that
\begin{equation*}
    \int_{\mathbb{R}^d}f(x)v_s^N(t,x) dx= \int_{\mathbb{R}^d}v^N(s,x)\mathbb{E}^N_{s,x}\left\{f(X^N_2(t));t<\zeta^N\right\} \;dx, \quad f\in C_b(\mathbb{R}^d),  
\end{equation*}
hence using the path-continuity property \eqref{eq:diffusion} we get that
\begin{equation*}
    [0,\infty)\ni t\mapsto v_s^N(t,x) dx \quad \mbox{is weakly right continuous}.
\end{equation*}

Further, let us show that one can choose $S\subset \tilde{S}$ of full measure on $(0,N)$ such that for every $s\in \tilde{S}$ we have that $\left(v_s^N(t,x) dx\right)_{t\in (0,N-s)}$ is a solution to the linear Fokker-Planck equation \eqref{eq:LFPshort}. 
We proceed as follows: First of all, notice that since $\overline{P}_t^N$ and $\overline{T}_t^N$ coincide as operators on $L^p([0,N)\times \mathbb{R}^d;\overline{\mu}_N)$ for all $t\geq 0$, we get that $\left(\overline{P}_t^N\right)_{t\geq 0}$ solves the Cauchy equation
\begin{equation}\label{eq:CauchyP}
    \overline{P}_{t}^Nf(s,x)-f(s,x)=\int_{0}^{t}\overline{P}_{r}^N \overline{\sf L} f(s,x) \; dr \; \overline{\mu}_N\mbox{-a.e.}, \quad 0<t<\infty, f\in C_c^\infty(\mathbb{R}^d).
\end{equation}
Let $g,h\in C_c^\infty((0,N);\mathbb{R})$ and $f\in C_c^\infty(\mathbb{R}^d)$, so that using the evolution character of $\left(\overline{T}_t^N\right)$ given by \eqref{eq:evol}, we have
\begin{align*}
    &\int_0^N g(s) \int_s^N h'(t)\int_{\mathbb{R}^d} f(x) v_s^N(t-s,x) \; dxdtds
    =\int_0^N g(s) \int_s^N h'(t) \int_{\mathbb{R}^d} \overline{P}_{t-s}f(s,x) v(s,x) \; dxdtds\\
    &=\int_0^N g(s) \int_{\mathbb{R}^d} \int_s^N h'(t) \overline{P}_{t-s}f(s,x) v(s,x) \; dxdtds\\
    &=\int_0^N g(s) \int_{\mathbb{R}^d} h(s)f(x) v(s,x)\;dxds
    -\int_0^N g(s) \int_s^N h(t) \int_{\mathbb{R}^d} \overline{P}_{t-s} \overline{\sf L} f(s,x) v^N(s,x) \; dxdtds\\
    &=\int_0^N g(s)h(s) \int_{\mathbb{R}^d}f(x) v(s,x)\;dxds
    -\int_0^N g(s) \int_s^N h(t) \int_{\mathbb{R}^d} {\sf L}_t f(x) v^N_s(t-s,x) \; dxdtds.
\end{align*}
Thus
\begin{align*}
     &\int_0^N g(s) \int_s^N\int_{\mathbb{R}^d} \left(h'(t) f(x) +h(t) {\sf L}_t f(x) v_s^N(t-s,x)\right) \; dxdtds=\int_0^N g(s)h(s) \int_{\mathbb{R}^d}f(x) v(s,x)\;dxds.
\end{align*}
Since $g$ is arbitrarily chosen and using a density argument, it is easy to see that one can find a measurable subset $S'\subset\tilde{S}$ of full Lebesgue measure on $(0,N)$ such that
\begin{equation*}
    \int_s^N\int_{\mathbb{R}^d} \left(\frac{d}{dt}f(t,x) + {\sf L}_t f(t,x) v_s^N(t-s,x)\right) \; dxdt= \int_{\mathbb{R}^d}f(s,x) v(s,x)\;dx
\end{equation*}
for all $s\in S'$ and $f\in C_c^\infty((0,N)\times \mathbb{R}^d)$.
Thus, for every $s\in S'$ and $f\in C_c^\infty((s,N)\times \mathbb{R}^d)$ we obtain that
\begin{equation*}
    \int_s^N\int_{\mathbb{R}^d} \left(\frac{d}{dt}f(t,x) + {\sf L}_t f(t,x) v_s^N(t-s,x)\right) \; dxdt= 0,
\end{equation*}
which means that $\left(v_s^N(t-s,x)dx\right)_{t\in [s,N)}$ is a weakly right-continuous solution to the linear Fokker-Planck equation \eqref{eq:LFPshort}.

To finish the proof of (iii) it remains to show that $\left(v_s^N(t-s,x)dx\right)_{t\in (s,N)}$ is vaguely left-continuous.
To this end, note that if $\mathcal{C}\subset C_c^\infty((0,N)\times \mathbb{R}^d)$ be countable and dense in $C_c((0,N)\times \mathbb{R}^d)$ with respect to the uniform convergence, so that by \eqref{eq:CauchyP} we get for every $t>0$, $ds$-a.e. on $(0,N)$ that
\begin{equation*}
\int_{\mathbb{R}^d} \overline{P}_tf(s,x) v^N(s,x)\;dx=\int_{\mathbb{R}^d} 1_{[0,N)}(t+s)f(s+t,x) v^N_s(t,x)\;dx,\quad f\in \mathcal{C}.  
\end{equation*}
and
\begin{equation*}
    \int_{\mathbb{R}^d} \overline{P}_tf(s,x) v^N(s,x)\;dx=\int_{\mathbb{R}^d} f(s,x) v^N(s,x)\;dx + \int_{\mathbb{R}^d} v^N(s,x) \int_0^t \overline{P}_r\overline{\sf L}f(s,x)\;drdx, \quad f\in \mathcal{C}.
\end{equation*}
Thus, on the one hand, there exists $S\subset S'$ of full Lebesgue measure on $(0,N)$ such that for every $s\in S'$
\begin{equation*}
    \int_{\mathbb{R}^d} \overline{P}_tf(s,x) v^N(s,x)\;dx=\int_{\mathbb{R}^d} f(s,x) v^N(s,x)\;dx + \int_{\mathbb{R}^d} v^N(s,x) \int_0^t \overline{P}_r\overline{\sf L}f(s,x)\;drdx, \quad f\in \mathcal{C}, t\in \mathbb{Q}_+,
\end{equation*}
in particular the above integrals are well defined.
On the other hand,
\begin{equation*}
    [0,\infty)\ni t\mapsto\int_{\mathbb{R}^d} \overline{P}_tf(s,x) v^N(s,x)\;dx=\int_{\mathbb{R}^d} f(s,x) v^N(s,x)\;dx + \int_{\mathbb{R}^d} v^N(s,x) \int_0^t \overline{P}_r\overline{\sf L}f(s,x)\;drdx
\end{equation*}
is right continuous for every $f\in C_c^\infty((0,N)\times\mathbb{R}^d)$ and $s\in \mathcal{S}$, hence we must have
\begin{equation*}
    \int_{\mathbb{R}^d} \overline{P}_tf(s,x) v^N(s,x)\;dx=\int_{\mathbb{R}^d} f(s,x) v^N(s,x)\;dx + \int_{\mathbb{R}^d} v^N(s,x) \int_0^t \overline{P}_r\overline{\sf L}f(s,x)\;drdx, \quad f\in \mathcal{C}, t\in [0,\infty).
\end{equation*}
This means that for every $s\in S$ and $f\in \mathcal{C}$ we have 
\begin{equation*}
    [0,\infty)\ni t \mapsto \int_{\mathbb{R}^d} \overline{P}_tf(s,x) v^N(s,x)\;dx=\int_{\mathbb{R}^d} 1_{[0,N)}(t+s)f(s+t,x) v^N_s(t,x)\;dx \in \mathbb{R} \quad \mbox{ is continuos}.
\end{equation*}
By a density argument we can further assume that $\mathcal{C}$ consists of functions with separate variables, so taking $f\in \mathcal{C}$ of the form $f(r,x)=g(r)h(x), (r,x)\in (0,N)\times \mathbb{R}^d$,
we get for every $s\in S$ that
\begin{equation*}
    [0,N-s)\ni t \mapsto\int_{\mathbb{R}^d} h(x) v^N_s(t,x)\;dx \in \mathbb{R} \quad \mbox{ is continuos}.
\end{equation*}
Since we assumed that the class of functions $h$ for which the above holds is dense in $C_c(\mathbb{R}^d)$ with respect to the uniform convergence, we get that
\begin{equation*}
    [0,N-s)\ni t \mapsto v^N_s(t,x)\;dx \in \mathcal{M}\left(\mathbb{R}^d\right) \quad \mbox{ is vaguely continuos}.
\end{equation*}
\end{proof}

\begin{rem}
We emphasize that although $X^N$ has continuous paths a.s. on $[0, \zeta^N)$ as stated in \eqref{eq:diffusion}, it does not mean that the the curve of sub-probability measures $\mathbb{P}_{s,x}\circ \left(X_2^N(t)\right)^{-1}$ on $\mathbb{R}^d$ is weakly left-continuous on $(s,N-s)$ because the process $t\mapsto 1_{[0,\zeta^N)}(t)$ is not guaranteed to be stochastically left continuous in $t$. 
In other words, based only on the results obtained in \cite{St99a}, it is not clear that $\zeta^N$ is a continuous random variable, at least on the event $\zeta^N<N-s$ for $s$ fixed; such a property would have solved the left-continuity issue in $t\in(s,N-s)$, since in that case, the jump event $\zeta^N=t$ would have zero probability. 
\end{rem}

\begin{lem}\label{coro:identification1}
    Assume that conditions $\mathbf{H_{a,b}^{\sqrt{u}}}$ and $\mathbf{H_{\sf \leq u}}$ are fulfilled.
    Then, the following assertions hold for every $0<s<N<\infty$.
    \begin{enumerate}
        \item[(i)] We have the identification
        \begin{equation}\label{eq:identification}
        \mathbb{P}^N_{\delta_s\otimes \mu_s}\circ \left(X_2^N(t)\right)^{-1}= \mu_{t+s}, \quad 0\leq t<N-s. \end{equation}
         \item[(ii)] We have the following refinement of \Cref{thm:Stannat}, (ii.3):
         \begin{equation}\label{eq:uniform_motion}
            \mathbb{P}^N_{\delta_s\otimes \mu_s}\left( X^N_1(t)\neq t+s; t<\zeta^N  \right)=0.
        \end{equation}
        \item[(iii)] The lifetime $\zeta^N$ satisfies
        \begin{equation}\label{eq:zeta}
         \mathbb{P}_{s,x}^N\left(\zeta^N=N-s\right)=1 \quad \mu_s\mbox{-a.e. } x\in \mathbb{R}^d,
    \end{equation}
    \end{enumerate}    
\end{lem}
\begin{proof}
(i)-(ii). Since $\mathbf{H_{a,b}^{\sqrt{u}}}$ is in force we apply \Cref{prop:densities} for $v^N(s,x)=u(s,x), (s,x)\in [0,\infty)\times \mathbb{R}^d$, noticing that we trivially have $u\in \mathcal{A}_{\leq u}^{0,N}$ for all $N>0$.   
Thus, for every $N>0$ there exists $S\subset [0, N)$ of full Lebesgue measure such that for every $s\in S$ there exists a family $\left(v_s^N(t,x)dx\right)_{t\geq 0}$ having all the properties derived in \Cref{prop:densities}.
In particular, $\left(v_s^N(t-s,x)dx\right)_{t\in [s,N)}$ is a weakly right-continuous and vaguely left-continuous solution to the linear Fokker-Planck equation \eqref{eq:LFPshort} on $[s,N)$ with initial condition $u(s,x)dx=\mu_s$.
Consequently, by $\mathbf{H_{\sf \leq u}}$ we have
\begin{equation*}
    \mathbb{P}^N_{\delta_s\otimes u(s,x)dx}\circ\left(X_2^N(t-s)\right)^{-1}=v_s^N(t-s,x)dx=u(t,x)dx, \quad s\in S,t\in [s,N), 
\end{equation*}
or equivalently, that
\begin{equation}\label{eq:s_in_S}
    \mathbb{P}^N_{\delta_s\otimes \mu_s}\circ\left(X_2^N(t)\right)^{-1}=v_s^N(t,x)dx=u(t+s,x)dx=\mu_{t+s}, \quad s\in S,t\in [0,N-s). 
\end{equation}
Furthermore, by \eqref{thm:Stannat}, (ii.2), and integrating \eqref{eq:diffusion} with respect to $\overline{\mu}_N$, we get
\begin{equation*}
    \int_0^N \overline{\mathbb{P}}^N_{\delta_s\otimes \mu_s}\left( X^N_1(t)\neq t+s; t<\zeta^N  \right) \;ds=0,
\end{equation*}
which means that by passing to a subset of $S$, we may assume without loss of generality that
\begin{equation}\label{eq:uniform_S}
 \overline{\mathbb{P}}^N_{\delta_s\otimes \mu_s}\left( X^N_1(t)\neq t+s; t<\zeta^N  \right)=0,\quad s\in S.
\end{equation}

To prove that \eqref{eq:s_in_S} and \eqref{eq:uniform_S} from above hold for all $0<s<N$, we use a bootstrap argument: Let $0<s<N$ and $s'\in S$ such that $s'<s$.
Then
\begin{align*}
    \delta_s\otimes u(s,x)dx 
    &= \delta_{s'+s-s'}\otimes u(s'+s-s',x)dx=\delta_{s'+s-s'}\otimes v_{s'}(s-s',x)dx\\
    &=\mathbb{P}^N_{\delta_{s'}\otimes \mu_{s'}}\circ\left(X_1^N(s-s'),X_2^N(s-s')\right)^{-1},
\end{align*}
where the last equality follows from \eqref{eq:s_in_S} and \eqref{eq:uniform_S}.
Hence, by the Markov property of $X^N$ we get
\begin{align*}
    \mathbb{P}^N_{\delta_{s}\otimes \mu_s}\circ\left(X_1^N(t),X_2^N(t)\right)^{-1}
    &=\mathbb{P}^N_{\mathbb{P}^N_{\delta_{s'}\otimes \mu_{s'}}\circ\left(X_1^N(s-s'),X_2^N(s-s')\right)^{-1}}\circ\left(X_1^N(t),X_2^N(t)\right)^{-1}\\
    &=\mathbb{P}^N_{\delta_{s'}\otimes \mu_{s'}}\circ\left(X_1^N(s-s'+t),X_2^N(s-s'+t)\right)^{-1},
\end{align*}
and thus for all $t<N-s$ we have that $t+s-s'<N-s'$, so by \eqref{eq:s_in_S} we obtain
\begin{align*}
    \mathbb{P}^N_{\delta_s\otimes u(s,x)dx}\circ\left(X_2^N(t)\right)^{-1}
    &=\mathbb{P}^N_{\delta_{s'}\otimes u(s',x)dx}\circ\left(X_2^N(s-s'+t)\right)^{-1}\\
    &=u(s+t,x)dx,
\end{align*}
so (i)-(ii) are completely proved.

To prove (iii) we make use of (i)-(ii) to deduce that for every $0<s<N<\infty$ we have
\begin{equation*}
    \mathbb{P}^N_{\delta_s\otimes \mu_s}\circ \left(X_1^N(t),X_2^N(t)\right)^{-1}= \delta_{s+t}\otimes\mu_{t+s}, \quad 0\leq t<N-s,
\end{equation*}
hence for $0\leq t<N-s$ 
\begin{equation*}
    \mathbb{E}_{\delta_s\otimes\mu_s}\left\{ 1_{[0,N-s)\times \mathbb{R}^d}(X^N(t)); t<\zeta^N\right\}=\int_{[0,N-s)\times\mathbb{R}^d} 1 \; d \mathbb{P}^N_{\delta_s\otimes \mu_s}\circ \left(X_1^N(t),X_2^N(t)\right)^{-1}=1,
\end{equation*}
in other words
$\mathbb{P}_{\delta_s\otimes\mu_s}\left\{ t<\zeta^N\right\}=1, \; 0\leq t<N-s$, or $\mathbb{P}_{\delta_s\otimes\mu_s}\left\{ N-s\leq\zeta^N\right\}=1$.
At the same time, by \eqref{eq:uniform_motion} we clearly have that $\mathbb{P}^N_{\delta_s\otimes \mu_s}\left\{N-s<\zeta^N\right\}=0$, since $\mathbb{P}^N_{\delta_s\otimes \mu_s}$-a.s. on the event $N-s<\zeta^N$ we have $X_1^N(N-s)=N$ which entails the contradiction $X^N(N-s)\notin [0,N)\times \mathbb{R}^d$.
In conclusion, we must have \eqref{eq:zeta}.
\end{proof}

\begin{lem}\label{coro:identification2}
Assume that conditions $\mathbf{H_{a,b}^{\sqrt{u}}}$ and $\mathbf{H_{\sf \lesssim u}}$ are fulfilled. 
Also, let $\left(\nu_t:=v(t,x)dx\right)_{t\in [0,N)}\subset \mathcal{P}(\mathbb{R}^d)$ be a weakly continuous solution to the linear Fokker-Planck equation \eqref{eq:LFPshort} on $[0,N)$ such that  $v\in\mathcal{A}_{\lesssim u}^{0,N}$.
Then for every $0<s<N<\infty$ we have the identification
    \begin{equation}\label{eq:identification_nu}
        \mathbb{P}^N_{\delta_s\otimes \nu_s}\circ \left(X_2^N(t)\right)^{-1}= \nu_{t+s}, \quad 0\leq t<N-s.
    \end{equation}
Moreover, there exists a unique solution $\eta^N\in\mathcal{P}\left(C([0,N];\mathbb{R}^d)\right)$ to the canonical MP associated to $\left({\sf L}_t\right)_{t\in [0,T]}$ such that for every $\varepsilon\in (0,N)$
\begin{equation*}
    \eta^N|_{C([\varepsilon,N);\mathbb{R}^d)}=\mathbb{P}^N_{\delta_\varepsilon\otimes \nu_\varepsilon}\circ \left[\left(X_2^N(t-\varepsilon)\right)_{t\in[\varepsilon, N)}\right]^{-1} \quad \mbox{as distributions on } C([\varepsilon,N);\mathbb{R}^d).
\end{equation*} 
\end{lem}
\begin{proof}
First of all, since $\mathbf{H_{\sf \lesssim u}}$ is stronger than $\mathbf{H_{\sf \leq u}}$ we have that the conclusions of \Cref{coro:identification1} are valid.
In particular, using \Cref{coro:identification1}, (iii), apllied to $v^N=v$, we have that the solutions $\left(v_s^N(t,x)dx\right)_{0\leq t<N-s}$ for $s\in S$ provided by \Cref{prop:densities} are weakly right-continuous curves in $\mathcal{P}(\mathbb{R}^d)$. 
Thus, by \Cref{rem:Trevisan_narrow}, they are weakly continuous, and the uniqueness condition $\mathbf{H_{\sf \lesssim u}}$ can be used to prove \eqref{eq:identification_nu} by following the same lines as in the proof of (i)-(ii) in \Cref{coro:identification1}.

Now, let us show the second part of the statement.
To this end, let $\varepsilon\in (0,T)$, $f\in C_c^{\infty}([\varepsilon,N]\times \mathbb{R}^d)$, and define the process
\begin{equation*}
    M(t):=f(X^N(t))-f(X^N(0))-\int_0^t\overline{\sf L}f(X^N(r))\;dr, \quad t\in [0,N-\varepsilon).
\end{equation*}
Recall that $h(\Delta)=0$ for any function $h:[0,N]\times \mathbb{R}^d \rightarrow \mathbb{R}$, where $\Delta$ is an abstract cemetery points; see the paragraph on right processes from \Cref{Appendix}.
Also, note that by \Cref{coro:identification1} we have
\begin{equation*}
    \mathbb{P}^N_{\delta_\varepsilon\otimes \nu_\varepsilon}\left(\zeta^N=N-\varepsilon\right)=1.
\end{equation*}
Furthermore, by \Cref{coro:identification1}, (ii),
\eqref{eq:identification_nu}, and \eqref{eq:bau}, we have
\begin{equation*}
    \mathbb{E}^N_{\delta_\varepsilon\otimes \nu_\varepsilon}\int_0^t\left|\overline{\sf L}f\right|(X^N(r))\;dr
    =\int_0^t\nu_{r+\varepsilon}\left(\left|\overline{\sf L}f\right|(r+\varepsilon,\cdot)\right)\;dr
    \leq c\int_0^N\mu_r\left(\left|\overline{\sf L}f\right|(r,\cdot)\right)\;dr<\infty, \quad t<N-\varepsilon.
\end{equation*}
Consequently, under $\mathbb{P}^N_{\delta_\varepsilon\otimes \nu_\varepsilon}$, the process $M(t),$ $\varepsilon\leq t\leq N$, is well-defined and integrable.
Moreover, since $X^N$ is a right process, it is a standard fact that $M$ is a continuous additive functional, namely
\begin{equation*}
    M(t+s)=M(s)+M(t)\circ \theta^N(s), \quad 0\leq s\leq N-\varepsilon-t, \quad \mathbb{P}^N_{\delta_\varepsilon\otimes \nu_\varepsilon}\mbox{-a.e.}
\end{equation*}
Consequently, $\left(M(t)\right)_{t\in [0,N-\varepsilon)}$ is a $\mathbb{P}^N_{\delta_\varepsilon\otimes \nu_\varepsilon}$-martingale as soon as we show that
\begin{equation*}
    \mathbb{E}^N_{\delta_\varepsilon\otimes \nu_\varepsilon} \left\{M(t)\right\}=0, \quad t\in [0,N-\varepsilon).
\end{equation*}
But this can be easily verified since again by \eqref{eq:identification_nu} and \Cref{coro:identification1} we have
\begin{align*}
    \mathbb{E}^N_{\delta_\varepsilon\otimes \nu_\varepsilon} \left\{M(t)\right\}
    &=\nu_{t+\varepsilon}(f(t+\varepsilon,\cdot))-\nu_{\varepsilon}(f(\varepsilon,\cdot))-\int_0^t\nu_{r+\varepsilon}(\overline{\sf L}f(r+\varepsilon,\cdot))\;dr\\
    &=\nu_{t+\varepsilon}(f(t+\varepsilon,\cdot))-\nu_{\varepsilon}(f(\varepsilon,\cdot))-\int_\varepsilon^{t+\varepsilon}\nu_{r}(\overline{\sf L}f(r,\cdot))\;dr\\
    &=0, \quad t< N-\varepsilon,
\end{align*}
where the last equality follows from the fact that $(\nu_t)_{t\in [0,N)}$ is assumed to be a solution to the linear Fokker-Planck equation \eqref{eq:LFPshort}.

Now, let us notice that $(M(t))_{t\in[0,N-\varepsilon)}$ being a $\mathbb{P}^N_{\delta_\varepsilon\otimes \nu_\varepsilon}$-martingale actually means that
\begin{equation*}
    f(t+\varepsilon,X^N_2(t))-f(\varepsilon,X^N_2(0))-\int_0^t\overline{\sf L}f(r+\varepsilon,X^N_2(r))\;dr, \quad t\in [0,N-\varepsilon),
\end{equation*}
is a $\mathbb{P}^N_{\delta_\varepsilon\otimes \nu_\varepsilon}$-martingale, 
hence, by a change o variable, that
\begin{equation*}
    f(t,X^N_2(t-\varepsilon))-f(\varepsilon,X^N_2(0))-\int_\varepsilon^{t}\overline{\sf L}f(r,X^N_2(r-\varepsilon))\;dr, \quad t\in [\varepsilon,N),
\end{equation*}
is a $\mathbb{P}^N_{\delta_\varepsilon\otimes \nu_\varepsilon}$-martingale.
Consequently,
\begin{equation*}
    \mathbb{P}^N_{\delta_\varepsilon\otimes \nu_\varepsilon}\circ\left[\left(X_2^N(t-\varepsilon)\right)_{t\in [\varepsilon,N)}\right]^{-1}\in \mathcal{P}\left(C([\varepsilon,N);\mathbb{R}^d)\right)
\end{equation*}
is a solution to the canonical MP associated to $\left({\sf L}_t\right)_{t\in [\varepsilon,N)}$, with initial distribution $\nu_\varepsilon$.
Finally, the second part of \Cref{coro:identification2} follows by \Cref{prop:uniqueness_MP}.
\end{proof}

\subsubsection{Proof of Theorem 2.12}\label{ss:proof of main0}
Let 
 \begin{equation*}
        X^N=\left(\Omega^N,\mathcal{F}^N,\left(\mathcal{F}_t^N\right)_{t\geq 0},\mathbb{P}_{s,x}^N, (s,x)\in [0,N)\times \mathbb{R}^d\right) \mbox{ with lifetime } \zeta^N, \quad N>0.
        \end{equation*}
be given by \Cref{thm:Stannat}.
In order to construct the projective limit of the above family of right processes, the idea is to first construct the projective limit of the analytic objects $\left(\overline{P}_t^N,t\geq 0\right)_{N>0}$ and $\left(\overline{\mathcal{U}}^N\right)_{N>0}$.
One this first part is achieved, we proceed to the construction of the corresponding stochastic process which shall stand as the projective limit of $(X^N)_{N>0}$.
The proof is a bit involved, so let us split it in several steps.

\medskip
\noindent{\bf Step OP1.1}:
{\it We claim that for all $N>0$ we have} 
\begin{equation*}
    \overline{P}_t^N((s,x),\cdot)=\overline{P}_t^{N+1}((s,x),\cdot)\big|_{[0,N)\times \mathbb{R}^d} \quad \overline{\mu}_N\mbox{-a.e. on } [0,N)\times \mathbb{R}^d, t\geq 0,
\end{equation*}
and thus, by taking the Laplace transform with respect to the $t$ variable,
\begin{equation*}
    \overline{U}_\alpha^N\left((s,x),\cdot\right)=\overline{U}_\alpha^{N+1}\left((s,x),\cdot\right)\big|_{[0,N)\times \mathbb{R}^d} \quad \overline{\mu}_N\mbox{-a.e. on } [0,N)\times \mathbb{R}^d, \alpha>0.
\end{equation*}
To prove the claim, let $h\in pb\mathcal{B}([0,\infty)\times \mathbb{R}^d)$, so that $hu \in \mathcal{A}_{\lesssim u}^{0,N}, N>0$ for $c:=\|h\|_\infty$. 
Further, set
\begin{equation*}
    v(t,x):=\frac{h(t,x)u(t,x)}{\int_{\mathbb{R}^d} h(t,x)u(t,x)dx}, \quad t\geq 0, x\in \mathbb{R}^d,
\end{equation*}
so that \Cref{prop:densities} as well as \Cref{coro:identification1} can be applied to deduce: For each $N>0$ there exists a set $S\subset (0,N+1)$ of full Lebesgue measure such that for every $s\in S$ one can construct the family of probability measures $\left(v_s^{N+1}(t-s,x)dx\right)_{t\in [s,N+1)}$ which is a weakly continuous solution to the linear Fokker-Planck equation \eqref{eq:LFPshort} on $[s,N+1)$ with $v_s^{N+1}(0,x)dx=v(s,x)dx$.
Thus, for $s\in S, t< N-s$, by \Cref{coro:identification1} and then \Cref{coro:identification2} we have
\begin{align*}
    \left(\delta_s\otimes v(s,x)dx\right) \circ \overline{P}_t^N 
    &= \mathbb{P}^N_{\delta_s\otimes  v(s,x)dx}\circ \left(X_1^N(t),X_2^N(t)\right)^{-1}\\
    &=\delta_{t+s}\otimes\mathbb{P}^N_{\delta_s\otimes  v(s,x)dx}\circ \left(X_2^N(t)\right)^{-1}\\
    &=\delta_{t+s}\otimes v_s^{N+1}(t-s,x)dx,\\
    \intertext{as well as}
    \left(\delta_s\otimes v(s,x)dx\right) \circ \overline{P}_t^{N+1} 
    &= \mathbb{P}^{N+1}_{\delta_s\otimes  v(s,x)dx}\circ \left(X_1^{N+1}(t),X_2^{N+1}(t)\right)^{-1}\\
    &=\delta_{t+s}\otimes\mathbb{P}^{N+1}_{\delta_s\otimes  v(s,x)dx}\circ \left(X_2^{N+1}(t)\right)^{-1}\\
    &=\delta_{t+s}\otimes v_s^{N+1}(t-s,x)dx.
\end{align*}
Consequently, for every $s\in S$ and $t<N-s$ we have
\begin{equation*}
    \left(\delta_s\otimes v(s,x)dx\right) \circ \overline{P}_t^N=\left(\delta_s\otimes v(s,x)dx\right) \circ \overline{P}_t^{N+1},
\end{equation*}
or equivalently,
\begin{equation*}
     \int_{\mathbb{R}^d} h(s,x)u(s,x)\overline{P}_t^N((s,x),\cdot) \;dx=\int_{\mathbb{R}^d} h(s,x)u(s,x)\overline{P}_t^{N+1}((s,x),\cdot)\;dx,
\end{equation*}
where the dot indicates that the equality is for the corresponding measures on $[0,N)\times \mathbb{R}^d$.
As a matter of fact, the above equality holds for any $t\geq 0$ since by \Cref{coro:identification1} it is easy to see that as measures on $[0,N)\times \mathbb{R}^d$,
\begin{equation*}
     \left(\delta_s\otimes v(s,x)dx\right) \circ \overline{P}_t^N=0=\left(\delta_s\otimes v(s,x)dx\right) \circ \overline{P}_t^{N+1}, \quad t\geq N-s.
\end{equation*}
Finally, integrating also with respect to $s$ we obtain
\begin{equation*}
    \int_{[0,N)\times \mathbb{R}^d} h(s,x)\overline{P}_t^N((s,x),\cdot) \;\overline{\mu}_N(ds,dx)=\int_{[0,N)\times \mathbb{R}^d} h(s,x)\overline{P}_t^{N+1}((s,x),\cdot)\;\overline{\mu}_N(ds,dx).
\end{equation*}
The claim now follows since $h$ was arbitrarily chosen.

\medskip
\noindent{\bf Step OP1.2}: {\it For $\alpha>0$, $0\leq f\in L^p\left([0,\infty)\times \mathbb{R}^d ; \overline{\mu}\right)$, $1\leq p\leq \infty$, set
\begin{equation*}
    \overline{U}_\alpha f:= \sup_{N} \left(1_{[0,N)\times \mathbb{R}^d} \overline{U}_\alpha^N (f|_{[0,N)\times \mathbb{R}^d})\right).
\end{equation*}
We claim that $\overline{\mathcal{U}}:=\left(\overline{U}_\alpha\right)_{\alpha>0}$ is a $C_0$-resolvent of Markov operators on $L^p\left([0,\infty)\times \mathbb{R}^d ; \overline{\mu}\right)$, $1\leq p< \infty$, and $\overline{\mu}$ is sub-invariant with respect to $\overline{\mathcal{U}}$.
}

\noindent{To prove this claim}, let $0\leq f\in L^p\left([0,\infty)\times \mathbb{R}^d ; \overline{\mu}\right)$, $1\leq p< \infty$ and note first that by {\bf Step Q1.1} from above we have 
\begin{equation*}
    1_{[0,N)\times \mathbb{R}^d} \overline{U}_\alpha^N (f|_{[0,N)\times \mathbb{R}^d}) \mbox{ is increasing with respect to } N>0.
\end{equation*}
Therefore, by monotone convergence
\begin{equation*}
    \overline{\mu}\left(\left(\alpha \overline{U}_\alpha f\right)^p\right)= \sup_N \overline{\mu}_N\left(\left(\alpha \overline{U}_\alpha^N \left(f|_{[0,N)\times \mathbb{R}^d}\right)\right)^p\right)\leq \sup_N \overline{\mu}_N\left(\left(f|_{[0,N)\times \mathbb{R}^d}\right)^p\right)=\overline{\mu}\left(f^p\right), \quad \alpha>0.
\end{equation*}
Furthermore, since by \Cref{coro:identification1}, (iii) we get $\overline{P}_t^N 1(s,x)=1_{[0,N)}(s+t)$ $\overline{\mu}_N$-a.e., $t\geq 0$, we deduce that $\alpha \overline{U}_\alpha 1 = 1$. 
Also, $\alpha \overline{U}_\alpha$ is clearly positivity preserving, consequently it is a Markov operator on $L^p(\overline{\mu})$ for every $1\leq p\leq \infty$, and it is straightforward to check that $\left(\overline{U}_\alpha\right)_\alpha$ satisfies the resolvent equation. 
To conclude this step it remains to show that 
\begin{equation*}
    \lim_{\alpha \to \infty}\alpha \overline{U}_\alpha = {\rm Id} \quad \mbox{ in } L^p(\overline{\mu}), 
\end{equation*}
which by a density argument boils down to testing the above convergence on smooth functions $f$ with compact support in $[0, N)\times \mathbb{R}^d$, for some arbitrarily fixed $N>0$. 
But for such $f$, by {\bf Step Q1.1} and \Cref{coro:identification1}, (ii), (iii), we have
\begin{equation}\label{eq:U=U^N}
    \alpha \overline{U}_\alpha f= \alpha \overline{U}_\alpha^N \left(f|_{[0,N)\times \mathbb{R}^d}\right)1_{[0,N)\times \mathbb{R}^d},
\end{equation}
which converges to $f$ in $L^p(\overline{\mu})$ as $\alpha \to \infty$ by \Cref{thm:Stannat}.

\medskip
\noindent{\bf Step OP1.3: }{\bf Proof of \Cref{thm:main0},(i)}
    Let us denote by $\left(\overline{\sf L}_p,\sf{D}(\overline{\sf{L}}_p)\right)$ the infinitesimal generators of the $C_0$-resolvent $\overline{\mathcal{U}}$ on $L^p\left([0,\infty)\times \mathbb{R}^d ; \overline{\mu}\right)$, that is
\begin{equation*}
    {\sf{D}}(\overline{\sf{L}}_p):=\overline{U}_1\left(L^p\left([0,\infty)\times \mathbb{R}^d ; \overline{\mu}\right)\right), \quad {\overline{\sf{L}}}_p \left(\overline{U}_1 f\right):= \overline{U}_1 f-f, \quad f\in L^p\left([0,\infty)\times \mathbb{R}^d ; \overline{\mu}\right), 1\leq p<\infty.
\end{equation*}
Since $\alpha\overline{U}_\alpha$ is a contraction on $L^p\left([0,\infty)\times \mathbb{R}^d ; \overline{\mu}\right)$ for every $\alpha>0$, we have that $\left(\overline{\sf L}_p,{\sf{D}(\overline{\sf{L}}_p)}\right)$ is a maximal dissipative operator on $L^p\left([0,\infty)\times \mathbb{R}^d ; \overline{\mu}\right)$, $1\leq p<\infty$.
Let us show that $\left(\overline{\sf L}_1,\sf{D}(\overline{\sf{L}}_1)\right)$ is an extension of $\left(\overline{\sf L},C_c^{\infty}((0,\infty)\times \mathbb{R}^d))\right)$ on $L^1\left([0,\infty)\times \mathbb{R}^d;\overline{\mu}\right)$: 
Let $f\in C_c^{\infty}((0,\infty)\times \mathbb{R}^d))$ so that $f\in C_c^{\infty}((0,N)\times \mathbb{R}^d))$ for some $N>0$ that we fix.
Thus, by \Cref{thm:Stannat}, $f\in {\sf D}\left(\overline{\sf{L}}_1^N\right)$ and there exists $f_0\in L^1\left([0,N)\times \mathbb{R}^d ; \overline{\mu}_N\right)$ such that
\begin{equation*}
    f=\overline{U}_1^Nf_0 \quad \mbox{and} \quad \overline{\sf L} f = \overline{U}_1^Nf_0-f_0 \quad \overline{\mu}_N\mbox{-a.e.}
\end{equation*}
Extending $f_0$ by zero on $[N,\infty)\times \mathbb{R}^d$ and using \eqref{eq:U=U^N}, we get
\begin{equation*}
    \overline{\sf L} f = \overline{U}_1^Nf_0-f_0=\overline{U}_1f_0-f_0=\overline{\sf L}_1\overline{U}_1f_0=\overline{\sf L}_1f \quad \overline{\mu}\mbox{-a.e.},
\end{equation*}
hence $\left(\overline{\sf L}_1,\sf{D}(\overline{\sf{L}}_1)\right)$ is an extension of $\left(\overline{\sf L},C_c^{\infty}((0,\infty)\times \mathbb{R}^d))\right)$ on $L^1\left([0,\infty)\times \mathbb{R}^d;\overline{\mu}\right)$.

\medskip
\noindent{\bf Step OP1.4}: {\it We claim that there exists a resolvent of Markov kernels $\overline{\mathcal{U}}^0:=\left(\overline{U}^0_\alpha\right)_{\alpha>0}$ on $[0,\infty)\times \mathbb{R}^d$ such that $\mathcal{E}(\mathcal{U}^0_\beta)$ is min stable, contains the constant functions, and generates $\mathcal{B}([0,\infty)\times \mathbb{R}^d))$ for one (hence all) $\beta >0$; that is, the equivalent statements in \Cref{prop 2.5} from \nameref{Appendix} hold.}

\noindent{This} claim follows mostly from \cite{BeBoRo06}.
More precisely, by \cite[Theorem 2.2 and Remark 2.3]{BeBoRo06}, we can directly deduce that a resolvent of sub-Markov kernels $\tilde{\mathcal{U}}$ exists having all the other claimed properties satisfied. 
Thus, we only need to prove the Markov property pointwise on $[0,\infty)\times\mathbb{R}^d$. 
To this end we use \Cref{lem:kernel_trick} which is also taken from \cite{BeBoRo06}, as follows: Let
\begin{align*}
    E_0
    &:= \left\{x\in [0,\infty)\times\mathbb{R}^d : \alpha \tilde{U}_\alpha 1(x)=1, \alpha \in \mathbb{Q}_+\right\}\\
    &\ =\left\{x\in [0,\infty)\times\mathbb{R}^d : \alpha \tilde{U}_\alpha 1(x)=1, \alpha >0\right\},
\end{align*}
where the (second) equality holds due to the resolvent equation and the fact that $\tilde{\mathcal{U}}$ is already sub-Markovian.
Since $\alpha \tilde{U}_\alpha 1=1$ $\overline{\mu}$-a.e. for every $\alpha>0$, we get that $\overline{\mu}([0,\infty)\times \mathbb{R}^d\setminus E_0)=0$. 
But now we can apply \Cref{lem:kernel_trick} simultaneously for the kernels $\alpha \tilde{U}_\alpha$, $0<\alpha \in \mathbb{Q}_+$ to find $F\in \mathcal{B}([0,\infty)\times \mathbb{R}^d)$ such that $F\subset E_0$,
\begin{equation*}
    \overline{\mu}([0,\infty)\times \mathbb{R}^d\setminus F)=0, \quad \mbox{and} \quad \tilde{U}_\alpha 1_{[0,\infty)\times \mathbb{R}^d\setminus F}=0 \mbox{ on } F, \quad 0<\alpha \in \mathbb{Q}_+, \mbox{ hence for all } \alpha>0.
\end{equation*}
Therefore, by \Cref{defi:restriction} and \Cref{defi:trivial modification} from \nameref{Appendix}, we can consider the resolvent of Markovian kernels $\overline{\mathcal{U}}^0:=\left(\overline{U}^0_\alpha\right)_{\alpha>0}$ on $[0,\infty)\times \mathbb{R}^d$ obtained by restricting first $\tilde{\mathcal{U}}$ to $F$, and then by taking its trivial extension to $[0,\infty)\times \mathbb{R}^d$.
$\overline{\mathcal{U}}^0$ enjoys all the properties required in the claim. 

\medskip
\noindent{\bf Step OP1.5}: {\it There exists $\mathcal{N}\in \mathcal{B}([0,\infty)\times \mathbb{R}^d)$ with 
\begin{equation*}
    \overline{\mu}(\mathcal{N})=0 \quad \mbox{and} \quad \overline{U}^0_\alpha 1_\mathcal{N}=\overline{U}^N_\alpha 1_{\mathcal{N}\cap [0,N)\times \mathbb{R}^d}=0 \mbox{ on } \left([0,N)\times \mathbb{R}^d\right)\setminus \mathcal{N} \mbox{ for all } \alpha, N>0,
\end{equation*}
such that if we set 
\begin{equation*}
    E:=\left([0,\infty)\times \mathbb{R}^d\right)\setminus \mathcal{N}, \quad E^N:=\left([0,N)\times \mathbb{R}^d\right)\setminus \mathcal{N}, \quad N>0,
\end{equation*}
then for every $N>0$ and $f\in b\mathcal{B}([0,\infty))$ we have
\begin{equation}\label{eq:uniform_U}
\overline{U}^0_\alpha\left((s,x),\cdot\right)|_{E^N}=\overline{U}^N_\alpha\left((s,x),\cdot\right), (s,x)\in E^N, \quad \mbox{ and } \quad \overline{U}^0_\alpha f(s,x)=\int_0^\infty e^{-\alpha t} f(s+t)\;dt, \; (s,x)\in E.
\end{equation}
}

\noindent{To} prove the above assertion, let $(f_k)_{k\geq 1}\in b\mathcal{B}([0,\infty)\times \mathbb{R}^d])$ and $(g_k)_{k\geq 1}\in b\mathcal{B}([0,\infty)\times \mathbb{R}^d])$ be such that the two families of functions separate the space of finite measures on the corresponding spaces.
Thus, using \Cref{coro:identification1},
\begin{equation*}
    \overline{U}_\alpha^0\left(f_k 1_{[0,N)\times \mathbb{R}^d}\right)=  \overline{U}_\alpha^N\left(f_k|_{[0,N)\times \mathbb{R}^d}\right)1_{[0,N)\times \mathbb{R}^d} \quad \overline{\mu}\mbox{-a.e. for all } \alpha, N>0, k\geq 1,
\end{equation*}
so if we set
\begin{equation*}
    \mathcal{N}^0:=\bigcup_{k\geq 1, N,\alpha \in \mathbb{Q}_+^\ast} \left[\overline{U}_\alpha^0\left(f_k 1_{[0,N)\times \mathbb{R}^d}\right)\neq \overline{U}_\alpha^N\left(f_k|_{[0,N)\times \mathbb{R}^d}\right)1_{[0,N)\times \mathbb{R}^d}\right],
\end{equation*}
then $\overline{\mu}(\mathcal{N}^0)=0$; throughout, for $f\in \mathcal{B}([0,N)\times \mathbb{R}^d)$, by $f1_{[0,N)\times \mathbb{R}^d}$ we denote the extension of $f$ to $[0,\infty)\times \mathbb{R}^d$ by setting it to be zero on $[N,\infty)\times \mathbb{R}^d$.
By the same \Cref{coro:identification1}, we get $\overline{\mu}$-a.e. for all  $\alpha>0, k\geq 1$ that
\begin{align*}
    \overline{U}_\alpha^0\left(g_k\right)
    &=  \sup_N\overline{U}_\alpha^N\left(g_k |_{[0,N)}\right)1_{[0,N)\times \mathbb{R}^d}=\sup_{N}\int_0^\infty e^{-\alpha t}g_k|_{[0,N)}(s+t) \;dt \;1_{[0,N)\times \mathbb{R}^d}\\
    &=\int_0^\infty e^{-\alpha t} g_k(s+t) \;dt.
\end{align*}
Thus, if we set 
\begin{equation*}
    \mathcal{N}^{00}:=\bigcup_{k\geq 1,\alpha \in \mathbb{Q}_+^\ast} \left[\overline{U}_\alpha^0 g_k(s,x)\neq \int_0^\infty e^{-\alpha t} g_k(s+t) \;dt\right],
\end{equation*}
then  $\overline{\mu}(\mathcal{N}^{00})=0$.
Now we set $\mathcal{N}':=\mathcal{N}^{0}\cup \mathcal{N}^{00}$ which satisfies $\overline{\mu}(\mathcal{N}')=0$, so we can apply \Cref{lem:kernel_trick} simultaneously for the kernels $\overline{U}^0\alpha, \alpha\in \mathbb{Q}^\ast_+$ to find $\mathcal{N}\in \mathcal{B}([0,\infty)\times \mathbb{R}^d)$ such that 
\begin{equation*}
    \mathcal{N}'\subset \mathcal{N} \quad \mbox{and} \quad \overline{U}^0_\alpha 1_{\mathcal{N}}=0 \mbox{ on } E \mbox{ for all } \alpha \in \mathbb{Q}^\ast_+, \mbox{ hence for all } \alpha>0.
\end{equation*}
Moreover, for every $(s,x)\in E^N, k\geq 1, \alpha\in \mathbb{Q}^\ast_+$ we have
\begin{equation*}
    \overline{U}_\alpha^0\left(f_k 1_{[0,N)\times \mathbb{R}^d}\right)(s,x)=  \overline{U}_\alpha^N\left(f_k|_{[0,N)\times \mathbb{R}^d}\right)(s,x)1_{[0,N)\times \mathbb{R}^d}(s,x),
\end{equation*}
and since $f_k,k\geq 1$ is measure separating the above equality holds with $f_k$ replaced by any $f\in b\mathcal{B}([0,\infty)\times \mathbb{R}^d)$.
Furthermore, by the resolvent equation, it is easy to see that the equality holds for all $\alpha>0$.
Thus, we obtain
\begin{equation*}
    \overline{U}^0_\alpha\left((s,x),\cdot\right)|_{E^N}=\overline{U}^N_\alpha\left((s,x),\cdot\right), \quad (s,x)\in E^N, \alpha>0.
\end{equation*}
By a similar reasoning we also have
\begin{equation*}
    \overline{U}^0_\alpha f(s,x)=\int_0^\infty e^{-\alpha t} f(s+t)\;dt, \quad f\in b\mathcal{B}([0,\infty)), (s,x)\in E, \alpha>0.
\end{equation*}

\medskip
\noindent{\bf Step OP1.6}: {\it Keeping the notations from above, let $\tilde{\mathcal{U}}^0:=\left(\tilde{U}^0_\alpha\right)_{\alpha>0}$ and $\tilde{\mathcal{U}}^N:=\left(\tilde{U}^N_\alpha\right)_{\alpha>0}$ be the restrictions of $\overline{\mathcal{U}}^0$ and  $\overline{\mathcal{U}}^N$ to $E$ and $E^N$, respectively, given by \Cref{defi:restriction}.
Furthermore, set
\begin{equation*}
    V:E\rightarrow [0,1], \quad V(s,x)=e^{-s}, (s,x)\in E.
\end{equation*}
Then $V$ is $\tilde{\mathcal{U}}^0$-excessive.}

\noindent Indeed, for all $(s,x)\in E$, by \eqref{eq:uniform_U} we have
\begin{equation*}
    \alpha \tilde{U}^0_\alpha V(s,x)=\alpha \int_0^\infty e^{-\alpha t} e^{-(s+t)}\;dt=\frac{\alpha}{\alpha+1}e^{-s}\leq e^{-s}, \quad \alpha>0,
\end{equation*}
and we clearly also have $\lim\limits_{\alpha \to \infty} \alpha \tilde{U}^0_\alpha V(s,x)=V(s,x)$. 

\medskip
\noindent{\bf Step OP1.7}: Let $(E_1,\tilde{\mathcal{U}}^1)$ be the saturation of $(E,\tilde{\mathcal{U}}^0)$ (with respect to $\tilde{\mathcal{U}}^0_{\beta}$ for some $\beta>0$) given by \Cref{defi:saturation}.
Consequently, by \Cref{rem:saturation} the function $V$ has a unique extension $V_1$ by fine continuity from $E$ to $E_1$; $V_1$ is $\tilde{\mathcal{U}}^1$-excessive. 
Furthermore, by \Cref{thm:saturation}, there exits a right process 
\begin{equation}\label{eq:X_tilde}
    \tilde{X}=(\tilde{\Omega}, \tilde{\mathcal{F}}, \tilde{\mathcal{F}}_t , \tilde{X}(t), \tilde{\theta(t)} , \tilde{\mathbb{P}}_z,z\in E_1)
\end{equation}
on $E_1$ with resolvent $\tilde{\mathcal{U}}^1$.

Let $\tilde{\mathcal{U}}^{1,N}$ be the resolvent of the process $\tilde{X}$ killed upon leaving the finely open set $[V_1>e^{-N}]$, $N>0$.
That is
\begin{equation}\label{eq:killed_U_tilde}
    \tilde{\mathcal{U}}^{1,N}=\tilde{\mathcal{U}}^1-B_{\alpha}^{[V_1\leq e^{-n}]}\tilde{\mathcal{U}}^1 \quad \mbox{on } [V_1>e^{-N}], \alpha>0,
\end{equation}
where recall that the balayage operator $B_{\alpha}^A$ is defined in \Cref{defi:balayage}; see also \eqref{Hunt-th}.

\medskip
\noindent{\it We claim that for every $N>0$, the resolvent $\tilde{\mathcal{U}}^{1,N}$ is an extension of $\tilde{\mathcal{U}}^{N}$ from $E^N$ to $[V_1>e^{-N}]$, in the sense of \Cref{defi:natural}.}

\medskip
\noindent To prove the claim, we need to check conditions $(i.1)-(i.2)$ from \Cref{defi:natural}, namely that 

\medskip
$\tilde{U}_{\alpha}^{1,N}\left(z,[V_1>e^{-N}]\setminus E^N\right)=0$, $z\in [V_1>e^{-N}]$,  
and that

\medskip 
$\tilde{U}_\alpha^{1,N}((s,x),\cdot)|_{E^N}=\tilde{U}_\alpha^{N}((s,x),\cdot)$ as measures on $E^N$, for every $(s,x)\in E^N$ and $N>0$.

\medskip
\noindent{}To check the first condition, note that
\begin{equation*}
    E\cap \left([V_1>e^{-N}]\setminus E^N\right)=\mathcal{N}\cap [0,N)\times \mathbb{R}^d, \quad \mbox{where } \mathcal{N} \mbox{ is given in Step OP1.5.}
\end{equation*}
Thus, $\tilde{U}_{\alpha}^{1,N}\left((s,x),E\cap\left([V_1>e^{-N}]\setminus E^N\right)\right)\leq \tilde{U}_{\alpha}^{1}\left((s,x),\mathcal{N}\right)=0$, $(s,x)\in E^N$, hence
\begin{equation*}
    \tilde{U}_{\alpha}^{1,N}\left([V_1>e^{-N}]\setminus E^N\right)\leq  \tilde{U}_{\alpha}^{1,N}\left(E_1\setminus E\right)\leq \tilde{U}_{\alpha}^{1}\left(E_1\setminus E\right)=0. 
\end{equation*}

Checking the second condition is a little bit more involved: Let $N>0$ and $(s,x)\in E^N$, so that $s<N$.
Note that
\begin{equation*}
\tilde{U}_\alpha^1\left((s,x),\cdot\right)|_{E^N}=\tilde{U}_\alpha^0\left((s,x),\cdot\right)|_{E^N}=\tilde{U}_\alpha^N\left((s,x),\cdot\right), \quad \alpha>0,
\end{equation*}
where the last equality has been obtained in \eqref{eq:uniform_motion}.
Moreover, since 
\begin{equation*}
\tilde{U}_\alpha^{1,N}=\tilde{U}_\alpha^{1}-B_\alpha^{[V_1\leq e^{-N}]}\tilde{U}_\alpha^{1},    
\end{equation*}
it is sufficient to show that
\begin{equation}\label{eq:BU=0}
    B_\alpha^{[V_1\leq e^{-N}]}\tilde{U}_\alpha^{1}\left((s,x), \cdot\right)|_{E_N}=0, \quad N>0.
\end{equation}
To this end, let $F\in b\mathcal{B}(E_N)$ with compact support in $E_N$, in particular
\begin{equation*}
    {\rm supp}(F)\subset [0,N-\varepsilon)\times \mathbb{R}^d \quad \mbox{ for some } \varepsilon>0.
\end{equation*}
We claim that
\begin{equation}\label{eq:UF=0}
    \tilde{U}_\alpha^1F=0 \quad \mbox{ on } [V_1\leq e^{-N}], \quad N>0,
\end{equation}
hence $B_\alpha^{[V_1\leq e^{-N}]}\tilde{U}_\alpha^{1}F(s,x)\leq \tilde{U}_\alpha^{1}F(s,x)=0$; from this, by a straightforward density argument, we get \eqref{eq:BU=0}.
To prove \eqref{eq:UF=0}, let $z\in [V_1\leq e^{-N}]$, so that by the fine density of $E$ in $E_1$, there exists a sequence $\left(z_k=(s_k,x_k)\right)_{k\geq 1}\subset E$ such that
\begin{equation*}
    e^{-s_k}=V(s_k,x_k)=V_1(s_k,x_k)\rightarrow_{k} V_1(z)\leq e^{-N}.
\end{equation*}
Also,
\begin{equation*}
    \tilde{U}^1_\alpha F(z)=\lim_k \tilde{U}^1_\alpha F(z_k)=\lim_k \tilde{U}^0_\alpha F(z_k).
\end{equation*}
Thus, if we set
\begin{equation*}
    f(s):=\sup_{x \in \mathbb{R}^d}F(s,x), \quad s\in [0, N),
\end{equation*}
and take into account that $\lim_k s_k\geq N$ and that $f(r)=0, r\geq N-\varepsilon$, we get
\begin{equation*}
    \tilde{U}^1_\alpha F(z) \leq \limsup_k \tilde{U}_\alpha^0 f(s_k,x_k)= \limsup_k \int_0^\infty e^{-\alpha t} f(s_k+t) \;dt = 0,
\end{equation*}
which concludes the claim.

\medskip
\noindent{\bf Step Q1.8}: Let $\left(\tilde{\mathcal{U}}^{N,1},E^{N,1}\right)$ be the saturation of $\left(\tilde{\mathcal{U}}^{N},E^{N}\right)$, $N\geq 1$.
Then, by {\bf Step OP1.7} and \Cref{rem:maximal}, we have that
$\left(\tilde{\mathcal{U}}^{N,1},E^{N,1}\right)$ is an extension of  $\left(\tilde{\mathcal{U}}^{1,N},[V_1>e^{-N}]\right)$, and that the latter is an extension of $\left(\tilde{\mathcal{U}}^{N},E^{N}\right)$, in the sense of \Cref{defi:natural}. 
Or, using the notation introduced in \Cref{defi:natural}, we have
\begin{equation}\label{eq:sub1}
\left(\tilde{\mathcal{U}}^{N},E^{N}\right)\subset  \left(\tilde{\mathcal{U}}^{1,N},\left[V_1>e^{-N}\right]\right)\subset   \left(\tilde{\mathcal{U}}^{N,1},E^{N,1}\right).
\end{equation}

Further, set
\begin{equation*}
    \hat{E}^N:=\left\{ (s,x)\in [0,N)\times \mathbb{R}^d : \overline{U}_0^N\left((s,x), [0,N)\times \mathbb{R}^d\setminus E^N\right)=0\right\}.
\end{equation*}
Note that $\hat{E}^N\supset E^N$ and $\hat{E}^N\in \mathcal{A}\left(\overline{\mathcal{U}}^N\right)$, $N\geq 1$; see \Cref{ss:restriction_process}.
Let
\begin{equation*}
\hat{\mathcal{U}}^N=\left(\hat{U}_\alpha^N\right)_{\alpha >0}:=\overline{\mathcal{U}}^N|_{\hat{E}^N}, \quad N\geq 1,
\end{equation*}
the restriction of $\overline{\mathcal{U}}^N$ to $\hat{E}^N$ given by \Cref{defi:restriction}.
Thus, we also have
\begin{equation}\label{eq:sub2}
\left(\tilde{\mathcal{U}}^{N},E^{N}\right)\subset  \left(\hat{\mathcal{U}}^{N},\hat{E}^N\right)\subset
\left(\tilde{\mathcal{U}}^{N,1},E^{N,1}\right).
\end{equation}
Further, set 
\begin{equation*}
\check{E}^N:=\hat{E}^N\cap [V_1>e^{-N}] \quad \mbox{in } E^{N,1},\quad \check{\mathcal{U}}^N:=\tilde{\mathcal{U}}^{N,1}|_{\check{E}^N}, \quad N>0.
\end{equation*}
Thus, we now have
\begin{align}\label{eq:sub22}
&\left(\tilde{\mathcal{U}}^{N},E^{N}\right)\subset  \left(\check{\mathcal{U}}^{N},\check{E}^N\right)\subset   
\left(\tilde{\mathcal{U}}^{1,N},\left[V_1>e^{-N}\right]\right)\subset
\left(\tilde{\mathcal{U}}^{N,1},E^{N,1}\right),\\
&\left(\tilde{\mathcal{U}}^{N},E^{N}\right)
\subset \left(\check{\mathcal{U}}^{N},\check{E}^N\right)
\subset \left(\hat{\mathcal{U}}^{N},\hat{E}^N\right)
\subset \left(\tilde{\mathcal{U}}^{N,1},E^{N,1}\right).
\end{align}
Moreover, note that $E^{N,1}\setminus [V_1>e^{-N}]$ and $E^{N,1}\setminus \hat{E}^N$ are polar sets in $E^{N,1}$; this follows by \Cref{thm:saturation}, since both resolvents $\hat{\mathcal{U}}^{N}$ and $\tilde{\mathcal{U}}^{1,N}$ are the resolvents of two right processes on the corresponding spaces.
Consequently, for every $N\geq 1$,
\begin{equation*}
E^{N,1}\setminus\check{E}^N \mbox{ is polar for } \tilde{\mathcal{U}}^{N,1},\quad \left[V_1>e^{-N}\right]\setminus \hat{E}^N \mbox{ is polar for } \tilde{\mathcal{U}}^{1,N}, \quad \mbox{and} \quad \hat{E}^N\setminus \check{E}^N \mbox{ is polar for } \hat{\mathcal{U}}^{N}. 
\end{equation*}

\medskip
\noindent{\bf Step OP1.9}: Recall that $X^N$ is a right process associated to $\overline{\mathcal{U}}^N$ on $[0,N)\times \mathbb{R}^d$, as in \Cref{thm:Stannat}, \Cref{coro:identification1}, and \Cref{coro:identification2}.
Also, since $\check{E}^N\subset [0,N)\times \mathbb{R}^d$ and $\check{E}^N\in \mathcal{A}\left(\overline{\mathcal{U}}^N\right)$, we can consider 
\begin{equation}\label{eq:restrictoin_hat}
    \check{X}^N:=X^N|_{\check{E}^N},
\end{equation}
the restriction of $X^N$ from $[0,N)\times \mathbb{R}^d$ to $\check{E}^N$, which is a right process on $\check{E}^N$ with lifetime $\zeta^N$ and (norm-)continuous paths, as it follows from \Cref{prop:restriction} and \Cref{defi:restr-process}.
Clearly, the resolvent of $\check{X}^N$ is $\tilde{\mathcal{U}}^{N}$, $N\geq 1$.

Also, let 
$\tilde{X}^{1,N}$ be the right process on $\check{E}^N$
obtained by first killing the right process $\tilde{X}$ upon leaving $\left[V_1>e^{-N}\right]$ and then taking the restriction to $\check{E}^N$, $N\geq 1$; this is done following the procedures in \Cref{ss:restriction_process}.
Thus, the resolvent of $\tilde{X}^{1,N}$ is precisely $\tilde{\mathcal{U}}^{1,N}|_{\check{
E}^N}=\hat{\mathcal{U}}^N, N\geq 1$.
Now, since the norm topology on $\check{E}^N$ is a natural topology with respect to $\check{\mathcal{U}}^N$, the latter being the resolvent of both processes $\check{X}^{N}$ and $\tilde{X}^{1,N}$, by \Cref{prop:tau-identification} we get
\begin{equation}\label{eq:same_law}
    \tilde{X}^{1,N} \mbox{ and } \check{X}^N \mbox{ have the same laws on the path space } C\left([0,\infty); \check{E}^N\right), 
\end{equation}
with $\check{E}^N$ endowed with the norm-topology, and for all starting points $(s,x)\in \check{E}^N, N\geq 1$.

Consequently, relying also on \Cref{coro:identification1}, 
{\it the right process $\tilde{X}$ given by \eqref{eq:X_tilde} satisfies for all $N>0$
\begin{align}
    &\tilde{\mathbb{P}}_{\delta_s\otimes \mu_s}\left( T_{\left[V_1\leq e^{-N}\right]}=N-s\right)
    =\mathbb{P}_{\delta_s\otimes \mu_s}\left( \zeta^N=N-s \right)=1, \quad s\in (0,N),\\
    &\tilde{\mathbb{P}}_{\delta_s\otimes \mu_s}\left( [0, N-s)\ni t\mapsto\tilde{X}(t)\ni \check{E}^N \mbox{ is well defined and norm-continuous }\right) =1, \quad s\in (0,N).\label{eq:norm-continuous}\\
    &\tilde{\mathbb{P}}_{(s,x)}\circ \left(\tilde{X}(t)\right)^{-1}
    =\mathbb{P}_{(s,x)}\circ \left(X^N(t)\right)^{-1},\; s\in(0,N),t\in (0,N-s), \mu_s\mbox{-a.e. } (s,x)\in \check{E}^N.\label{eq:ident_tilde_N}
\end{align}
}

\medskip
\noindent{\bf Step OP1.10}: Define 
\begin{equation*}
   \check{E}:=\mathop{\bigcup}\limits_{N\geq 1} \check{E}^N, \quad \mbox{ so that we clearly have } \overline{\mu}\left([0,\infty)\times \mathbb{R}^d\setminus \check{E}\right)=0.
\end{equation*}
By {\bf Step OP1.9}, we have
\begin{align}
    &\tilde{\mathbb{P}}_{\overline{\mu}}\left( \sup_{N}T_{\left[V_1\leq e^{-N}\right]}=\infty \right)=1,\\
    &\tilde{\mathbb{P}}_{\overline{\mu}}\left( [0, \infty)\ni t\mapsto\tilde{X}(t)\ni \check{E} \mbox{ is well defined and norm-continuous }\right) =1.\label{eq:norm-continuous2}
\end{align}
In particular, if we extend $\overline{\mu}$ from $E$ to $E_1$ by setting $\overline{\mu}(E_1\setminus E):=0$, then $E_1\setminus \check{E}$ is $\overline{\mu}$-polar and $\overline{\mu}$-negligible (with respect to $\mathcal{U}^1$), hence by \Cref{rem 2.2}, (iii), there exists a further $\mathcal{B}([0,\infty)\times \mathbb{R}^d)$-measurable set $\check{E}_0\subset \check{E}$ such that
$E_1\setminus \check{E}_0$ is $\overline{\mu}$-inessential (see \Cref{defi:smallsets}).

By \Cref{prop:restriction}, $\tilde{X}$ can be restricted to $\check{E}_0$, the restricted process being denoted by $\check{X}^0$.

Furthermore, note that each $\check{E}^N$ is finely open in $E_1$, since it is finely open in $[V_1>e^{-N}]$, and the latter is finely open in $E_1$.
Consequently, by \eqref{eq:same_law}, if $f\in C_b([0,\infty)\times \mathbb{R}^d)$ then 
\begin{equation*}
    [0,\infty)\ni t\mapsto f(\check{X}^0(t))\in \mathbb{R} \mbox{ is (right) continuous at } t=0,\quad  \check{\mathbb{P}}_{(s,x)}\mbox{-a.s. for all } (s,x)\in \check{E}_0,
\end{equation*}
hence by \Cref{rem:fine_continuity_0} we easily deduce that the norm topology on $\check{E}^0$ is a natural topology with respect to the restriction $\mathcal{U}^1|_{\check{E}^0}$.
We can now apply \eqref{eq:norm-continuous2} and \Cref{prop:continuous_modification} to deduce that there exists another subset $E
^\dagger\subset\check{E}_0$ set with $\check{E}_0\setminus E
^\dagger$ $\overline{\mu}$-inessential such that if we restrict the process $\check{X}_0$ to $E
^\dagger$, then it has $\check{\mathbb{P}}^{(s,x)}$-a.s. norm-continuous paths in $E
^\dagger$ for every $(s,x)\in E
^\dagger$.
We denote this last right process with by 
\begin{equation}\label{eq:X^dagger}
    X^\dagger=\left(\mathbb{P}_{(s,x)}^\dagger, (s,x)\in E^\dagger, X^\dagger(t)=(X^\dagger_1(t),X^\dagger_2(t),t\geq 0)\right), \quad \mbox{with resolvent } \mathcal{U}^\dagger.
\end{equation}
Note that $X^\dagger$ has $\mathbb{P}_{(s,x)}^\dagger$-a.s. norm-continuous paths for every $(s,x)\in E^\dagger$, and it is conservative since its resolvent $\mathcal{U}^\dagger$ is the restriction of $\tilde{\mathcal{U}}^1$ from $E^1$ to $E^\dagger$, hence
\begin{equation}\label{eq:X_dagger_conservative}
    \alpha U^\dagger_\alpha 1(x)=\alpha \tilde{U}^1_\alpha 1_{E^\dagger}(x)=\alpha \tilde{U}^1_\alpha 1(x)=1, \quad x\in E^\dagger.
\end{equation}

\medskip
\noindent{\bf Step OP1.11}: 
{\it The right process $X^\dagger$ satisfies
\begin{equation}\label{eq:uniform_dagger}
    \mathbb{P}_{(s,x)}^\dagger\left( X^\dagger_1(t)=t+s, t\geq 0\right) =1, \quad (s,x)\in E^\dagger.
\end{equation}
}
Indeed, we can argue as follows: First, note that by \eqref{eq:uniform_U} and the fact that $\tilde{\mathcal{U}}^0$ is the restriction of $\overline{\mathcal{U}}^0$ from $[0,\infty)\times \mathbb{R}^d$ to $E$, we have 
\begin{equation*}
    \tilde{U}^0_\alpha f|_E(s,x)=\int_0^\infty e^{-\alpha t} f(s+t)\;dt, \; (s,x)\in E, f\in b\mathcal{B}([0,\infty)).
\end{equation*}
Since $E$ is finely dense in $E_1$, it is straightforward that, for all $(s,x)\in E^1$,
    \begin{equation*}
        \tilde{U}^1_\alpha \tilde{f}(s,x)=\int_0^\infty e^{-\alpha t} f(s+t)\;dt, \quad f\in b\mathcal{B}([0,\infty)), 
    \end{equation*}
    where $\tilde{f}$ is any bounded and $\mathcal{B}_1$-measurable extension of $f|_{E}$ to $E_1$.
    Consequently, since $\mathcal{U}^\dagger$ is the restriction of $\mathcal{U}^1$ to $E^\dagger$, and $E^\dagger\in \mathcal{B}\left([0,\infty)\times \mathbb{R}^d\right)$, we get
    \begin{equation*}
        U^\dagger_\alpha f|_{E^\dagger}(s,x)=\int_0^\infty e^{-\alpha t} f(s+t)\;dt, \quad f\in b\mathcal{B}([0,\infty)), (s,x)\in E^\dagger.
    \end{equation*}
It is now straightforward to deduce \eqref{eq:uniform_dagger}.

\medskip
\noindent{\bf Step OP1.12}: 
{\it For every $N\geq 1$, $s\in (0,N)$ and $t\in [0,N-s)$ we have the identification
    \begin{equation}\label{eq:identification_hat}
        \mathbb{P}_{(s,x)}^\dagger\circ \left(X^\dagger(t)\right)^{-1}=\mathbb{P}^N_{(s,x)}\circ \left(X^N(t)\right)^{-1}, \quad \mu_s\mbox{-a.e. } (s,x)\in E^\dagger\cap \check{E}^N.
    \end{equation}
    }
Indeed, note first that by the previous steps, 
\begin{equation*}
    E^\dagger\subset \check{E}
    =\bigcup\limits_{N\geq 1}\check{E}^N\subset [0,\infty)\times \mathbb{R}^d, \quad E^\dagger\in \mathcal{A}(\mathcal{U}^1), \quad X^\dagger=\tilde{
    X}|_{E^\dagger}.
\end{equation*}
Hence, using also \eqref{eq:ident_tilde_N} we get the identification
\begin{equation*}
    \mathbb{P}_{(s,x)}^\dagger\circ \left(X^\dagger(t)\right)^{-1}
    =\tilde{\mathbb{P}}_{(s,x)}\circ \left(\tilde{X}(t)\right)^{-1}
    =\mathbb{P}^N_{(s,x)}\circ \left(X^N(t)\right)^{-1}, \quad \mu_s\mbox{-a.e. } (s,x)\in E^\dagger\cap \check{E}^N.
\end{equation*}
    
\medskip
\noindent{\bf Step OP1.13}: 
{\it We have $\overline{\mu}([0,\infty)\times \mathbb{R}^d \setminus E^\dagger)=0$. 
    Furthermore, for every $s>0$ let us set 
    \begin{equation}\label{eq:hat(s)}
        E^\dagger(s):=\left\{ x\in \mathbb{R}^d : (s,x)\in E^\dagger\right\}.
    \end{equation}
    Then
    \begin{equation}\label{eq:indentification_hat_mu}
         \mu_s(E^\dagger(s))=1 \mbox{ for all } s\in (0,\infty) \quad \mbox{and}\quad \mathbb{P}^\dagger_{\delta_s\otimes\mu_s}\circ \left(X^\dagger(t)\right)^{-1}=\delta_{t+s}\otimes\mu_{t+s} \quad \mbox{ for every } s> 0 \mbox{ and } t\geq 0.
    \end{equation}
    }
    Indeed, using {\bf Step OP1.10},
    \begin{equation*}
        \overline{\mu}([0,\infty)\times \mathbb{R}^d \setminus E^\dagger)
        = \overline{\mu}([0,\infty)\times \mathbb{R}^d \setminus\check E)
        =0.
    \end{equation*}
    Moreover, from \eqref{eq:identification} and  \eqref{eq:identification_hat} we obtain that
    \begin{equation*}
         \mathbb{P}^\dagger_{\delta_s\otimes\mu_s}\circ \left(X^\dagger(t)\right)^{-1}=\delta_{t+s}\otimes\mu_{t+s} \quad \mbox{ for every } s>0, t\geq 0.
    \end{equation*}

\medskip
\noindent{\bf Step OP1.14}: 
By construction, it is clear that $\overline{\mu}$ is sub-invariant for the transition function $\left(P^\dagger_t\right)_{t\geq 0}$ of the process $X^\dagger$, hence $\left(P^\dagger_t\right)_{t\geq 0}$ can be extended to a $C_0$-semigroup $\left(T_t^\dagger\right)_{t\geq 0}$ on every $L^p\left([0,\infty)\times \mathbb{R}^d;\overline{\mu}\right)$, $1\leq p<\infty$.
Moreover, the resolvent of $\left(T^\dagger_t\right)_{t\geq 0}$ is precisely $\overline{\mathcal{U}}$ constructed at {\bf Step OP1.2}.
Consequently,
\begin{equation*}
    P^\dagger_tf=f+\int_0^t P^\dagger_s \overline{{\sf L}}f \;ds \quad \overline{\mu}\mbox{-a.e. for all } t\geq 0 \mbox{ and } f\in C_c^\infty\left([0,\infty)\times \mathbb{R}^d\right). 
\end{equation*}

\medskip
\noindent{\bf Step OP1.15}: 
Let $s>0$ and $\mathcal{P}(\mathbb{R}^d)\ni\nu_s\leq c\mu_s$ for some constant $c\in (0,\infty)$, and set
    \begin{equation*}
        \nu_{t+s}:=\mathbb{P}^\dagger_{\delta_s\otimes\nu_s}\circ \left(X^\dagger_2(t)\right)^{-1}, \quad t\geq 0.
    \end{equation*}
    {\it We claim  that for all $N>s$, the family $\left(\nu_{t+s}\right)_{t\in [0,N-s)}$ is the unique weakly continuous solution to the LFPE \eqref{eq:LFPshort} on $[s,N)$ which consists of probability measures, belongs to $\mathcal{A}_{\lesssim u}^{s,N}$, and whose initial condition (at time $s$) is $\nu_s$.}

    Indeed, the fact that $\left(\nu_{t+s}\right)_{t\in [0,N-s)}$ is a weakly continuous curve in $\mathcal{P}(\mathbb{R}^d)$ is now clear since $\left(X^\dagger_2(t)\right)_{t\geq 0}$ is conservative and has a.s. continuous paths.
    
    The fact that $\left(\nu_{t+s}\right)_{t\in [0,N-s)}$ is a solution to the LFPE \eqref{eq:LFPshort} on $[s,N)$ follows e.g. similarly to the proof of \Cref{prop:densities}, (iii).
     Finally,
    \begin{equation*}
    \nu_{t+s}=\int_{\mathbb{R}^d}\mathbb{P}^\dagger_{s,x}\circ \left(X^\dagger_2(t)\right)^{-1} \frac{d\nu_s}{d\mu_s} \mu_s(dx)\leq c\mathbb{P}^\dagger_{\delta_s\otimes\nu_s}\circ \left(X^\dagger_2(t)\right)^{-1}= c\mu_{t+s}, \quad t\geq 0.
    \end{equation*}
Consequently, the uniqueness part also follows since condition $\mathbf{H_{\sf \lesssim u}}$ is in force.

\medskip
\noindent{\bf Step OP1.16: }{\bf Proof of \Cref{thm:main0},(ii).} 
Let $E^\dagger$ and $X^\dagger$ defined in \eqref{eq:X^dagger}.
Note that $X^\dagger$ is a conservative norm-continuous right process on $E$, so it satisfies (ii.1) from \Cref{thm:main0} due to \eqref{eq:indentification_hat_mu}.

\medskip
Concerning assertion (ii.2), note that by \Cref{thm 4.6}, (i), we have that the (trace of the) norm topology on $E^\dagger$ is a natural topology, hence for any $f\in C_b([0,\infty)\times\mathbb{R}^d)$ we have that $f|_{E^\dagger}$ is finely continuous, hence $P^\dagger_tf|_{E^\dagger}$ is finely continuous due to \Cref{rem:finely feller}.
Furthermore, recall that $\mathcal{U}^\dagger$ is a $\overline{\mu}$-version of $\overline{\mathcal{U}}$, hence by {\bf Step OP1.3}, (ii.2) is as well satisfied by $X^\dagger$.

\medskip
Assertion (ii.3) is also fulfilled by $X^\dagger$ due to the conclusion in {\bf Step OP1.10}.

\medskip
Assertion (ii.4) with $X^\dagger$ instead of $X$ follows by \eqref{eq:uniform_dagger} and \eqref{eq:indentification_hat_mu}.

\medskip
Regarding assertion (ii.5), let us notice first that by a straightforward density argument we have that $C_c^2([0,\infty)\times \mathbb{R}^d)\in {\sf D(\overline{L}_1)}$.
Let $f\in C_c^2([0,\infty)\times \mathbb{R}^d)$ and set
\begin{equation}\label{eq:Mf}
    M_f(t):=f(X^\dagger(t))-f(X^\dagger(0))-\int_0^{t} \overline{\sf L}f(X^\dagger(r))\; dr, \quad t\geq 0.
\end{equation}
Let us define $\overline{\mu}_1\in \mathcal{P}([0,\infty)\times \mathbb{R}^d)$ by
\begin{equation*}
    \overline{\mu}_1:=e^{-t}u(t,x)\;dt\otimes dx \quad \mbox{on } [0,\infty)\times \mathbb{R}^d.
\end{equation*}
Notice that by \Cref{thm:main0}, (i)-(ii.2), proved above, we have
\begin{equation}\label{eq:Ufinite}
    \overline{\mu}_1\left(U^\dagger_\alpha |\overline{\sf L}f|\right)<\infty \quad \mbox{for every } \alpha> 0,
\end{equation}
and also that $(P^\dagger_tf)_{t\geq 0}$ solves the Cauchy problem for $\left(\overline{\sf L}_1,\sf{D}(\overline{\sf{L}}_1)\right)$ on $L^p\left([0,\infty)\times \mathbb{R}^d ; \overline{\mu}\right)$.
Consequently, using also the continuity of $t\mapsto M(t)$ we easily get
\begin{equation*}
    \mathbb{E}^\dagger_{s,x}\left\{M_f(t)\right\}=0,\quad t\geq 0,\quad (s,x)\in E^\dagger, \; \mathbb{P}^\dagger_{\overline{\mu}_1}\mbox{-a.e.}, 
\end{equation*}
and thus,
\begin{equation*}
    \mathbb{E}^\dagger_{\tilde{\nu}}\left\{M_f(t)\right\}=0,\quad t\geq 0,\quad \mbox{for every } \tilde{\nu}\in \mathcal{P}(E^{\dagger}) \mbox{ for which there exists } c>0 \mbox{ such that } \tilde{\nu}\leq c\overline{\mu}_1.
\end{equation*}
Now, by \Cref{prop:martingale_absorbing}, there exists a subset $E_f\subset E^\dagger$ which is $\mu$-inessential and such that
\begin{equation}\label{eq:M_f}
    \left(M_f(t)\right)_{t\geq 0}\quad  \mbox{is an } \left(\mathcal{F}_t^{\dagger}\right)\mbox{-martingale under } \mathbb{P}^\dagger_{\delta_{(s,x)}} \mbox{ for every } (s,x)\in E_f.
\end{equation}

Furthermore, for a countable subset $\mathcal{T}\subset C_c^\infty((0,\infty)\times \mathbb{R}^d)$ which is dense with respect to the uniform convergence up to (and including) the second derivatives,
set
\begin{equation*}
    E:=\left(\bigcap\limits_{f\in \mathcal{T}} E_f\right)  \cap \left(\bigcap\limits_{n}\left[U^\dagger_1 |\overline{\sf L}\psi_n|<\infty\right]\right),
\end{equation*}
where $(\psi_n)_{n\geq 1}\subset C_c^\infty((0,\infty)\times\mathbb{R}^d)$ is such that 
\begin{equation*}
   \psi_n\geq 0 \quad \mbox{and} \quad \psi_n=1 \mbox{ on }[1/n,n]\times B(0,n), \quad n\geq 1.
\end{equation*}
Note that by \eqref{eq:Ufinite}, \Cref{rem:exc_absorbing}, and \Cref{rem:absorbing_intersection}, we have that $\left[U^\dagger_1 |\overline{\sf L}\psi_n|=\infty\right]$ is $\tilde{\mu}$-inessential for every $f\in \mathcal{T}$,
as well as $E^\dagger\setminus E$.

Now, we take $X$ to be the restriction of $X^\dagger$ from $E^\dagger$ to $E$ in the sense of \Cref{defi:restr-process}.
It is clear that (ii.1)-(ii.4) are still fulfilled, whilst (ii.5) follows easily from \eqref{eq:M_f} and the above mentioned density of $\mathcal{T}$ in $C_c^\infty((0,\infty)\times \mathbb{R}^d)$.

\medskip
\noindent{(ii.6)} Let $s>0$ and $\nu_s\in \mathcal{P}(\mathbb{R}^d)$ such that $\nu_s\leq c \mu_s$ for some finite $c>0$.
By \eqref{eq:bau} and \eqref{eq:indentification_hat_mu} we easily get that
\begin{equation*}
    \mathbb{E}^\dagger_{\delta_s\otimes\mu_s}\left\{|M(t)|\right\}<\infty, \quad t\geq 0.
\end{equation*}
Now, by assertion (ii.5) and the notation $M_f$ introduced in \eqref{eq:M_f}, we deduce that 
\begin{equation*}
    M_f \mbox{ is an } (\mathcal{F}(t))\mbox{-martingale under } \mathbb{P}_{\delta_s\otimes\nu_s}, \quad f\in C_c^\infty((0,\infty)\times \mathbb{R}^d).
\end{equation*}
Now, assertion (ii.6) follows from {\bf Step OP1.15}.

\subsubsection{Proof of Theorem 2.15}\label{ss: proof of main1}
Because the process ensured by \Cref{thm:main0} requires some modifications, let us use the notation 
\begin{equation*}
X^\dagger=\left(\Omega^\dagger,\mathcal{F}^\dagger,\mathcal{F}^\dagger(t),X^\dagger(t),\mathbb{P}^\dagger_{s,x}, (s,x)\in E^\dagger, t\geq 0\right)    
\end{equation*}
for the process constructed in \Cref{thm:main0}, with transition function denoted by $\left(P^\dagger_t\right)_{t\geq 0}$ and resolvent $\mathcal{U}^\dagger=\left(U^\dagger_\alpha\right)_{\alpha>0}$.
Note that we have also used the notation $X^\dagger$ to define the process from \eqref{eq:X^dagger} in {\bf Step OP1.10}; it is useful to note that the process we now denote by $X^\dagger$, namely the one furnished by \Cref{thm:main0}, is in fact the restriction (in the sense of \Cref{defi:restr-process}) of $X^\dagger$ from \eqref{eq:X^dagger}, so it inherits all the properties deduced for $X^\dagger$ in all steps that follows after (including) {\bf Step OP1.10}.

Proving that in the conclusions (ii.1), (ii.4) and (ii.7) from \Cref{thm:main0} we can allow $s=0$, if condition $\mathbf{H_0^{TV}}$ is in force, is highly nontrivial. 
We show this systematically, in what follows.

\medskip
\noindent{\bf Step OP2.1}:  
Firs of all, since the uniqueness assumption $\mathbf{H_{\sf \lesssim u}}$ holds, one can apply \cite[Lemma 2.12]{Tr16} to deduce that the family $\left(\eta^N\right)_{N>0}$ given by \Cref{rem:Trevisan_superposition} is a consistent family (see \cite{Pa67} for the notion), hence by the general theory (see again \cite[Theorem 4.1]{Pa67}) admits a unique projective limit $\eta \in \mathcal{P}\left(C([0,\infty);\mathbb{R}^d)\right)$, with one dimensional marginals $(\mu_t)_{t\in [0,\infty)}$.
Let 
\begin{equation*}
(\eta^x)_{x\in \mathbb{R}^d} \subset \mathcal{P}\left(C([0,\infty);\mathbb{R}^d)\right)    
\end{equation*}
be the Borel family of probability measures obtained by disintegrating the law $\eta \in \mathcal{P}\left(C([0,\infty);\mathbb{R}^d)\right)$  with respect to $\mu_0$, in the same spirit of \Cref{rem:Trevisan_superposition}. 
Also, let 
\begin{equation*}
    (\eta^x_t)_{t\in [0,\infty)} \mbox{ be the one dimensional marginals of } \eta^x \mbox{ for every } x\in \mathbb{R}^d.
\end{equation*}

\medskip
\noindent{\bf Step OP2.2}: {\it We claim that there exists $M\in \mathcal{B}(\mathbb{R}^d)$ such that $\mu_0(M)=1$ and for all $x\in M$ and $f\in b\mathcal{B}\left([0,\infty)\times \mathbb{R}^d\right)$
\begin{equation}\label{eq:eta_P}
    \eta_s^x\left(P^\dagger_{t}f|_{E^\dagger}(s,\cdot)\right)=\eta_{s+t}^x\left(f(s+t,\cdot)\right),\quad \; ds\otimes dt\mbox{-a.e. } (s,t)\in (0,\infty)^2.
\end{equation}
}
To prove the claim, note first that by a monotone class argument, it is sufficient to fix $f\in b\mathcal{B}\left([0,\infty)\times \mathbb{R}^d\right)$ and prove \eqref{eq:eta_P} with $M$ possibly depending on $f$.
Further, let $\nu_0 \in \mathcal{P}(\mathbb{R}^d)$ such that $\nu_0\leq c  \mu_0$ for some finite constant $c$, and define $\eta^{\nu} \in\mathcal{P}\left(C([0,\infty];\mathbb{R}^d)\right)$ by
\begin{equation}\label{eq:eta_nu}
    \eta^{\nu_0}(A):=\int_A \eta^x(A) \nu_0(dx), \quad A\in \mathcal{B}\left(C([0,\infty];\mathbb{R}^d)\right).
\end{equation}
Consequently, for every $s>0$ we have $\eta^{\nu_0}_s\leq c\mu_s$, and hence by \Cref{thm:main0}, (ii.7) we have that
\begin{equation*}
    \nu_{t+s}:=\mathbb{P}^\dagger_{\delta_s\otimes\eta_s^{\nu_0}}\circ \left(X^\dagger_2(t)\right)^{-1}, \quad t\geq 0,
\end{equation*}
is the unique weakly continuous solution to the LFPE \eqref{eq:LFPshort} on $[s,N)$ which consists of probability measures, belongs to $\mathcal{A}_{\lesssim u}^{s,N}$, and whose initial condition (at time $s$) is $\eta_s^{\nu_0}$, for all $N>s$.
But  $\left(\eta_{t+s}^{\nu_0}\right)_{t\geq 0}$ is also such a solution (see \Cref{rem:Trevisan_superposition}, in fact \cite[Proposition 2.7]{Tr16}), so we must have
\begin{equation}\label{eq:nu=eta}
    \nu_{s+t}=\eta_{s+t}^{\nu_0}, \quad t\geq 0,
\end{equation}
hence, using also \Cref{thm:main0}, (ii.4),
\begin{equation*}
\eta_s^{\nu_0}\left(P^\dagger_{t}f|_{E^\dagger}(s,\cdot)\right)= \mathbb{E}^{\dagger,\delta_s\otimes \eta^{\nu_0}_s}\left\{f|_{E^\dagger}(X^\dagger(t))\right\}=\eta_{s+t}^{\nu_0}\left(f(s+t,\cdot)\right), \quad s>0,t\geq 0.
\end{equation*}
Since $\nu_0$ is arbitrarily chosen such that $\nu_0\leq c  \mu_0$, the claim follows.

\begin{rem}
    It is worth to point out that relation \eqref{eq:eta_P} is not obtained for every $s,t>0$ because we do not know whether the left hand side of the equality from \eqref{eq:eta_P} is continuous in both arguments $(s,t)$, the challenging continuity being the one in the $s$ variable.
\end{rem}

\medskip
\noindent{\bf Step OP2.3}: {\it We claim that there exists $M'\in\mathcal{B}(\mathbb{R}^d)$ such that
\begin{equation*}
    \mu_0(M')=1, \quad \eta^x_0=\delta_x, \quad \mbox{and} \quad \eta_t^x(E^\dagger(t))=1, \quad x\in M',\; dt\mbox{-a.e. } t > 0,
\end{equation*}
where recall that $E^\dagger(t)$ is given by \eqref{eq:hat(s)}.
}

Indeed, using \Cref{thm:main0}, (ii.4), we get
\begin{equation*}
\mu_0(\cdot)=\int_{\mathbb{R}^d} \eta^x_0(\cdot) \mu_0(dx), \quad \mbox{and} \quad \int_{\mathbb{R}^d}\eta^x_t(E^\dagger(t))\mu_0(dx)=\eta_t(E^\dagger(t))=\mu_t(E^\dagger(t))=1, \quad t>0, 
\end{equation*}
from which the claim can be easily deduced.

\medskip
\noindent{\bf Step Q2.4}: {\it
Let us set
\begin{equation*}
    M_0:=M\cap M', \quad E^{\dagger,0}:=\left(\{0\}\times M_0\right) \cup \left(E^\dagger\setminus \{0\}\times E^{\dagger}(0)\right),
\end{equation*}
and
\begin{equation}\label{eq:Pdagger0}
    P_t^{\dagger,0}f(s,x):=
    \begin{cases}
        \eta^x_t\left(f(t,\cdot)\right), \quad &\mbox{ if } s=0\\
        P_t^{\dagger}\left(f|_{E^\dagger}\right)(s,x), \quad &\mbox{ if } s>0
    \end{cases}\;,
    \quad f\in b\mathcal{B}(E^\dagger_0),\; t\geq 0,\; (s,x)\in E^{\dagger,0}.
\end{equation}
We claim that
\begin{enumerate}[(1)]
    \item $P^{\dagger, 0}_0={\sf Id}$ and $P^{\dagger, 0}_tP^{\dagger, 0}_{t'}=P^{\dagger, 0}_{t+t'}$ as kernels on $E^{\dagger,0}$, $dt\otimes dt'$-a.e. $(t,t')\in (0,\infty)^2$.
    \item $\lim\limits_{t\to 0}P^{\dagger, 0}_tf=f \quad$ pointwise on $E^\dagger_0$ for all $f\in C_b(E^{\dagger,0})$. 
    \item If we denote by $\mathcal{U}^{\dagger,0}:=\left(U^{\dagger,0}_\alpha=\int_0^\infty e^{-\alpha t} P^{\dagger,0}_t\;dt\right)_{\alpha>0}$ the resolvent of $\left(P^{\dagger,0}_t\right)_{t\geq 0}$, then $\mathcal{U}^{\dagger,0}$ is a resolvent family of kernels on $E^{\dagger,0}$ such that 
    \begin{equation}\label{eq:alphaU_alpha}
        \alpha U^{\dagger,0}_\alpha 1=1, \;\alpha>0, \quad \mbox{and} \quad \lim\limits_{\alpha\to \infty}\alpha U^{\dagger,0}_\alpha f=f, \;  f\in C_b(E^{\dagger,0}).  
    \end{equation}
\end{enumerate}
}
To prove the claim, first of all notice that
\begin{equation*}
    \mbox{the mapping } [0,\infty)\times E^{\dagger,0}\ni(t,s,x)\mapsto P^{\dagger,0}_tf(s,x)\in \mathbb{R} \mbox{ is measurable for every } f\in b\mathcal{B}(E^{\dagger,0}).
\end{equation*}
Let us prove (1), by noticing first that $P^{\dagger, 0}_0={\sf Id}$ follows from the fact that $\eta_0^x=\delta_x, x\in M_0$ and $\delta_{(s,x)}\circ P^{\dagger, 0}_0=\delta_{(s,x)}, (s,x)\in E^{\dagger, 0}$.
To show the semigroup property for a.e. $t,t'>0$, note that taking first $s=0$ we have by \eqref{eq:eta_P} that
\begin{align*}
    P^{\dagger, 0}_t P^{\dagger, 0}_{t'}f(0,x)
    &=\eta_t^x\left(P^{\dagger, 0}_{t'}f(t,\cdot)\right)
    =\eta_t^x\left(P^{\dagger}_{t'}\left(f|_{E^{\dagger}}\right)(t,\cdot)\right)
    =\eta^x_{t+t'}\left(f(t+t',\cdot)\right)\\
    &=P^{\dagger, 0}_{t+t'}f(0,x), \quad dt\otimes dt'\mbox{-a.e.}
\end{align*}
The case $s>0$ is immediate since $P^{\dagger}_tP^{\dagger}_{t'}=P^{\dagger}_{t+t'}$ on $E^{\dagger}$.

To prove (2), it is sufficient to show the desired convergence for every $f$ which is bounded and Lipschitz on $E^{\dagger,0}$.
With this reduction, if $(s,x)\in E^\dagger$ then by using the fact that $\left(P^\dagger_t\right)_{t\geq 0}$ is the transition function of a right process with a.s. continuous paths in $E^\dagger$, we get
\begin{equation*}
    \lim\limits_{t\to 0}P^{\dagger,0}_tf(s,x)=\lim\limits_{t\to 0}P^\dagger_t\left(f|_{E^\dagger}\right)(s,x)=f(s,x). 
\end{equation*}
If $s=0$ and $x\in M_0$, then
\begin{align*}
    \lim\limits_{t\to 0}\left|P^{\dagger,0}_tf(0,x)-f(0,x)\right|
    &=\lim\limits_{t\to 0}\left|\eta^x_t\left(f(t,\cdot)\right)-f(0,x)\right|\\
    &=\limsup\limits_{t\to 0}\left[\eta^x_t\left(\min\left(\|\cdot-x\|,2\|f\|_{\infty}\right)\right)+|f|_{\sf Lip}t\right]=0,
\end{align*}
where $|f|_{\sf Lip}$ denotes the Lipschitz norm of $f$.

The proof of (3) is now straightforward, using (1) and (2) from above, as well as {\bf Step OP2.3} and the fact that $P^\dagger_t 1=1$ on $E^\dagger$.

\medskip
\noindent{\bf Step OP2.5}: {\it We claim that there exists $M_{00}\in \mathcal{B}(\mathbb{R}^d)$, $M_{00}\subset M_0$ such that
$\mu_0(M_{00})=1$, and if we set
\begin{equation*}
E^{\dagger,00}:=\left(\{0\}\times M_{00}\right)\cup E^\dagger, 
\end{equation*}
then 
\begin{equation*}
    U^{\dagger,0}_1\left((s,x),E^{\dagger,0}\setminus E^{\dagger,00}\right)=0, \quad (s,x)\in E^{\dagger,00},
\end{equation*}
and the restricted resolvent $\mathcal{U}^{\dagger,00}:=\mathcal{U}^{\dagger,0}|_{E^{\dagger,00}}$ is a Markovian resolvent of kernels with no branch points, that is for all $(s,x)\in E^{\dagger,00}, f,g\in b\mathcal{B}(E^{\dagger,00})$ we have
\begin{equation}\label{eq:branch}
    \lim\limits_{\alpha\to \infty}\alpha U^{\dagger,00}_{\alpha+1}\left(U^{\dagger,00}_1f\wedge U^{\dagger,00}_1g\right)(s,x)=U^{\dagger,00}_1f(s,x)\wedge U^{\dagger,00}_1g(s,x).
\end{equation}
}

The proof of this step is involved, so let us split it in several sub-steps:
\begin{enumerate}[(1)]
    \item Let us set
    \begin{equation*}
        E^{\dagger,00}:=\left\{(s,x)\in E^{\dagger,0}:\lim\limits_{\alpha\to \infty}\alpha U^{\dagger,0}_{\alpha+1}\left(U^{\dagger,0}_1f\wedge U^{\dagger,0}_1g\right)(s,x)= U^{\dagger,0}_1f(s,x)\wedge U^{\dagger,0}_1g(s,x), f,g\in b\mathcal{B}(E^{\dagger,0})\right\}
    \end{equation*}
    Then by \cite[Theorem 1.2.11 and assertion 1 of Remark at page 23]{BeBo04a}
    we have that $E^{\dagger,00}\in \mathcal{B}(E^{\dagger,0})$ and $U_1^{\dagger,0}(\cdot,E^{\dagger,0}\setminus E^{\dagger,00})=0$ on $E^{\dagger,0}$. 
    Moreover, notice that if we set $E':=E^\dagger\setminus \{0\}\times E^\dagger(0)$ then $E'\in \mathcal{A}(\mathcal{U}^\dagger)$ (use e.g. \Cref{thm:main0}, (ii.4)), $\mathcal{U}^{\dagger,0}|_{E'}=\mathcal{U}^{\dagger}|_{E'}$, and $\mathcal{U}^{\dagger}$ has no branch points since it is the resolvent of a right process on $E^\dagger$; see \Cref{prop 2.5} and the remark after it.
    Thus,
    \begin{align*}
        E^{\dagger,0}\setminus E^{\dagger,00}\subset \{0\}\times M_0\subset \{0\}\times E^{\dagger,0}(0),
    \end{align*}
    and it remains to show that 
    \begin{equation*}
        \mu_0\left(E^{\dagger,00}(0)\right)=1.
    \end{equation*}
    \item The next step is to notice that by the monotone class argument \cite[Lemma 1.2.10]{BeBo04a}, it is sufficient to fix $f,g\in b{\sf Lip}([0,\infty)\times \mathbb{R}^d)$ with compact support, where $b{\sf Lip}([0,\infty)\times \mathbb{R}^d)$ denotes the space of bounded and Lipschitz functions on $[0,\infty)\times \mathbb{R}^d$, and show that
    \begin{equation*}
    \mu_0\left(\left\{x\in E^{\dagger,0}(0): \lim\limits_{\alpha\to \infty}\alpha U^{\dagger,0}_{\alpha+1}\left(U^{\dagger,0}_1f\wedge U^{\dagger,0}_1g\right)(0,x)=U^{\dagger,0}_1f(0,x)\wedge U^{\dagger,0}_1g(0,x)\right\}\right)=1.
    \end{equation*}
    \item In light of (2) from above, let us fix $f,g\in b{\sf Lip}([0,\infty)\times \mathbb{R}^d)$ and define for $x\in M_0$ and $\alpha>0$
    \begin{align*}
        &\psi_\alpha(x)
        :=\alpha U^{\dagger,0}_{\alpha+1}\left(U^{\dagger,0}_1f\wedge U^{\dagger,0}_1g\right)(0,x)
        =\int_0^\infty \alpha e^{-(\alpha+1) t} \eta_t^x\left(\left(U^{\dagger}_1f\wedge U^{\dagger}_1g\right)(t,\cdot)\right) \;dt,\\
        &\psi(x):=U^{\dagger,0}_1f(0,x)\wedge U^{\dagger,0}_1g(0,x)\\
        \intertext{Note that $U^{\dagger,0}_1f\wedge U^{\dagger,0}_1g$ is $\mathcal{U}^{\dagger,0}_1$-supermedian, hence $\psi_\alpha\leq \psi_{\beta}\leq \psi$ for every $0<\alpha\leq \beta$, and we can define}
        &\psi_\infty:=\lim\limits_{\alpha \to \infty} \psi_\alpha=\sup\limits_{\alpha>0} \psi_\alpha\leq \psi, \quad \mbox{pointwise on } M_0.
    \end{align*}
    Consequently, to achieve {\bf Step OP2.5} it is sufficient to show that
    \begin{equation}\label{eq:h}
        \left(\mu_0\left(\psi_\infty\right)=\right)\lim\limits_{\alpha\to \infty} \mu_0\left(\psi_\alpha\right)=\mu_0\left(\psi\right).
    \end{equation}
    \item Proceeding to prove \eqref{eq:h}, let us notice first that using e.g. \Cref{thm:main0}, (ii.4), and the fact that $P^\dagger_{t}\left(U^{\dagger}_1f\wedge U^{\dagger}_1g\right)\leq e^{t}\left(U^{\dagger}_1f\wedge U^{\dagger}_1g\right), t\geq 0$, we get
    \begin{align*}
        \mu_0\left(\psi_\alpha\right)
        &=\int_0^\infty \alpha e^{-(\alpha+1) t}\mu_t\left(\left(U^{\dagger}_1f\wedge U^{\dagger}_1g\right)(t,\cdot)\right) \;dt\\
        e^{-t}\mu_t\left(\left(U^{\dagger}_1f\wedge U^{\dagger}_1g\right)(t,\cdot)\right)
        &=e^{-t}\mu_s\left(\left(P^\dagger_{t-s}\left(U^{\dagger}_1f\wedge U^{\dagger}_1g\right)\right)(s,\cdot)\right)\\ 
        &\leq e^{-s}\mu_s\left(\left(U^{\dagger}_1f\wedge U^{\dagger}_1g\right)(s,\cdot)\right) \;dt, \quad s\leq t,
    \end{align*}
    we obtain that $t\mapsto e^{-t}e^{-t}\mu_t\left(\left(U^{\dagger}_1f\wedge U^{\dagger}_1g\right)(t,\cdot)\right)$ is non-decreasing.
    This easily implies that \eqref{eq:h} is satisfied if we show that
    \begin{equation}\label{eq:t_n}
        \exists (t_n)_n\searrow 0: \quad\lim_n\mu_{t_n}\left(\left(U^{\dagger}_1f\wedge U^{\dagger}_1g\right)(t_n,\cdot)\right)=\mu_{0}\left(\psi\right).
    \end{equation}
    \item We focus now on proving \eqref{eq:t_n}. 
    To this end, we write for $t>0$
    \begin{align*}
    &\mu_t\left(\left(U^{\dagger}_1f\wedge U^{\dagger}_1g\right)(t,\cdot)\right) = \mu_t\left(1_{E^{\dagger,0}(0)}\psi\right) + e_t,\\
    &|e_t|
    \leq \mu_t\left(\left|1_{E^{\dagger,0}(t)}(\cdot)U^{\dagger,0}_1f(t,\cdot)-1_{E^{\dagger,0}(0)}(\cdot)U^{\dagger,0}_1f(0,\cdot)\right|\right.\\
    &\quad\quad \quad\quad \left.+\left|1_{E^{\dagger,0}(t)}(\cdot)U^{\dagger,0}_1g(t,\cdot)-1_{E^{\dagger,0}(0)}(\cdot)U^{\dagger,0}_1g(0,\cdot)\right|\right).
    \end{align*}
    Hence, using $\mathbf{H^{\sf TV}_0}$ we have
    \begin{align*}
        \limsup\limits_{t\to 0} \left|\mu_t\left(1_{E^{\dagger,0}(0)}\psi\right)-\mu_0\left(\psi\right)\right|
        \leq
        \|\psi\|_\infty\limsup\limits_{t\to 0}\int_{\mathbb{R}^d}\left|u(t,x)-u(0,x)\right|dx=0,
    \end{align*}
    hence it remains to prove that $\exists (t_n)_n\searrow0: \quad \lim_n e_{t_n} =0$, which boils down to show that for every $f\in b{\sf Lip}([0,\infty)\times \mathbb{R}^d)$ with compact support we have
    \begin{equation*}
       \exists (t_n)_n\searrow 0: \quad \lim\limits_n\mu_{t_n}\left(\left|1_{E^{\dagger,0}(t_n)}(\cdot)U^{\dagger,0}_1f(t_n,\cdot)-1_{E^{\dagger,0}(0)}(\cdot)U^{\dagger,0}_1f(0,\cdot)\right|\right)=0.
    \end{equation*}
    \begin{rem}
    Here and in general, if $v:F\rightarrow \mathbb{R}$ and $w:A\rightarrow \mathbb{R}$ with $A\subset F$, then by $vw$ we understand wither $v|_Aw$ or $v(1_A w)$, where recall that $1_Aw$ stands for the extension of $w$ from $A$ to $F$ by setting $w=0$ on $F\setminus A$.
    \end{rem}
    Furthermore, by $\mathbf{H^{\sf TV}_0}$, we deduce that {\bf Step OP2.5} is achieved if we show that for every $f\in b{\sf Lip}([0,\infty)\times \mathbb{R}^d)$ with compact support,
    \begin{equation}\label{eq:t_n_n}
        \exists (t_n)_n\searrow 0: \; \lim\limits_n\int_{\mathbb{R}^d}\left|u(t_n,x)U^{\dagger,0}_1f(t_n,x)-u(0,x)U^{\dagger,0}_1f(0,x)\right| \;dx =0.
    \end{equation}
    \item\label{it:6} To prove \eqref{eq:t_n_n} we need first a compactness result.
    More precisely, {\it fixing $f\in b{\sf Lip}([0,\infty)\times \mathbb{R}^d)$ with support (for simplicity) in $[0,1)\times \mathbb{R}^d$ and $h\in C_c^\infty(\mathbb{R}^d)$, we claim that by setting
    \begin{equation*}
        w(t,x):=u(t,x)U^{\dagger}_1f(t,x)h(x), \quad x\in \mathbb{R}^d,t\in (0,1],
    \end{equation*}
    then there exists a $dt\otimes dx$-version $\tilde{w}$ such that
    \begin{equation*}
        \tilde{w}\in C\left([0,1];L^2(\mathbb{R}^d)\right).
    \end{equation*}
    }
    In fact, the claim follows by Lions-Magenes interpolation lemma \cite[Chapter 3, Theorem 3.1]{LiMa72}, once we show that
    \begin{equation}\label{eq:Lions-Magenes}
        w\in L^2\left([0,1];H^1(\mathbb{R}^d)\right) \quad \mbox{ and } \quad \frac{dw}{dt}\in L^2\left([0,1];H^{-1}(\mathbb{R}^d)\right).
    \end{equation}
    To this end, following \cite{St99a} let us consider $\left(\mathcal{A}_0,{\sf D}\left(\mathcal{A}_0\right)\right)$ the closure in $L^2\left([0,1]\times \mathbb{R}^d; \overline{\mu}_1\right)$ of the symmetric bilinear form
    \begin{equation*}
        \mathcal{A}_0(\varphi,\psi):= \int_{[0,1]\times \mathbb{R}^d} \langle A(t,x)\nabla \varphi(t,x),\nabla \psi(t,x) \rangle \; \overline{\mu}_1(dt,dx), \quad \varphi,\psi\in C_0^\infty\left((0,1)\times \mathbb{R}^d\right).
    \end{equation*}
    Also, note that by \Cref{thm:main0}, (i), on $L^p([0,1]\times \mathbb{R}^d;\overline{\mu}_1)$ we can identify $U^{\dagger}_1$ with $\overline{U}_1^1$, where $\overline{U}_1^1$ is given in \Cref{thm:Stannat}.
    Consequently, by \cite[Theorem 1.7]{St99a}, b) and c), we have that
    \begin{equation}\label{eq:stannat_domains}
       U^{\dagger}_1f \mbox{ and } q(\cdot,\cdot):=h(\cdot)U^{\dagger}_1f(\cdot,\cdot) \mbox{ belong to } b{\sf{D}(\overline{\sf{L}}_1^1)}\subset {\sf D}\left(\mathcal{A}_0\right),
    \end{equation}
    where ${\sf{D}(\overline{\sf{L}}_1^1)}$ is given in \Cref{thm:Stannat}, 
    and also that
    \begin{equation}\label{eq:A_0}
        -\int_{[0,1]\times \mathbb{R}^d} \frac{d\varphi}{dt}q\;d\overline{\mu}_1=\mathcal{A}_0(q,\varphi)+\int_{[0,1]\times \mathbb{R}^d}\langle b-b^0, \nabla \varphi \rangle q \; d\overline{\mu}_1 + \int_{[0,1]\times \mathbb{R}^d} \varphi \overline{\sf{L}}_1^1 q\; d\overline{\mu}_1,
    \end{equation}
    where 
    \begin{equation*}
        b_i^0:=\sum\limits_{i=1}^d \left(\partial_{x_j} a_{ij} + 2a_{ij}\frac{\partial_{x_j}\sqrt{u}}{\sqrt{u}}\right), \quad 1\leq i\leq d.
    \end{equation*}

    Now, let us show that 
    $w\in L^2\left([0,1];H^1(\mathbb{R}^d)\right)$.
    To this end, note that it is not a priori clear that we can consider the classical weak gradient in the $x$-variable, $\nabla w$.
    Thus, to make the computations rigorous, let us consider
    \begin{equation*}
        \left(v_k\right)_k \subset C_0^\infty((0,1)\times \mathbb{R}^d)\subset {\sf D}\left(\mathcal{A}_0\right) \mbox{ uniformly bounded},
    \end{equation*}
    and such that
    \begin{equation*}
        \lim_k\mathcal{A}_0\left(v_k-U^{\dagger}_1f,v_k-U^{\dagger}_1f\right)=0,\quad
        q_k:=hv_k, \quad 
        w_k:=u q_k, \quad k\geq 1.
    \end{equation*}
    For later use, note that it is straightforward to see that $\lim_k\mathcal{A}_0\left(q_k-q,q_k-q\right)=0$.
    Furthermore, it is easy to see that for our current goal it is sufficient to show that
    \begin{equation}\label{eq:w_k_bounded}
        (w_k)_{k} \mbox{ is bounded in } L^2\left([0,1];H^1(\mathbb{R}^d)\right).
    \end{equation}
    By the product rule,
    \begin{equation*}
        \nabla w_k
        =q_k\nabla u+u\nabla q_k
        =2q_k\sqrt{u}\nabla \sqrt{u}+u\nabla q_k, 
    \end{equation*}
    so that by $\mathbf{H^{\sqrt{u}}}$ and using that $q_k$ is uniformly bounded with respect to $k,t,x$, and has compact support uniformly with respect to $k$, the first summand $2q_k\sqrt{u}\nabla \sqrt{u}$ can be easily bounded in $L^2([0,1];L^2(\mathbb{R}^d))$. Hence, it remains to show that
    \begin{equation*}
        \left(u\nabla q_k\right)_k \mbox{ is bounded in } L^2\left([0,1]\times \mathbb{R}^d; dtdx\right).
    \end{equation*}
    But, according to \eqref{eq:stannat_domains},
    \begin{equation*}
        \int_{[0,1]\times \mathbb{R}^d} \left|u\nabla q_k\right|^2 \;dtdx=\int_{[0,1]\times \mathbb{R}^d} u\left|\nabla q_k\right|^2 d\overline{\mu}_1 \leq c  \mathcal{A}_0\left(q_k,q_k\right),
    \end{equation*}
    where $c$ is a constant depending on the ellipticity of $A$ and the $L^\infty$-norm of $u$ on $[0,1]\times {\sf supp}(h)$, where ${\sf supp}(h)$ is the support of $h$ in $\mathbb{R}^d$.
    Now, \eqref{eq:w_k_bounded} follows since it is easy to see that $\sup\limits_k \mathcal{A}_0\left(q_k,q_k\right)<\infty$ due to the fact that  $\sup\limits_k \mathcal{A}_0\left(v_k,v_k\right)<\infty$.
    
    \noindent{Finally}, let us show that 
    $\frac{dw}{dt}\in L^2\left([0,1];H^{-1}(\mathbb{R}^d)\right)$, namely that there exists
    $\dot{w}\in L^2\left([0,1];H^{-1}(\mathbb{R}^d)\right)$
and a constant $C$ such that
    \begin{equation*}
        \left|\int_{[0,1]\times \mathbb{R}^d}\frac{d\varphi}{dt} w\; dtdx\right| \leq C \|\varphi\|_{L^2((0,1)\times H^1(\mathbb{R}^d))}, \quad \mbox{for all } \varphi \in C_0^\infty\left((0,1)\times \mathbb{R}^d\right).
    \end{equation*}
    To do so, note that by \eqref{eq:A_0}
    \begin{align*}
        \left|\int_{[0,1]\times \mathbb{R}^d}\frac{d\varphi}{dt} w\; dtdx \right|
        &=\left|\int_{[0,1]\times \mathbb{R}^d}\frac{d\varphi}{dt} qu\; dtdx \right|\\ 
        &\leq \left|\mathcal{A}_0(q,\varphi)\right|+\left|\int_{[0,1]\times \mathbb{R}^d}\langle b-b^0, \nabla \varphi \rangle q\; d\overline{\mu}_1\right| + \left|\int_{[0,1]\times \mathbb{R}^d} \varphi \overline{\sf{L}}_1^1 q\; d\overline{\mu}_1\right|
    \end{align*}
    We treat now each of the three terms in the last sum:
    \begin{align*}
        \left|\mathcal{A}_0(q,\varphi)\right|
        &=\lim_k\left|\mathcal{A}_0(q_k,\varphi)\right|
        \leq\sup_k\int_{[0,1]\times {\sf supp}(h)} \|A^{1/2}\nabla q_k\|\|A^{1/2}\varphi\|\; d\overline{\mu}_1\\
        &\leq\sup_k\mathcal{A}_0(q_k,q_k)^{1/2}
        \left(\int_{[0,1]\times {\sf supp}(h)} \|A^{1/2}\nabla\varphi\|^2\; d\overline{\mu}_1\right)^{1/2}\\ &\leq C\sup_k\mathcal{A}_0(q_k,q_k)^{1/2}\|\varphi\|_{L^2([0,1]; H^1(\mathbb{R}^d))},
    \end{align*}
    where the constant $C$ is independent of $k$ and $\varphi$, and its existence is ensured by $\mathbf{H_a}$ and $\mathbf{H^{\sqrt{u}}}$.
    also, note that $\sup_k\mathcal{A}_0(q_k,q_k)<\infty$.

    Let us now treat the second term, for which we have
    \begin{align*}
        \left|\int_{[0,1]\times \mathbb{R}^d}\langle b-b^0, \nabla \varphi \rangle q \; d\overline{\mu}_1\right|
        &\leq \|\varphi\|_{L^2((0,1)\times H^1(\mathbb{R}^d))}\left(\int_{[0,1]\times \mathbb{R}^d} \|b-b^0\|^2 q^2u \; d\overline{\mu}_1\right)^{1/2}\\
        &\leq C\|\varphi\|_{L^2([0,1]; H^1(\mathbb{R}^d))},
    \end{align*}
    where the constant $C$ is independent of $\varphi$, and its existence is ensured by $\mathbf{H_{a,b}^{\sqrt{u}}}$. 
    
    We now treat the last term, as follows:
    \begin{equation*}
        \left|\int_{[0,1]\times \mathbb{R}^d} \varphi \overline{\sf{L}}_1^1 q\; d\overline{\mu}_1\right|
        \leq \|\varphi\|_{L^2((0,1)\times \mathbb{R}^d)}\|u\overline{\sf{L}}_1^1 q\|_{L^2((0,1)\times \mathbb{R}^d;\overline{\mu}_1)},
    \end{equation*}
    whilst $\|u\overline{\sf{L}}_1^1 q\|_{L^2((0,1)\times \mathbb{R}^d;\overline{\mu}_1)}$ is finite since $\overline{\sf{L}}_1^1 q$ is bounded with compact support by \cite[Theorem 1.7, (c)]{St99a}, whilst $u$ is bounded on compact sets by $\mathbf{H^{\sqrt{u}}}$.

    Thus, the claim in item (6) is proved.
    \item\label{it:7} {\it We now claim that there exists $F:[0,1]\times \mathbb{R}^d\rightarrow\mathbb{R}$ measurable, which is a $dt\otimes dx$-version of $u(\cdot,\cdot)U^{\dagger}_1f(\cdot,\cdot)$ on $[0,1]\times \mathbb{R}^d$ such that
    \begin{equation*}
        hF\in C\left([0,1]; L^2(\mathbb{R}^d)\right), \quad \mbox{ for every } h\in C_c^\infty(\mathbb{R}^d).
    \end{equation*}
    }
    To prove this, for each $n\geq 0$ let $h_n\in C_c^\infty(\mathbb{R}^d)$ such that $h_n=1$ on $B(0,n)$.
    Then, by \eqref{it:6}, let $\tilde{w}_n\in C([0,1];L^2(\mathbb{R}^d))$ be a $dt\otimes dx$-version of $h_n(\cdot)u(\cdot,\cdot)U^{\dagger}_1f(\cdot,\cdot)$, and for every $t\in [0,1]$
    define
    \begin{equation*}
        F(t,\cdot):=\sum_{n\geq 0}\tilde{w}_{n+1}(t,\cdot)1_{B(0,n+1)\setminus B(0,n)}(\cdot). 
    \end{equation*}
    It is now easy to see that $F$ satisfies the requirements from \eqref{it:7}.
    \item Let $F$ be as in item \eqref{it:7}. 
    {\it We claim that 
    \begin{equation*}
        \forall \;(t_n)_n\searrow 0 \mbox{ there exists a subsequence } (t_{n_k})_k \mbox{ such that} \quad \lim_k F(t_{n_k},x)=F(0,x)\quad dx\mbox{-a.e.}
    \end{equation*} 
    }
    To show this, let $V(x):=e^{-\|x\|}, x\in \mathbb{R}^d$, and $(t_k)_k\subset (0,1]$ such that $\lim_k t_k=0$.
    Then
    \begin{align*}
        \|V(\cdot)F(t_k,\cdot)-V(\cdot)F(0,\cdot)\|_{L^1(\mathbb{R}^d)}
        &\leq \sum_{n\geq 0}\int_{B(0,n+1)\setminus B(0,n)}V(x)|\tilde{w}_{n+1}(t_k,x)-\tilde{w}_{n+1}(0,x)|\;dx\\
        &\leq \sum_{n\geq 0}e^{-n}\int_{B(0,n+1)\setminus B(0,n)}|\tilde{w}_{n+1}(t_k,x)-\tilde{w}_{n+1}(0,x)|\;dx\\
        &=\sum_{n\geq 0}e^{-n}a_{k,n},
    \end{align*}
    where
    \begin{equation*}
        a_{k,n}:=\int_{B(0,n+1)\setminus B(0,n)}|\tilde{w}_{n+1}(t_k,x)-\tilde{w}_{n+1}(0,x)|\;dx, \quad k\geq 1,n\geq 0,
    \end{equation*}
    whilst $\tilde{w}_n,n\geq 0$ are given in \eqref{it:7}.
    Now, on the one hand we have that $\tilde{w}_{n+1}\in C([0,1];L^2(\mathbb{R}^d))$, hence
    $
        \lim_k a_{k,n}=0.
    $
    On the other hand,
    \begin{align*}
        \sup_ka_{k,n}
        &\leq 2 \sup_{t\in [0,1]}\int_{B(0,n+1)}|\tilde{w}_{n+1}(t,x)|\;dx
        =2 \|\tilde{w}_{n+1}(t,x)\|_{L^\infty([0,1];L^1(B(0,n+1)))}\\
        &\leq 2 \|f\|_{L^\infty([0,1]\times \mathbb{R}^d)}<\infty.
    \end{align*}
    Therefore, by dominated convergence, we obtain 
    \begin{equation*}
        \lim_{k} V(\cdot)F(t_k,\cdot)=V(\cdot)F(0,\cdot) \quad \mbox{in } L^1(\mathbb{R}^d),
    \end{equation*}
    hence, taking also into account that $V>0$, there exists a subsequence $(t_{k_n})_{n\geq 1}$ such that 
    \begin{equation*}
        \lim_n F(t_{k_n},\cdot)=F(0,\cdot)\quad dx\mbox{-a.e.}
    \end{equation*}
    \item\label{it:9} {\it We claim that
    \begin{equation}\label{eq:exists_F_0}
        \exists\; F_0:\mathbb{R}^d\rightarrow\mathbb{R} \mbox{ and } (t_n)_{n}\searrow 0 \mbox{ such that } \lim_n 1_{E^{\dagger,0}(t_n)}(\cdot)U^{\dagger, 0}_1 f(t_n,\cdot)=F_0(\cdot) \quad dx\mbox{-a.e.}
    \end{equation}
    }
    Indeed, first of all, since $F$ given in \eqref{it:7} is a $dt\otimes dx$-version of $u(\cdot,\cdot)U_1^{\dagger}f(\cdot,\cdot)$ on $[0,1]\times \mathbb{R}^d$, there exists $(t_n)_{n\geq 1}\searrow 0$ such that 
    \begin{equation*}
        F(t_n,\cdot)=u(t_n,\cdot)U_1^{\dagger}(t_n,\cdot),\quad n\geq 1, dx\mbox{-a.e.}
    \end{equation*}
    Therefore, using also that $U_1^{\dagger,0}(t,\cdot)=U_1^{\dagger}(t,\cdot)$ for $t>0$, we have
    \begin{equation*}
        \exists\; (t_{n_k})_{k}: \quad\lim_k u(t_{n_k},\cdot)U_1^{\dagger,0}(t_{n_k},\cdot)=F(0,\cdot) \quad dx\mbox{-a.e}
    \end{equation*}
    Now, since $\lim_k u(t_{n_k},\cdot)=u(0,\cdot)$ in $L^1(\mathbb{R}^d)$ and $\int_{E^{\dagger,0}(t)}u(t,x)dx=1, t\in [0,1]$, we can easily deduce that, by passing to a further subsequence,
    \begin{equation*}
        \exists\; (t_{n_{k_l}})_{l}: \quad\lim_l 1_{E^{\dagger,0}(t_{n_{k_l}})}(\cdot)U_1^{\dagger,0}(t_{n_{k_l}},\cdot)=1_{E^{\dagger,0}(0)}(\cdot)\frac{F(0,\cdot)}{u(0,\cdot)} \quad dx\mbox{-a.e.},
    \end{equation*}
    which proves the claim.
    \item Let us fix $s\geq 0$, $f\in b{\sf Lip}([0,\infty)\times \mathbb{R}^d)$ with compact support, and define
    \begin{equation*}
        \varphi_t(x)=\varphi_t^s(x)=1_{E^{\dagger,0}(t)}(x)P^{\dagger, 0}_sf(t,x), \quad t\in [0,1], x\in \mathbb{R}^d,
    \end{equation*}
    where $1_{E^{\dagger,0}(t)}(x)g(x),x\in \mathbb{R}^d$ denotes the extension by zero (from $E^{\dagger,0}(t)$ to $\mathbb{R}^d$) of any given function $g:E^{\dagger,0}(t)\rightarrow\mathbb{R}$.
    
    {\it We claim that
    \begin{equation}\label{eq:0_identification}
        \lim\limits_{t\to 0}\mu_0\left(v_0\varphi_t\right)=\mu_0\left(v_0\varphi_0\right) \quad \mbox{for all } v_0\in \mathcal{A}_0,
    \end{equation}
    where $\mathcal{A}_0$ is the class of probability densities employed in $\mathbf{H^{\sf TV}_0}$.
    }
    
    \noindent{}To prove the above assertion, let $v_0\in \mathcal{A}_0$, 
    \begin{equation*}
        \nu_0:=v_0(x)u(0,x) dx \leq \|v_0\|_{\infty}\mu_0, \quad \mbox{and } \eta_t^{\nu_0}, t\geq 0 \mbox{ be given by } \eqref{eq:eta_nu}.
    \end{equation*}
    Then,
    \begin{align*}
        \mu_0\left(v_0\varphi_t\right)
        &=\int_{\mathbb{R}^d}v_0(x)\varphi_t(x)\mu_0(dx)
        = \int_{\mathbb{R}^d}v_0(x)u(0,x)\varphi_t(x)\; dx\\
        &= \int_{\mathbb{R}^d}\left(v_0(x)u(0,x)-\frac{d\eta_t^{\nu_0}}{dx}(x)\right)\varphi_t(x)\; dx+ \int_{\mathbb{R}^d}\varphi_t(x)\eta_t^{\nu_0}(dx)\\
        &=\int_{\mathbb{R}^d}\left(v_0(x)u(0,x)-\frac{d\eta_t^{\nu_0}}{dx}(x)\right)\varphi_t(x)\; dx+\eta_t^{\nu_0}\left(P^{\dagger, 0}_sf(t,\cdot)\right)\\
        \intertext{hence by \eqref{eq:eta_P} we can continue with}
        &=\int_{\mathbb{R}^d}\left(v_0(x)u(0,x)-\frac{d\eta_t^{\nu_0}}{dx}(x)\right)\varphi_t(x)\; dx+\eta_{s+t}^{\nu_0}\left(f(s+t,\cdot)\right).
    \end{align*}
    Now, on the one hand, since $\mathbf{H^{\sf TV}_0}$ is in force we have that the first term in the right-hand side of the last equality converges to zero when $t\to 0$.
    On the other hand, by the weak continuity of $\left(\eta_{t}^{\nu_0}\right)_{t\geq 0}$ and the fact that $f\in b{\sf Lip}([0,\infty)\times \mathbb{R}^d)$, we easily get
    \begin{equation*}
        \lim\limits_{t\to 0}\eta_{s+t}^{\nu_0}\left(f(s+t,\cdot)\right)=\eta_{s}^{\nu_0}\left(f(s,\cdot)\right)=\int_{\mathbb{R}^d}v_0(x)\varphi_0(x)\mu_0(dx).
    \end{equation*}
    \item Let us now finish the proof of the claim in {\bf Step OP2.5}.
    First, notice that
    \begin{equation*}
        \int_0^\infty e^{- s} \varphi_t^s(x)\;ds
        =1_{E^{\dagger,0}(t)}(x)U^{\dagger, 0}_1 f(t,x),\quad t\geq 0, x\in \mathbb{R}^d.
    \end{equation*}
    Consequently, on the one hand we have
    \begin{equation*}
        \lim\limits_{t\to 0}\mu_0\left(v_0(\cdot)1_{E^{\dagger,0}(t)}(x)U^{\dagger, 0}_1 f(t,\cdot)\right)=\mu_0\left(v_0(\cdot)1_{E^{\dagger,0}(0)}(\cdot)U^{\dagger, 0}_1 f(0,\cdot)\right) \quad \mbox{for all } v_0\in \mathcal{A}_0,
    \end{equation*}
    On the other hand, by item (9) and using dominated convergence, we deduce that
    \begin{equation*}
        \exists (t_n)_n\searrow0: \lim\limits_{t\to 0}\mu_0\left(v_0(\cdot)1_{E^{\dagger,0}(t)}(x)U^{\dagger, 0}_1 f(t,\cdot)\right)=\mu_0\left(v_0(\cdot)F_0(\cdot)\right) \quad \mbox{for all } v_0\in \mathcal{A}_0,
    \end{equation*}
    where $F_0$ is the function entailed in \eqref{it:9}.
    Since $\mathcal{A}_0$ is assumed to have the separation property from $\mathbf{H^{\sf TV}_0}$, we deduce that
    \begin{equation}\label{eq:F_0=U}
        F_0=1_{E^{\dagger,0}(0)}(\cdot)U^{\dagger, 0}_1 f(0,\cdot) \quad \mu_0\mbox{-a.e.}
    \end{equation}
    Finally, by \eqref{eq:exists_F_0} and  \eqref{eq:F_0=U} we can ensure \eqref{eq:t_n_n}, hence {\bf Step OP2.5} is accomplished.
\end{enumerate}

\medskip
\noindent{\bf Step OP2.6}: Let $\mathcal{U}^{\dagger,00}=\mathcal{U}^{\dagger,0}|_{E^{\dagger,00}}$ be the resolvent of Markov kernels constructed in {\bf Step OP2.5}.
Also, recall that $\mathcal{U}^{\dagger}$ is the resolvent of the right process $X^\dagger$ on $E^\dagger$, as in \eqref{eq:X^dagger}, and note that using {\bf Step OP1.10} we have
\begin{equation*}
    U_1^{\dagger}(\cdot,\{0\}\times E^{\dagger}(0))=0 \quad \mbox{ and } E^\dagger\setminus\{0\}\times E^{\dagger}(0)\in \mathcal{A}(\mathcal{U}^\dagger).
\end{equation*}
Hence, we can consider the restricted resolvent and process $(X',\mathcal{U}')$ of $(X^\dagger,\mathcal{U}^\dagger)$ from $E^\dagger$ to $E':=E^\dagger\setminus\{0\}\times E^{\dagger}(0)$.
Now, let us notice that we have the natural extension (see \Cref{defi:natural})
\begin{equation*}
    \left(E',\mathcal{U}'\right)\subset\left(E^{\dagger,00},\mathcal{U}^{\dagger,00}\right).
\end{equation*}
Moreover, by {\bf Step OP2.4} and {\bf Step OP2.5}, one can esily see that $\mathcal{E}(\mathcal{U}^{\dagger,00}_\beta)$ is min stable, contains the constant functions, and generates $\mathcal{B}(E^{\dagger,00})$ for one (hence all) $\beta >0$; see \Cref{Appendix} for details.
Consequently, by \Cref{prop 2.5} and \Cref{thm:natural}, we deduce that there exists a right process on $E^{\dagger,00}$ denoted by $X^{\dagger,00}$
\begin{equation*}
X^{\dagger,00}
:=
\left(\Omega^{\dagger,00}, \mathcal{F}^{\dagger,00}, \mathcal{F}^{\dagger,00}_t , X^{\dagger,00}(t), \theta^{\dagger,00}(t) , \mathbb{P}^{\dagger,00}_{(s,x)}, (s,x)\in E^{\dagger,00}\right),
\end{equation*}
which is a natural extension of $X'$ from $E'$ to $E^{\dagger,00}$; using the notation in \Cref{defi:natural}, this means that
\begin{equation}\label{eq:naturalX'}
    \left(E',X'\right)\subset\left(E^{\dagger,00},X^{\dagger,00}\right).
\end{equation}
In particular, 
\begin{equation}\label{eq:polar_00}
    \{0\}\times E^{\dagger,00}(0) \mbox{ is polar with respect to }  \mathcal{U}^{\dagger,00}.
\end{equation}
Note that $X^{\dagger,00}$ is also conservative, namely
\begin{equation*}
    \mathbb{P}^{\dagger,00}_{s,x}\left(X^{\dagger,00}(t)\in E^{\dagger,00} \mbox{ for all } t\geq 0\right)=1, \quad \mbox{ for all } (s,x)\in E^{\dagger,00}.
\end{equation*}
{\it We claim that for all $(s,x)\in E^{\dagger,00}, s>0,$ we have
\begin{equation} \label{eq:continuous s>0}
\mathbb{P}^{\dagger,00}_{s,x}\left(X^{\dagger,00} \mbox{ has norm-continuous paths in } E^{\dagger,00}\subset[0,\infty)\times \mathbb{R}^d\right)=1.
\end{equation}
}
Indeed, if $(s,x)\in E^{\dagger,00}$ with $s>0$ then
\begin{equation*}
    \mathbb{P}^{\dagger,00}_{s,x}\circ \left(X^{\dagger,00}\right)^{-1}=\mathbb{P}^{\dagger}_{s,x}\circ \left(X^{\dagger}\right)^{-1},
\end{equation*}
and in this case the claim is clear since $X^\dagger$ has norm-continuous paths a.s. for all starting points.

\medskip
\noindent{\bf Step OP2.7}:
{\it We claim that 
\begin{equation} \label{eq:continuous s=0}\mathbb{P}^{\dagger,00}_{\delta_0\otimes\mu_0}\left([0,\infty)\ni t\mapsto X^{\dagger,00}(t) \mbox{ is norm-continuous in } E^{\dagger,00}\subset[0,\infty)\times \mathbb{R}^d\right)=1.
\end{equation}
}
To show this end, note first that by \eqref{eq:naturalX'}- \eqref{eq:polar_00}, and the fact that $X'$ has a.s. (for every starting point) norm-continuous paths in $E'$, we deduce that it is enough to show
\begin{equation} \label{eq:continuous_on_0N}
\mathbb{P}^{\dagger,00}_{\delta_0\otimes\mu_0}\left([0,N)\ni t\mapsto X^{\dagger,00}(t) \mbox{ is norm-continuous in } E^{\dagger,00}\right)=1, \mbox{ for some } N>0.
\end{equation}
As a matter of fact, by the same \eqref{eq:naturalX'}-\eqref{eq:polar_00} we clearly have
\begin{equation*} \mathbb{P}^{\dagger,00}_{\delta_0\otimes\mu_0}\left((0,N)\ni t\mapsto X^{\dagger,00}(t) \mbox{ is norm-continuous in } E^{\dagger,00}\right)=1, \mbox{ for every } N>0,
\end{equation*}
so that showing \eqref{eq:continuous_on_0N} boils down to showing
\begin{equation} \label{eq:continuous_at_0}
\mathbb{P}^{\dagger,00}_{\delta_0\otimes\mu_0} \left(\lim_{t\to 0}X^{\dagger,00}(t)=X^{\dagger,00}(0) \quad \mbox{with respect to the norm-topology}\right)=1.
\end{equation}
To show this, recall that by the above and {\bf Step OP1.9} we have
\begin{equation*}
    X^{\dagger,00}|_{E'}=X',\;  E^{\dagger,00}\setminus E' \mbox{ is polar}, \quad X'=X^\dagger|_{E'},\quad X^\dagger=\check{X}^0|_{E^\dagger},\quad \check{X}^0=\tilde{X}|_{\check{E}_0},
\end{equation*}
whilst by {\bf Step OP1.8} we have
\begin{equation*}
    \tilde{\mathbb{P}}_{\delta_s\otimes \mu_s}\circ\left[\left(\tilde{X}(t)\right)_{t\in[0,N-s)}\right]^{-1}
    =
    \mathbb{P}^N_{\delta_s\otimes \mu_s}\circ\left[\left(X^N(t)\right)_{t\in[0,N-s)}\right]^{-1},\quad 0<s<N.
\end{equation*}
Consequently, we obtain
\begin{equation*}
    \mathbb{P}^{\dagger,00}_{\delta_s\otimes \mu_s}\circ\left[\left(X^{\dagger,00}(t)\right)_{t\in[0,N-s)}\right]^{-1}
    =
    \mathbb{P}^N_{\delta_s\otimes \mu_s}\circ\left[\left(X^N(t)\right)_{t\in[0,N-s)}\right]^{-1},\quad 0<s<N,
\end{equation*}
hence, by \Cref{coro:identification2}, 
there exists $\eta^N\in\mathcal{P}\left(C([0,N];\mathbb{R}^d)\right)$ such that
\begin{equation*}
    \eta^N|_{C([\varepsilon,N);\mathbb{R}^d)}
    =\mathbb{P}^{\dagger,00}_{\delta_\varepsilon\otimes \mu_\varepsilon}\circ\left[\left(X^{\dagger,00}(t-\varepsilon)\right)_{t\in[\varepsilon,N)}\right]^{-1} \quad \mbox{as probabilities on } C([\varepsilon,N);\mathbb{R}^d), \quad \mbox{for every } \varepsilon\in (0,N).
\end{equation*}
Furthermore, using \eqref{eq:Pdagger0} we get
\begin{equation}\label{eq:Udagger0=eta}
    U_\alpha^{\dagger,00}f(0,x)=\int_0^\infty e^{-\alpha t} \eta_t^x(f(t,\cdot))\;dt,\quad \alpha>0,f\in b\mathcal{B}(E^{\dagger,00}).
\end{equation}
Consequently,
\begin{equation*}
    \delta_0\otimes \mu_0\left(U_\alpha^{\dagger,00}f\right)=\int_0^\infty e^{-\alpha t} \delta_t\otimes\mu_t(f)\;dt, \quad \alpha>0,
\end{equation*}
which means that
\begin{equation*}
    \mathbb{P}^{\dagger,00}_{\delta_0\otimes\mu_0}\circ \left[X^{\dagger,00}(t)\right]^{-1}=\delta_t\otimes\mu_t,\quad t\geq 0.
\end{equation*}
Consequently, using also that $X^{\dagger,00}$ is a (right) Markov process, we get
\begin{align*}
    \mathbb{P}^{\dagger,00}_{\delta_0\otimes \mu_0}\circ\left[\left(X_2^{\dagger,00}(t)\right)_{t\in[\varepsilon,N)}\right]^{-1}
    &=\mathbb{P}^{\dagger,00}_{\delta_\varepsilon\otimes \mu_\varepsilon}\circ\left[\left(X_2^{\dagger,00}(t-\varepsilon)\right)_{t\in[\varepsilon,N)}\right]^{-1}\\
    &=\eta^N|_{C([\varepsilon,N);\mathbb{R}^d)}
    \quad \mbox{on } C([\varepsilon,N);\mathbb{R}^d) \quad 0<\varepsilon<N,
\end{align*}
Getting rid of $\varepsilon$, the above can be rewritten as
\begin{equation}\label{eq:dagg00_eta_identification}
    \mathbb{P}^{\dagger,00}_{\delta_0\otimes \mu_0}\circ\left[\left(X_2^{\dagger,00}(t)\right)_{t\in(0,N)}\right]^{-1}
    =\eta^N|_{C((0,N);\mathbb{R}^d)}, \quad \mbox{ for some (in fact all) } N>0.
\end{equation}
Since $\eta^N\in \mathcal{P}(C([0,N);\mathbb{R}^d))$, the above \eqref{eq:dagg00_eta_identification} shows that
\begin{equation}\label{eq:exists_Z}
    \exists\; Z:\Omega\rightarrow\mathbb{R}^d \mbox{ which is } \mathcal{F}^{\dagger,00}(0)\mbox{-measurable such that } \lim_{t\to 0} \|X_2^{\dagger,00}(t)-Z\|=0\quad   \mathbb{P}^{\dagger,00}_{\delta_0\otimes \mu_0}\mbox{-a.e.} 
\end{equation}
Moreover, by \eqref{eq:Udagger0=eta} it is easy do deduce that
\begin{equation*}
    X_1^{\dagger,00}(t)=t, \quad t\geq 0, \quad \mathbb{P}^{\dagger,00}_{(0,x)}\mbox{-a.e. for every } (0,x)\in E^{\dagger,00}. 
\end{equation*}
Hence, in order to finish {\bf Step OP2.7}, we need to prove the identification
\begin{equation}\label{eq:Z_identification}
    X_2^{\dagger,00}(0)=Z \quad \mathbb{P}^{\dagger,00}_{\delta_0\otimes \mu_0}\mbox{-a.e.} 
\end{equation}
To prove this, let us set for each $z\in \mathbb{R}^d$
\begin{equation*}
    F_z\in C_b(E^{\dagger,00}),\quad F_z(t,x):=\|z-x\|\wedge 1, \quad (t,x)\in\mathbb{R}^d,
\end{equation*}
and note that by \eqref{eq:alphaU_alpha} we have
\begin{equation*}
    \lim_{\alpha\to \infty}\alpha U_\alpha^{\dagger,00}F_z(0,x)=F_z(0,x), \quad (0,x)\in E^{\dagger,00}.
\end{equation*}
Hence, using that $X^{\dagger,00}$ is a Markov process with transition function $\left(P_t^{\dagger,00}\right)_{t\geq 0}$ and resolvent $\mathcal{U}^{\dagger,00}$, we have
\begin{align}\label{eq:lim_identification}
    \lim_{\alpha\to \infty}\int_0^\infty e^{-\alpha t}
    \mathbb{E}^{\dagger,00}_{\delta_0\otimes\mu_0}&\left\{\left\|X_2^{\dagger,00}(t)-X_2^{\dagger,00}(0)\right\|\wedge 1\right\}\;dt\\
    &=\lim_{\alpha\to\infty}\int_0^\infty e^{-\alpha t}\mathbb{E}^{\dagger,00}_{\delta_0\otimes\mu_0}\left\{F_{X_2^{\dagger,00}(0)}(X^{\dagger,00}(t))\right\}\; dt\\
    &=\lim_{\alpha\to\infty}\int_0^\infty e^{-\alpha t}\int_{E^{\dagger,00}(0)}\left[P_t^{\dagger,00}F_{x}\right]((0,x))\;\mu_0(dx)\\
    &=\lim_{\alpha\to\infty}\int_{E^{\dagger,00}(0)}\alpha \left[U_\alpha^{\dagger,00}F_x\right]((0,x))\;\mu_0(dx)\\
    &=\int_{E^{\dagger,00}(0)}\alpha F_x((0,x))\;\mu_0(dx)=0.
\end{align}
The above convergence entails that
\begin{equation*}
    \exists (t_n)_n\searrow 0 \quad\mbox{ such that }\quad \lim_n\mathbb{E}^{\dagger,00}_{\delta_0\otimes\mu_0}\left\{\left\|X_2^{\dagger,00}(t_n)-X_2^{\dagger,00}(0)\right\|\wedge 1\right\}=0,
\end{equation*}
hence by passing to a subsequence, we get
\begin{equation*}\label{eq:Z_identification_subsequnce}
    \exists (t_n)_n\searrow 0 \quad\mbox{ such that }\quad \lim_n\left\|X_2^{\dagger,00}(t_n)-X_2^{\dagger,00}(0)\right\|\wedge 1=0, \quad \mathbb{P}^{\dagger,00}_{\delta_0\otimes \mu_0}\mbox{-a.e.}.
\end{equation*}
Thus, using \eqref{eq:exists_Z} and \eqref{eq:Z_identification_subsequnce}, we get \eqref{eq:Z_identification}, hence {\bf Step OP2.7} is finished.

\medskip
\noindent{\bf Step OP2.8:} Let now
\begin{equation*}
     E^{\dagger,00}_0:=\left\{(s,x)\in  E^{\dagger,00} : \mathbb{P}_{(s,x)}^{\dagger,00}\left([0,\infty)\ni t\mapsto X^{\dagger,00}(t)\in E^{\dagger,00} \mbox{ is norm-continuous} \right)=1\right\},
\end{equation*}
and notice that from the previous steps we have that
\begin{equation*}
    E^{\dagger,00}_0(t)=E^{\dagger}(t), \quad \mbox{for every } t>0.
\end{equation*}
Using now \eqref{eq:continuous s>0} and \eqref{eq:continuous s=0}, we can apply \Cref{prop:continuous_modification_special} to deduce that $E_0^{\dagger,00}$ is $\delta_0\otimes\mu_0$-inessential and satisfies
\begin{equation*}
    E_0^{\dagger,00}(0)\in \mathcal{B}(\mathbb{R}^d),\quad E_0^{\dagger,00}(0)\subset E^{\dagger,00}(0),\quad \mu_0(E_0^{\dagger,00}(0))=1.
\end{equation*}

\medskip
\noindent{\bf Step OP2.9: Finalizing the proof of \Cref{thm:main1}.} Based on the previous step, let us consider the right process
\begin{equation*}
    X=\left(\Omega, \mathcal{F}, \mathcal{F}(t) , X(t), \theta(t) , \mathbb{P}_{(s,x)}, (s,x)\in E\right),  
\end{equation*}
with trasition semigroup $(P_t)_{t\geq 0}$ and resolvent $\mathcal{U}=(U_\alpha)_{\alpha>0}$, obtained by taking the restriction of $X^{\dagger,00}$ from $E^{\dagger,00}$ to $E:=E^{\dagger,00}_0$ as in \Cref{prop:restriction}-\Cref{defi:restr-process}.
Then clearly $X$ has all the desired properties in \Cref{thm:main0}, with the mention that assertion (ii.5) from \Cref{thm:main0} is at the moment satisfied for every $(s,x)\in E$ with $s>0$.
To extend (ii.5) to $s=0$ we first notice that
\begin{equation}\label{eq:delta0mu0}
    \delta_0\otimes\mu_0\left(U_\alpha |\overline{L}f|\right)<\infty, \quad f\in C_c^\infty((0,\infty)\times\mathbb{R}^d),\quad \alpha>0.
\end{equation}
Then, by the same arguments as in {\bf Step OP1.3} that follow after \eqref{eq:M_f}, we can find a further set $\tilde{E}\subset E$ such that $E\setminus E_0$ is $\delta_0\otimes \mu_0$-inessential, and $M_f$ given by \eqref{eq:Mf} with $X^\dagger$ replaced by $X$, is well defined and integrable under $\mathbb{P}_{(0,x)}$ for every $(0,x)\in \tilde{E}$ and $f\in C_c^\infty((0,\infty)\times\mathbb{R}^d)$.
Moreover, $[0,1]\ni t\mapsto M_f(t)$ is right continuous in $L^1(\mathbb{P}_{\delta_{(0,x)})})$.
Consequently, using the fact that $X(s)\in E^{\dagger,00}, s>0$, for every $t>s>0$
\begin{align*}
    \mathbb{E}_{(0,x)}( M_f(t))
    &=\mathbb{E}_{(0,x)}\left\{ M_f(s)+M_f(t)\circ \theta(s)\right\}
    =\mathbb{E}_{(0,x)}\left\{ M_f(s)\right\}+\mathbb{E}_{(0,x)}\left\{\mathbb{E}_{X(s)}\left\{M_f(t)\right\}\right\}\\
    &=\mathbb{E}_{(0,x)}\left\{ M_f(s)\right\}.
\end{align*}
By the above mentioned right continuity in $L^1(\mathbb{P}_{\delta_{(0,x)})})$, and using that we clearly have $\mathbb{E}_{(0,x)}\left\{ M_f(0)\right\}=0$, it follows by \Cref{lem:martingale 1} that $M_f$ is a martingale under $\mathbb{P}_{\delta_{(0,x)})}$ for every $(0,x)\in \tilde{E}$.
Taking the restriction of the process $X$ from $E$ to $\tilde{E}$ we can thus ensure (ii.5) from \Cref{thm:main0}.

To conclude, (ii.1') and (ii.4') in \Cref{thm:main1} are clearly satisfied, whilst (ii.6') follows similarly to {\bf Step OP1.16}, (ii.6).

\section{Consequences of the main results}\label{s:consequences}

In this part we use the right processes constructed in \Cref{s:reg_superposition_linear}, \Cref{thm:main0} and \Cref{thm:main1}, in order to deduce new results about time-dependent linear PDEs.
More precisely: We construct fundamental flow solutions; we solve parabolic Dirichlet problems by probabilistic means; by means of Lyapunov functions we derive maximal tail estimates for the path-space law obtained by the superposition of the fundamental flow solution to the time-dependent linear Fokker-Planck equation \eqref{eq:LFP}; we solve Fokker-Planck equations perturbed by nonlocal operators.

\subsection{Existence of fundamental flow solutions to linear Fokker-Planck equations}

Let $({\sf L}_t)_{t\geq 0}$ be given by \eqref{eq:L_t}-\eqref{eq:b_sigma} and $(\mu_t)_{t\geq 0}\subset \mathcal{P}(\mathbb{R}^d)$ be a weakly continuous solution to the linear Fokker-Planck equation \eqref{eq:LFP} on $[0,\infty)$.
\begin{defi}\label{eq:sol_flow}
A pair $(E,\Gamma_\mu)$, where $E\in \mathcal{B} ([0,\infty)\times \mathbb{R}^d)$ and $\Gamma_\mu$ is a mapping
\begin{equation*}
    [0,\infty)\times E\times\mathcal{B}(\mathbb{R}^d) \ni(t,s,x,A)\overset{\Gamma_\mu}{\longmapsto}\Gamma_{\mu,t}^{s,x}(A)\in [0,1],
\end{equation*}
is called a $\mu$-restricted fundamental solution flow to the linear Fokker-Plank equation \eqref{eq:LFPshort} if the following properties are fulfilled:
\begin{enumerate}
    \item[(i)] We have
    $
        \mu_s(E(s))=1=\Gamma_{\mu,t}^{s,x}(E(s+t)) \quad \mbox{for every } t\geq 0, (s,x)\in E.
    $
    \item[(ii)] $\Gamma_\mu$ is a probability kernel from  $[0,\infty)\times E$  to  $\mathbb{R}^d$
    \item[(iii)] For every $t\geq 0, (s,x)\in E$
    it holds that
    $\left(\Gamma_{\mu,t}^{s,x}\right)_{t\geq 0}$ is a weakly continuous solution to the linear Fokker-Planck equation \eqref{eq:LFPshort} on $[s,\infty)$, with initial condition $\Gamma_{\mu,0}^{s,\nu_s}=\delta_x$,
    \item[(iv)] Let us set
    \begin{equation*}
        \Gamma_{\mu,t}^{s,\xi}:=\int_{E(s)}\Gamma_{\mu,t}^{s,x} \; \xi(dx), \quad t,s\geq 0, \xi\in \mathcal{P}(E(s)).
    \end{equation*}
    For every $s\geq 0$ and $\nu_s\in \mathcal{P}(E(s))$, the following flow property holds:
    \begin{equation*}
        \Gamma_{\mu,t+r}^{s,\nu_s}
        =\Gamma_{\mu,t+r}^{s+r,\Gamma_{\mu,r}^{s,\nu_s}}, \quad r,t\geq 0.
    \end{equation*}
    \item[(v)] For every $s\geq 0$ and $\nu_s\in \mathcal{P}(E(s))$ for which
    \begin{equation*}
        \int_s^N\int_{B(0,N)} \left(\|b(t,x)\|+\|a(t,x)\| \right)\; \Gamma_{\mu,t}^{s,\nu_s}(dx) dt <\infty, \quad N>0,
    \end{equation*}
    it holds that $\left(\Gamma_{\mu,t}^{s,\nu_s}\right)_{t\geq 0}$ is a weakly continuous solution to the linear Fokker-Planck equation \eqref{eq:LFPshort} on $[s,\infty)$.
\end{enumerate}
\end{defi}

\begin{coro}\label{coro:fundamental}
Let $\left(\mu_t:=u(t,x)dx\right)_{t\geq 0}$ be a given solution to the linear Fokker-Planck equation \eqref{eq:LFPshort}, and assume that $\mathbf{H_{a,b}^{\sqrt{u}}}$, $\mathbf{H_{\sf \lesssim u}}$, and $\mathbf{H^{\sf TV}_0}$ hold.
Let
\begin{equation*}
X=\left(\Omega,\mathcal{F},\mathcal{F}(t),X(t)=(X_1(t),X_2(t)),\mathbb{P}_{s,x}, (s,x)\in E, t\geq 0\right)
\end{equation*} 
be the right process provided by both \Cref{thm:main0} and \Cref{thm:main1}.
Further, set
\begin{equation*}
    \Gamma_{\mu}:=\left(\Gamma_{\mu,t}^{s,x}\right)_{t\geq 0, (s,x)\in E},\quad \Gamma_{\mu,t}^{s,x}(A)
    :=\mathbb{P}_{s,x}\left(X_2(t)\in A\cap E(t+s)\right)
    ,\quad A\in \mathcal{B}(\mathbb{R}^d), s,t\geq 0.
\end{equation*}
Then $(E,\Gamma_{\mu})$ is a $\mu$-restricted fundamental solution flow to the linear Fokker-Planck equation \eqref{eq:LFPshort}.
\end{coro}
\begin{proof}
Note that (i) is automatically satisfied by \Cref{thm:main1}, whilst (ii) follows immediately from the properties of $(P_t)_{t\geq 0}$. 

Assertion (iii) can be deduced directly from \Cref{thm:main0}-\Cref{thm:main1}, (ii.5), and a similar observation as in \Cref{rem:mp_LFP}.

Assertion (iv) follows from the semigroup property of $(P_t)_{t\geq 0}$ and the fact that $X_1$ is the uniform motion to the right.

Finally, (v) follows by \Cref{thm:main1}, (ii.6').
\end{proof}

\begin{rem} Under the hypotheses of \Cref{coro:fundamental}, it follows that for every $s\geq 0$ and $\nu_s\in \mathcal{P}(E(s))$ such that there exists a constant $c\in (0,\infty)$ with $\nu_s\leq \mu_s$, it holds that it holds that
$\left(\Gamma_{\mu,t}^{s,\nu_s}\right)_{t\geq 0}$ is the unique weakly continuous solution to the linear Fokker-Planck equation \eqref{eq:LFP} on $[s,\infty)$ which satisfies $\Gamma_{\mu,t}^{s,\nu_s}\leq c \mu_{t+s}$ for every $t\geq 0$.
\end{rem}

\subsection{Backward Kolmogorov equations and the parabolic Dirichlet problem}\label{ss:parabolic Dirichlet}

Let $\left(\mu_t:=u(t,x)dx\right)_{t\geq 0}$ be a given solution to the linear Fokker-Planck equation \eqref{eq:LFPshort}, and assume that $\mathbf{H_{a,b}^{\sqrt{u}}}$, $\mathbf{H_{\sf \lesssim u}}$, and $\mathbf{H^{\sf TV}_0}$ hold.
Also, recall that ${\sf L}_t$ is given by \eqref{eq:L_t}.

In this part, given an open subset $D\subset \mathbb{R}^d$ and a time horizon $0<T<\infty$, we are interested in giving a meaning of and constructing a solution to the following final value problem
\begin{equation}\label{eq:BKE}
\left\{\begin{aligned}
    \frac{d}{dt}\rho(t,x)+{\sf L}_t\rho(t,x)&=0, \quad &(t,x)\in (0,T)\times D, \;dt\otimes \mu_t(dx)\mbox{-a.e.}\\
    \phantom{ \frac{d}{dt}\rho(t,x)+{\sf L}_t}\rho(t,x) &=F(t,x), \quad &(t,x)\in (0,T)\times \partial D\\
    \phantom{ \frac{d}{dt}\rho(t,x)+{\sf L}_t}\rho(T,x) &=F(T,x), \quad &\quad \;\;x\in D,
\end{aligned}\right.\;,
\end{equation}
where $F\in C^2_c([0,\infty)\times \mathbb{R}^d)$ is given.

In order to solve \eqref{eq:BKE}, we adopt the strategy developed by E.B. Dynkin in \cite{Dy65II} in order to treat the so called {\it stochastic Dirichlet problem}; see also \cite[Section 9]{Ok03}, or \cite{Be11}. 
Crucially, we shall rely on the regularity properties of the balayage operator, as derived in \Cref{lem:B_is_finely_feller} and \Cref{prop:B-boundary} in \Cref{Appendix}, which essentially allow us to rigorously employ Doob's maximal inequality for continuous-time martingales, by ensuring the right-continuity of the trajectories. 
Let us point out that such a regularity argument seems to have been overlooked in the proof of \cite[Theorem 9.2.5]{Ok03} in order to guarantee that, using the notation from \cite{Ok03}, the martingale $(N_t)_t$ is right-continuous.

\begin{prop}\label{coro:BKE}
Let $\left(\mu_t:=u(t,x)dx\right)_{t\geq 0}$ be a given solution to the linear Fokker-Planck equation \eqref{eq:LFPshort}, and assume that $\mathbf{H_{a,b}^{\sqrt{u}}}$, $\mathbf{H_{\sf \lesssim u}}$, and $\mathbf{H^{\sf TV}_0}$ hold.
Further, let
\begin{equation*}
X=\left(\Omega,\mathcal{F},\mathcal{F}(t),X(t)=(X_1(t),X_2(t)),\mathbb{P}_{s,x}, (s,x)\in E, t\geq 0\right)
\end{equation*} 
be the right process provided by both \Cref{thm:main0} and \Cref{thm:main1}, $D\subset \mathbb{R}^d$ be open, $T\in (0,\infty)$, and $F\in b\mathcal{B}([0,\infty)\times \mathbb{R}^d)$ for which there exists $N>0$ such that $F(s,\cdot)=0, s\geq N$.
Furthermore, let us set
\begin{align*}
    \tau
    &:=\inf\left\{t>0 : X(t)\in [0,\infty)\times \mathbb{R}^d\setminus(0,T)\times D\right\}, \mbox{ and note that } \tau \leq T-s \;\mathbb{P}_{s,x}\mbox{-a.s.} \mbox{ for } s\in [0,T],\\
    \rho(t,x)&
    :=\mathbb{E}_{t,x}\left\{F(X(\tau))\right\},\quad t\in[0,\infty), x\in E(t).
\end{align*}
The following assertions hold.
\begin{enumerate}
    \item[(i)] $\rho$ is a solution to problem \eqref{eq:BKE} in the following sense:
    \begin{enumerate}
        \item[(i.1)] $\rho$ is $\mathcal{B}(E)$-measurable, finely continuous on $[[0,T)\times D]\cap E$ (it is finely continuous on $E$ if $F$ is finely continuous), and under $\mathbb{P}_{s,x}, \;(s,x)\in \left[[0,T)\times D\right]\cap E$,
        \begin{equation*}
            \left(\rho(X(t\wedge\tau))_{t\geq 0}\right) \mbox{ is a c\`adl\`ag } (\mathcal{F}(t\wedge \tau))\mbox{-martingale}\quad \mbox{and}\quad \lim\limits_{t\nearrow \tau }\rho(X(t))=F(X(\tau))\; \mbox{a.s.}
        \end{equation*}
        Moreover, we have
        $
            \rho(T,x)=F(T,x), \quad x\in D \;\mu_T\mbox{-a.e.}
        $
        \item[(i.2)] If $F$ is finely continuous and $F\in {\sf D({\overline L}_p)}$ for some $p\in (1,\infty)$ then $\rho \in {\sf D({\overline L}_q)}$ and
        \begin{equation*}
            {\sf L_q}\rho=0 \quad \mbox{on } [0,T)\times D \;\overline{\mu}\mbox{-a.e.} \quad \mbox{for every } 1\leq q\leq p.
        \end{equation*}
    \end{enumerate}
    \item[(ii)] $\rho$ is the unique solution to problem \eqref{eq:BKE} in both of the following meanings:
    \begin{enumerate}
        \item[(ii.1)] Let $v\in b\mathcal{B}(E)$ be such that under $\mathbb{P}_{s,x}, \;(s,x)\in \left[[0,T)\times D\right]\cap E$,
        \begin{equation*}
            \left(v(X(t\wedge\tau))_{t\geq 0}\right) \mbox{ is a c\`adl\`ag } (\mathcal{F}(t\wedge \tau))\mbox{-martingale}\quad \mbox{and}\quad \lim\limits_{t\nearrow \tau }v(X(t))=F(X(\tau))\; \mbox{a.s.}
        \end{equation*}
        Then $v(s,x)=\rho(s,x)$ for every $(s,x)\in [[0,T)\times D]\cap E$.
        \item[(ii.2)] Let $v\in b\mathcal{B}(E)$ be finely continuous such that
        \begin{equation}\label{eq:v_parabolic}
            v \in {\sf D({\overline L}_1)},\quad    {\sf L_1}v=0 \quad \mbox{on } [0,T)\times D \;\overline{\mu}\mbox{-a.e.}, \quad \mbox{and } \lim\limits_{t\nearrow \tau }v(X(t))=F(X(\tau)) \quad \mathbb{P}_{\overline{\mu}|_{(0,T)\times D}}\mbox{-a.e.}
        \end{equation} 
        Then there exists a set $E'\in \mathcal{B}(E)$ such that $E\setminus E'$ is $\overline{\mu}$-inessential and $v(t,x)=\rho(t,x)$ for every $(t,x)\in [[0,T)\times D]\cap E'$.
    \end{enumerate}
\end{enumerate}
\end{prop}

\begin{proof}
(i.1). Note first that in terms of the balayage operator $B^A$ given by \Cref{defi:balayage} and \Cref{defi:balayage_kernel}, we have the identification
\begin{equation*}
    \rho(t,x)=\left(B^{A}F\right)(t,x), \quad t\in [0,T],x\in E(t), \quad \mbox{where } A=E\setminus (0,T)\times D.
\end{equation*}
The fact that $\rho$ is $\mathcal{B}$-measurable and finely continuous on $[[0,T)\times D]\cap E$ (and finely continuous on $E$ if $F$ is finely continuous) follow from \Cref{Borel-measurab} and \Cref{lem:B_is_finely_feller}.
Further, note that $X$ is clearly a Hunt process (being a conservative right process with continuous paths), $F$ is finely continuous (since the norm topology is a natural topology), whilst $\tau$ is a predictable stopping time under $\mathbb{P}_{s,x}$ for $(s,x)\in \left[[0,T)\times D\right]\cap E$ since if we define
\begin{equation*}
    \tau_k:=\inf\left\{t>0 : X(t)\in [0,\infty)\times\mathbb{R}^d\setminus D_k\right\}, \quad \mbox{with } \overline{D_k}\subset D_{k+1}, D_k\nearrow [0,T)\times D, k\geq 1
\end{equation*}
then, under $\mathbb{P}_{s,x}$ for $(s,x)\in \left[[0,T)\times D\right]\cap E$, and using the path-continuity of $X$, there exists $k_0$ such that $\tau_k<\tau$ for $k\geq k_0$ and $\tau_k\nearrow \tau$ a.s.
Consequently, the first part of assertion (i.1) follows from \Cref{prop:B-boundary}, (v).

To prove the second part, recall that $\mu_T(E(T))=1$. 
By the path-continuity of $X$ and the fact that $X_1$ is the uniform motion to the right, we get
\begin{equation*}
    \tau=0 \quad \mathbb{P}_{T,x}\mbox{-a.s.} \quad x\in E(T)\cap D.
\end{equation*}
(i.1) is now completely proved.

\medskip
(i.2).
First of all,  since there exists $N>0$ such that $F(s,\cdot)=0, s\geq N$, by the Markov property we easily get 
\begin{align*}
    \overline{\mu}\left(|\rho|^q\right)
    =\int_0^\infty \delta_0\otimes\mu_0\left(P_t |\rho|^q\right)\;dt
    &\leq\int_0^\infty \delta_0\otimes\mu_0\left(\mathbb{E}_{(\cdot,\cdot)}\left\{|F|^q(t+\tau\circ\theta(t),X_2(t+\tau\circ\theta(t)))\right\} \right)\;dt\\
    &=\int_0^N \delta_0\otimes\mu_0\left(\mathbb{E}_{(\cdot,\cdot)}\left\{|F|^q(t+\tau\circ\theta(t),X_2(t+\tau\circ\theta(t)))\right\} \right)\;dt\\
    &\leq N\|F\|_\infty^q, \quad 1\leq q<\infty.
\end{align*}
Thus, $\rho\in L^q(\overline{\mu})$ for every $q\geq 1$.

Furthermore, by the strong Markov property and using that $\tau$ is a {\it terminal} time (\cite[p.65]{Sh88}), hence $\tau=t+\tau\circ \theta(t)$ on $\tau>t$ a.s. for every $t\geq 0$, we have
\begin{align*}
    P_t\rho(s,x)
    &=\mathbb{E}_{s,x}\left\{\mathbb{E}_{X(t)}\left\{F(X(\tau))\right\}\right\}
    =\mathbb{E}_{s,x}\left\{F(X(t+\tau\circ\theta(t)))\right\}\\
    &=\mathbb{E}_{s,x}\left\{F(X(\tau);\tau>t\right\}+\mathbb{E}_{s,x}\left\{F(X(t+\tau\circ\theta(t)));\tau\leq t\right\}\\
    &=\rho(s,x)+\mathbb{E}_{s,x}\left\{F(X(t+\tau\circ\theta(t)))-F(X(\tau));\tau\leq t\right\}, \quad (s,x)\in E.
\end{align*}
Since $t+(\tau\wedge t)\circ\theta(t))$ is also a stopping time (see e.g. \cite[Theorem (8.6)]{BlGe68}), and using that $X$ solves the martingale problem (see \Cref{thm:main0} and \Cref{thm:main1}), we deduce that 
\begin{align*}
    \mathbb{E}_{s,x}
    &\left\{F(X(t+\tau\circ\theta(t)))-F(X(\tau));\tau\leq t\right\}\\
    &=\mathbb{E}_{s,x}\left\{F(X(t+(\tau\wedge t)\circ\theta(t)))-F(X(\tau\wedge t));\tau\leq t\right\}\\
    &=\mathbb{E}_{s,x}\left\{\mathbb{E}_{s,x}\left\{F(X(t+(\tau\wedge t)\circ\theta(t)))-F(X(\tau\wedge t)) | \mathcal{F}(\tau\wedge t)\right\};\tau\leq t\right\}\\
    &=\mathbb{E}_{s,x}\left\{\mathbb{E}_{s,x}\left\{\int_{\tau\wedge t}^{t+(\tau\wedge t)\circ\theta(t)}\overline{\sf L} F(X(r))\;dr\;\bigg|\; \mathcal{F}(\tau\wedge t)\right\};\tau\leq t\right\}\\
    &=\mathbb{E}_{s,x}\left\{\int_{\tau\wedge t}^{t+(\tau\wedge t)\circ\theta(t)}\overline{\sf L} F(X(r))\;dr\;;\;\tau\leq t\right\}, \quad (s,x)\in E.
\end{align*}
Thus, corroborating the above we get
\begin{align*}
\left|\frac{P_t\rho(s,x)-\rho(s,x)}{t}\right|
    &\leq\mathbb{E}_{s,x}\left\{\frac{1}{t}\int_0^{2t}\left|\overline{\sf L} F\right|(X(r))\;dr\;;\;\tau\leq t\right\}\\
    &\leq\frac{1}{t}\int_0^{2t}P_s\left(\left|\overline{\sf L} F\right|\right)\;dr,
\end{align*}
hence, by e.g. Jensen's inequality,
\begin{align*}
\overline{\mu}\left(\left|\frac{P_t\rho(s,x)-\rho(s,x)}{t}\right|^p\right)
    &\leq2^{p-1}\frac{1}{t}\int_0^{2t}\overline{\mu}\left(P_r\left(\left|\overline{\sf L} F\right|^p\right)\right)\;dr\\
    &\leq2^p \overline{\mu}\left(\left|\overline{\sf L} F\right|^p\right).
\end{align*}
Since by \Cref{thm:main0} we have ${\overline{\sf L}}F\in L^p(\overline{\mu})$, we deduce that
\begin{equation*}
    \sup_{t\in [0,1]} \left\|\frac{P_t\rho-\rho}{t}\right\|_{L^p(\overline{\mu})}<\infty.
\end{equation*}
But now it is a standard fact in the theory of strongly continuous semigroups that $\rho\in {\sf D}({\sf \overline{L}_p})$; see e.g. the discussion before Proposition 4.4 in \cite{BeCi18}.
Furthermore, note that ${\sf \overline{L}_p} \rho =0$ on $(N,\infty)\times \mathbb{R}^d$ $\overline{\mu}$-a.e., where $N$ has been in the beginning of the proof of (iii).
Thus, since $\overline{\mu}([0,N]\times \mathbb{R}^d)=N<\infty$, we deduce that ${\sf \overline{L}_p} \rho \in L^1(\overline{\mu})$.
Since we also have $\rho \in L^1(\overline{\mu})$, we get that $\rho\in {\sf D}({\sf \overline{L}_1})$ as well. 

Moreover, using H\"older's inequality,
\begin{align*}
\left|\frac{P_t\rho(s,x)-\rho(s,x)}{t}\right|
    &\leq\mathbb{E}_{s,x}\left\{\frac{1}{t}\int_0^{2t}\left|\overline{\sf L} F\right|(X(r))\;dr\;;\;\tau\leq t\right\}\\
    &\leq2^{1/p}\left[\mathbb{E}_{s,x}\left\{\frac{1}{t}\int_0^{2t}\left|\overline{\sf L} F\right|^p(X(r))\;dr\right\}\right]^{1/p}\left[\mathbb{P}_{s,x}\left(\tau\leq t\right)\right]^{(p-1)/p}\\
    &=2^{1/p}\left[\frac{1}{t}\int_0^{2t}P_s\left(\left|\overline{\sf L} F\right|^p\right)\;dr\right]^{1/p}\left[\mathbb{P}_{s,x}\left(\tau\leq t\right)\right]^{(p-1)/p}.
\end{align*}
Consequently,
\begin{align*}
    &\int_{(0,T)\times D}\left|\frac{P_t\rho(s,x)-\rho(s,x)}{t}\right|\;\overline{\mu}(dt,dx)\\
    &\leq 2^{1/p}\overline{\mu}\left(\left[\frac{1}{t}\int_0^{2t}P_s\left(\left|\overline{\sf L} F\right|^p\right)\;dr\right]\right)^{1/p}\left(\int_{(0,T)\times D}\mathbb{P}_{s,x}\left(\tau\leq t\right)\;\overline{\mu}(dt,dx)\right)^{(p-1)/p}\\
    &\leq 2^{2/p}\overline{\mu}\left(\left|\overline{\sf L} F\right|^p\right)^{1/p}\left(\int_{[(0,T)\times D]\cap E}\mathbb{P}_{s,x}\left(\tau\leq t\right)\;\overline{\mu}(dt,dx)\right)^{(p-1)/p}
\end{align*}
which converges to $0$ by dominated convergence since due to \Cref{rem:finely open prob} we have
\begin{equation*}
    \mathbb{P}_{s,x}\left(\tau>0\right)=1, \quad (s,x)\in [(0,T)\times D]\cap E.
\end{equation*}
This clearly completes the proof of (i.2).

\medskip
(ii.1). Let $(\tau_k)_{k\geq 1}$ be as in the proof of (i.1) and $\tau=0$, so that from the hypothesis of (ii.2) we deduce that, under $\mathbb{P}_{s,x}, (s,x)\in [[0,T)\times D]\cap E$,
$(v(X(\tau_k)))_{k\geq 0}$ is a $(\mathcal{F}(\tau_k))$-martingale which is closed by $F(X(\tau))$. 
Therefore, 
\begin{equation*}
    v(s,x)=\mathbb{E}_{s,x}\left\{v(X(0))\right\}=\mathbb{E}_{s,x}\left\{F(X(\tau))\right\}, \quad (s,x)\in [[0,T)\times D]\cap E,
\end{equation*}
which proves the claim.

\medskip
(ii.2). 
Let $v:E\rightarrow \mathbb{R}$ be $\mathcal{B}(E)$-measurable, bounded, and finely continuous such that \eqref{eq:v_parabolic} is satisfied, and consider the process
\begin{equation*}
    M(t):=v(X(t))-v(X(0))-\int_0^t {\sf \overline{L}_1}v(X(s))\;ds,\quad t\geq 0, 
\end{equation*}
which, under $\mathbb{P}_{\overline{\mu}}$, is well defined; in particular it does not depend on the $\overline{\mu}$-version of ${\sf \overline{L}_1}v$.
Let us take $g:E\rightarrow\mathbb{R}$ $\mathcal{B}(E)$-measurable such that $g(x)=0, x\in (0,T)\times D$ and $g={\sf \overline{L}_1}v$ $\overline{\mu}$-a.e. and consider the process
\begin{equation*}
    \tilde{M}(t):=v(X(t))-v(X(0))-\int_0^t g(X(s))\;ds,\quad t\geq 0. 
\end{equation*}
Since $P_tv=v+\int_0^t P_s g \;ds$ $\overline{\mu}$-a.e., it follows by \Cref{lem:martingale 1} and \Cref{prop:martingale_absorbing} that there exists a set $E'\in \mathcal{B}(E)$ such that $E\setminus E'$ is $\overline{\mu}$-inessential and $(\tilde{M}(t))_{t\geq 0}$ is a well-defined right-continuous $\mathcal{F}(t)$-martingale under $\mathbb{P}_{s,x}$ for every $(s,x)\in E'$.
Moreover, since $g(x)=0, x\in (0,T)\times D$, we have $\mathbb{P}_{s,x}$-a.s. for $(s,x)\in [(0,T)\times D]\cap E'$ that
\begin{equation*}
    v(s,x)=v(X(t))-\tilde{M}(t), \quad t\in [0,\tau],
\end{equation*}
hence by \eqref{eq:v_parabolic} we get
\begin{equation*}
    v(s,x)=F(X(\tau))-\tilde{M}(\tau-).
\end{equation*}
Note that since $X$ is in fact a Hunt process, it follows from \cite[Theorem (47.6) and Exercise (47.7)]{Sh88}, that $(\mathcal{F}(t))_{t\geq 0}$ is {\it quasi-left continuous} (see \cite{Sh88} for details), hence by e.g. \cite[pp. 191-192]{Pr12}, no $(\mathcal{F}(t))$-martingale can jump at predictable times.
Thus, $\mathbb{P}_{s,x}$-a.s. for $(s,x)\in [(0,T)\times D]\cap E'$ we have 
\begin{equation*}
    \tilde{M}(\tau-)=\tilde{M}(\tau) \quad\mbox{and}\quad  F(X(\tau))=v(X(\tau)).
\end{equation*}
Since $\tau$ is bounded a.s. we get $\mathbb{P}_{s,x}$-a.s. for $(s,x)\in [(0,T)\times D]\cap E'$ that $\mathbb{E}_{s,x}\left\{M(\tau)\right\}=0$, hence finally
\begin{equation*}
    v(s,x)=\mathbb{E}_{s,x}\left\{F(X(\tau))\right\}\left(=\rho(s,x)\right).
\end{equation*}
\end{proof}

\begin{rem}
Let $\left(\mu_t:=u(t,x)dx\right)_{t\geq 0}$ be a given solution to the linear Fokker-Planck equation \eqref{eq:LFPshort}, and assume that $\mathbf{H_{a,b}^{\sqrt{u}}}$, $\mathbf{H_{\sf \lesssim u}}$, and $\mathbf{H^{\sf TV}_0}$ hold.
Further, let $T>0$, $F\in C^2_c(\mathbb{R}^d)$, and $\rho$ be the solution from \Cref{coro:BKE} for $D=\mathbb{R}^d$.
Then
\begin{equation}\label{eq:representation}
    \int_{\mathbb{R}^d} F(x)\; \mu_T(dx)=\int_{\mathbb{R}^d} \rho(0,x)\; \mu_0(dx).
\end{equation}
To see this, first of all note that since $D=\mathbb{R}^d$ whilst $\left(X_1(t)\right)_{t\geq 0}$ is the uniform motion to the right as ensured by \Cref{thm:main0}, it follows that $\tau$ given in \Cref{coro:BKE} satisfies
\begin{equation*}
    \tau=T-t \quad \mathbb{P}_{t,x}\mbox{-a.s.}\quad t\in [0,T], x\in E(t)
\end{equation*}
Thus,
\begin{equation*}
    \rho(t,x)
    :=\mathbb{E}_{t,x}\left\{F(X_2(T-t))\right\},\quad t\in[0,T], x\in E(t),
\end{equation*}
and the claim follows by \Cref{thm:main1}, (ii.4').

From the numerical standpoint, \eqref{eq:representation} indicates that in order to compute the integral $\mu_T(F)$, given $\mu_0,T,$ and $F$, one could attempt to solve backward in time equation \eqref{eq:BKE} in order to compute $\rho(0,\cdot)$, and then use \eqref{eq:representation}.
However, solving \eqref{eq:BKE} should be carefully understood in the $\mu$-restricted sense given by \Cref{coro:BKE}; from a probabilistic numerical perspective, one should be able to reverse in time the diffusion $(X_2(t))_{t\in [0,T]}$.
\end{rem}

\subsection{Lyapunov functions}
Let $\left(\mu_t:=u(t,x)dx\right)_{t\geq 0}$ be a given solution to the linear Fokker-Planck equation \eqref{eq:LFPshort}, and assume that $\mathbf{H_{a,b}^{\sqrt{u}}}$, $\mathbf{H_{\sf \lesssim u}}$, and $\mathbf{H^{\sf TV}_0}$ hold.
Further, let $X$ be the process provided by \Cref{thm:main0} and \Cref{thm:main1}, with transition function $(P_t)_{t\geq 0}$ and resolvent $\mathcal{U}=(U_\alpha)_{\alpha>0}$. 
In this part we are interested in deriving general maximal tail estimates of the type
\begin{equation*}
    \mathbb{P}_{(s,x)}\left( \sup_{t\in [0,T]} V(X_2(t)) \geq \varepsilon \right)\leq \frac{e^{c T}}{\varepsilon}V(x),
\end{equation*}
where $V$ is a suitable Lyapunov function.
Such estimates have been obtained in \cite{BoRoSh21} with $\mathbb{P}_{(s,x)}$ replaced by $\mathbb{P}_{\delta_0\otimes\mu_0}$. 
Here, we are interested in obtaining such estimates for every $(s,x)\in E$; see \Cref{prop:LV} below.

Throughout, for a function $V\in p\mathcal{B}(\mathbb{R}\times \mathbb{R}^d)$, by $P_tV$ and $U_\alpha V$ we mean $P_t(V|_E)$ and $U_\alpha (V|_E)$, respectively.

We also introduce the following measures
\begin{equation*}
    \overline{\mu}_\alpha(dx,dt):= e^{-\alpha t}\mu_t(dx)dt, \quad \alpha>0.
\end{equation*}

\begin{prop}\label{prop:mu-V}
Assume that there exists a Borel measurable function $V:[0,\infty)\times\mathbb{R}^d\rightarrow [0,\infty]$ and a constant $\delta_V\in (0,\infty)$ 
such that
\begin{equation}\label{eq:Lyapunov}
    \mu_0(V(0,\cdot))<\infty \quad \mbox{and} \quad \mu_t(V(t,\cdot))\leq e^{\delta_Vt}\mu_0(V(0,\cdot)), \quad t\geq 0.
\end{equation}
Then the following hold:
\begin{enumerate}
    \item[(i)] 
    $\overline{\mu}_\beta\left(U_{\alpha+\beta} F\right) \leq   \overline{\mu}(F)\wedge\frac{\overline{\mu}_\beta(F)}{\alpha}, \quad \beta,\alpha>0, \quad F\in p\mathcal{B}([0,\infty)\times \mathbb{R}^d).   
    $
    \item[(ii)]$\delta_0\otimes\mu_0\left(U_{\alpha} F\right) \leq   \overline{\mu}_\alpha(F)\quad F\in p\mathcal{B}([0,\infty)\times \mathbb{R}^d)$.
    \item[(iii)] $\overline{\mu}_\beta\left(U_\alpha V\right)\leq\frac{\mu_0(V(0,\cdot))}{(\alpha-\delta_V)(\beta-\delta_V)}, \quad \alpha,\beta>\delta_V$.
    \item[(iv)]
    $\delta_0\otimes\mu_0\left(U_\alpha V\right)\leq \frac{\mu_0(V(0,\cdot))}{\alpha-\delta_V}, \quad \alpha>\delta_V.$
\end{enumerate}
\end{prop}
\begin{proof}
\begin{align*}
    \overline{\mu}_\beta\left(U_{\alpha+\beta} F\right)
    &=\int_0^\infty e^{-(\alpha+\beta) t} \overline{\mu}_\beta\left(P_t F(\cdot,\cdot)\right)\;dt
    =\int_0^\infty e^{-(\alpha+\beta) t} \int_0^\infty e^{-\beta s}\mu_s\left(P_t F(s,\cdot)\right)\;ds\;dt\\
    &=\int_0^\infty e^{-(\alpha+\beta) t} \int_0^\infty e^{-\beta s}\mu_{t+s}\left( F(t+s,\cdot)\right)\;ds\;dt\\
    &=\int_0^\infty e^{-\alpha t} \int_0^\infty e^{-\beta (t+s)}\mu_{t+s}\left( F(t+s,\cdot)\right)\;ds\;dt\\
    &\leq\overline{\mu}(F)\wedge\left( \overline{\mu}_\beta(F)\int_0^\infty e^{-\alpha t}\;dt\right)\\
    &=\overline{\mu}(F)\wedge\frac{\overline{\mu}_\beta(F)}{\alpha}, \quad \beta,\alpha>0, F\in p\mathcal{B}([0,\infty)\times \mathbb{R}^d).\\
    \delta_0\otimes\mu_0\left(U_\alpha F\right)
    &=\int_0^\infty e^{-\alpha t} \mu_0\left(P_t F(0,\cdot)\right)\;dt
    =\int_0^\infty e^{-\alpha t} \mu_t\left(F(t,\cdot)\right)\;dt=\overline{\mu}_\alpha(F),\quad F\in p\mathcal{B}([0,\infty)\times \mathbb{R}^d)\\
    \overline{\mu}_\beta\left(U_\alpha V\right)
    &=\int_0^\infty e^{-\alpha t} \overline{\mu}_\beta\left(P_t V(\cdot,\cdot)\right)\;dt
    =\int_0^\infty e^{-\alpha t} \int_0^\infty e^{-\beta s}\mu_s\left(P_t V(s,\cdot)\right)\;ds\;dt\\
    &=\int_0^\infty e^{-\alpha t} \int_0^\infty e^{-\beta s}\mu_{t+s}\left( V(t+s,\cdot)\right)\;ds\;dt\\
    &\leq \mu_0(V(0,\cdot))\int_0^\infty e^{-(\alpha-\delta_V) t} \int_0^\infty e^{-(\beta-\delta_V)s}\;ds\;dt \\
    &=\frac{\mu_0(V(0,\cdot))}{(\alpha-\delta_V)(\beta-\delta_V)}, \quad \alpha,\beta >\delta_V.\\
    \delta_0\otimes\mu_0\left(U_\alpha V\right)
    &=\int_0^\infty e^{-\alpha t} \mu_0\left(P_t V(0,\cdot)\right)\;dt
    =\int_0^\infty e^{-\alpha t} \mu_t\left(V(t,\cdot)\right)\;dt\\
    &\leq \mu_0(V(0,\cdot))\int_0^\infty e^{-(\alpha-\delta_V)t} \;dt\\
    &=\frac{\mu_0(V(0,\cdot))}{\alpha-\delta_V}, \quad \alpha>\delta_V.
    \end{align*}
\end{proof}

\begin{prop}\label{prop:liminfV_n}
Assume that there exists $0\leq V_n\in C_b^2([0,\infty)\times \mathbb{R}^d),n\geq 1$, and $\delta_V\in(0,\infty)$ such that
\begin{equation}\label{eq:V_n}
\overline{\sf L}V_n(s,x)\leq \delta_V V(s,x),\quad V(s,x):=\liminf_n V_n(s,x), \quad (s,x)\in E.
\end{equation}
Then
\begin{equation}\label{eq:PtV<V}
    P_tV\leq e^{\delta_V t}V, \quad t\geq 0.
\end{equation}
Furthermore, if 
$\liminf_n \left({\sf \overline{L}}V_n\right)^-\geq {\sf \overline{L}}V$, then
\begin{equation}\label{eq:|LV|}
    \int_0^{t} P_r|\overline{\sf L}V|\; dr\leq V(2e^{\delta_Vt}-1), \quad t\geq 0.
\end{equation}
\end{prop}
\begin{proof}
Let us consider
\begin{equation*}
    \tau_k:=\inf\{t>0:X(t)\in E\setminus [0,k)\times B(0,k)\}, \quad k\geq 1.
\end{equation*}
Since $X$ has continuous paths and it is conservative, we deduce that
\begin{equation}\label{eq:tau_localization}
    \tau_k\nearrow_k \infty \quad \mathbb{P}_{s,x}\mbox{-a.s.}, \quad (s,x)\in E.
\end{equation}
Indeed, if we set $\tau_\infty:=\sup_k \tau_k$, then, since $X_1$ is the uniform motion to the right,
\begin{equation*}
    \lim_k\|X_2(\tau_k)\|=\infty \quad \mbox{on } \tau_\infty<\infty \quad \mbox{a.s.}
\end{equation*}
But $\sup\limits_{t\in [0,\tau_\infty]}\|X_2(t)\|<\infty $ on $\tau_\infty<\infty$ a.s., hence we must have \eqref{eq:tau_localization}.

Further, for each $n,k\geq 1$ there exists $0\leq V_{n,k}\in C_c^2([0,\infty)\times \mathbb{R}^d)$ such that
\begin{equation}\label{eq:Vn=Vnk}
    V_{n,k-1}=V_n \quad \mbox{on } (-1/k,k)\times B(0,k).
\end{equation}
On the one hand, by \Cref{thm:main0} and \Cref{thm:main1}, for every $(s,x)\in E$ and $k\geq 1$ we have that the stopped process
\begin{equation}\label{eq:mpVnk}
   V_{n,k}(X(t\wedge\tau_k))-V_{n,k}(X(0))-\int_0^{t\wedge\tau_k} \overline{\sf L}V_{n,k}(X(r))\; dr, \quad t\geq 0
\end{equation}
is a continuous $\left(\mathcal{F}_t\right)$-martingale under $\mathbb{P}_{(s,x)}$.
Consequently, using \eqref{eq:Vn=Vnk}, taking expectations in \eqref{eq:mpVnk}, and using also \eqref{eq:V_n}, we have for every $(s,x)\in E$
\begin{align*}
    \mathbb{E}_{(s,x)}\left\{V_{n,k}(X(t\wedge\tau_k))\right\}
    &= V_{n,k}(s,x)+\mathbb{E}\left\{\int_0^{t\wedge\tau_k} \overline{\sf L}V_{n,k}(X(r))\; dr\right\}\\
    &= V_{n,k}(s,x)+\mathbb{E}\left\{\int_0^{t\wedge\tau_k} \overline{\sf L}V_{n}(X(r))\; dr\right\}\\
    &\leq V_{n,k}(s,x)+\delta_V\mathbb{E}\left\{\int_0^{t\wedge\tau_k} V(X(r))\; dr\right\}\\
    &\leq V_{n,k}(s,x)+\delta_V\int_0^{t} \mathbb{E}\left\{V(X(r\wedge\tau_k))\; dr\right\}, \quad t\geq0, n,k\geq 1.
\end{align*}
On the other hand, note that
\begin{equation*}
    \liminf_nV_{n,k}=V \quad \mbox{on } [0,k)\times B(0,k), \quad k\geq 1.
\end{equation*}
Therefore, by Fatou's lemma we get for $k\geq 1$ and $(s,x)\in E\cap [0,k)\times B(0,k)$ that
\begin{equation*}
\mathbb{E}_{(s,x)}\left\{V(X(t\wedge\tau_k))\right\}\leq V(s,x)+\delta_V\int_0^{t} \mathbb{E}\left\{V(X(r\wedge\tau_k))\; dr\right\}, \quad t\geq 0.
\end{equation*}
By Gronwall's lemma we thus have for $k\geq 1$ and $(s,x)\in (s,x)\in E\cap [0,k)\times B(0,k)$ that
\begin{equation*}
    \mathbb{E}_{(s,x)}\left\{V(X(t\wedge\tau_k))\right\}\leq e^{\delta_V t}V(s,x), \quad t\geq 0.
\end{equation*}
Now, \eqref{eq:PtV<V} follows by letting $k$ go to infinity.

To prove \eqref{eq:|LV|}, we use the second equality after \eqref{eq:mpVnk} to get
\begin{align*}
    \mathbb{E}\left\{\int_0^{t\wedge\tau_k} \left(\overline{\sf L}V_{n}\right)^-(X(r))\; dr\right\}
    &\leq V_{n,k}(s,x)+\mathbb{E}\left\{\int_0^{t\wedge\tau_k} \left(\overline{\sf L}V_{n}\right)^+(X(r))\; dr\right\}\\
    &\leq V_{n,k}(s,x)+\delta_V\mathbb{E}\left\{\int_0^{t\wedge\tau_k} V(X(r))\; dr\right\}, \quad t\geq0, n,k\geq 1.
\end{align*}
Letting $k$ go to infinity we thus get
\begin{align*}
    \mathbb{E}_{s,x}\left\{\int_0^{t} \left(\overline{\sf L}V_{n}\right)^-(X(r))\; dr\right\}
    &\leq V_{n}(s,x)+\delta_V\mathbb{E}\left\{\int_0^{t} V(X(r))\; dr\right\}\\
    &\leq V_{n}(s,x)+\delta_V V\int_0^t e^{\delta_Vs}\;ds\\
    &=V_{n}(s,x)+V(e^{\delta_Vt}-1), \quad t\geq 0,n\geq 1.
\end{align*}
Letting now $n$ go to infinity and using Fattou's lemma we get
\begin{equation*}
    \int_0^{t} P_r\left(\overline{\sf L}V\right)^-\; dr\leq e^{\delta_Vt}V, \quad t\geq 0.
\end{equation*}
Since we clearly have
\begin{equation*}
    \int_0^{t} P_r\left(\overline{\sf L}V\right)^+\; dr
    \leq (e^{\delta_Vt}-1)V, \quad t\geq 0,
\end{equation*}
relation \eqref{eq:|LV|} follows.
\end{proof}

\begin{prop}
Let $V\in p\mathcal{B}([0,\infty)\times \mathbb{R}^d)$ and $\delta_V\in(0,\infty)$ such that
\begin{equation*}
    \mu_0(V(0,\cdot))<\infty, \quad P_tV\leq e^{\delta_V t}V,\quad t\geq 0.
\end{equation*}
Then
$
    \mu_t(V(t,\cdot))\leq e^{\delta_Vt}\mu_0(V(0,\cdot)), \quad t\geq 0,
$
and hence assertions (i)-(iii) from \Cref{prop:mu-V} hold.
\end{prop}
\begin{proof}
First of all notice that     
\begin{equation*}
    e^{-\delta_V t}P_t(V\wedge n)\leq V\wedge n,\quad t\geq 0,n\geq 1.
\end{equation*}
Consequently,
\begin{equation*}
    e^{-\delta_V t}\mu_0\left(P_t(V\wedge n)(0,\cdot)\right)\leq \mu_0\left(V(0,\cdot)\wedge n\right),
\end{equation*}
hence
\begin{equation*}
    e^{-\delta_V t}\mu_{t}\left(V(t,\cdot)\wedge n\right)\leq \mu_0\left(V(0,\cdot)\wedge n\right), \quad n\geq 1,
\end{equation*}
from which the statement follows.
\end{proof}

The main result of this subsection is the following.
\begin{prop}\label{prop:LV}
Let $V$ be such that
\begin{equation*}
    0\leq V\in C^2(\mathbb{R}^d), \quad \overline{\sf L}V\leq \delta_V V \quad \mbox{for some } \delta_V\in(0,\infty).
\end{equation*}    
Then the following assertions hold:
\begin{enumerate}
    \item[(i)] $P_tV\leq e^{\delta_V t}V, t\geq 0$.
    \item[(ii)] $U_{\alpha+\delta_V} |{\sf \overline{L}}V|\leq 2\frac{(\alpha+\delta_V)}{\alpha}V, \quad \alpha>0$.
    \item[(iii)] For every $\varepsilon>0$ we have
    \begin{equation}\label{eq:doob}
        \mathbb{P}_{(s,x)}\left( \sup_{t\in [0,T]} V(X_2(t)) \geq \varepsilon \right)\leq \frac{e^{\delta_V T}}{\varepsilon}V(x), \quad (s,x)\in E.
    \end{equation}
\end{enumerate}
\end{prop}
\begin{proof}
(i). For each $N\geq 1$ let $\psi_N\in C_b^\infty(\mathbb{R})$ such that
\begin{align*}
     \psi_N(t)=t, t\leq N-1,\quad \psi_N(t)=N, t>N+1, \quad 0\leq \psi_N'\leq 1, \quad\psi_N''\leq 0.
\end{align*}
Now, as in \cite[Proof of Proposition 2.2]{BoRoSh21}, we have
\begin{equation*}
\overline{\sf L}\psi_N(V)=\psi'_N(V)\overline{\sf L}V +\psi''_N(V)\|\sqrt{a}\nabla V\|^2\leq \psi'_N(V)\overline{\sf L}V,
\end{equation*}    
from which we deduce, by setting $V_N:=\psi_N(V)\in C_b^2(\mathbb{R}^d), N\geq 1$, that
\begin{equation*}
  \overline{\sf L}V_N\leq \delta_V V,\quad N\geq 1.  
\end{equation*}
By applying \Cref{prop:liminfV_n} we conclude that $P_tV\leq e^{\delta_V t}V, t\geq 0$.

\medskip
(ii). We first note that 
\begin{equation*}
    \lim_N\overline{\sf L} V_N = \overline{\sf L} V \quad \mbox{pointwise on } [0,\infty)\times \mathbb{R}^d.
\end{equation*}
Therefore, we can apply \Cref{prop:liminfV_n} to get \eqref{eq:|LV|}, namely
\begin{equation*}
    \int_0^{t} P_r|\overline{\sf L}V|\; dr\leq V(2e^{\delta_Vt}-1), \quad t\geq 0.
\end{equation*}
Consequently, integrating by parts
\begin{align*}
    \int_0^{t} e^{-(\alpha+\delta_V)r} P_r|\overline{\sf L}V|\; dr
    &=e^{-(\alpha+\delta_V)t}\int_0^{t} P_r|\overline{\sf L}V|\; dr
    +(\alpha+\delta_V)\int_0^{t} e^{-(\alpha+\delta_V)s}\int_0^s P_r|\overline{\sf L}V|\; dr\\
    &
    \leq e^{-\alpha t}V+2(\alpha+\delta_V)V\int_0^{t} e^{-\alpha s}\;ds, \\
    &=e^{-\alpha t}V+2\frac{(\alpha+\delta_V)}{\alpha}V (1-e^{-\alpha t}), \quad t\geq 0.
\end{align*}
Letting $t$ go to infinity we get \eqref{eq:|LV|}.

\medskip
(iii). To prove \eqref{eq:doob}, note that by the first part we have that $V$ is $\delta_V$-excessive, hence (see e.g. \Cref{thm 4.6}) $\left(e^{-\delta_Vt}V(X_2(t))\right)_{t\geq 0}$ is a right-continuous $\mathcal{F}(t)$-supermartingale under $\mathbb{P}_{(s,x)}$ for every $(s,x)\in E$.
\end{proof}
Thus, by Doob's maximal inequality we get
\begin{align*}
     \mathbb{P}_{(s,x)}\left( \sup_{t\in [0,T]} V(X_2(t)) \geq \varepsilon \right)
     &\leq\mathbb{P}_{(s,x)}\left( \sup_{t\in [0,T]} e^{-\delta_Vt}V(X_2(t)) \geq e^{-\delta_VT}\varepsilon \right)\\
    &\leq \frac{e^{\delta_V T}}{\varepsilon}V(x), \quad (s,x)\in E.
\end{align*}

\medskip
\noindent{\bf Example.} We follow \cite[Example 2.3]{BoRoSh21} and take
\begin{equation*}
    V(s,x):=V(x)=\log(1+\|x\|^2), \quad (s,x)\in \mathbb{R}\times \mathbb{R}^d.
\end{equation*}
Assume further that there exists $C\in (0,\infty)$ such that
\begin{equation*}
    \|a(t,x)\|\leq C \|x\|^2V(x)+C, \quad \langle b(t,x),x\rangle\leq C\|x\|^2V(x)+C.
\end{equation*}
It follows that there exists some constant $\delta_V\in (0,\infty)$ such that
\begin{equation*}
    \overline{\sf L}V\leq \delta_V V+\delta_V, \quad \mbox{or equivalently,}\quad \overline{\sf L}(V+1)\leq \delta_V (V+1)
\end{equation*}
Thus, by considering $V+1$ instead of $V$, the assumptions (and hence the conclusion) in \Cref{prop:LV} are verified.

\subsection{Adding jumps to time-dependent linear Fokker-Planck equations}\label{ss:jumps}
Let $\left(\mu_t:=u(t,x)dx\right)_{t\geq 0}$ be a given solution to the linear Fokker-Planck equation \eqref{eq:LFPshort}, and assume that $\mathbf{H_{a,b}^{\sqrt{u}}}$, $\mathbf{H_{\sf \lesssim u}}$, and $\mathbf{H^{\sf TV}_0}$ hold.
Also, let $X$ be the process on $E$ provided by \Cref{thm:main0} and \Cref{thm:main1}, with transition function $(P_t)_{t\geq 0}$ and resolvent $\mathcal{U}=(U_\alpha)_{\alpha>0}$. 

Let $\left(\mathcal{L},\mathcal{D}_b(\mathcal{L})\right)$ be the generator of $X$ defined through the resolvent as in \eqref{eq:D(L)_weak}.

\begin{rem}
    It is not hard to deduce from \Cref{thm:main0}, (ii.5), (see e.g. \Cref{lem:b locally bounded} from below),  that 
    \begin{equation*}
        {\sf \overline{L}}f(t,x)=\mathcal{L}f(t,x), \quad (t,x)\in E, \quad f\in C_c^2([0,\infty)\times \mathbb{R}^d)\cap \mathcal{D}_b(\mathcal{L}).
    \end{equation*}
    However, despite that both spaces $C_c^2([0,\infty)\times \mathbb{R}^d)$ and $\mathcal{D}_b(\mathcal{L})$ are rich enough in the sense that they both separate the set of finite measures on $E$, their intersection can be small if $b$ merely satisfies $\bf H_b$ from the beginning of \Cref{ss:more regularity}.
    For this reason, in what follows we pay special attention to the two operators, distinctly. 
\end{rem}

Let
$
    K:[0,\infty)\times \mathbb{R}^d\times \mathcal{B}(\mathbb{R}^d)\rightarrow [0,\infty)
$
be a bounded kernel from $[0,\infty)\times \mathbb{R}^d$ to $\mathbb{R}^d$, that is $K$ is measurable in the first two arguments, and a measure in the third argument, and
\begin{equation*}
   \sup_{t\geq 0, x\in \mathbb{R}^d}K(t,x,\mathbb{R}^d)<\infty. 
\end{equation*}

Let us consider the perturbed operators
\begin{align}
\mathcal{L}^Kf(t,x)
&:= \mathcal{L}f(t,x) +\int_{\mathbb{R}^d} \left[f(t,y)-f(t,x)\right]\; K(t,x,dy), \quad f\in \mathcal{D}_b(\mathcal{L}), \quad (t,x)\in E\nonumber\\
{\sf \overline{L}^K}f(t,x)
&:= {\sf \overline{L}}f(t,x) +\int_{\mathbb{R}^d} \left[f(t,y)-f(t,x)\right]\; K(t,x,dy), \quad f\in C^2_b([0,\infty)\times \mathbb{R}^d), \quad (t,x)\in [0,\infty)\times \mathbb{R}^d.\label{eq:L^K}
\end{align}

\begin{prop}\label{prop:X^K}
Let us keep the context fixed in the beginning of this section. 
Then there exists a right process $X^K$ on $E$,
\begin{equation*}
X^K=\left(\Omega^K,\mathcal{F}^K,\mathcal{F}^K(t),X^K(t)=(X^K_1(t),X^K_2(t)),\mathbb{P}^K_{s,x}, (s,x)\in E, t\geq 0\right)
\end{equation*}
which has a.s. c\`adl\`ag trajectories with respect to the (trace of the) norm topology on $E$, and such that the following properties hold:
\begin{enumerate}
    \item[(i)] For $(s,x)\in E$ we have
            \begin{equation*}
            \mathbb{P}^K_{s,x}\left( X_1^K(t)= t+s; 0\leq t<\infty  \right)=1.
            \end{equation*}
    \item[(ii)] For every $\psi\in pb\left(\mathcal{B}(\mathbb{R}^d)\otimes\mathcal{B}(\mathbb{R}^d)\right)$, $ \psi(y,y)=0, y\in \mathbb{R}^d$, we have
    \begin{equation*}
    \mathbb{E}^K_{s,x}
    \left\{\sum_{r\leq t}
    \psi(X_2^K(r-), X_2^K(r))\right\}
    =
    \mathbb{E}^K_{s,x}
    \left\{\int_0^t \int_{\mathbb{R}^d} 
    \psi(X_2^K(r), y) K(s+r,X_2^K(r), dy) dr\right\}, \quad t\geq 0.
    \end{equation*}
    \item[(iii)] If there exists $\delta_V \in (0,\infty)$ and $0\leq V\in C^2(\mathbb{R}^d)$ such that
    \begin{equation*}
        \overline{\sf L}V\leq \delta_V V \quad\mbox{and} \quad K(t,x,V(\cdot))\leq \delta_V V(x), \;t\geq 0, x\in \mathbb{R}^d,
    \end{equation*}
    then the transition semigroup $(P^K_t)_{t\geq 0}$ of $X^K$ satisfies
    \begin{equation*}
        P^K_t V\leq e^{2\delta_V t} V, \quad t\geq 0.
    \end{equation*}
    Moreover, the bound \eqref{eq:doob} holds for $X_2^K$ instead of $X_2$ and $2\delta_V$ instead of $\delta_V$.
    \item[(iv)]  We have that $\mathcal{L}^K$ coincides with the generator on $b\mathcal{B}(E^K)$ associated to $X^K$ in the sense given by \eqref{eq:D(L)_weak}, $\mathcal{D}_b(\mathcal{L})=\mathcal{D}_b(\mathcal{L}^K)$, and for every $f\in \mathcal{D}_b(\mathcal{L})=\mathcal{D}_b(\mathcal{L}^K)$ and every $(s,x)\in E$, the process
    \begin{equation}\label{eq:martingale_K}
        f(X^K(t))-  f(X^K(0)) -\int_0^t \mathcal{L}^Kf\left(X^K(s)\right) ds,\quad t\geq 0,
    \end{equation}
    is a c\`adl\`ag $(\mathcal{F}^K(t))$-martingale under $\mathbb{P}^K_{s,x}$.
\end{enumerate}
\end{prop}
\begin{proof}
The idea is to regard $K$ as a bounded kernel on $E$,
\begin{equation*}
    Kf(t,x):=K(t,x,f(t,\cdot)), \quad f\in b\mathcal{B}(E), \quad (t,x)\in E,
\end{equation*}
where $f(t,\cdot)$ is extended by zero on the complement of $E(t)=\{x\in \mathbb{R}^d : (t,x)\in E\}$.
Thus, by \Cref{prop:perturbation} there exists a c\`adl\`ag right Markov process 
\begin{equation*}
X^K=\left(\Omega^K,\mathcal{F}^K,\mathcal{F}^K(t),X^K(t)=(X^K_1(t),X^K_2(t)),\mathbb{P}^K_{s,x}, (s,x)\in E, t\geq 0\right)
\end{equation*}
with transition semigroup $(P^K_t)_{t\geq 0}$ and resolvent $\mathcal{U}^K:=(U^K_\alpha)_{\alpha>0}$ such that
\begin{equation}\label{eq:resolvent_formula}
    P_t^K=P_t'+\int_0^tP_s'KP^K_s \;ds \quad \mbox{and} \quad U_\alpha^K=U_\alpha'+U_\alpha'KU_\alpha^K, \quad t,\alpha>0,
\end{equation}
where $(P_t')_{t\geq 0}$ and $(U_\alpha')_{\alpha>0}$ de denote the transition semigroup and the resolvent of the process $X$ killed by the multiplicative functional induced by the function $K1$, as in \Cref{ss:adding jumps}.

Now, since $K$ in fact acts only on the spatial variable, it is straightforward to check, using \eqref{eq:resolvent_formula}, that $X^K_1$ is the uniform motion to the right, i.e. (i) holds. 

Assertions (ii) and (iv) are now just consequences of \Cref{prop:perturbation} and \Cref{prop:adding jumps diffusion}.

Let us prove (iii). 
First of all, notice that from \Cref{prop:LV} and the fact that $\mathcal{U}'$ is obtained from $\mathcal{U}$ by killing, we have
\begin{equation*}
    U_{\alpha+\delta_V}'V\leq U_{\alpha+\delta_V} V\leq\frac{1}{\alpha} V, \quad \alpha>0. 
\end{equation*}
Thus, by \eqref{eq:resolvent_formula} we get
\begin{align*}
    U_{\alpha+2\delta_V}^KV
    &=U_{\alpha+2\delta_V}'V+U_{\alpha+2\delta_V}'\sum\limits_{n\leq 1}(KU_{\alpha+2\delta_V}')^nV\\
    &\leq \frac{1}{\alpha+\delta_V}V+\frac{1}{\alpha+\delta_V}V\sum\limits_{n\geq 1}\left(\frac{\delta_V}{\alpha+\delta_V}\right)^n\\
    &=\frac{1}{\alpha}V, \quad \alpha>0.
\end{align*}
Now (iii) follows from by the inverse Laplace transform the proof of \Cref{prop:LV}, (iii).
\end{proof}

We now go further and present two cases when in \eqref{eq:martingale_K} we can take $f\in C_c^2([0,\infty)\times \mathbb{R}^d)$ and replace $\mathcal{L}^K$ by ${\sf \overline{L}^K}$. 

In the first case we consider locally bounded coefficients, and we start with the following lemma:

\begin{lem}\label{lem:b locally bounded}
 If $b$ is locally bounded then   \begin{equation*}
        C_c^2([0,\infty)\times \mathbb{R}^d)\subset \mathcal{D}(\mathcal{L})\quad \mbox{and}\quad {\sf \overline{L}}f=\mathcal{L}f, \quad f\in C_c^2([0,\infty)\times \mathbb{R}^d).
    \end{equation*} 
\end{lem}
\begin{proof}
    Recall that by \Cref{thm:main0}, (ii.5) we have that for every $(s,x)\in E$ and every $f\in C_c^{2}([0,\infty)\times \mathbb{R}^d)$ we have that
    \begin{equation*}
        f(X(t))-f(X(0))-\int_0^{t} \overline{\sf L}f(X(r))\; dr, \quad t\geq 0
    \end{equation*}
    is a continuous $\left(\mathcal{F}_t\right)$-martingale under $\mathbb{P}_{(s,x)}$.
    Moreover, since $b$ is locally bounded by assumption, whilst $a$ is locally bounded since $\bf H_a$ is fulfilled, it follows that 
    \begin{equation*}
        {\sf \overline{L}}f\in b\mathcal{B}(E), \quad f\in C_c^{2}([0,\infty)\times \mathbb{R}^d).
    \end{equation*}
    Now the claim follows from \Cref{prop:connection}.
\end{proof}

\begin{prop}\label{prop: b locally bounded}
    If $b$ is locally bounded then
    \begin{equation}\label{eq:ident L=cal{L}}
        C_c^2([0,\infty)\times \mathbb{R}^d)\subset \mathcal{D}(\mathcal{L}^K)\quad \mbox{and}\quad {\sf \overline{L}^K}f=\mathcal{L}^Kf, \quad f\in C_c^2([0,\infty)\times \mathbb{R}^d).
    \end{equation} 
    Consequently, if $X^K$ is the right process constructed in \Cref{prop:X^K}, then for every $(s,x)\in E$ and every $f\in C_c^2([0,\infty)\times \mathbb{R}^d)$ we have that
    \begin{equation*}
        \mathbb{E}^K_{(s,x)}\left\{\int_0^{t} \left|{\sf \overline{L}^K}f(X^K(r))\right|\; dr
               \right\}<\infty, \quad t\geq 0,
    \end{equation*}
    and 
    \begin{equation*}
        \left(f(X^K(t))-f(X^K(0))-\int_0^{t} {\sf \overline{L}^K}f(X^K(r))\; dr\right)_{t\geq 0} \; \mbox{ is a c\`adl\`ag } \left(\mathcal{F}^K_t\right)\mbox{-martingale under } \mathbb{P}^K_{(s,x)}.
    \end{equation*}
\end{prop}
\begin{proof}
Clearly, if we prove \eqref{eq:ident L=cal{L}}, then the second part of the statement is just a direct consequence of \Cref{prop:X^K}, (iv).

Let us prove \eqref{eq:ident L=cal{L}}. 
Recall first that by \Cref{prop:perturbation}, (iv), it follows that  the generator $(\mathcal{L}^K, \mathcal{D}_b(\mathcal{L}^K))$ associated to $X^K$ in the sense of \eqref{eq:D(L)_weak} satisfies $\mathcal{D}_b(\mathcal{L})=\mathcal{D}_b(\mathcal{L}^K)$, and $\mathcal{L}^K$ is given by \eqref{eq:L^K}.
Now, \eqref{eq:ident L=cal{L}} follows by \Cref{lem:b locally bounded}.
\end{proof}

In the second case the idea is to regard ${\sf \overline{L}^K}$ as a bounded perturbation of ${\sf \overline{L}}$ in $L^1(\overline{\mu})$, where recall that $\overline{\mu}$ denotes the measure $\mu_t(dx)dt$ on $[0,\infty)\times \mathbb{R}^d$.
To this end we introduce the following condition:

\medskip
\noindent{$\bf H_u(K)$} There exists $\delta \in (0,\infty)$
\begin{equation}\label{eq:KisL1}
    \int_0^\infty \int_{\mathbb{R}^d} K(t,x,f(t,\cdot))\; \mu_t(dx)\;dt\leq \delta \int_0^\infty \int_{\mathbb{R}^d} f(t,x) \;\mu_t(dx)\;dt, \quad f\in L^1(\overline{\mu}).
\end{equation}

\begin{rem}
    If the kernel $K$ admits a density $K(t,x,dy)=k(t,x,y)dy$ then \eqref{eq:KisL1} is equivalent to 
    \begin{equation*}
        \int_{\mathbb{R}^d}k(t,x,y)u(t,x)\;dx\leq \delta u(t,y) \quad dy\otimes dt\mbox{-a.e.}
    \end{equation*}
\end{rem}

\begin{lem}\label{lem:0mu0}
    Let $X^K$ be the right process given by \Cref{prop:X^K}, with resolvent $\mathcal{U}^K:=(U^K_\alpha)_{\alpha>0}$.
    If $\bf H_u(K)$ is fulfilled, then
    \begin{equation*}
        \|U^K_\alpha f\|_{L^1(E;\delta_0\otimes\mu_0)}\leq  \left(1+\frac{\delta}{\alpha}\right)\| f\|_{L^1(E;\overline{\mu})}, \quad f\in L^1(E;\overline{\mu}).
    \end{equation*}
\end{lem}
\begin{proof}
First of all note that by \Cref{prop:mu-V} we have 
\begin{equation*}
    \|U_\alpha f\|_{L^1(E;\delta_0\otimes\mu_0)}\leq  \| f\|_{L^1(E;\overline{\mu}_\alpha)}, \quad f\in L^1(E;\overline{\mu}_\alpha), \alpha>0.
\end{equation*}
By the fact that $U_\alpha\leq U_\alpha'$ as kernels, the fact that $\overline{\mu}$ is sub-invariant for $\mathcal{U}$, and using also $\bf H_u(K)$, we have
\begin{equation*}
    \overline{\mu}\left(KU'_\alpha f\right)\leq \delta \overline{\mu}\left(U'_\alpha f\right)\leq\delta \overline{\mu}\left(U_\alpha f\right)\leq\frac{\delta}{\alpha}\overline{\mu}(f), \quad f\in p\mathcal{B}(E).    
\end{equation*}
Now we note that by employing the resolvent formula \eqref{eq:resolvent_formula} we get
\begin{equation*}
U_\alpha^K f=U_\alpha'f+U_\alpha'\sum\limits_{n\geq 1}\left(KU'_\alpha\right)^nf, \quad f\in p\mathcal{B}(E),\alpha>0.
\end{equation*}
Consequently, for $f\in p\mathcal{B}(E)$ we have
\begin{align*}
\delta_0\otimes\mu_0\left(U_{\alpha+\delta}^K f\right)
&\leq \overline{\mu}(f)+\overline{\mu}\left(\sum\limits_{n\geq 1}\left(KU'_{\alpha+\delta}\right)^nf\right)\\
&\leq \overline{\mu}(f)+\overline{\mu}(f)\sum\limits_{n\geq 1} \left(\frac{\delta}{\alpha+\delta}\right)^n\\
&=\left(1+\frac{\delta}{\alpha}\right)\overline{\mu}(f), \quad f\in p\mathcal{B}(E).
\end{align*}
\end{proof}

\begin{prop}\label{prop: HuK}
Assume that $\bf H_u(K)$ hold.
Furthermore, let $({\sf \overline{L}_1}, {\sf D}({\sf \overline{L}_1}))$ be the generator on $L^1([0,\infty)\times \mathbb{R}^d ; \overline{\mu})$ provided by \Cref{thm:main0}.

Then there exists a maximal operator ${\sf \overline{L}^K_1}$ on $L^1([0,\infty)\times \mathbb{R}^d ; \overline{\mu})$ which extends $\sf \overline{L}^K$ given in \eqref{eq:L^K}, namely
\begin{align*}
    {\sf D}({\sf \overline{L}_1})&\ni f\mapsto {\sf \overline{L}^K_1}f \in L^1([0,\infty)\times \mathbb{R}^d ; \overline{\mu}) \\
    {\sf \overline{L}^K_1}f(t,x)&:= {\sf \overline{L}_1}f(t,x) +\int_{\mathbb{R}^d} \left[f(t,y)-f(t,x)\right]\; K(t,x,dy), \quad (t,x)\in E \;\overline{\mu}\mbox{-a.e.}, \quad f\in {\sf D}({\sf \overline{L}_1}),
\end{align*}
which is well defined and generates a $C_0$-semigroup $(T_t^K)_{t\geq 0}$ on $L^1([0,\infty)\times \mathbb{R}^d ; \overline{\mu})$.
Moreover, if $X^K$ is the right process constructed in \Cref{prop:X^K}, then the following properties hold:
\begin{enumerate}
    \item[(i)] The transition semigroup $(P_t^K)_{t\geq 0}$ of $X^K$ coincides with $(T_t^K)_{t\geq 0}$ as family of operators on $L^1([0,\infty)\times \mathbb{R}^d ; \overline{\mu})$
    \item[(ii)] There exists $E^K\in \mathcal{B}(E)$ such that $E\setminus E^K$ is $\overline{\mu}$-inessential  with respect to $\mathcal{U}^K$, $\mu_t\left(E^K(t)\right)=1$, $t\geq 0$, where recall that for a set $A\in [0,\infty)\times \mathbb{R}^d$ we define $A(t):=\left\{x\in \mathbb{R}^d : (t,x)\in A\right\}$,
    and for every $(s,x)\in E^K$ and every $f\in C_c^{\infty}((0,\infty)\times \mathbb{R}^d)$ we have that
            \begin{equation}\label{eq:LbarK finite}
                \mathbb{E}^K_{(s,x)}\left\{\int_0^{t} \left|{\sf \overline{L}^K}f(X^K(r))\right|\; dr
               \right\}<\infty, \quad t\geq 0,
            \end{equation}
            and 
            \begin{equation}\label{eq:LbarK mp}
                \left(f(X^K(t))-f(X^K(0))-\int_0^{t} {\sf \overline{L}^K}f(X^K(r))\; dr\right)_{t\geq 0} \; \mbox{ is a c\`adl\`ag } \left(\mathcal{F}^K_t\right)\mbox{-martingale under } \mathbb{P}^K_{(s,x)}.
            \end{equation}
\end{enumerate}
\end{prop}
\begin{proof}
From $\bf H_u(K)$ we deduce that $K$ induces a bounded operator on $L^1([0,\infty)\times \mathbb{R}^d; \overline{\mu})$. 
Therefore, the existence of the operator ${\sf \overline{L}^K_1}$ with domain ${\sf D(\overline{L}_1)}$ and generating a $C_0$-semigroup on $L^1([0,\infty)\times \mathbb{R}^d;\overline{\mu})$ follows by \Cref{prop:L1_approach}.

Moreover, assertion (i) follows from the same \Cref{prop:L1_approach}.

Let us prove (ii). 
To this end, 
first of all note that by \eqref{eq:bau} we deduce that
\begin{equation*}
    1_{[0,N]\times B(0,N)}\left(\|a\|+\|b\|\right)\in L^1([0,\infty)\times \mathbb{R}^d; \overline{\mu}), \quad N \geq 1,
\end{equation*}
hence, by (i) and \Cref{lem:0mu0},
\begin{equation*}
    U^K_1\left(1_{[0,N]\times B(0,N)}\left(\|a\|+\|b\|\right)\right)\in L^1([0,\infty)\times \mathbb{R}^d; \overline{\mu}+\delta_0\otimes\mu_0), \quad N\geq 1.
\end{equation*}
Consequently, by setting
\begin{equation*}
    \tilde{E}^K:=\bigcap\limits_{N\geq 1} \left[U^K_1\left(1_{[0,N]\times B(0,N)}\left(\|a\|+\|b\|\right)\right)<\infty\right]
\end{equation*}
we get $\tilde{E}^K\in \mathcal{A}(\mathcal{U}^K)$, see \Cref{rem:absorbing_intersection}, and $\overline{\mu}(E\setminus \tilde{E}^K)+\delta_0\otimes\mu_0(E\setminus \tilde{E}^K)=0$, hence $E\setminus \tilde{E}^K$ is $(\overline{\mu}+\delta_0\otimes\mu_0)$-inessential with respect to $\mathcal{U}^K$.

Note that \eqref{eq:LbarK finite} holds for every $(s,x)\in \tilde{E}^K$. 

Let us fix $f\in C_c^\infty((0,\infty)\times \mathbb{R}^d)$, and apply once again  \Cref{prop:L1_approach} to deduce that there exists a set $E_f\in \mathcal{B}(E)$ such that $E\setminus E_f$ is $\overline{\mu}$-inessential and for every $(s,x)\in E_f$ we have
that
\begin{equation*}
f(X^K(t))-f(X^K(0))-\int_0^{t} {\sf \overline{L}_1^K}f(X^K(r))\; dr, \quad t\geq 0
\end{equation*}
is a c\`adl\`ag $\left(\mathcal{F}^K_t\right)$-martingale under $\mathbb{P}^K_{(s,x)}$.
Now let us note that by (i) we have
\begin{equation*}
    \overline{\mu}\left(U_1^K\left|{\sf \overline{L}_1^K}f-{\sf \overline{L}^K}f\right|\right)=0,
\end{equation*}
hence, by \Cref{rem:exc_absorbing}, by passing to a further subset, we can assume that $E_f$ is such that
\begin{equation*}
    U_1^K\left|{\sf \overline{L}_1^K}f-{\sf \overline{L}^K}f\right|(s,x)=0, \quad (s,x)\in E_f,
\end{equation*}
and thus, for every $(s,x)\in E_f$,
\begin{equation*}
    \mathbb{E}^K_{(s,x)}\left\{\int_0^{t} \left|{\sf \overline{L}_1^K}f-{\sf \overline{L}^K}f\right|(X^K(r))\; dr
\right\}=0, \quad t\geq 0.
\end{equation*}

Let now $\mathcal{C}\subset C_c^2((0,\infty)\times \mathbb{R}^d)$ be dense (with respect to the uniform convergence up to the second derivatives) and also countable, and consider
\begin{equation*}
    E^K_0:=\tilde{E}^K\cap\bigcap\limits_{f\in \mathcal{C}} E_f.
\end{equation*}
Then $E^K_0\in \mathcal{A}(\mathcal{U}^K)$, see \Cref{rem:absorbing_intersection}, and $\overline{\mu}(E\setminus E^K_0)=0$, hence $E\setminus E^K_0$ is $\overline{\mu}$-inessential with respect to $\mathcal{U}^K$.
Moreover, for every $f\in \mathcal{C}$ we have that \eqref{eq:LbarK mp} is fulfilled.
Finally, by dominated convergence, it is easy to see that \eqref{eq:LbarK mp} holds for every $f\in C_c^\infty((0,\infty)\times \mathbb{R}^d)$.

Let us show that $\mu_t\left(E^K_0(t)\right)=1$ for $t> 0$ as well.
To this end, note that $\overline{\mu}(E\setminus E^K_0)=0$, there exists $t_n\searrow0$ such that $\mu_{t_n}(E^K_0(t_n))=1, n\geq 1$.
Moreover, from \eqref{eq:resolvent_formula} we have that for every $(s,x)\in E^K_0$
\begin{equation*}
    P_t^K (A)(s,x)=0 \quad \mbox{implies} \quad P_t (A)(s,x)=0, \quad A\in \mathcal{B}(E).
\end{equation*}
Since $X_1^K$ is the uniform motion to the right and $E^K_0\in \mathcal{A}(\mathcal{U}^K)$ (see \Cref{rem8.25}), we have
\begin{equation*}
  P_{t-t_n}^K([0,\infty)\times (E(t)\setminus E^K_0(t)))(t_n,x)=0,\quad   t>t_n>0, (t_n,x)\in E^K_0.
\end{equation*}
Consequently,
\begin{equation*}
  P_{t-t_n}([0,\infty)\times (E(t)\setminus E^K_0(t)))(t_n,x)=0,\quad   t>t_n>0, (t_n,x)\in E^K_0,
\end{equation*}
and thus
\begin{equation*}
  \mu_t(E(t)\setminus E_0^K(t)))=\mu_{t_n}\left(P_{t-t_n}([0,\infty)\times (E(t)\setminus E^K_0(t)))(t_n,\cdot)\right)=0,\quad   t>t_n>0.
\end{equation*}
This means that 
\begin{equation*}
    \mu_t(E^K_0(t))=1,\quad t>0.
\end{equation*}

Now, let us define
\begin{equation*}
    E^K_{00}:=E^K_0\setminus\left(\{0\}\times E^K_0(0)\right) ,
\end{equation*}
and note that since $X^K_1$ is the uniform motion to the right, we have $E^K_{00}\in \mathcal{A}(\mathcal{U}^K)$ and in fact that $E^K\setminus E^K_{00}$ is $\overline{\mu}$-inessential.

By \Cref{lem:0mu0} we have that
\begin{equation*}
    \delta_0\otimes\mu_0\left(U^K_11_{E\setminus E_{00}^K}\right)\leq \overline{\mu}\left(E\setminus E_{00}^K\right)=0,
\end{equation*}
from which we deduce that there exist $0<t_n\searrow 0$ such that
\begin{equation}\label{eq:t_n2}
    \delta_0\otimes\mu_0\left(P_{t_n}^K\left(1_{E^K_{00}}\right)\right)=1, \quad n\geq 1.
\end{equation}
Let us consider
\begin{equation*}
    Z_0:=\left\{x\in E^K(0):P_{t_n}^K\left(1_{E^K_{00}}\right)(0,x)=1,\; n\geq 1\right\},
\end{equation*}
so that, using \eqref{eq:t_n2}, we have $\mu_0(Z_0)=1$.
Now let us set
\begin{equation*}
    \quad E^K:=\left[\left(\{0\}\times Z_0\right)\cup E^K_{00}\right]\cap \tilde{E}^K.
\end{equation*}
Note that by \eqref{eq:t_n2}, the fact that $X^K_1$ is the uniform motion to the right, and using also \Cref{rem:absorbing_intersection}, we have that $E^K\in \mathcal{A}(\mathcal{U}^K)$, and in fact that $E\setminus E^K$ is $\overline{\mu}$-inessential.
Moreover, from the above discussion we also have 
\begin{equation*}
    \mu_t\left(E^K(t)\right)=1, \quad t\geq 0.
\end{equation*}

Furthermore, notice that if $(s,x)\in E^K$ with $s>0$, we have already obtained that \eqref{eq:LbarK mp} holds for every $f\in C_c^\infty((0,\infty)\times \mathbb{R}^d)$.
It thus remains to show that \eqref{eq:LbarK mp} holds for every $f\in C_c^\infty((0,\infty)\times \mathbb{R}^d)$ and for every $(0,x)\in E^K$.
To this end, note that $X^K(s)\in E^K_{00}, s>0$, hence if we denote by $\left(M_f(t)\right)_{t\geq 0}$ the process in \eqref{eq:LbarK mp}, then for every $t>s>0$ 
\begin{align*}
    \mathbb{E}^K_{(0,x)}( M_f(t))
    &=\mathbb{E}^K_{(0,x)}\left\{ M_f(s)+M_f(t)\circ \theta^K(s)\right\}
    =\mathbb{E}^K_{(0,x)}\left\{ M_f(s)\right\}+\mathbb{E}^K_{(0,x)}\left\{\mathbb{E}_{X^K(s)}^K\left\{M_f(t)\right\}\right\}\\
    &=\mathbb{E}^K_{(0,x)}\left\{ M_f(s)\right\}.
\end{align*}
Since $[0,\infty)\ni t\mapsto M_f(t)$ is right continuous in $L^1(\mathbb{P}^K_{\delta_{(0,x)})})$, and using that we clearly have $\mathbb{E}^K_{(0,x)}\left\{ M_f(0)\right\}=0$, it follows by \Cref{lem:martingale 1} that $M_f$ is a martingale under $\mathbb{P}^K_{\delta_{(0,x)})}$ for every $(0,x)\in E^K$.
\end{proof}

\begin{defi}
Let $0\leq s< T<\infty$. 
Let $\left(\nu_t\right)_{t\in (s,T)}\subset \mathcal{M}(\mathbb{R}^d)$ be a Borel curve in $\mathcal{M}(\mathbb{R}^d)$.

\begin{enumerate}
    \item[(i)] $\left(\nu_t\right)_{t\in (s,T)}$ is called a (weak, or distributional) solution on $(s,T)$ to the linear Fokker-Planck equation  
    \begin{equation}\label{eq:LFPshortKbar}
        \frac{d}{dt} \nu_t = \left({\sf L}_t^K\right)^\ast \nu_t, \quad t\in (s,T),
    \end{equation}
    if
    \begin{equation*}
        \int_s^T\int_{B(0,R)} \|b(t,x)\|+\|a(t,x)\| \; \nu_t(dx) dt <\infty, \quad R>0,
    \end{equation*}
    and
    \begin{equation*}
        \int_s^T\int_{\mathbb{R}^d} {\sf \overline{L}^K}f(t,x) \;\nu_t(dx) dt = 0, \quad f\in C_c^\infty \left((s,T)\times \mathbb{R}^d\right).
    \end{equation*}  
    \item[(ii)] $\left(\nu_t\right)_{t\in [s,T]}$ is called a (weak, or distributional) solution on $(s,T)$ to the linear Fokker-Planck equation  
    \begin{equation}\label{eq:LFPshortKcal}
        \frac{d}{dt} \nu_t = \left(\mathcal{L}_t^K\right)^\ast \nu_t, \quad t\in (s,T),
    \end{equation}
    if it is weakly continuous and
    \begin{equation*}
        \nu_T(f(T,\cdot))=\nu_s(f(s,\cdot))+\int_s^T\int_{\mathbb{R}^d} {\mathcal{L}^K}f(t,x) \;\nu_t(dx) dt , \quad f\in \mathcal{D}_b(\mathcal{L}),
    \end{equation*}  
    where recall that $\mathcal{D}_b(\mathcal{L})$ is given by \eqref{eq:D(L)_weak}, whilst $\mathcal{L}^K$ is given by \eqref{eq:L^K}; see also \Cref{prop:X^K}, (iv).
\end{enumerate}
If $(\nu_t)_{t\in[s, T]}\subset \mathcal{M}(\mathbb{R}^d)$ is weakly continuous and is a solution to \eqref{eq:LFPshortKcal} or \eqref{eq:LFPshortKbar} from above, then $\nu_s$ is considered to be the initial condition of the corresponding Fokker-Planck equation.
\end{defi}

\begin{defi}\label{eq:sol_flow_K}
A pair $(E,\Gamma_\mu)$, where $E\in \mathcal{B} ([0,\infty)\times \mathbb{R}^d)$ and $\Gamma_\mu$ is a mapping
\begin{equation*}
    [0,\infty)\times E\times\mathcal{B}(\mathbb{R}^d) \ni(t,s,x,A)\overset{\Gamma_\mu}{\longmapsto}\Gamma_{\mu,t}^{s,x}(A)\in [0,1],
\end{equation*}
is called a $\mu$-restricted fundamental solution flow to the linear Fokker-Plank equation \eqref{eq:LFPshortKcal} (respectively \eqref{eq:LFPshortKbar}) associated to $\mathcal{L}^K$ (respectively $\sf \overline{L}^K$)
if (i)-(v) from \Cref{eq:sol_flow} are fulfilled with the linear Fokker-Planck equation \eqref{eq:LFP} replaced by \eqref{eq:LFPshortKcal} (respectively \eqref{eq:LFPshortKbar}).
\end{defi}

\begin{coro}
Let $X^K$ be the right process provided by \Cref{prop:X^K}.
Further, for every $A\in \mathcal{B}(\mathbb{R}^d), s,t\geq 0, (s,x)\in E$, set
\begin{equation*}
    \Gamma_{\mu}^K:=\left(\Gamma_{\mu,t}^{K,s,x}\right)_{t\geq 0, (s,x)\in E},\quad \Gamma_{\mu,t}^{K,s,x}(A)
    :=\mathbb{P}^{K}_{s,x}\left(X_2^K(t)\in A\cap E(t+s)\right).
\end{equation*}
Then the following assertions hold
\begin{enumerate}
    \item[(i)] $(E,\Gamma_{\mu})$ is a $\mu$-restricted fundamental solution flow to the linear Fokker-Planck equation \eqref{eq:LFPshortKcal}.
    Moreover, if $b$ is locally bounded then $(E,\Gamma_{\mu})$ is a $\mu$-restricted fundamental solution flow to the linear Fokker-Planck equation \eqref{eq:LFPshortKbar}.
    \item[(ii)] If $\bf H_u(K)$ is fulfilled and $E^K$ is given by \Cref{prop: HuK}, then $\left(E^K, \left(\Gamma_{\mu,t}^{K,s,x}\right)_{t\geq 0, (s,x)\in E^K}\right)$ is a $\mu$-restricted fundamental solution flow to the linear Fokker-Planck equation \eqref{eq:LFPshortKbar}.
\end{enumerate}
\end{coro}
\begin{proof}
    The proof is very similar to that of \Cref{coro:fundamental}, so we omit it.
\end{proof}

\section{Construction of reference densities satisfying \texorpdfstring{${\bf H_{a,b}^{\sqrt{u}}}$, ${\bf H_{\lesssim u}}$ and ${\bf H_0^{TV}}$}{}}\label{s:construction}

In this section we first show that condition $\bf H^{\sqrt{u}}$ from the beginning of \Cref{s:reg_superposition_linear} is valid under very mild additional assumptions.
We mention that when the diffusion coefficients $a$ satisfy a Lipschitz condition in the space variable and a $1/2$-H\"older in the time variable, then $\bf H^{\sqrt{u}}$ can be ensured through \cite[Theorem 7.4.1]{BoKrRoSh22}.
However, it turns out (see \Cref{coro: sqrt}) that $\bf H^{\sqrt{u}}$ is valid under the much weaker conditions \eqref{eq:b da in L2 loc duplicate}, \eqref{eq:ellipticity_duplicate}, and \eqref{eq:hu in H-1} from below. 

Then, we show that the Trevisan's well-posendess conditions from \cite{Tr16} in the nondegenerate case are explicit sufficient conditions for the main general assumptions $\bf H_{a,b}^{\sqrt{u}}$, ${\bf H_{\lesssim u}}$, and ${\bf H_0^{TV}}$ from \Cref{s:reg_superposition_linear} to be fulfilled.

\subsection{The square root \texorpdfstring{$H^1$}{}-regularity for Fokker-Planck equations}
Let $N>0$, $\left(u(t,x)dx\right)_{t\in [0,N]}\subset \mathcal{P}(\mathbb{R}^d)$ be a weakly-continuous solution on $[0,N]$ to the linear Fokker-Planck equation
\begin{equation}\label{eq:LFP_duplicate}
    \int_{(0,N)\times \mathbb{R}^d} \left[\frac{d}{dt}f(t,x)+ {\sf L}_t f(t,x)  \right] u(t,x) \;dx dt = 0, \quad f\in C_c^\infty \left((0,N)\times \mathbb{R}^d\right),
\end{equation}
where
\begin{equation*}
    b:[0,\infty)\times \mathbb{R}^d \rightarrow \mathbb{R}^d \quad \mbox{ and } \quad a:[0,\infty)\times\mathbb{R}^d\rightarrow \mathbb{R}^{d\times d}
\end{equation*}
are Borel measurable coefficients, such that $A(t,x):=\left(a_{i,j}(t,x)\right)_{1\leq i,j\leq d}$ is a symmetric and non-negative definite real matrix for all $(t,x)\in [0,\infty)\times \mathbb{R}^d$.

Furthermore, we assume that for every $r>0$ we have
    \begin{equation}\label{eq:b da in L2 loc duplicate}
        b_i \mbox{ and } \partial_{x_j} a_{i,j} \mbox{ belong to }L^2\left([0,N)\times B(0,r)\right) , \quad 1\leq i,j\leq d,
    \end{equation} 
and there exists a constant $C=C(r,N)>0$ such that
    \begin{equation}\label{eq:ellipticity_duplicate}
        C^{-1}\|\xi\|^2\leq \langle A(t,x) \xi, \xi\rangle \leq C \|\xi\|^2, \quad \xi\in \mathbb{R}^d, (t,x)\in [0,N)\times B(0,r).
    \end{equation}

Let
\begin{equation}\label{eq:b_tilde}
    \tilde{b}_i(t,x):=\frac{1}{2}\sum_{i=1}^d\partial_{x_j}a_{i,j}(t,x),\quad (t,x)\in [0,N]\times \mathbb{R}^d,
\end{equation}
so that the operator $L_t$ can be rewritten as
\begin{equation*}
    L_tf=\langle b-\tilde{b},\nabla f\rangle+\frac{1}{2}{\sf div}\left(A\nabla f\right),\quad f\in C_c^\infty \left((0,N)\times \mathbb{R}^d\right).
\end{equation*}

Further, we assume that
\begin{equation}\label{eq:hu in H-1}
    hu\in L^\infty([0,N]\times \mathbb{R}^d; \mathbb{R}) \quad \mbox{and} \quad h\nabla u\in L^2([0,N)\times \mathbb{R}^d;\mathbb{R}^d), \quad \mbox{for all } h\in C_c^\infty(\mathbb{R}^d). 
\end{equation}
\begin{lem}\label[lem]{lem:hu in H-1}
If \eqref{eq:b da in L2 loc duplicate}, \eqref{eq:ellipticity_duplicate}, and \eqref{eq:hu in H-1} hold, then
\begin{equation*}
    hu\in W^{1,2}([0,N];H^{-1}(\mathbb{R}^d)), \mbox{ hence } hu\in C([0,T];L^2(\mathbb{R}^d)),\quad h\in C_c^\infty(\mathbb{R}^d).
\end{equation*}  
\end{lem}
\begin{proof}
Let $h\in C_c^\infty(\mathbb{R}^d)$ and $f\in C_c^\infty \left((0,N)\times \mathbb{R}^d\right)$, so that
\begin{equation*}
    \int_{(0,N)\times \mathbb{R}^d} \left[\frac{d}{dt}(hf)+ \langle b-\tilde{b},\nabla (hf)\rangle+\frac{1}{2}{\sf div}\left(A\nabla (hf)\right)  \right] u \;dx dt = 0.
\end{equation*}
Consequently,
\begin{align*}
    \int_{(0,N)\times \mathbb{R}^d} hu\frac{d}{dt}f\;dxdt
    &= -\int_{(0,N)\times \mathbb{R}^d}u\langle b-\tilde{b},\nabla (hf)\rangle\;dxdt
    - \int_{(0,N)\times \mathbb{R}^d} u{\sf div}\left(A\nabla (hf)\right)\;dxdt\\
    &=-\int_{(0,N)\times \mathbb{R}^d} u\langle b-\tilde{b},\nabla (hf)\rangle\;dxdt
    - \int_{(0,N)\times \mathbb{R}^d} \langle A\nabla u,\nabla (hf)\rangle\;dxdt.
\end{align*}
Thus, by the Cauch-Schwartz inequality
\begin{align*}
    \left|\int_{(0,N)\times \mathbb{R}^d} hu\frac{d}{dt}f\;dxdt\right|
    &\leq\int_{(0,N)\times \mathbb{R}^d} \left(\|u (b-\tilde{b})\|+\|A\nabla u\|\right)\left(\|\nabla h\|+|h|\right)\left(|f|+\|\nabla f\|\right)\;dxdt\\
    &\leq\left(\int_{(0,N)\times \mathbb{R}^d} \left(\|u (b-\tilde{b})\|+\|A\nabla u\|\right)^2\left(\|\nabla h\|+|h|\right)^2\;dxdt\right)^{1/2}\\
    &\quad \times \left(\int_{(0,N)\times \mathbb{R}^d} \left(|f|+\|\nabla f\|\right)^2\;dxdt\right)^{1/2}.
\end{align*}
Now, using \eqref{eq:b da in L2 loc duplicate}, \eqref{eq:ellipticity_duplicate}, and \eqref{eq:hu in H-1}, it is straightforward to get that there exists a constant $C<\infty$, independent of $f$, such that
\begin{equation*}
    \left|\int_{(0,N)\times \mathbb{R}^d} hu\frac{d}{dt}f\;dxdt\right|\leq  C\|f\|_{L^2([0,N];H^{1}(\mathbb{R}^d))},
\end{equation*}
which proves the first part of the statement.
The second part of the statement follows by the first one, \eqref{eq:hu in H-1}, and Lions-Magenes lemma.
\end{proof}

\begin{rem}\label{rem: implies TV}
    Note that by \Cref{lem:hu in H-1} and the fact that $(\mu_t)_{t\in[0,N]}$ is weakly continuous, hence tight, it is straightforward to deduce that 
    \begin{equation*}
        \lim_{t\to 0} u(t,\cdot)=u(0,\cdot) \quad \mbox{in } L^1(\mathbb{R}^d).
    \end{equation*}
\end{rem}

The following regularity result is the first main step towards the main result of this subsection, namely \Cref{coro: sqrt} from below. 
For its proof we got inspired from \cite{BoDaRo96}, where the uniform ellipticity is employed to prove the regularity of invariant measures for time-independent operators.

\begin{prop}\label{prop:u gamma}
If \eqref{eq:b da in L2 loc duplicate}, \eqref{eq:ellipticity_duplicate}, and \eqref{eq:hu in H-1} hold, then for every $\gamma\in(0,1)$ we have 
    \begin{equation*}
        \int_0^N\int_{\mathbb{R}^d}h\frac{\|\nabla (hu)\|^2}{(hu)^\gamma}\;dxdt<\infty,\quad 0\leq h\in C_c^\infty((0,N)\times \mathbb{R}^d).
    \end{equation*}
\end{prop}
\begin{proof}
Let $0\leq h\in C_c^\infty((0,N)\times \mathbb{R}^d)$ and set
\begin{equation*}
    u_{c}:= hu+c,\quad v_c:=(hu+c)^{1-\gamma},\quad c>0.
\end{equation*}
By Fatou's lemma we have
\begin{equation*}
    \int_0^N\int_{\mathbb{R}^d}h\frac{\|\nabla (hu)\|^2}{(hu)^\gamma}\;dxdt\leq \liminf_{c\to 0}\int_0^N\int_{\mathbb{R}^d}h\frac{\|\nabla u_c\|^2}{u_c^\gamma}\;dxdt.
\end{equation*}
We proceed with the key fact that
\begin{align*}
    \int_0^N\int_{\mathbb{R}^d}h\frac{\|\nabla u_c\|^2}{u_c^\gamma}\;dxdt
    &\leq C\int_0^N\int_{\mathbb{R}^d}\frac{h}{u_{c}^\gamma}\langle A\nabla u_{c},\nabla u_{c}\rangle\;dxdt\\
    &=\frac{C}{1-\gamma}\int_0^N\int_{\mathbb{R}^d}h\langle \nabla v_{c},A\nabla u_{c}\rangle\;dxdt\\
    &=\frac{C}{1-\gamma}\int_0^N\int_{\mathbb{R}^d}\langle \nabla (hv_{c}),A\nabla u_{c}\rangle\;dxdt
    -\frac{C}{1-\gamma}\int_0^N\int_{\mathbb{R}^d}\langle v_{c}\nabla h,A\nabla u_{c}\rangle\;dxdt.
\end{align*}
Further, note that for $0<c\leq 1$
\begin{align*}
    \left|\int_0^N\int_{\mathbb{R}^d}\langle v_{c}\nabla h,A\nabla u_{c}\rangle\;dxdt\right|
    &= \left|\int_0^N\int_{\mathbb{R}^d}\langle v_{c}A\nabla h,\nabla (hu)\rangle\;dxdt\right|\\
    &\leq\int_0^N\int_{\mathbb{R}^d} v_{c}\|A\nabla h\| \|\nabla (hu)\|\;dxdt\\
    &\leq \left(\int_0^N\int_{\mathbb{R}^d} (hu+1)^{2(1-\gamma)}\|A\nabla h\|^2 \;dxdt\right)^{1/2} \left(\int_0^N\int_{\mathbb{R}^d} \|\nabla (hu)\|^2\;dxdt\right)^{1/2}.
\end{align*}
The last two integrals are finite by \eqref{eq:ellipticity_duplicate}, \eqref{eq:hu in H-1}, and the fact that $h\in C_c^{\infty}((0,N)\times \mathbb{R}^d)$.

It remains to show that
\begin{equation}\label{eq:lim_inf_c is finte gamma}
    \liminf_{c\to 0}\int_0^N\int_{\mathbb{R}^d}\langle \nabla (hv_{c}),A\nabla u_{c}\rangle\;dxdt<\infty.
\end{equation}
To this end, note that the equation \eqref{eq:LFP_duplicate} satisfied by $u$ can be rewritten as
\begin{align}\label{eq:LFP divergence}
    0=\int_0^N\int_{\mathbb{R}^d}\left(\frac{d}{dt}f\right)u+f\langle b-\tilde{b},\nabla u\rangle
    -\langle\nabla f,A\nabla u\rangle\;dxdt,
\end{align}
hence, for every $f\in C_c^\infty((0,\infty)\times\mathbb{R}^d)$ we have
\begin{align}
    \int_0^N\int_{\mathbb{R}^d}\langle\nabla f,A\nabla (hu)\rangle\;dxdt
    &=\int_0^N\int_{\mathbb{R}^d}\langle h\nabla f,A\nabla u\rangle\;dxdt+\langle \nabla f,uA\nabla h\rangle\;dxdt\\
    &=\int_0^N\int_{\mathbb{R}^d}\langle \nabla (hf),A\nabla u\rangle\;dxdt
    -\langle f\nabla h,A\nabla u\rangle
    +\langle \nabla f,uA\nabla h\rangle\;dxdt\\
    &=\int_0^N\int_{\mathbb{R}^d}u\frac{d}{dt}(hf)+hf\langle b-\tilde{b},\nabla u\rangle
    -\langle f\nabla h,A\nabla u\rangle
    +\langle \nabla f,uA\nabla h\rangle\;dxdt,
    \label{eq:LFP-H^-1 gamma}
\end{align}
where for the last equality we have employed \eqref{eq:LFP divergence}.
Also, note that 
by a density argument, \eqref{eq:LFP-H^-1 gamma} remains valid for 
\begin{equation*}
f\in W_0^{1,2}((0,N); L^2(\mathbb{R}^d))\cap L^2((0,N);H^1(\mathbb{R}^d))\cap L^\infty((0,N)\times \mathbb{R}^d).
\end{equation*}

Let us now consider the following (smoothing) bounded linear operators
\begin{equation*}
    \Gamma_\alpha:=\alpha(\alpha-\Delta)^{-1}:L^2(\mathbb{R}^d)\rightarrow H^2(\mathbb{R}^d), \quad \alpha >0,
\end{equation*}
and recall that $\Gamma_\alpha$ extend to linear bounded operators 
\begin{equation*}
    \Gamma_\alpha : H^{-1}(\mathbb{R}^d)\rightarrow H^1(\mathbb{R}^d), \quad \Gamma_\alpha : L^p(\mathbb{R}^d)\rightarrow L^p(\mathbb{R}^d), \quad p\in [1,\infty].
\end{equation*}
Moreover, we have that $\Gamma_\alpha$ is a Markov operator on $L^p(\mathbb{R}^d)$ and
\begin{equation}\label{eq:Gamma-convergence}
   \lim_{\alpha\to \infty}\Gamma_\alpha u=u \;\mbox{ for } u\in H^1(\mathbb{R}^d),\quad \lim_{\alpha\to \infty}\Gamma_\alpha u=u\; \mbox{ for } u\in L^p(\mathbb{R}^d), p\in [1,\infty)
\end{equation}
Let us define
\begin{equation*}
u_{\alpha,c}:=\Gamma_\alpha u_c, \;v_{\alpha,c}:= u_{\alpha,c}^{1-\gamma},\quad \alpha,c>0,
\end{equation*}
so that, by \Cref{lem:hu in H-1} and \eqref{eq:hu in H-1},
\begin{equation*}
\Gamma_\alpha (hv_{\alpha,c})\in W_0^{1,2}((0,N); L^2(\mathbb{R}^d))\cap L^2((0,N);H^1(\mathbb{R}^d))\cap L^\infty((0,N)\times \mathbb{R}^d), \quad \alpha, c>0.
\end{equation*}
Plugging $\Gamma_\alpha (hv_{\alpha,c})$ instead of $f$ in \eqref{eq:LFP-H^-1 gamma} we obtain
\begin{align*}
&\int_0^N\int_{\mathbb{R}^d}
    \langle \nabla \Gamma_\alpha (hv_{\alpha,c}), A\nabla (hu)\rangle\;dxdt\\
    &=\int_0^N\int_{\mathbb{R}^d}u\frac{d}{dt}(h\Gamma_\alpha (hv_{\alpha,c}))
    +h\Gamma_\alpha (hv_{\alpha,c})\langle b-\tilde{b},\nabla u\rangle\\
    &\phantom{=\int_0^N\int_{\mathbb{R}^d}u\frac{d}{dt}(h\Gamma_\alpha (hv_{\alpha,c}))} 
    -\langle \Gamma_\alpha (hv_{\alpha,c})\nabla h,A\nabla u\rangle
    +\langle \nabla \Gamma_\alpha (hv_{\alpha,c}),uA\nabla h\rangle\;dxdt.
\end{align*}
Note that by \eqref{eq:Gamma-convergence}, \eqref{eq:hu in H-1}, and the fact that $\nabla$ and $\Gamma_\alpha$ commute, and $\Gamma_\alpha$ is $L^2(\mathbb{R}^d)$-symmetric, we have
\begin{align*}
\lim_{\alpha\to\infty}\int_0^N\int_{\mathbb{R}^d}
    \langle \nabla (\Gamma_\alpha (hv_{\alpha,c})), A\nabla (hu)\rangle\;dxdt
    &=\lim_{\alpha\to\infty}\int_0^N\int_{\mathbb{R}^d}
    \langle \nabla ( hv_{\alpha,c}), \Gamma_\alpha(A\nabla (hu)\rangle)\;dxdt\\
    &=\int_0^N\int_{\mathbb{R}^d}
    \langle \nabla (hv_{c}), A\nabla (hu)\rangle\;dxdt,
\end{align*}
hence showing \eqref{eq:lim_inf_c is finte gamma} boils down to showing
\begin{align}\label{eq:lim_inf_c 3 gamma}
    \liminf_{c\to 0} \lim_{\alpha\to\infty}
    &\left|\int_0^N\int_{\mathbb{R}^d}u\frac{d}{dt}(h\Gamma_\alpha (hv_{\alpha,c}))
    +h(\Gamma_\alpha (hv_{\alpha,c}))\langle b-\tilde{b},\nabla u\rangle\right.\\
    &\left.\phantom{\left(\frac{d}{dt}(h\Gamma_\alpha (hv_{\alpha,c}))\right)u} 
    -\langle (\Gamma_\alpha (hv_{\alpha,c}))\nabla h,A\nabla u\rangle
    +\langle \nabla (\Gamma_\alpha (hv_{\alpha,c})),uA\nabla h\rangle\;dxdt\right|<\infty.
\end{align}
In what follows we treat all the four terms in the last integral, separately:
For the first term, since $\frac{d}{dt}$ and $\Gamma_\alpha$ commute, and $\Gamma_\alpha$ is $L^2(\mathbb{R}^d)$-symmetric, we have
\begin{align*}
    \int_0^N\int_{\mathbb{R}^d}
    &u\frac{d}{dt}\left(h\Gamma_\alpha (hv_{\alpha,c}) \right)\;dxdt\\
    &=\int_0^N\int_{\mathbb{R}^d}u\Gamma_\alpha (hv_{\alpha,c})\frac{d}{dt}h
    +u_{\alpha,c}\frac{d}{dt}\left( hv_{\alpha,c} \right)\;dxdt\\
    &=\int_0^N\int_{\mathbb{R}^d}hv_{\alpha,c}\Gamma_\alpha\left(u\frac{d}{dt}h\right) 
    +\left(\frac{d}{dt}h\right)u_{\alpha,c}^{2-\gamma}
    +h(1-\gamma)u_{\alpha,c}^{1-\gamma}\frac{d}{dt}u_{\alpha,c}\;dxdt\\
    &=\int_0^N\int_{\mathbb{R}^d}hv_{\alpha,c}\Gamma_\alpha\left(u\frac{d}{dt}h\right)
    +\left(\frac{d}{dt}h\right)u_{\alpha,c}^{2-\gamma}
    +h\frac{1-\gamma}{2-\gamma}\frac{d}{dt}u_{\alpha,c}^{2-\gamma}\;dxdt\\
    &=\int_0^N\int_{\mathbb{R}^d}hv_{\alpha,c}\Gamma_\alpha\left(u\frac{d}{dt}h\right)
    +\left(\frac{d}{dt}h\right)u_{\alpha,c}^{2-\gamma}
    -\left(\frac{d}{dt}h\right)\frac{1-\gamma}{2-\gamma}u_{\alpha,c}^{2-\gamma}\;dxdt\\
    &\mathop{\longrightarrow}_{\alpha\to\infty}\int_0^N\int_{\mathbb{R}^d}hv_{c}u\frac{d}{dt}h
    +\left(\frac{d}{dt}h\right)u_{c}^{2-\gamma}
    -\left(\frac{d}{dt}h\right)\frac{1-\gamma}{2-\gamma}u_{c}^{2-\gamma}\;dxdt\\
    &\mathop{\longrightarrow}_{c\to 0}\int_0^N\int_{\mathbb{R}^d}(hu)^{2-\gamma}\frac{d}{dt}h
    +\left(\frac{d}{dt}h\right)u^{2-\gamma}
    -\left(\frac{d}{dt}h\right)\frac{1-\gamma}{2-\gamma}u^{2-\gamma}\;dxdt.
\end{align*}

Proceeding to the second term in \eqref{eq:lim_inf_c 3 gamma}, we have
\begin{align*} 
\lim_{\alpha\to\infty}
\int_0^N\int_{\mathbb{R}^d}h\Gamma_\alpha (hv_{\alpha,c})\langle b-\tilde{b},\nabla u\rangle\;dxdt
&=\lim_{\alpha\to\infty}
\int_0^N\int_{\mathbb{R}^d} hv_{\alpha,c}\Gamma_\alpha(h\langle b-\tilde{b},\nabla u)\rangle\;dxdt\\
&=
\int_0^N\int_{\mathbb{R}^d}  h^2(hu+c)^{1-\gamma}\langle b-\tilde{b},\nabla u\rangle\;dxdt\\
&\mathop{\longrightarrow}_{c\to 0}\int_0^N\int_{\mathbb{R}^d}  h^2(hu)^{1-\gamma}\langle b-\tilde{b},\nabla u\rangle\;dxdt
<\infty \mbox{ by } \eqref{eq:b da in L2 loc duplicate},\eqref{eq:hu in H-1}. 
\end{align*}

For the third term in \eqref{eq:lim_inf_c 3 gamma}, we have
\begin{align*}
\int_0^N\int_{\mathbb{R}^d}\langle \Gamma_\alpha (hv_{\alpha,c})\nabla h,A\nabla u\rangle\;dxdt
&=\int_0^N\int_{\mathbb{R}^d}\Gamma_\alpha (hv_{\alpha,c})\langle \nabla h,A\nabla u\rangle\;dxdt\\
&=\int_0^N\int_{\mathbb{R}^d}hv_{\alpha,c}\Gamma_\alpha \left(\langle \nabla h,A\nabla u\rangle\right)\;dxdt\\
&\mathop{\longrightarrow}_{\alpha\to \infty}\int_0^N\int_{\mathbb{R}^d}h(hu+c)^{1-\gamma}\langle \nabla h,A\nabla u\rangle\;dxdt\\
&\mathop{\longrightarrow}_{c\to 0}\int_0^N\int_{\mathbb{R}^d}h(hu)^{1-\gamma}\langle \nabla h,A\nabla u\rangle\;dxdt.
\end{align*}

For the last term in \eqref{eq:lim_inf_c 3 gamma}, we have
\begin{align*}
    \int_0^N\int_{\mathbb{R}^d}
    \langle \nabla \Gamma_\alpha (hv_{\alpha,c}),uA\nabla h\rangle\;dxdt
    &=\int_0^N\int_{\mathbb{R}^d}
    \langle \nabla  (hv_{\alpha,c}),\Gamma_\alpha(uA\nabla h)\rangle\;dxdt\\
    &=-\int_0^N\int_{\mathbb{R}^d}
    \langle  hv_{\alpha,c},\nabla\Gamma_\alpha(uA\nabla h)\rangle\;dxdt\\
    &\mathop{\longrightarrow}_{\alpha\to \infty}-\int_0^N\int_{\mathbb{R}^d}
    \langle  h(hu+c)^{1-\gamma},\nabla(uA\nabla h)\rangle\;dxdt\\
    &\mathop{\longrightarrow}_{c\to 0}-\int_0^N\int_{\mathbb{R}^d}
    \langle  h(hu)^{1-\gamma},\nabla(uA\nabla h)\rangle\;dxdt\\
    &<\infty \mbox{ by } \eqref{eq:b da in L2 loc duplicate}, \eqref{eq:ellipticity_duplicate}, \eqref{eq:hu in H-1}.
\end{align*}
Thus, the proof is finished.
\end{proof}

We are now able to improve \Cref{prop:u gamma} by allowing $\gamma=1$:
\begin{prop}\label{coro: sqrt}
If \eqref{eq:b da in L2 loc duplicate}, \eqref{eq:ellipticity_duplicate}, and \eqref{eq:hu in H-1} hold, then condition $\mathbf{H^{\sqrt{u}}}$ is fulfilled.
\end{prop}
\begin{proof}
Let $0\leq h\in C_c^\infty((0,N)\times \mathbb{R}^d)$ and set
\begin{equation*}
    u_{c}:= hu+c,\quad v_c:=\log(hu+c),\quad c>0.
\end{equation*}
By Fatou's lemma we have
\begin{equation*}
    \int_0^N\int_{\mathbb{R}^d}h^2\frac{\|\nabla (hu)\|^2}{hu}\;dxdt
    \leq \liminf_{c\to 0}\int_0^N\int_{\mathbb{R}^d}h^2\frac{\|\nabla u_c\|^2}{u_c}\;dxdt.
\end{equation*}
We proceed with the key fact that
\begin{align*}
    \int_0^N\int_{\mathbb{R}^d}h^2\frac{\|\nabla u_c\|^2}{u_c}\;dxdt
    &\leq C\int_0^N\int_{\mathbb{R}^d}\frac{h^2}{u_{c}}\langle A\nabla u_{c},\nabla u_{c}\rangle\;dxdt
    =\frac{C}{1-\gamma}\int_0^N\int_{\mathbb{R}^d}h^2\langle \nabla v_{c},A\nabla u_{c}\rangle\;dxdt\\
    &=\frac{C}{2(1-\gamma)}\int_0^N\int_{\mathbb{R}^d}\langle \nabla (h^2v_{c}),A\nabla u_{c}\rangle\;dxdt
    -\frac{C}{2(1-\gamma)}\int_0^N\int_{\mathbb{R}^d}\langle hv_{c}\nabla h,A\nabla u_{c}\rangle\;dxdt.
\end{align*}
Further, note that for $0<c\leq 1$
Note that
\begin{align*}
    &\left|\int_0^N\int_{\mathbb{R}^d}\langle hv_{c}\nabla h,A\nabla u_{c}\rangle\;dxdt\right|\\
    &= \left|\int_0^N\int_{\mathbb{R}^d}\langle hv_{c}A\nabla h,\nabla (hu)\rangle\;dxdt\right|
    =\left|\int_0^N\int_{\mathbb{R}^d} {\sf div} \left(hv_{c}A\nabla h\right) hu\;dxdt\right|\\
    &=\left|\int_0^N\int_{\mathbb{R}^d} \langle \nabla (hu),hA\nabla h\rangle \frac{hu}{hu+c}\;dxdt
    +\int_0^N\int_{\mathbb{R}^d} {\sf div} \left(hA\nabla h\right) hu\log(hu+c)\;dxdt\right|\\
    &\leq \int_0^N\int_{\mathbb{R}^d} \left|\langle h\nabla u,hA\nabla h\rangle\right| \;dxdt
    +\int_0^N\int_{\mathbb{R}^d} \left|{\sf div} \left(hA\nabla h\right)\right| |hu\log(hu+c)|\;dxdt\\
    &\leq \int_0^N\int_{\mathbb{R}^d} \left|\langle \nabla (hu),hA\nabla h\rangle\right| \;dxdt
    +\int_0^N\int_{\mathbb{R}^d} \left|{\sf div} \left(hA\nabla h\right)\right| \left[hu\log(hu+1)+e^{-1}+1\right]\;dxdt\\
    &<\infty \mbox{ by } \eqref{eq:b da in L2 loc duplicate}, \eqref{eq:ellipticity_duplicate}, \mbox{ and } \eqref{eq:hu in H-1},
\end{align*}
where we have also used
\begin{equation}\label{eq:xlogx}
    |x\log(x+c)|\leq |x\log(x)|+c, \quad |x\log(x)|\leq e^{-1}+x\log(x+1), \quad x\geq 0,c>0.
\end{equation}

It remains to show that
\begin{equation}\label{eq:lim_inf_c is finte log}
    \liminf_{c\to 0}\int_0^N\int_{\mathbb{R}^d}\langle \nabla (h^2v_{c}),A\nabla u_{c}\rangle\;dxdt<\infty.
\end{equation}

Let us define
\begin{equation*}
u_{\alpha,c}:=\Gamma_\alpha u_c, \;v_{\alpha,c}:= \log(u_{\alpha,c}),\quad \alpha,c>0,
\end{equation*}
so that, by \Cref{lem:hu in H-1} and \eqref{eq:hu in H-1},
\begin{equation*}
\Gamma_\alpha (hv_{\alpha,c})\in W_0^{1,2}((0,N); L^2(\mathbb{R}^d))\cap L^2((0,N);H^1(\mathbb{R}^d))\cap L^\infty((0,N)\times \mathbb{R}^d), \quad \alpha, c>0.
\end{equation*}
Plugging $\Gamma_\alpha (hv_{\alpha,c})$ instead of $f$ in \eqref{eq:LFP-H^-1 gamma} we obtain
\begin{align*}
&\int_0^N\int_{\mathbb{R}^d}
    \langle \nabla \Gamma_\alpha (h^2v_{\alpha,c}), A\nabla (hu)\rangle\;dxdt\\
    &=\int_0^N\int_{\mathbb{R}^d}u\frac{d}{dt}(h\Gamma_\alpha (h^2v_{\alpha,c}))
    +h\Gamma_\alpha (h^2v_{\alpha,c})\langle b-\tilde{b},\nabla u\rangle\\
    &\phantom{=\int_0^N\int_{\mathbb{R}^d}u\frac{d}{dt}(h\Gamma_\alpha (h^2v_{\alpha,c}))} 
    -\langle \Gamma_\alpha (h^2v_{\alpha,c})\nabla h,A\nabla u\rangle
    +\langle \nabla \Gamma_\alpha (h^2v_{\alpha,c}),uA\nabla h\rangle\;dxdt.
\end{align*}
Note that by \eqref{eq:Gamma-convergence}, \eqref{eq:hu in H-1}, and the fact that $\nabla$ and $\Gamma_\alpha$ commute, and $\Gamma_\alpha$ is $L^2(\mathbb{R}^d)$-symmetric, we have
\begin{align*}
\lim_{\alpha\to\infty}\int_0^N\int_{\mathbb{R}^d}
    \langle \nabla (\Gamma_\alpha (h^2v_{\alpha,c})), A\nabla (hu)\rangle\;dxdt
    &=\lim_{\alpha\to\infty}\int_0^N\int_{\mathbb{R}^d}
    \langle \nabla ( hv_{\alpha,c}), \Gamma_\alpha(A\nabla (hu)\rangle)\;dxdt\\
    &=\int_0^N\int_{\mathbb{R}^d}
    \langle \nabla (h^2v_{c}), A\nabla (hu)\rangle\;dxdt,
\end{align*}
hence showing \eqref{eq:lim_inf_c is finte log} boils down to showing
\begin{align}\label{eq:lim_inf_c 3 log}
    \liminf_{c\to 0} \lim_{\alpha\to\infty}
    &\left|\int_0^N\int_{\mathbb{R}^d}u\frac{d}{dt}(h\Gamma_\alpha (h^2v_{\alpha,c}))
    +h(\Gamma_\alpha (h^2v_{\alpha,c}))\langle b-\tilde{b},\nabla u\rangle\right.\\
    &\left.\phantom{\left(\frac{d}{dt}(h\Gamma_\alpha (h^2v_{\alpha,c}))\right)u} 
    -\langle (\Gamma_\alpha (h^2v_{\alpha,c}))\nabla h,A\nabla u\rangle
    +\langle \nabla (\Gamma_\alpha (h^2v_{\alpha,c})),uA\nabla h\rangle\;dxdt\right|<\infty.
\end{align}
We treat all the four terms in the last integral, separately:
For the first term, since $\frac{d}{dt}$ and $\Gamma_\alpha$ commute, and $\Gamma_\alpha$ is $L^2(\mathbb{R}^d)$-symmetric, we have
\begin{align*}
    \int_0^N\int_{\mathbb{R}^d}
    &u\frac{d}{dt}\left(h\Gamma_\alpha (h^2v_{\alpha,c}) \right)\;dxdt\\
    &=\int_0^N\int_{\mathbb{R}^d}u\Gamma_\alpha (h^2v_{\alpha,c})\frac{d}{dt}h
    +u_{\alpha,c}\frac{d}{dt}\left( h^2v_{\alpha,c} \right)\;dxdt\\
    &=\int_0^N\int_{\mathbb{R}^d}h^2v_{\alpha,c}\Gamma_\alpha\left(u\frac{d}{dt}h\right) 
    +\left(\frac{d}{dt}h^2\right)u_{\alpha,c}\log(u_{\alpha,c})
    +h^2\frac{d}{dt}u_{\alpha,c}\;dxdt\\
    &=\int_0^N\int_{\mathbb{R}^d}h^2v_{\alpha,c}\Gamma_\alpha\left(u\frac{d}{dt}h\right)
    +\left(\frac{d}{dt}h^2\right)u_{\alpha,c}\log(u_{\alpha,c})
    -\left(\frac{d}{dt}h^2\right)u_{\alpha,c}\;dxdt\\
    &\mathop{\longrightarrow}_{\alpha\to\infty}\int_0^N\int_{\mathbb{R}^d}h^2v_{c}u\frac{d}{dt}h
    +\left(\frac{d}{dt}h^2\right)u_{c}\log(u_{c})
    -\left(\frac{d}{dt}h^2\right)u_{c}\;dxdt\\
    &=\int_0^N\int_{\mathbb{R}^d}h^2u\log(hu+c)\frac{d}{dt}h
    +\left(\frac{d}{dt}h^2\right)(hu+c)\log(hu+c)
    -\left(\frac{d}{dt}h^2\right)(hu+c)\;dxdt\\
    &\mathop{\longrightarrow}_{c\to 0}\int_0^N\int_{\mathbb{R}^d}h^2u\log(hu)\frac{d}{dt}h
    +\left(\frac{d}{dt}h^2\right)hu\log(hu)
    -\left(\frac{d}{dt}h^2\right)hu\;dxdt.
\end{align*}

Proceeding to the second term in \eqref{eq:lim_inf_c 3 log}, we have
\begin{align} 
&\lim_{\alpha\to\infty}
\left|\int_0^N\int_{\mathbb{R}^d}h\Gamma_\alpha (h^2v_{\alpha,c})\langle b-\tilde{b},\nabla u\rangle\;dxdt\right|\nonumber\\
&=\left|\lim_{\alpha\to\infty}
\int_0^N\int_{\mathbb{R}^d} h^2v_{\alpha,c}\Gamma_\alpha(h\langle b-\tilde{b},\nabla u)\rangle\;dxdt\right|
=
\left|\int_0^N\int_{\mathbb{R}^d}  h^3\log(hu+c)\langle b-\tilde{b},\nabla u\rangle\;dxdt\right|\label{eq:treat b-btilde}\\
&\leq\int_0^N\int_{\mathbb{R}^d}  |h^2\log(hu+c)\langle b-\tilde{b},\nabla (hu)\rangle|
+|h^2u\log(hu+c)\langle b-\tilde{b},\nabla h\rangle|\;dxdt\nonumber\\
&\leq \int_0^N\int_{\mathbb{R}^d}  \|h\log(hu+c)\nabla (hu)\|\|h(b-\tilde{b})\|
+|h^2u\log(hu+c)|\| b-\tilde{b}\|\|\nabla h\|\;dxdt\nonumber\\
&\leq\left(\int_0^N\int_{\mathbb{R}^d}  h^2\left(\|\log(hu+c)\nabla (hu)\|\right)^2 \;dxdt\right)^{1/2}\left(\int_0^N\int_{\mathbb{R}^d}  h^2\|(b-\tilde{b})\|^2 \;dxdt\right)^{1/2}\nonumber\\
&\quad+\int_0^N\int_{\mathbb{R}^d}  |h^2u\log(hu+c)\langle b-\tilde{b},\nabla h\rangle|\;dxdt.\nonumber
\end{align}
Now, let us note that
\begin{align*}
    &\int_0^N\int_{\mathbb{R}^d} h^2\|b-\tilde{b}\|^2\;dxdt<\infty, \quad \mbox{ by } \eqref{eq:b da in L2 loc duplicate}\\
    &\int_0^N\int_{\mathbb{R}^d}  |h^2u\log(hu+c)\langle b-\tilde{b},\nabla h\rangle|\;dxdt\mathop{\rightarrow}\limits_{c\to 0}\int_0^N\int_{\mathbb{R}^d}  |h^2u\log(hu)\langle b-\tilde{b},\nabla h\rangle|\;dxdt<\infty\quad \mbox{ by } \eqref{eq:b da in L2 loc duplicate}, \eqref{eq:hu in H-1}.
\end{align*}
Furthermore, by fixing $\gamma\in(0,1)$,
\begin{align*}
\int_0^N\int_{\mathbb{R}^d} h^2|\log(hu+c)|^2\|\nabla (hu)\|^2\;dxdt
&=\int_0^N\int_{\mathbb{R}^d} h^2(hu)^\gamma|\log(hu+c)|^2\frac{\|\nabla (hu)\|^2}{(hu)^\gamma}\;dxdt\\
&\leq \|h(hu)^\gamma|\log(hu+c)|^2\|_{L^\infty((0,N)\times \mathbb{R}^d)}\int_0^N\int_{\mathbb{R}^d} h\frac{\|\nabla (uh)\|^2}{(hu)^\gamma}\;dxdt.
\end{align*}
Finally, 
\begin{equation*}
    \sup_{c\in (0,1)}\|h(hu)^\gamma|\log(u+c)|^2\|_{L^\infty((0,N)\times \mathbb{R}^d)}<\infty \quad \mbox{by using } \eqref{eq:hu in H-1},
\end{equation*}
whilst
\begin{equation*}
    \int_0^N\int_{\mathbb{R}^d} h\frac{\|\nabla (hu)\|^2}{(hu)^\gamma}\;dxdt<\infty \quad \mbox{ by } \Cref{prop:u gamma}.
\end{equation*}

For the third term in \eqref{eq:lim_inf_c 3 log}, we have
\begin{align*}
&\int_0^N\int_{\mathbb{R}^d}\langle \Gamma_\alpha (h^2v_{\alpha,c})\nabla h,A\nabla u\rangle\;dxdt\\
&=\int_0^N\int_{\mathbb{R}^d}\Gamma_\alpha (h^2v_{\alpha,c})\langle \nabla h,A\nabla u\rangle\;dxdt
=\int_0^N\int_{\mathbb{R}^d}h^2v_{\alpha,c}\Gamma_\alpha \left(\langle \nabla h,A\nabla u\rangle\right)\;dxdt\\
&\mathop{\longrightarrow}_{\alpha\to \infty}\int_0^N\int_{\mathbb{R}^d}h^2\log(hu+c)\langle A\nabla h,\nabla u\rangle\;dxdt,
\end{align*}
and this last term can be further treated precisely like $\int_0^N\int_{\mathbb{R}^d}  h^3\log(hu+c)\langle b-\tilde{b},\nabla u\rangle\;dxdt$ in \eqref{eq:treat b-btilde}.

For the last term in \eqref{eq:lim_inf_c 3 log}, we have
\begin{align*}
    &\int_0^N\int_{\mathbb{R}^d}
    \langle \nabla \Gamma_\alpha (h^2v_{\alpha,c}),uA\nabla h\rangle\;dxdt\\
    &=\int_0^N\int_{\mathbb{R}^d}
    \langle \nabla  (h^2v_{\alpha,c}),\Gamma_\alpha(uA\nabla h)\rangle\;dxdt
    =\mathop{\longrightarrow}_{\alpha\to \infty}\int_0^N\int_{\mathbb{R}^d}
    \langle  \nabla(h^2\log(hu+c)),uA\nabla h\rangle\;dxdt
    \\
    &=-\int_0^N\int_{\mathbb{R}^d}
    h^2\log(hu+c){\sf div}(uA\nabla h)\;dxdt
    =-\int_0^N\int_{\mathbb{R}^d}
    h^2\log(hu+c)[\langle\nabla u,A\nabla h\rangle+u{\sf div}(A\nabla h)]\;dxdt.
\end{align*}
Now, the integral
\begin{equation*}
    \int_0^N\int_{\mathbb{R}^d}
    h^2\log(hu+c)\langle\nabla u,A\nabla h\rangle \; dx dt
\end{equation*}
can be further treated precisely like $\int_0^N\int_{\mathbb{R}^d}  h^3\log(hu+c)\langle b-\tilde{b},\nabla u\rangle\;dxdt$ in \eqref{eq:treat b-btilde},
    whilst
    \begin{align*}
    &\int_0^N\int_{\mathbb{R}^d}
    h^2\log(hu+c)u{\sf div}(A\nabla h)\;dxdt\\
    &\leq \sup_{c\in (0,1)}\|h^2u\log(hu+c)\|_{L^\infty((0,N)\times \mathbb{R}^d)}\int_0^N\int_{\mathbb{R}^d}
    |{\sf div}(A\nabla h)|\;dxdt<\infty \mbox{ by } \eqref{eq:b da in L2 loc duplicate}, \eqref{eq:hu in H-1}.
    \end{align*}
Thus, the proof is finished.
\end{proof}

If conditions \eqref{eq:b da in L2 loc duplicate}, \eqref{eq:ellipticity_duplicate}, and \eqref{eq:hu in H-1} are satisfied globally in space, then, by following essentially the same steps as above, the results from \Cref{lem:hu in H-1} and \Cref{coro: sqrt} are satisfied globally in space as well.
More precisely, we have the following.
\begin{prop}\label{prop:global_root_reg}
Assume the following global in space versions of \eqref{eq:b da in L2 loc duplicate}, \eqref{eq:ellipticity_duplicate}, and \eqref{eq:hu in H-1}:
\begin{align*}
{\scalebox{0.7}{$\bullet$}\quad}
&b_i \mbox{ and } \partial_{x_j} a_{i,j} \mbox{ belong to }L^2\left([0,N)\times \mathbb{R}^d\right) , \quad 1\leq i,j\leq d,\\
{\scalebox{0.7}{$\bullet$}\quad}
&C^{-1}\|\xi\|^2\leq \langle A(t,x) \xi, \xi\rangle \leq C \|\xi\|^2, \quad \xi\in \mathbb{R}^d, (t,x)\in [0,N)\times \mathbb{R}^d\\
{\scalebox{0.7}{$\bullet$}\quad}
& u\in L^\infty([0,N]\times \mathbb{R}^d; \mathbb{R}) \quad \mbox{and} \quad \nabla u\in L^2([0,N)\times \mathbb{R}^d;\mathbb{R}^d). 
\end{align*}      
Then
\begin{equation*}
    u\in W^{1,2}([0,N];H^{-1}(\mathbb{R}^d)), \mbox{ hence } u\in C([0,T];L^2(\mathbb{R}^d)),\quad\mbox{and}\quad
    \sqrt{u}
    \in L^2([0,T];H^1(\mathbb{R}^d)).
\end{equation*}
\end{prop}

\subsection{Trevisan's conditions}\label{ss:trevisan}
Let  us consider the diffusion operator $\left({\sf L}_t\right)_{t\in [0,\infty)}$ with coefficients $a,b$ given by \eqref{eq:L_t}-\eqref{eq:a,b}, such that the well-posedeness conditions in \cite{Tr16} for the elliptic case are satisfied for every $N>0$, namely:
\begin{align}
    &(T1)\quad\exists C=C(N)>0 : \langle A(t,x)\xi,\xi\rangle\geq C \|\xi\|^2, \quad t\in[0,N],x,\xi\in \mathbb{R}^d,\label{eq:T1}\\
    &(T2)\quad a\in L^\infty([0,N]\times \mathbb{R}^d),\quad \frac{d}{dt}a\in L^{\infty}([0,N)\times\mathbb{R}^d),\label{eq:T2}\\
    &(T3)\quad a\in L^1([0,N);W^{1,p}(\mathbb{R}^d)) \mbox{ for some } p\in[2,\infty],\quad \left(\sum_{i,j} \frac{\partial^2}{\partial x_i\partial x_j}a_{i,j}\right)^+\in L^1([0,N];L^\infty(\mathbb{R}^d))\label{eq:T3}\\
    &(T4)\quad b\in L^1([0,N];L^\infty(\mathbb{R}^d)).\label{eq:T4}
\end{align}
Then, by \cite[Theorem 3.3]{Tr16} and its proof (see also the comm):

{\it For every $\mu_0=u_0\;dx\in \mathcal{P}(\mathbb{R}^d)$ with $u_0\in L^r(\mathbb{R}^d)$, $r\geq 2p/(p-2)$, and for every $N>0$, there exists a unique weakly continuous solution $\mu:=\left(\mu(t)=u(t,x) \;dx\right)_{t\in [0,N]}\subset \mathcal{P}(\mathbb{R}^d)$ of the linear Fokker-Planck equation \eqref{eq:LFPshort}, with initial condition $\mu_0$ and such that $u\in L^\infty([0,N];L^r(\mathbb{R}^d))$.
Moreover, it satisfies
\begin{equation}\label{eq:extra_regularity}
    u\in L^2([0,N];W^{1,2}(\mathbb{R}^d)), \quad \mbox{for every } N>0.
\end{equation}
}
\begin{rem}\label{rem:well-posedeness T}
    \begin{enumerate}
        \item[(i)] Note that in \cite[Theorem 3.3]{Tr16}, the uniqueness part is apparently claimed in the class of weakly continuous solutions which belong to $L^\infty([0,N];L^r(\mathbb{R}^d))$ and are necessarily probability solutions.
        However, by inspecting the proof of \cite[Theorem 3.3]{Tr16}, one can see that uniqueness holds among all weakly continuous solutions to \eqref{eq:LFPshort} whose densities belongs to $L^\infty([0,N];L^r(\mathbb{R}^d))$. 
        \item[(ii)] We emphasize that the uniqueness is guaranteed in the space $u\in L^\infty([0,N];L^r(\mathbb{R}^d))$, whilst \eqref{eq:extra_regularity} is a byproduct of the construction of the solution; see \cite[Subsection 3.3]{Tr16} and the discussion on page 20.
    \end{enumerate}
\end{rem}

We also need the following slight modifications of (T3) and (T4):

\begin{enumerate}
    \item[$(T3')$] $ \partial_{x_j} a_{i,j}\in L^2_{\sf loc}([0,\infty)\times \mathbb{R}^d),\quad 1\leq i,j\leq d$
    \item[$(T4')$] $b_i\in L^2_{\sf loc}([0,\infty)\times \mathbb{R}^d),\quad 1\leq i\leq d$
\end{enumerate}

\begin{coro}
    Assume that conditions (T1)-(T4), and (T3')-(T4') are satisfied.
    Further, let $u_0(x)dx\in \mathcal{P}(\mathbb{R}^d)$ be such that $u_0\in L^\infty(\mathbb{R}^d)$, and consider $(\mu(t)=u(t,x)dx)_{t\in [0,T]}\subset \mathcal{P}(\mathbb{R}^d)$ the unique weakly continuous solution to the linear Fokker-Planck equation \eqref{eq:LFPshort}, with initial condition $u_0(x)dx$, and such that $u\in L^\infty([0,N]\times\mathbb{R}^d)$.
    Then the conditions $\mathbf{H_{a,b}^{\sqrt{u}}}$, $\mathbf{H_{\sf \lesssim u}}$, and $\mathbf{H^{\sf TV}_0}$ are satisfied, hence the conclusion of \Cref{thm:main0} \& \Cref{thm:main1} hold.
\end{coro}
\begin{proof}
First, by (T1),(T2), (T3'), and the fact that $u$ is bounded,  we get that $\mathbf{H_a}$ is fulfilled. 

Further, $\mathbf{H_b}$ follows from (T4') and the fact that $u$ is bounded.

Condition $\mathbf{H^{\sqrt{u}}}$ is ensured by \Cref{coro: sqrt} since $u$ is bounded and \eqref{eq:extra_regularity} holds.

Further, $\mathbf{H_{\sf \lesssim u}}$ follows from \cite[Theorem 3.3]{Tr16} presented above, and \Cref{rem:well-posedeness T}, (i).

Finally, $\mathbf{H^{\sf TV}_0}$ follows by \Cref{rem: implies TV} and \Cref{eq:uniqueness eta}.
\end{proof}

\section[Regularization of the superposition principle for nonlinear FPE]{Regularization of the superposition principle for nonlinear time-dependent Fokker-Planck equations}\label{s:nonlinear_superposition}

In this part we extend to nonlinear Fokker-Planck equations the results presented above for the linear case. 
We construct Choquet capacities for such nonlinear equations by probabilistic means and we investigate explicit sufficient conditions on the coefficients for a class of generalized porous media equations to which our theory applies.

\subsection{The strong Markov property and a Choquet capacity for the solution of the corresponding McKean-Vlasov equation}
In this section we start from a general nonlinear Fokker-Planck equation, namely
we consider the $(t, \mu)$-dependent diffusion operator
\begin{equation}\label{eq:L_t,nu}
    {\sf L}_{t,\nu}f(t,x):=b(t,x,\nu)\nabla_x f(t,x) + \frac{1}{2}\sum\limits_{i,j=1}^d a_{i,j}(t,x,\nu)\frac{\partial^2}{\partial x_i \partial x_j} f(t,x), \quad (t,x,\nu) \in (0,\infty)\times \mathbb{R}^d\times \mathcal{P}(\mathbb{R}^d),  
\end{equation}
that acts on test functions $f\in C_c^\infty((0,\infty)\times \mathbb{R}^d)$ which are twice continuously differentiable in the second argument, whilst $b$ and $a$ are functions between the following spaces
\begin{equation*}
    b:[0,\infty)\times \mathbb{R}^d\times \mathcal{P}(\mathbb{R}^d) \rightarrow \mathbb{R}^d \quad \mbox{ and } \quad a:[0,\infty)\times\mathbb{R}^d\times \mathcal{P}(\mathbb{R}^d)\rightarrow \mathbb{R}^{d\times d}_{{\sf sym}+}.
\end{equation*}
\begin{rem}
Note that above we did not assume any kind of regularity/measurability for the coefficients $b$ and $a$, the reason being that our aim is to cover irregular coefficients, like those from below which depend on point-wise evaluations of the density of $\nu$ assuming it exists, also called Nemytskii-type operators.
This is in fact the setting adopted in \cite{BaRo20}, to which we refer for more details and examples. 
\end{rem}

Let us now extend the notion of solution to the Fokker-Planck equation from the linear case (see \Cref{def:solution}) to the nonlinear one:
\begin{defi} \label{def:Nsolution}
Let $0\leq s< T<\infty$. 
\begin{enumerate}
    \item[(i)] A family $\left(\mu_t\right)_{t\in (s,T)}\subset \mathcal{M}(\mathbb{R}^d)$ is called a (weak, or distributional) solution on $(s,T)$ to the nonlinear Fokker-Planck equation  \begin{equation}\label{eq:NLFPshort}
        \frac{d}{dt} \mu_t = {\sf L}^\ast_{t,\mu_t} \mu_t, \quad t\in (s,T),
    \end{equation}
    if $\left(\mu_t\right)_{ t\in (s,T)}$ is a Borel curve in $\mathcal{M}(\mathbb{R}^d)$,
    \begin{align}
        \scalebox{0.7}{$\bullet\quad$}
        &[0,\infty)\times\mathbb{R}^d\ni(t,x)\mapsto b(t,x,\mu_t)\in \mathbb{R}^d \mbox{ and } (t,x)\mapsto a(t,x,\mu_t)\in \mathbb{R}^{d\times d} \mbox{ are Borel measurable},\nonumber\\
        \scalebox{0.7}{$\bullet\quad$}
        &\int_{(s,T)\times B(0,R)} \|b(t,x,\mu_t)\|+\|A(t,x,\mu_t)\| \; \mu_t(dx) dt <\infty, \quad R>0 \label{eq:Na,b}\\
        \scalebox{0.7}{$\bullet\quad$}
        &\int_{(s,T)\times \mathbb{R}^d} \left[\frac{d}{dt}f(t,x)+ {\sf L}_{t,\mu_t} f(t,x)  \right] \;\mu_t(dx) dt = 0, \quad f\in C_c^\infty \left((s,T)\times \mathbb{R}^d\right).\label{eq:NLFPmu}
    \end{align}
\item[(ii)] We say that $\left(\mu_t\right)_{t\in [s,T)}\subset \mathcal{M}(\mathbb{R}^d)$ is a solution to the nonlinear Fokker-Planck equation \eqref{eq:NLFPshort} with initial condition $\mu_s$ if $\left(\mu_t\right)_{t\in (s,T)}$ is a solution on $(s,T)$ as in (i), and $\lim\limits_{t\searrow s} \mu_t=\mu_s$ weakly.
\item[(iii)] We say that $\left(\mu_t\right)_{t\in [s,T)}\subset \mathcal{M}(\mathbb{R}^d)$ is a weakly continuous solution to the nonlinear Fokker-Planck equation \eqref{eq:NLFPshort} if it is a solution on $(s,T)$ in the sense of (i), and the curve of measures $[s,T)\ni t\mapsto \mu_t\in \mathcal{M}(\mathbb{R}^d)$ is weakly continuous; this definition is obviously extended when $(\mu_t)_{t\in [s,T)]}$ is merely weakly right-continuous. 
\end{enumerate}
\end{defi}

Further, if $\mu=(\mu_t=u(t,x)dx)_{t\geq 0}\subset \mathcal{P}(\mathbb{R}^d)$ is a Borel curve, we consider the following "linearized" objects:
\begin{align*}
    b^u(t,x)
    &:=b(t,x,\mu_t),\quad a^u_{ij}(t,x):=a_{ij}(t,x,\mu_t),\quad (t,x)\in [0,\infty)\times\mathbb{R}^d\\
    {\sf L}^uf(t,x)
    &:=b^u(t,x)\nabla_x f(t,x) + \frac{1}{2}\sum\limits_{i,j=1}^d a^\mu_{ij}(t,x)\frac{\partial^2}{\partial x_i \partial x_j} f(t,x)\\
    &\phantom{:}={\sf L}_{t,\mu_t}f(t,x),\quad f\in C_c^\infty((0,\infty)\times\mathbb{R}^d),\quad (t,x)\in [0,\infty)\times\mathbb{R}^d.
\end{align*}

Let us now recall the counterpart of the nonlinear Fokker-Planck equations considered above to the theory of McKean-Vlasov SDEs, also called distribution dependent SDEs. 
First, assume that 
\begin{equation*}
    \sigma:[0,\infty)\times\mathbb{R}^d\times \mathcal{P}(\mathbb{R}^d)\rightarrow \mathbb{R}^{d\times d} \quad\mbox{is such that}\quad a=\sigma\sigma^\ast.
\end{equation*}
\begin{defi}\label{defi:sol MV}
    Let $\left(\Omega,\mathcal{F},\left(\mathcal{F}(t)\right)_{t\geq 0},\mathbb{P}\right)$ be a filtered probability space. 
    An $\mathcal{F}(t)$-adapted and $\mathbb{P}$-a.s. path-continuous stochastic process $X=(X(t))_{t\geq 0}$ with values in $\mathbb{R}^d$ is called a {\rm weak solution} to the McKean-Vlasov SDE
\begin{equation}\label{eq:MV SDE}
    dX(t)=b\left(t,x,\mathcal{L}_{X(t)}\right)dt+\sigma\left(t,x,\mathcal{L}_{X(t)}\right)dW(t), \quad X(0)\sim \mu_0\in\mathcal{P}(\mathbb{R}^d)
\end{equation}
where $W$ is a $d$-dimensional standard Brownian motion defined on $\left(\Omega,\mathcal{F},\left(\mathcal{F}(t)\right)_{t\geq 0},\mathbb{P}\right)$, whilst $\mathcal{L}_{X(t)}:=\mathbb{P}\circ \left(X(t)\right)^{-1}, \;t\geq 0$,
if
\begin{equation*}
    \int_0^N\int_{B(0,R)} \left(\|b(t,x,\mathcal{L}_{X(s)})\|+\|a(t,x,\mathcal{L}_{X(s)})\|\right) \;\mathcal{L}_{X(s)}(dx) dt <\infty, \quad N>0,R>0,
\end{equation*}
and
\begin{equation*}
    M(t):=f(t,X(t))-f(0,X(0))-\int_0^t\left(\frac{d}{ds}f+{\sf L}_{s,\mathcal{L}_{X(s)}}\right)(s,X(s))\;ds,\quad t\geq 0
\end{equation*}
is an $\mathcal{F}(t)$-martingale for every $f\in C_c^\infty([0,\infty)\times\mathbb{R}^d)$.
\end{defi}

In what follows, conditions $\mathbf{H_{\sf \lesssim u}}, \quad \mathbf{H^{\sf TV}_0}$ introduced in \Cref{ss:main results} and related to a general (time-dependent) linear Kolmogorov operator ${\sf L}$, are going to be employed for the linearized operator ${\sf L}^u$ obtained from the nonlinear one ${\sf L}_{t,\nu}$ from \eqref{eq:L_t,nu}, by freezing the particular solution $u$ in the coefficients; to stress this, we shall write $\mathbf{H_{\sf \lesssim u}}({\sf L}^u)$ and $\mathbf{H^{\sf TV}_0}({\sf L}^u)$, instead of $\mathbf{H_{\sf \lesssim u}}$ and  $ \quad \mathbf{H^{\sf TV}_0}$, respectively.

\begin{thm}\label{thm:MVE}
Let $(\mu_t)_{t\geq 0}\subset \mathcal{P}(\mathbb{R}^d)$ be a weakly continuous solution to the nonlinear Fokker-Planck equation \eqref{eq:NLFPshort} on $(0,\infty)$, such that $\mu_t:= u(t,x) dx,\; t\geq 0$ for a jointly measurable density
$u:[0,\infty)\times \mathbb{R}^d\rightarrow \mathbb{R}_+$.  
Furthermore, assume that the "linearized" conditions
\begin{equation*}
\mathbf{H_{a^u,b^u}^{\sqrt{u}}},\quad \mathbf{H_{\sf \lesssim u}}({\sf L}^u), \quad \mathbf{H^{\sf TV}_0}({\sf L}^u) \quad \mbox{ are fulfilled}.
\end{equation*}
Then there exists a set $E\in \mathcal{B}([0,\infty)\times \mathbb{R}^d)$ and a conservative right (hence strong Markov) process on $E$
\begin{equation*}
X=\left(\Omega,\mathcal{F},\mathcal{F}(t),X(t),\mathbb{P}_{s,x}, (s,x)\in E, t\geq 0\right), \quad \mbox{with transition function } (P_t)_{t\geq 0},
\end{equation*} 
such that the following assertions hold:
\begin{enumerate}
            \item[(i)] $\overline{\mu}([0,\infty)\times \mathbb{R}^d\setminus E)=0$ and $\mu_s\left(\left\{x\in \mathbb{R}^d, (s,x)\in E\right\}\right)=1$, $s\geq 0$.
            \item[(ii)] For $(s,x)\in E$ we have
            \begin{equation*}
                \mathbb{P}_{s,x}\left( [0,\infty)\ni t\mapsto X(t) \in E\subset [0,\infty)\times \mathbb{R}^d \mbox{ is continuous }  \right)=1.
            \end{equation*}
            \item[(iii)] If we set $X(t)=(X_1(t),X_2(t)), t\geq 0$, then for $(s,x)\in E$ we have
            \begin{equation*}
            \mathbb{P}_{s,x}\left( X_1(t)= t+s; 0\leq t<\infty  \right)=1.
            \end{equation*}
            \item[(iv)] The process $(X_2(t))_{t\geq 0}$ has the following properties:
            \begin{enumerate}
                \item[(iv.1)] For every $s,t\geq 0$
                \begin{equation*}
                \mathbb{P}_{s,\mu_s}\circ \left(X_2(t)\right)^{-1}=\mu_{t+s}.
                \end{equation*}
            \item[(iv.2)] Regarded on the filtered probability space $\left(\Omega, \mathcal{F},\mathcal{F}(t),\mathbb{P}_{0,\mu_0}\right)$, the process $(X_2(t))_{t\geq 0}$ is the unique weak solution to the McKean-Vlasov SDE \eqref{eq:MV SDE}, with initial condition $\mathcal{L}_{X_2(0)}=\mu_0$ and such that $\mathcal{L}_{X_2(t)}\leq c \mu_t, \;t\geq 0$ for some constant $c\in (0,\infty)$.
            \item[(iv.3)] Regarded on the filtered probability space $\left(\Omega, \mathcal{F},\mathcal{F}(t),\mathbb{P}_{0,\mu_0}\right)$ the process $(X_2(t))_{t\geq 0}$ is strong Markov, in the sense that for every finite $\mathcal{F}(t)$-stopping times $\tau$ we have that
            \begin{equation*}
                \mathbb{E}_{\delta_0\otimes \mu_0}\left\{ f(t+\tau,X_2(t+\tau)) \;\big|\;\mathcal{F}(\tau) \right\}=P_tf(\tau,X_2(\tau)),\quad t\geq 0, f\in b\mathcal{B}([0,\infty)\times \mathbb{R}^d).
            \end{equation*}
            \end{enumerate}
\end{enumerate}
\end{thm}
\begin{proof}
The existence of $X$ with all the desired properties follows by the main results \Cref{thm:main0} and \Cref{thm:main1}, with $a,b$ replaced by $a^u,b^u$. 
\end{proof}

\begin{rem}
Let $(X_2(t))_{t\geq 0}$ be the weak solution constructed in \Cref{thm:MVE}, and define for every $A\in \mathcal{B}(\mathbb{R}^d)$
\begin{equation*}
    \tau^{(2)}_A:=\inf\left\{t> 0 : X_2(t)\in A\right\} \mbox{ (hitting time)}, \quad D^{(2)}_A:=\inf\left\{t\geq 0 : X_2(t)\in A\right\} \mbox{ (entry time)}.
\end{equation*}
Then both $\tau^{(2)}_A$ and $D^{(2)}_A$ are $\left(\mathcal{F}(t)\right)$-stoping times.
\end{rem}

\begin{coro}\label{coro: capacity}
Let $\mu:=(\mu_t)_{t\geq 0}\subset \mathcal{P}(\mathbb{R}^d)$ be a weakly continuous solution to the nonlinear Fokker-Planck equation \eqref{eq:NLFPshort}, such that $\mu_t:= u(t,x) dx,\; t\geq 0$ for a jointly measurable density
$u:[0,\infty)\times \mathbb{R}^d\rightarrow \mathbb{R}_+$.  
Furthermore, assume that the "linearized" conditions
\begin{equation*}
\mathbf{H_{a^u,b^u}^{\sqrt{u}}},\quad \mathbf{H_{\sf \lesssim u}}({\sf L^u}), \quad \mathbf{H^{\sf TV}_0}({\sf L^u}) \quad \mbox{ are fulfilled}.
\end{equation*}
Let $X$ be the process provided by \Cref{thm:MVE}, $\alpha>0$, and consider the mapping
\begin{equation*}
     \mathbb{R}^d\supset A\longmapsto {\sf Cap}^\alpha_{(\mu_t)}(A):= \inf\left\{\mathbb{E}_{\delta_0\otimes \mu_0}\left\{e^{-\alpha D^{(2)}_{G}}\right\} : A\subset G, \; G\subset \mathbb{R}^d \mbox{ open}\right\}.
\end{equation*}
Then ${\sf Cap}_{(\mu_t)}^\alpha$ is a Choquet capacity on $\mathbb{R}^d$ endowed with the norm topology.
Moreover, ${\sf Cap}_{(\mu_t)}^\alpha$ is tight.
\end{coro}
\begin{proof}
Recall that $X$ is a path-continuous right process  on $E$.
Therefore, we can consider the (tight) Choquet capacity ${\sf Cap}^\alpha_{\delta_0\otimes\mu_0}$ given by \eqref{eq:capacity}, with   $\mathcal{T}$ being the norm topology on $E$, $\nu=\delta_0\otimes\mu_0$, and $f=1/\alpha$.
Then it is straightforward to check that
\begin{equation*}
    {\sf Cap}_{(\mu_t)}^\alpha(A)={\sf Cap}^\alpha_{\delta_0\otimes\mu_0}([0,\infty)\times A), \quad A\subset \mathbb{R}^d.
\end{equation*}
is a Choquet capacity on $\mathbb{R}^d$ endowed with the norm topology, which is also tight.
\end{proof}

\begin{rem}
    The capacity ${\sf Cap}^{\alpha}_{\mu_0}$ constructed in \Cref{coro: capacity} depends only on the given solution $\mu=(\mu_t)_{t\geq 0}$ of the nonlinear Fokker-Planck equation \eqref{eq:NLFPshort}, through the corresponding superposition on the path space $C([0,\infty); \mathbb{R}^d)$ provided by \Cref{rem:Trevisan_superposition}. 
\end{rem}

\subsection{Example: Generalized porous media equations}\label{ss:example_nonlinear}

We place in the framework of \cite{BaRo23} and consider the nonlinear Fokker-Planck equation
\begin{align}\label{eq:BR-FP}
\frac{d}{dt}u(t,x;u_0)
&=\Delta \beta(x,u(t,x;u_0))-{\sf div} (D(x)b(u(t,x;u_0))u(t,x;u_0)), 
\; t>0, x\in \mathbb{R}^d\\
u(0,\cdot)
&=u_0(\cdot).\nonumber
\end{align}
As we will see, like in \Cref{s:construction}, in order to apply our results to \eqref{eq:BR-FP}, we have to construct a reference density satisfying the linearized conditions $\mathbf{H_{a^u,b^u}^{\sqrt{u}}},\quad \mathbf{H_{\sf \lesssim u}}({\sf L^u}), \quad \mathbf{H^{\sf TV}_0}({\sf L^u})$.

To this end, we consider the following set of assumptions on $(\beta,D,b)$, most of them being taken over from \cite{BaRo23}: 
\begin{equation*}
\beta:\mathbb{R}^2\times\mathbb{R}\rightarrow\mathbb{R} \quad \mbox{and} \quad D:\mathbb{R}^d\times\mathbb{R}\rightarrow \mathbb{R}, \quad b:\mathbb{R}\rightarrow\mathbb{R}
\end{equation*}
satisfy the following:
\begin{enumerate}
    \item[($H_\beta$)] $\beta\in C^2(\mathbb{R}^d\times \mathbb{R};\mathbb{R})$ is such that:
    \begin{enumerate}
        \item[($H_{\beta,1}$)] 
        $
            \beta_r(x,r):=\frac{d}{dr}\beta(x,r)\geq \alpha_0>0,\;(x,r)\in \mathbb{R}^{d+1}, \quad \beta_{r}\in L^\infty\left(\mathbb{R}^{d+1}\right),\quad \nabla_x\beta_r, \beta_{rr}\in L^\infty(\mathbb{R}^d\times [0,N]),\quad N>0,\\[2mm] 
            \beta(x,0)=0,\quad x\in \mathbb{R}^d.
        $
        \item[($H_{\beta,2}$)] 
        $
            \int_{\mathbb{R}^d}\sup\limits_{|r|\leq N}|\Delta_x \beta(x,r)| \;dx<\infty, \quad N>0.
        $
    \end{enumerate}
    \item[($H_D$)] For $m>d/2$ if $d\geq 2$, and $m=1$  if $d=1$,
    \begin{equation*}
        D\in L^\infty(\mathbb{R}^d;\mathbb{R}^d), \quad {\sf div} D\in L^m_{\sf loc}(\mathbb{R}^d), \quad \left({\sf div} D\right)^-\in L^\infty(\mathbb{R}^d).
    \end{equation*}
    \item[($H_b$)] 
    $ 
    b\in C^1(\mathbb{R})\cap C_b(\mathbb{R}) \quad \mbox{and} \quad b\geq 0.
    $
    \item[($H_{b,\beta}$)]
    $
        |b(r)r-b(s)s|\leq c |\beta(x,r)-\beta(x,s)|,\quad x\in \mathbb{R}^d,s,t\in\mathbb{R}, \quad \mbox{ for some } c\in (0,\infty).
    $
\end{enumerate}

The following well posedness result follows from \cite[Theorem 6.1]{BaRo23}:
{\it For every $u_0\in L^1(\mathbb{R}^d)$ there exists a unique mild solution $u(\cdot,\cdot;u_0)\in C([0,\infty);L^1(\mathbb{R}^d))$ to the nonlinear Fokker-Planck equation \eqref{eq:BR-FP};}
see \cite[Definition 1.1]{BaRo23} for details on the notion of solution.

In fact, the above mentioned well-posedness result is valid under more general assumptions on the coefficients, see \cite[Theorem 2.1 \& Theorem 6.1]{BaRo23}.
The role of the assumptions $(H_\beta), (H_D), (H_b), (H_{b,\beta})$ from above, besides the well-posedness part, is that they also guarantee the following key properties of the solution, according to \cite{BaRo23}.
\begin{enumerate}
    \item[(1)] Cf. \cite[(1.16) \& (2.4)]{BaRo23}:
    
    $u(t+s,\cdot;u_0)=u(t,\cdot;u(s,\cdot;u_0)), \quad t,s\geq 0.
    $
    \item[(2)] Cf. \cite[(1.17) \& (2.4)]{BaRo23}:
    
    $
        \|u(t,\cdot;u_0)-u(t,\cdot;v_0)\|_{L^1(\mathbb{R}^d)}\leq  \|u_0-v_0\|_{L^1(\mathbb{R}^d)}, \quad t\geq0,\; u_0,v_0\in L^1(\mathbb{R}^d).
    $
    \item[(3)] Cf. \cite[Theorem 6.1]{BaRo23}: 
    
    If $u_0(x)dx \in \mathcal{P}(\mathbb{R}^d)$ then $u_t(x)dx\in \mathcal{P}(\mathbb{R}^d)$ for every $t>0$.
    \item[(4)] Cf. \cite[Theorem 6.1]{BaRo23}:
    
    If $u_0\in L^1(\mathbb{R}^d)\cap L^\infty(\mathbb{R}^d)$ then $u(\cdot,\cdot;u_0)\in L^\infty((0,N)\times \mathbb{R}^d)$, for every $N>0$.
    \item[(5)] Cf. \cite[Theorems 2.2 \& 5.2 \& Section 6]{BaRo23}: 
    
    If $u_0\in L^1(\mathbb{R}^d)\cap L^2(\mathbb{R}^d)$ then 
    \begin{align*}
        &(5.1)\quad  \beta(\cdot,u(\cdot,\cdot;u_0))\in L^2((0,N);H^1(\mathbb{R}^d))\cap L^\infty((0,N);L^2(\mathbb{R}^d))\\[2mm]
        &(5.2)\quad \frac{d}{dt}u(t,\cdot;u_0)\in L^2((0,N);H^{-1}(\mathbb{R}^d)), \quad N>0.
    \end{align*}
\end{enumerate}

\begin{lem}\label{lem:beta_hat}
Under $(H_\beta)$ there exists a function $\hat{\beta}$
\begin{equation*}
    \hat{\beta}:\mathbb{R}^d\times\mathbb{R}\rightarrow\mathbb{R},\quad \beta(x,\hat{\beta}(x,r))=r=\hat{\beta}(x,\beta(x,r)), \quad (x,r)\in \mathbb{R}^d\times \mathbb{R},
\end{equation*}
such that the following properties hold
\begin{equation}\label{eq:beta_hat_prop}
    \hat{\beta}\in C^1(\mathbb{R}^d\times\mathbb{R}),\quad \hat{\beta_{r}},\; \nabla_x\hat{\beta} \in L^\infty(\mathbb{R}^d\times [0,N]),\quad N>0
\end{equation}
\end{lem}
\begin{proof}
First of all, since $\beta \in C^2$, from $(H_{\beta,1})$ we deduce that for every $x\in\mathbb{R}^d$ there is a $C^1(\mathbb{R})$-function $\hat{\beta}(x,\cdot)$ such that
\begin{equation*}
    \beta(x,\hat{\beta}(x,r))=r=\hat{\beta}(x,\beta(x,r)),\quad r\in \mathbb{R}.
\end{equation*}
Moreover, by the implicit function theorem and the injectivity of $\beta(x,\cdot)$ for each $x\in \mathbb{R}^d$ we deduce that $\hat{\beta}(\cdot,r)$ is $C^1$  for every $r\in \mathbb{R}$; thus, $\hat{\beta}$ is $C^1$ in each variable.

Furthermore, from $(H_{\beta,1})$ we get
\begin{equation*}
    \nabla_x\beta(x,t)=\int_0^t\nabla_x\beta_r(x,r)\;dr, \;t\in \mathbb{R},\quad \mbox{hence}\quad \|\nabla_x\beta\|_{L^{\infty}(\mathbb{R}^d\times[0,N])}\leq N \|\nabla_x\beta_r\|_{L^{\infty}(\mathbb{R}^d\times[0,N])}<\infty,\quad N>0,
\end{equation*}
as well as
\begin{equation*}
    r=\beta(x,\widehat{\beta}(x,r))\geq \alpha_0 \widehat{\beta}(x,r), \quad \mbox{hence}\quad \widehat{\beta}\in L^\infty(\mathbb{R}^d\times[0,N]),\quad N>0.
\end{equation*}
Now, since
\begin{equation*}
    \nabla_x\hat{\beta}(x,r)=\frac{-\nabla_x\beta(x,t)|_{t=\hat{\beta}(x,r)}}{\beta_r(x,t)|_{t=\hat{\beta}(x,r)}}, \quad \hat{\beta}_r(x,r)=\frac{1}{\beta_r(x,t)|_{t=\hat{\beta}(x,r)}}, \quad (x,r)\in \mathbb{R}^d\times\mathbb{R}.
\end{equation*}
the boundedness properties in \eqref{eq:beta_hat_prop} follow immediately, and these prove in particular that $\hat{\beta}$ is jointly continuous (being locally Lipschitz). 
Thus, $\hat{\beta}$ is jointly $C^1$, and the proof is complete.
\end{proof}

\begin{prop}\label{prop:u_H1}
    Assume that $\beta,D,b$ satisfy conditions $(H_\beta), (H_D), (H_b), (H_{b,\beta})$, let $u_0\in L^1(\mathbb{R}^d)\cap L^\infty(\mathbb{R}^d)$ and consider $u(\cdot,\cdot):=u(\cdot,\cdot;u_0)$ to be the solution to \eqref{eq:BR-FP}. 
    Then
    \begin{equation*}
        \nabla_x u(\cdot,\cdot)\in L^\infty([0,T]\times \mathbb{R}^d)+L^2([0,T];L^2(\mathbb{R}^d))\subset L^2([0,T];H^1_{\sf loc}(\mathbb{R}^d)), \quad T>0.
    \end{equation*}
\end{prop}

\begin{proof}
Let $\hat{\beta}$ be the function provided by \Cref{lem:beta_hat}, so that
\begin{equation*}
    u(t,x)=\hat{\beta}(x,\beta(x,u(t,x))), \quad x,t\in \mathbb{R}^d\times\mathbb{R} \mbox{ a.e.} 
\end{equation*}
Consequently, the statement follows if we show that 
\begin{equation*}
    (t,x)\mapsto\nabla_x\left[\hat{\beta}(x,\beta(x,u(t,x)))\right]\in L^2([0,T];L^2(\mathbb{R}^d)).
\end{equation*}
To this end, by \eqref{eq:beta_hat_prop} and the fact that we already know that $\beta(u)\in L^2([0,T];H^1(\mathbb{R}^d))$, we have
\begin{align*}
(t,x)\mapsto\nabla_x\left[\hat{\beta}(x,\beta(x,u(t,x)))\right]
&=
\underbrace{\nabla_x\hat{\beta}(x,r)|_{r=\beta(x,u(t,x))}}_{\in L^\infty([0,T]\times\mathbb{R}^d)}
+\underbrace{\hat{\beta}_r(x,\beta(x,u(t,x)))}_{\in L^\infty([0,T]\times\mathbb{R}^d)}\underbrace{\nabla_x\left[\beta(x,u(t,x))\right]}_{\in L^2([0,T];L^2(\mathbb{R}^d))}\\
&\in L^\infty([0,T]\times \mathbb{R}^d)+L^2([0,T];L^2(\mathbb{R}^d)).
\end{align*}
\end{proof}

Let $0\leq u_0\in L^\infty(\mathbb{R}^d)$ such that $\int_{\mathbb{R}^d}u_0(x)\;dx=1$, and let $u(\cdot,\cdot;u_0)$ be the corresponding solution presented above.
If we set $\mu_t(dx):=u(t,x;u_0)dx$, $t\geq 0$, then \eqref{eq:BR-FP} recasts as: 

For every $f\in C_c^\infty((0,\infty)\times \mathbb{R}^d)$
\begin{align}\label{eq:NL-ab}
        &\int_0^\infty\int_{\mathbb{R}^d} \frac{d}{dt}f(t,x)+ \underbrace{\frac{\beta(x,u(t,x))}{u(t,x)}}_{=:a(x,\mu_t)}\Delta f(t,x)
        +\langle \underbrace{D(x)b(u(t,x))}_{=:b(x,\mu_t)},\nabla f(t,x)\rangle \; \mu_t(dx)\;dt
        =0,\\
        &u(0,\cdot)
        =u_0(\cdot).
\end{align}
Since $\mu:=(\mu_t)_{t\geq 0}\subset \mathcal{P}(\mathbb{R}^d)$ is also continuous in the total variation norm, we thus have that $\mu$ is a weakly continuous solution (in the sense of \Cref{def:Nsolution}) to the non-linear Fokker-Planck equation 
\begin{align}
&\frac{d}{dt} \mu_t 
= {\sf L}^\ast_{t,\mu_t} \mu_t,\label{eq:a^ub^u}\\ 
&\underbrace{{\sf L}_{t,\mu_t}}_{{\sf L}^u_t}f(t,x)
:=\underbrace{b(x,\mu_t)}_{=:b^u(t,x)}\nabla_x f(t,x) + \frac{1}{2} \underbrace{2a(x,\mu_t)}_{=:a^u(t,x)}\Delta_xf(t,x), \quad (t,x) \in (0,\infty)\times \mathbb{R}^d, \nonumber
\end{align}
where $b(x,\mu_t)$ and $a(x,\mu_t)$ are given in \eqref{eq:NL-ab}.

Now, we need a convenient well-posedness result for the {\it linearized} Fokker-Planck equation
\begin{equation}\label{eq:linearizedFPE}
    \frac{d}{dt} \nu_t = \left({\sf L}^{u}_{t}\right)^\ast \nu_t, \quad \mbox{where } {\sf L}^{u}_{t} \mbox{ is defined in } \eqref{eq:a^ub^u} \mbox{ with } u \mbox{ being fixed in the coefficients}.
\end{equation}

\begin{prop}\label{prop:BrisLyonsBaRo}
Assume that $\beta,D,b$ satisfy conditions $(H_\beta), (H_D), (H_b), (H_{b,\beta})$, let $0\leq u_0\in L^\infty(\mathbb{R}^d)$ such that $u_0(x)dx\in \mathcal{P}(\mathbb{R}^d)$, and let $u(\cdot,\cdot;u_0),\; \mu_t(dx):=u(t,x;u_0)dx, t\geq 0$ be the solution to \eqref{eq:BR-FP}/\eqref{eq:NL-ab}.
The following assertions hold
\begin{enumerate}
    \item[(i)] Let $0\leq s<T$ and $\nu_s=v_s(x)dx$ such that $0\leq v_s\in L^1(\mathbb{R}^d)\cap L^\infty(\mathbb{R}^d)$.
    
    \noindent{}If $\nu^{(1)}:=\left(\nu^{(1)}_t=v^{(1)}(t,x)dx\right)_{t\in [s,T]}$ and $\nu_t^{(2)}:=\left(\nu^{(2)}_t=v^{(2)}(t,x)dx\right)_{t\in [s,T]}$ are two right weakly-continuous solutions to the linearized Fokker-Planck equation \eqref{eq:linearizedFPE}, such that $v^{(1)},v^{(2)}\in L^\infty([s,T]\times \mathbb{R}^d)$, sharing the same initial condition $v_s(x)dx$, then
    \begin{equation*}
       \nu^{(1)}=\nu^{(2)}.
    \end{equation*}
    \item[(ii)] 
    Let $\nu_0=v_0(x)dx$ such that $0\leq v_0\in L^1(\mathbb{R}^d)\cap L^\infty(\mathbb{R}^d)$.
    Then there exists a solution $\left(\nu_t=v(t,x)dx\right)_{t\in [0,T]}$ to the linearized Fokker-Planck equation \eqref{eq:linearizedFPE} such that
    \begin{equation*}
        v\in L^\infty([0,T];L^1\cap L^\infty(\mathbb{R}^d)) \quad \mbox{and} \quad \sqrt{a^u}\nabla v\in L^2([0,T];L^2(\mathbb{R}^d)).
    \end{equation*}
    Moreover, it follows that
    \begin{equation}\label{eq:vinCloc}
        v\in C([0,T];L^2_{\sf loc}(\mathbb{R}^d)), \quad T>0,
    \end{equation}
    and if $v_0\leq c u_0$ for some $c\in (0,\infty)$, then 
    \begin{equation}\label{eq:vinC}
        v\in C([0,T];L^2(\mathbb{R}^d)), \quad T>0.
    \end{equation}
\end{enumerate}
\end{prop}
\begin{proof}
(i) It follows directly from \cite[Corollary 3.5]{BaRo23} as well as the last line on page 34 from \cite{BaRo23}.

\medskip
(ii). 
Fix $T>0$. 
The idea is to rely on the results \cite[Proposition 2 \& Proposition 3]{BrLi08}, so let us note that  $\left(\nu_t=v(t,x)dx\right)_{t\in [s,T]}$ being a solution to \eqref{eq:linearizedFPE}, becomes, in divergence form \cite[(5.8)]{{BrLi08}},
    \begin{equation*}
        \frac{d}{dt} v + {\sf div}\left(\left(b^u-\frac{1}{2}\nabla a^u\right)v\right)-\frac{1}{2}{\sf div}(a^u \nabla v)=0 
    \end{equation*}
In fact, it turns out that the Girsanov-type result \cite[Proposition 3]{BrLi08} is much more suitable to our context, so we go on and rewrite the above equation in the special form
    \begin{equation*}
       \frac{d}{dt} v + {\sf div}\left(\sqrt{a^u}\theta v\right)-\frac{1}{2}{\sf div}(a^u \nabla v)=0, \quad \mbox{where } \theta:= \frac{b^u}{\sqrt{a^u}}-\frac{\nabla a^u}{2\sqrt{a^u}}.
   \end{equation*}
Now, assertion (i) follows by \cite[Proposition 4]{BrLi08} as soon as we check the following conditions (see \cite[(6.12)-(6.13)]{BrLi08}), 
\begin{enumerate}
    \item[(ii.1)] $\sqrt{a^u}\in L^2([0,T];W^{1,2}_{\sf loc}(\mathbb{R}^d))$, $\quad \frac{\sqrt{a^u}}{1+\|x\|}\in L^2([0,T];L^2+L^\infty(\mathbb{R}^d))$
    \item[(ii.2)] $\theta \in L^2([0,T];L^2+L^\infty(\mathbb{R}^d))$.
\end{enumerate}

Let us check (ii.1) first.
Note that by condition $(H_{\beta})$ and the fact that $u\in L^\infty([0,T];\mathbb{R}^d)$ we have
\begin{equation}\label{eq:beta_bounds}
    \mbox{On }[0,T]\times\mathbb{R}^d, \;\exists c\in (0,\infty):\quad a^u=\beta(u)/u\leq c,\quad \frac{\left|\beta_r(u)u-\beta(u)\right|}{u^2}\leq c, \quad \frac{\left|\nabla_x\beta(x,u)\right|}{u}\leq c.
\end{equation}
Now, by the first bound in \eqref{eq:beta_bounds}, the second part of (ii.1) is immediate.

\noindent{}Regarding the first part of (ii.1), note that by $(H_{\beta})$ we have 
\begin{equation}\label{eq:beta_low_bound}
     a^u=\frac{\beta(u)}{u}\geq \alpha_0 \quad \mbox{on } [0,\infty)\times \mathbb{R}^d. 
\end{equation}
Consequently, using also \eqref{eq:beta_low_bound}, \eqref{eq:beta_bounds}, we have
\begin{align*}
\left|\nabla\sqrt{a^u}\right|
&=\frac{\left|\nabla a^u\right|}{\sqrt{a^u}}
\leq\frac{\left|\nabla a^u\right|}{\sqrt{\alpha_0}}
\leq \frac{1}{\sqrt{\alpha_0}}|\nabla a^u|\\
&\leq \frac{1}{\sqrt{\alpha_0}}|\nabla u|\frac{\left|\nabla_x\beta(x,u)u+\beta_r(u)u-\beta(u)\right|}{u^2}\\
&\leq \frac{1}{\sqrt{\alpha_0}}|\nabla u|\left(\frac{\left|\nabla_x\beta(x,u)\right|}{u}+\frac{\left|\beta_r(u)u-\beta(u)\right|}{u^2}\right)\\
&\leq \frac{2c}{\sqrt{\alpha_0}}|\nabla u|\in L^\infty([0,T]\times \mathbb{R}^d)+L^2([0,T];L^2(\mathbb{R}^d)),\quad \mbox{ due to } \Cref{prop:u_H1}.
\end{align*}
Thus, assertion (ii.1) is completely proved.
Assertion (i.2) is also clear now, by \eqref{eq:beta_low_bound}, the fact that $b\in C_b(\mathbb{R})$ according to $(H_b)$, and the property $\nabla \sqrt{ a^u}\in L^\infty([0,T]\times\mathbb{R}^d)+L^2([0,T]; L^2(\mathbb{R}^d))$ which was proved above.

In order to completely prove assertion (ii), we need to justify that $v\in C([0,T];L^2(\mathbb{R}^d))$.
To this end, note that by the first part of (ii) proved above, using also \eqref{eq:beta_low_bound}, it follows that $v\in L^2([0,T];H^1(\mathbb{R}^d))$.
To conclude, we can now apply \Cref{lem:hu in H-1} for ${\sf L}_t^u$ instead of ${\sf L}_t$ to conclude that $hv\in C([0,T];L^2(\mathbb{R}^d)), h\in C_c^\infty(\mathbb{R}^d)$, that is,
\begin{equation*}
    v\in C([0,T];L^2_{\sf loc}(\mathbb{R}^d)).
\end{equation*}
Finally, if $v_0\leq c u_0$, and if $\eta$ is given by \Cref{rem:Trevisan_superposition} applied to ${\sf L^u}$ and the fixed solution $(\mu_t)_{t\geq 0}$, and if we set $\tilde\nu_0:=\nu_0/\nu_0(\mathbb{R}^d)$ then
$(\eta^{\tilde\nu_0}_t)_{t\in [0,T]}$ gives a weakly continuous probability solution to the linearized Fokker-Planck equation \eqref{eq:linearizedFPE}, which satisfies $\eta^{\tilde\nu_0}_t\leq c \mu_t,t\geq 0$, hence it admits a density in $L^\infty([0,T]\times \mathbb{R}^d)$. 
Consequently, by (i) we get
\begin{equation*}
    (\eta^{\tilde\nu_0}_t)_{t\in [0,T]}=\left(\tilde\nu_t\right)_{t\in [0,T]},\quad \mbox{where}\quad \tilde\nu_t:=\frac{\nu_t}{\nu_0(\mathbb{R}^d)}, t\in [0,T].
\end{equation*}
Thus, $(\nu_t)_{t\in [0,T]}$ is a tight family of finite measures, and corroborating with \eqref{eq:vinCloc} in the same spirit as in \Cref{rem: implies TV}, we get \eqref{eq:vinC}.
\end{proof}

\begin{thm}
Assume that $\beta,D,b$ satisfy conditions (i)-(iv), let $0\leq u_0\in L^\infty(\mathbb{R}^d)$ such that $\int_{\mathbb{R}^d}u_0(x)\;dx=1$, and let $u(\cdot,\cdot;u_0),\; \mu_t(dx):=u(t,x;u_0)dx, t\geq 0$ be the solution to \eqref{eq:BR-FP}/\eqref{eq:NL-ab}.
Then the "linearized" conditions
\begin{equation*}
\mathbf{H_{a^u,b^u}^{\sqrt{u}}},\quad \mathbf{H_{\sf \lesssim u}}({\sf L^u}), \quad \mathbf{H^{\sf TV}_0}({\sf L^u})
\end{equation*}
are fulfilled, where $a^u, b^u,$ and ${\sf L^u}$ are given in \eqref{eq:a^ub^u}.
Consequently, the conclusions of \Cref{thm:MVE} and \Cref{coro: capacity} are valid in this situation.
\end{thm}
\begin{proof}
Condition $\mathbf{H_{a^u}}$ is fulfilled by e.g. (ii.1) from the proof of \Cref{prop:BrisLyonsBaRo}, \eqref{eq:beta_bounds}, \eqref{eq:beta_low_bound}, and the fact that $u\in L^\infty([0,T]\times \mathbb{R}^d), T>0$.
Condition $\mathbf{H_{b^u}}$ follows from $(H_b)$ and $(H_D)$.
Condition $\mathbf{H^{\sqrt{u}}}$ follow from \Cref{coro: sqrt}.
Thus, condition $\mathbf{H_{a^u,b^u}^{\sqrt{u}}}$ is fulfilled.

Furthermore, it is straightforward to see that condition $\mathbf{H_{\sf \lesssim u}}({\sf L^u})$ is fulfilled by \Cref{prop:BrisLyonsBaRo}, (i).

Finally, condition $\mathbf{H^{\sf TV}_0}({\sf L^u})$ is also satisfied due to \Cref{prop:BrisLyonsBaRo}, (ii), the dominated convergence, and using that $u\in C([0,\infty);L^1(\mathbb{R}^d))$.
\end{proof}

\section{Elements of right processes}\label{Appendix}

Throughout this section we follow mainly the terminology of \cite{BlGe68}, \cite{Sh88}, and \cite{BeBo04a}. 
The aim here is two-fold: 

\medskip
\noindent{$\bullet\quad$} We present a concise introduction to the theory of probabilistic potential theory, with special emphasis on those tools that are fundamentally used in order to derive the main results from \Cref{s:reg_superposition_linear} and \Cref{s:consequences}. 

\medskip
\noindent{$\bullet\quad$} We prove new  results on right processes that are of general nature and which are, as well, needed in \Cref{s:reg_superposition_linear} and \Cref{s:consequences}; see e.g. \Cref{lem:B_is_finely_feller}, \Cref{prop:B-boundary}, \Cref{prop:continuous_modification_special}, \Cref{Borel-measurab}, \Cref{prop:martingale_absorbing}, as well as those from \Cref{ss:adding jumps}.

\subsection{Fundamentals on right processes and their potential theory}
Let $(E, \mathcal{B})$ be a Lusin measurable space, i.e. it is measurably isomorphic to a Borel subset of a compact metric space endowed with the corresponding Borel $\sigma$-algebra.
We denote by $(b)p\mathcal{B}$ the set of all numerical, (bounded) positive $\mathcal{B}$-measurable functions on $E$.
Throughout, by $\mathcal{U}=(U_\alpha)_{\alpha>0}$ we denote a resolvent family of (sub-)Markovian kernels on $(E, \mathcal{B})$; see e.g. \cite[Subsection 1.1]{BeBo04a}.
If $\beta>0$, we set $\mathcal{U}_\beta:=(U_{\beta+\alpha})_{\alpha>0}$.

$A\subset E$ is called {\it universally} $\mathcal{B}$-measurable if for every (positive) finite measure $\mu$ on $(E,\mathcal{B})$ there exist $A_1,A_2\in \mathcal{B}$ such that $A_1\subset A\subset A_2$ and $\mu(A_2\setminus A_1)=0$.
By $\mathcal{B}^u$ we denote the $\sigma$-algebra of all universally $\mathcal{B}$-measurable sets in $E$.

\begin{defi} \label{defi 4.1}
A universally $\mathcal{B}$-measurable function $v:E\rightarrow \overline{\mathbb{R}}_+$ is called 
$\mathcal{U}$--{\rm excessive} 
provided that  
$\alpha U_\alpha v \leq v$ for all $\alpha >0$ and $\mathop{\sup}\limits_\alpha \alpha U_\alpha v =v $ point-wise; by $\mathcal{E(\mathcal{U})}$ we denote the convex cone of all $\mathcal{B}$-measurable 
 $\mathcal{U}$--excessive functions.
\end{defi}

\begin{rem} \label{rem:U_exc}
Recall that if $f\in p\mathcal{B}$, then $U_\alpha f$ is $\mathcal{U}_\alpha$-excessive.
\end{rem}

If a universally $\mathcal{B}$-measurable function $w: E \rightarrow \overline{\mathbb{R}}_+$ is merely $\mathcal{U}_\beta$-supermedian (i.e. $\alpha U_{\beta+\alpha} w \leq w$ for all $\alpha > 0$), then its $\mathcal{U}_\beta$--{\it excessive regularization} 
is defined as
$\widehat{w}:=\mathop{\sup}\limits_\alpha \alpha U_{\beta +\alpha}w$. 
The function $\widehat{w}$ is 
$\mathcal{U}_\beta$--excessive and 
$\widehat{w} \in \mathcal{E(U_\beta)}$ provided that 
$w$ is  in addition $\mathcal{B}$-measurable.

\begin{defi} \label{defi: fine_topology}
The {\it fine topology} on $E$ (associated with $\mathcal{U}$) is the coarsest topology on $E$ 
such that every $\mathcal{U}_\alpha$-excessive function is continuous for some (hence all) $\alpha >0$.
\end{defi}

\begin{prop}\label{prop:zero_open}
    If $u$ is $\mathcal{U}_\beta$-excessive for some $\beta>0$, then the set $[u=0]$ is finely open.
\end{prop}

\begin{defi} \label{defi:natual_topology}
A topology $\tau$ on $E$ is called {\rm natural} if it is a Lusin topology (i.e. $(E,\tau)$ is homeomorphic to a Borel subset of a compact metrizable space) which is coarser than the fine topology, and whose Borel $\sigma$-algebra is $\mathcal{B}$.    
\end{defi}

\begin{rem} \label{rem 4.3}
\begin{enumerate}[(i)]
    \item The necessity of considering natural topologies comes from the fact that, in general, the fine topology is neither metrizable, nor countably generated.
    \item Let $(E,\mathcal{B})$ be a Lusin measurable space, and $(f_n)_{n\geq 1}\subset \mathcal{B}$ be such that it separates the points of $E$. 
    Then the topology on $E$ generated by $(f_n)_{n\geq 1}$ is a Lusin topology; see e.g. \cite[Theorem A2.12]{Sh88} for the similar Radonian case.
\end{enumerate}
\end{rem}

There is a convenient class of natural topologies to work with (as we do in Section 2), especially when the aim is to construct a right process associated with  $\mathcal{U}$ (see Definition \ref{defi 4.4}). These topologies are called Ray topologies, and are defined as follows.

\begin{defi} \label{defi 4.5}
\begin{enumerate}[(i)]
\item If $\beta >0$ then a Ray cone associated with $\mathcal{U}_\beta$ is a convex cone $\mathcal{R}$ of bounded, $\cb$-measurable $\mathcal{U}_\beta$-excessive functions which is separable in the supremum norm, min-stable, contains the constant function $1$, generates $\mathcal{B}$, $U_\alpha(\mathcal{R}) \subset \mathcal{R}$ for all $\alpha > 0$, and $U_\beta((\mathcal{R}-\mathcal{R})_+) \subset \mathcal{R}$.
\item A {\it Ray topology} on $E$ is a topology generated by a Ray cone $\mathcal{R}$, and it is denoted by $\tau_{\mathcal{R}}$.
\end{enumerate} 
\end{defi}

In order to ensure the existence of a Ray topology, as well as many other useful properties related to the fine topology and excessive functions, we need the following condition:\\

\noindent
$\mathbf{(H_{\mathcal{C}})}.$ $\mathcal{C}$ is a min-stable convex cone of non-negative $\mathcal{B}$-measurable functions such that
\begin{enumerate}[(i)]
    \item $1\in \mathcal{C}$ and there exists a countable subset of $\mathcal{C}$ which separates the points of $E$,
    \item $U_\alpha(\mathcal{C})\subset \mathcal{C}$ for all $f\in \mathcal{C}$ and $\alpha>0$,
    \item $\lim\limits_{\alpha\to\infty}\alpha U_{\alpha}f=f$ pointwise on $E$ for all $f\in \mathcal{C}$.
\end{enumerate}

The following result is basically Corollary 2.3 from \cite{BeRo11a}:
\begin{prop} \label{prop 2.5}
The following assertions are equivalent.
\begin{enumerate}[(i)]
    \item There exists a cone $\mathcal{C}$ such that $\mathbf{(H_{\mathcal{C}})}$ is fulfilled.
    \item $\mathcal{E}(\mathcal{U}_\beta)$ is min stable, contains the constant functions, and generates $\mathcal{B}$ for one (hence all) $\beta >0$.
    \item For any $\beta>0$ there exists a Ray cone associated with $\mathcal{U}_\beta$. 
\end{enumerate}
\end{prop}

For the rest of this section we assume that one (hence all) of the assertions from \ref{prop 2.5} is valid.
This is always satisfied when there is a right Markov process with resolvent $\mathcal{U}$, whose definition is given below.

\begin{rem} \label{rem: ray_topology}
\begin{enumerate}[(i)]
\item It is clear that any Ray topology is a natural topology.
A useful converse is true: For any natural topology there exists a finer Ray topology; see Proposition 2.1 from \cite{BeBo05}.
In fact, a key ingredient here is that given a countable family of finite $\mathcal{U}_\beta$-excessive functions, one can always find a Ray cone that contains it; we shall use this fact later on.
\item As a matter of fact, if $(\mathbf{H_\mathcal{C}})$ holds, then one can explicitly construct plenty of Ray topologies, following e.g. \cite{BeBo04a}, Proposition 1.5.1, or \cite{BeRo11a}, Proposition 2.2.
As we shall need such a construction later on, let us detail it here.
Let $\mathcal{A}_0\subset bp\mathcal{B}$ be countable and such that it separates the set of finite measures on $(E,\mathcal{B})$.
The {\rm Ray cone generated by $\mathcal{A}_0$} and associated with  $\mathcal{U}_\beta$ is defined inductively as follows: 
\begin{align*}
&\mathcal{R}_0:=U_\beta (\mathcal{A}_0)\cup \mathbb{Q}_+,\\
&\mathcal{R}_{n+1}:=\mathbb{Q}_+\cdot \mathcal{R}_n \cup \left(\mathop{\sum}\limits_f\mathcal{R}_n \right)\cup \left(\mathop{\bigwedge}\limits_f \mathcal{R}_n\right) \cup \left(\mathop{\bigcup}\limits_{\alpha \in \mathbb{Q}_+}U_\alpha(\mathcal{R}_n)\right)\cup U_\beta ((\mathcal{R}_n-\mathcal{R}_n)_+),
\end{align*}
where by $\mathop{\sum}\limits_f \mathcal{R}_n$ resp. $\mathop{\bigwedge}\limits_f\mathcal{R}_n$ we denote the space of all finite sums (resp. infima) of elements from $\mathcal{R}_n$.
Then, the Ray cone $\mathcal{R}_{\mathcal{A}_0}$ generated by $\mathcal{A}_0$ is obtained by taking the closure of $\bigcup\limits_n \mathcal{R}_n$ w.r.t. the supremum norm.
\end{enumerate}
\end{rem}

\noindent
{\bf Right processes.} 
Let now $X=(\Omega, \mathcal{F}, \mathcal{F}(t) , X(t), \theta(t) , \mathbb{P}_x)$ be a normal Markov process with state space $E$, shift operators $\theta(t):\Omega\rightarrow \Omega, \; t\geq 0$, and lifetime $\zeta$; if $\Delta$ denotes the cemetery point attached to $E$, then any numerical function $f$ on $E$ shall be extended to $\Delta$ by setting $f(\Delta)=0$. 

 We assume that $X$ has the resolvent $\mathcal{U}$ fixed above, i.e. for all $f\in b\mathcal{B}$ and $\alpha >0$
$$
U_\alpha f(x)=\mathbb{E}_{x} \! \int_0^{\infty}e^{-\alpha t}f(X(t))dt,\quad  x\in E.
$$

To each probability measure $\mu$ on $(E, \mathcal{B})$ we associate the probability 
$$\mathbb{P}^\mu (A):=\mathop{\int} \mathbb{P}^x(A)\; \mu(dx)$$
for all $A \in \mathcal{F}$, and we consider the following enlarged filtration
$$
\widetilde{\mathcal{F}}(t):= \bigcap\limits_\mu \mathcal{F}^\mu(t), \; \; \widetilde{\mathcal{F}}:= \bigcap\limits_\mu \mathcal{F}^\mu,
$$
where $\mathcal{F}^\mu$ is the completion of $\mathcal{F}$ under $\mathbb{P}^\mu$, and $\mathcal{F}^\mu(t)$ is the completion of $\mathcal{F}(t)$ in $\mathcal{F}^\mu$ w.r.t. $\mathbb{P}^\mu$; in particular, $(x,A)\mapsto\mathbb{P}^{x}(A)$ is assumed to be a kernel from $(E,\mathcal{B}^{u})$ to $(\Omega, \mathcal{F})$, where $\mathcal{B}^{u}$ denotes the $\sigma$-algebra of all universally measurable subsets of $E$.

\begin{defi} \label{defi 4.4}
The Markov process $X$ is called a right (Markov) process if the following additional hypotheses are satisfied:
\begin{enumerate}[(i)]
\item The filtration $(\mathcal{F}(t))_{t\geq 0}$ is right continuous and $\mathcal{F}(t)=\widetilde{\mathcal{F}}(t), t\geq 0$.
\item For one (hence all) $\alpha>0$ and for each $f \in \mathcal{E}(\mathcal{U}_\alpha)$ the process $f(X)$ has right continuous paths $\mathbb{P}^{x}$-a.s. for all $x\in E$.
\item There exists a natural topology on $E$ with respect to which the paths of $X$ are $\mathbb{P}^{x}$-a.s. right continuous for all $x\in E$.
\end{enumerate}
\end{defi}

\begin{rem}[Strong Markov property]\label{rem:strong Markov}
    Recall that if $X$ is a right process on $(E,\mathcal{B})$ with transition function $(P_t)_{t\geq 0}$, then it is a strong Markov process, i.e. for any $\mu\in \mathcal{P}(E)$, $t\geq 0$, and $\tau$ an $(\mathcal{F}(t))$-stopping time we have
    \begin{equation*}
        \mathbb{E}_\mu\left\{f(X(t+\tau)) \mid \mathcal{F}_\tau\right\}=P_tf(X(\tau)) \quad \mathbb{P}_\mu\mbox{-a.s.}, \quad f\in b\mathcal{B}(E); \mbox{ recall that } f(\Delta):=0.
    \end{equation*}
    Moreover, for any $\mu, \tau$ as above and any $Y\in b\mathcal{F}$, the strong Markov property with shifts holds, namely
    \begin{equation*}
        \mathbb{E}_\mu\left\{Y\circ \theta(\tau) \mid \mathcal{F}_\tau\right\}=\mathbb{E}_{X(\tau)}\left\{Y\right\} \quad \mathbb{P}_\mu\mbox{-a.s.}, \quad Y\in b\mathcal{F}; \mbox{ here, } Y\circ \theta(\infty):=0.
    \end{equation*}
\end{rem}

Sometimes, for shortness and if there is no risk of confusion, we use $\left(X, \mathbb{P}_x,x\in E\right)$ to denote a right process $X$, instead of specifying the full tuple $X=(\Omega, \mathcal{F}, \mathcal{F}_t , X(t), \theta(t) , \mathbb{P}^x)$. 

\begin{rem}
It is worth to mention that if $X=(\Omega, \mathcal{F}, \mathcal{F}(t) , X(t), \theta(t) , \mathbb{P}_x)$ is a right process as defined above, and $\tau$ is a given natural topology on $E$, then then one can construct another right process $\tilde{X}=(\tilde{\Omega}, \tilde{\mathcal{F}}, \tilde{\mathcal{F}}(t), \tilde{X}(t), \tilde{\theta}(t), \tilde{\mathbb{P}}_x)$ sharing the same resolvent (or transition function) as $X$, i.e. $X$ and $\tilde{X}$ are {\rm equivalent} in the sense of \cite[Definition 4.1]{BlGe68}, such that assertion (iii) from \Cref{defi 4.4} can be replaced by a stronger one, namely:
\begin{enumerate}
    \item[(iii')] The path $[0,\infty)\ni t\rightarrow \tilde{X}(t,\omega)\in E$ is right-continuous with respect to $\tau$ for every $\omega \in \tilde{\Omega}$.
\end{enumerate}
Indeed, one can construct such $\tilde{X}$ by considering first a second process equivalent to $X$ which is of {\rm function space type}, according to \cite[Definition 4.2 and Theorem 4.3]{BlGe68}, and then simply deleting from the path space of this second process the set $\Lambda$ of those paths that are not right-continuous with respect to $\tau$.
Such a removal is allowed since $\Lambda$ is universally measurable according to \cite[Ch. IV, Theorem 34]{DeMe78}.
\end{rem}

The probabilistic description of the fine topology is given by the following key result, according to \cite[Chapter II, Theorem 4.8]{BlGe68}, or \cite[Proposition 10.8 and Exercise 10.18]{Sh88}, and \cite[Corollary A.10]{BeCiRo20}. 
The second assertion states  the relation between excessive functions and right-continuous supermartingales (cf. e.g.,  
\cite[Proposition 1]{BeCi18a}, see also \cite{BeCi18}). 

\begin{thm} \label{thm 4.6} Let  $X$ be a right process and $f$ a universally $\mathcal{B}$-measurable function. 
Then the following assertions hold.

$(i)$ The function $f$ is finely continuous if and only if $(f(X(t)))_{t\geq 0}$ has $\mathbb{P}^{x}$-a.s. right continuous paths for all $x\in E$.
In particular, $X$ has a.s. right continuous paths in any natural topology on $E$.

$(ii)$ If $\beta >0$ then the  function $f$ is 
$\mathcal{U}_\beta$-excessive if and only if $(e^{-\beta t} f(X(t)))_{t\geq 0}$ is 
a right continuous 
$\mathcal{F}_t$-supermartingale w.r.t. $\mathbb{P}^{x}$ for all  $x\in E$.
\end{thm}
\begin{rem}\label{rem:fine_continuity_0}
    Inspecting the proof of  \cite[Chapter II, Theorem 4.8]{BlGe68}, we see that if $f$ is universally $\mathcal{B}$-measurable and the mapping
    $[0,\infty)\ni t\mapsto f(X(t))\in \mathbb{R}$ is right continuous merely at $t=0$ $\mathbb{P}^{x}$-a.s. for all $x\in E$, then $f$ is finely continuous.
\end{rem}

\begin{rem}[{$P_t$ is finely Feller}]\label{rem:finely feller}
A remarkable property which has been relatively overlooked in the literature is that given a right process $X$ on a general Lusin measurable space $(E,\mathcal{B})$, the operators $P_t$ have the property that $P_tf$ is finely continuous whenever $f$ is finely continuous, bounded, and $\mathcal{B}^u$-measurable; see e.g. \cite[Exercise 10.24]{Sh88}, \cite[Exercise 4.14]{BlGe68}, or \cite[Lemma 2.8]{BeCi18} for a proof of this result.
In other words, $P_t$ is Feller with respect to the fine topology. 
\end{rem}

\begin{prop}\label{prop:tau-identification}
Let $X=(\Omega, \mathcal{F}, \mathcal{F}(t) , X(t), \theta(t) , \mathbb{P}_x), x\in E$ and $\tilde{X}=(\tilde{\Omega}, \tilde{\mathcal{F}}, \tilde{\mathcal{F}}(t), \tilde{X}(t), \tilde{\theta}(t), \tilde{\mathbb{P}}_x, x\in E)$ be two right processes on a Lusin measurable space $E$ sharing the same resolvent $\mathcal{U}$. 
Further, let $\tau$ be a natural (hence Lusin) topology on $E$ (with respect to $\mathcal{U}$), such that $X$ has continuous paths with respect to $\tau$ $\mathbb{P}_\mu$-a.s. for some $\mu\in \mathcal{P}$.
Then
\begin{equation}\label{eq:laws=}
    \mathbb{P}_\mu\circ \left(X\right)^{-1} = \tilde{\mathbb{P}}^\mu\circ \left(\tilde{X}\right)^{-1} \quad \mbox{as laws on the path-space }  C([0,\infty);E,\tau), 
\end{equation}
where $C([0,\infty);E,\tau)$ denotes the space of continuous functions defined on $[0,\infty)$ with values in $E$ endowed with the topology $\tau$.
\end{prop}
\begin{proof}
Let $D$ be the countable set of dyadics in $[0,\infty)$, and denote by $W^E$ the space of the restrictions to $D$ of all paths from $[0,\infty)$ to $E$ which are $\tau$-continuous.
Also, consider the product $E^D$ endowed with the canonical $\sigma$-algebra, and consider the laws
$$
\nu:=\mathbb{P}_\nu\circ (X)^{-1} \quad \mbox{and} \quad \tilde{\nu}:=\tilde{\mathbb{P}}_\mu\circ (\tilde{X})^{-1}
$$
of the two processes $X$ and $\tilde{X}$ on $E^D$
Now, on the one hand any distribution on ${E}^D$ is uniquely determined by its finite dimensional marginals, hence, by the Markov property, $\nu$ and $\tilde{\nu}$ are uniquely determined by their one-dimensional marginals. 
The latter marginals are in turn uniquely determined by their Laplace transforms, and since $X$ and $\tilde{X}$ share the same resolvent $\mathcal{U}$, we deduce that
\begin{equation*}
    \nu=\tilde{\nu}.
\end{equation*}
Next, by \cite[Chapter IV, Theorem 34; see also pages 91-92]{DeMe78}, we have that $W^E$ is a universally measurable subset of ${E}^D$. 
Consequently, we have 
$$
\nu(W^E)=\tilde{\nu}(W^E)=1,
$$
so, on the one hand, the paths of $\tilde{X}$ are restrictions to $D$ of $\tau$-continuous paths in $E$ $\tilde{\mathbb{P}}_\mu$-a.s. 
On the other hand, by \Cref{thm 4.6} it follows that $\tilde{X}$ has right-continuous paths with respect to $\tau$ $\mathbb{P}_x$-a.s. for every $x\in E$. 
This means that the right process $\tilde{X}$ has $E$ and has $\tau$-continuous paths in $E$ $\tilde{\mathbb{P}}_\mu$-a.s., and clearly \eqref{eq:laws=} is satisfied since a probability measure on $C([0,\infty);E,\tau)$ is uniquely determined by its finite dimensional time-marginals.
\end{proof}

We end this paragraph by recalling the analytic and probabilistic descriptions of {\it polar} (and other "small") sets.
\begin{defi}\label{defi:balayage}
Let $\alpha\geq 0$
\begin{enumerate}[(i)]
\item If $u\in \mathcal{E}(\mathcal{U}_\alpha)$ and $A \in \mathcal{B}$, then the $\alpha$-order reduced function of $u$ on $A$ is given by
$$
R_\alpha^A u = \inf \{ v \in \mathcal{E}(\mathcal{U}_\alpha): \, v\geqslant u \mbox{ on } A \}.
$$
$R_\alpha^A u$ is merely 
$\mathcal{U}_\alpha$-supermedian  and we denote by $B_\alpha^A u :=\widehat{R_\alpha^A u}$ its 
$\mathcal{U}_\alpha$--excessive regularization, called the {\it balayage} of $u$ on $A$.

One has 
$B_\alpha^A u = 
R_\alpha^A u$ 
on $E\setminus A$.
If in addition $A$ is finely open then $B_\alpha^A u=R_\alpha^A u\in\mathcal{E}(\mathcal{U_\alpha})$.

\item A set $A\in \mathcal{B}$ is called polar if $B_\alpha^A 1\equiv 0$.
\end{enumerate}
When $\alpha =0$ we drop the index from notation and we simply write $R^A$ and $B^A$, respectively.
\end{defi}

\begin{defi}\label{defi:smallsets} 
Let $m$ be a $\sigma$-finite measure on $E$.
\begin{enumerate}[(i)]
\item A set $A\in \mathcal{B}$ is called
\begin{enumerate}
\item[\scalebox{0.7}{$\circ$}] \say{ $\mathcal{U}$-negligible} if $U_q (1_A)\equiv 0$ for one (hence all) $q> 0$.
\item[\scalebox{0.7}{$\circ$}] \say{ polar}  if $B_\alpha^A 1\equiv 0$ for some $\alpha>0$.
\item[\scalebox{0.7}{$\circ$}] \say{ $m$-polar}  if $B_\alpha^A 1=0$ $m$-a.e. for some $\alpha>0$. 
\item[\scalebox{0.7}{$\circ$}] \say{ $m$-inessential}  provided that it is $m$-negligible and $R_\alpha^{A} 1(x)=0$ for all $x\in E\setminus A$, for some $\alpha>0$.
\end{enumerate}
\item A property is said to hold $m${\it -quasi-everywhere} w.r.t. $\mathcal{U}$ (resp. $\mathcal{U}$-a.e.), if there exists an $m$-inessential (resp. a $\mathcal{U}$-negligible) set $N$ such that the property holds for all $x\in E\setminus N$; on short, we write $m${\it -q.e.} instead of $m$-quasi-everywhere.
\end{enumerate}
\end{defi}

\begin{rem} \label{rem:polar_characterization}
\begin{enumerate}
    \item[(i)] It is well known that if $V\in \mathcal{E}(\mathcal{U}_\alpha)$ for some $\alpha\geq 0$ such that $U_1\left(1_{[V=\infty]}\right)\equiv 0$, then the set $[V=\infty]$ is polar.
    \item[(ii)] We have the following probabilistic characterization due to G.A. Hunt holds (see e.g. \cite{DeMe78}): If $X$ is a right process with resolvent $\mathcal{U}$, then for all $\alpha>0$, $u\in \mathcal{E}(\mathcal{U}_\alpha)$, $A\in\mathcal{B}$,  and $x\in E$
    \begin{equation} \label{Hunt-th}
        R_\alpha^A u(x)
        =\mathbb{E}_x\{e^{-\alpha D_A} u (X(D_A))\}\, ,\, 
        B_\alpha^A u(x) 
        =\mathbb{E}_x\{e^{-\alpha T_A} u (X(T_A))\},   
    \end{equation}
    where 
    $D_A:=\inf \{ t\geq 0 :  X(t)\in A \}$ is the {\rm entry time} of $A$ and 
    $T_A:=\inf \{ t>0 :  X(t)\in A \}$ is the {\rm hitting time} of $A$. 
    In particular, $A$ is polar if and only if $\mathbb{P}_{x}(T_A<\infty)=0$ for all $x\in E$.
    Furthermore, $A$ is $m$-polar if and only if $\mathbb{P}_{x}(T_A<\infty)=0$ $m$-a.e. $x\in E$.

    A similar interpretation holds for $m$-inessential sets: $A\in \mathcal{B}$ is $m$-inessential if and only if $m(A)=0$ and $\mathbb{P}_x(T_A<\infty)=0$ for every $x\in E\setminus A$.
\end{enumerate}
\end{rem}

\begin{rem}[Probabilistic description of finely open sets]\label{rem:finely open prob}
    Recall that a set $D\in \mathcal{B}$ is finely open if and only if 
    \begin{equation*}
        \mathbb{P}_x(T_{E\setminus D}>0)=1,\quad x\in D;
    \end{equation*}
    see, e.g. \cite[p. 52-53]{Sh88}.
\end{rem}

\begin{defi}[The kernels $R^A_\alpha$ and $B^A_\alpha$]\label{defi:balayage_kernel}
Let $A\in \cb$. 
We consider the extension of the reduced function, 
$\mathcal{E}(\cu_\alpha)\ni u\mapsto R^A_\alpha u$ to a sub-Markovian kernel on $\cb^u$ and clearly, by  (\ref{Hunt-th})
we have
\begin{equation*}
R_\alpha^A f(x)=\mathbb{E}_x\{e^{-\alpha D_A} f (X(D_A)), \quad f\in p\cb^u.
\end{equation*}
Analogously, $B_\alpha^A$ is extended to a sub-Markovian kernel on $\cb^u$.
\end{defi} 

\begin{rem} \label{rem 2.2} 
For the reader convenience, let us recall several potential theoretic facts;
cf. \cite{BeBo04a}, \cite{BeRo11a}, and \cite{BeCoRo11}. 
\begin{enumerate}[(i)]
\item If $A \in \mathcal{B}$ is finely open and $\mathcal{U}$-negligible, then $A=\emptyset$.
\item If $m$ is a $\sigma$-finite measure on $E$ such that 
$m(A)=0$ implies $U_1(1_A)=0$ $m$-a.e. for all $A\in \mathcal{B}$, then any finely open and $m$-negligible set $A\in \mathcal{B}$ is $m$-polar. 
\item Any $m$-inessential set is $m$-polar and $m$-negligible. 
The following converse holds (see e.g. \cite{BeBo04a}, page 168): 
{\rm A set which is $m$-polar and $m$-negligible is the subset of an $m$-inessential set.} 
For the reader's convenience we present here a sketch of the proof of this sentence. 
Let $A\in \mathcal{B}$. 
Assume that $A$ is
$m$-polar and $m$-negligible, or equivalently,
 $m(R^A_\alpha 1)=0$, 
since $R^A_\alpha 1=B^A_\alpha 1$ on $E\setminus A$.
We may suppose that $m$ is a finite measure and further we argue as in the first part of the proof 
of Theorem 1.7.29 from \cite{BeBo04a}.
Let $(v_n)_n$ be a decreasing sequence of bounded 
$\mathcal{U}_\alpha$-excessive functions such that
$v_n\geq R^A_\alpha 1$ for all $n$ and
$\inf_n v_n= R^A_\alpha 1$ $m$-a.e.
Let $A_o:= \{\inf_n v_n >0\}$.
Then $A\subset A_o$ and 
$R^{A_o}_\alpha 1=0$ on $E\setminus A_o$.
Because $m(\inf_n v_n)=0$ it follows that $m(A_o)=0$, hence $A_o$ is $m$-inessential.  

\item If $u\in b\mathcal{B}$ and $v$ is $\mathcal{B}$-measurable such that $v=u$ q.e., then $B^A_{1}v$ is well defined and equal to $B^A_{1}u$ $m$-a.e.
\end{enumerate}
\end{rem}

The following two results generalize several classical tools due to Dynkin \cite[Chapter XII, Sections 4 \and 5]{Dy65II} that concern the balayage operator and are key ingredients in solving the {\it stochastic parabolic Dirichlet problem} in \Cref{ss:parabolic Dirichlet}.

\begin{lem}  \label{lem:B_is_finely_feller}
Let $X$ be a right process on $E$, $u\in b\mathcal{B}$, $A \in \mathcal{B}$ and $\alpha\geq 0$. 
The following assertions hold:
\begin{enumerate}
    \item[(i)] The function $B_{\alpha}^A u$ is finely continuous on $({E\setminus A})^{\circ_f}$, where $B^{\circ_f}$ denotes the interior of a set $B\subset E$ with respect to the fine topology.
    \item[(ii)] If $u$ is finely continuous then $B_{\alpha}^A u$ is finely continuous.
\end{enumerate}
\end{lem}

\begin{proof}
Let $A \in \mathcal{B}$ and set $f := B_{\alpha}^A u $.
In order to show that $f$ is finely continuous (resp. on $({E\setminus A})^{\circ_f}$) it is sufficient to prove that if 
$\varepsilon > 0$ then $x$ is irregular for 
$V := f^{-1}([f(x) + \varepsilon, \infty))$ and for $ W:=f^{-1}((-\infty, f(x) - \varepsilon])$, for every $x\in E$ (resp. $x\in ({E\setminus A})^{\circ_f}$). 
We treat only the case of $V$, the one for $W$ being similar.
Fix $x\in E$.
Let $(V_n)_n$ be an increasing sequence of finely closed sets such that 
$T_{V_n} \searrow T_V \; \mathbb{P}_x$-a.s., which exists by e.g. \cite[Chapter I, Exercise (10.30)]{Sh88} \cite[Theorem 10.19]{BlGe68}.
By the zero-one law (\cite[Chapter I, Corollary (3.11)]{Sh88} or \cite[Proposition 5.17]{BlGe68}), $x$ is either regular for $V$, i.e. $\mathbb{P}_x([T_V = 0])=1$, or irregular for $V$, i.e. $\mathbb{P}_x([T_V = 0])=0$. 

Assume that $x$ is regular.
By the strong Markov property and recalling that $u(\Delta)=0$ we have
\begin{align*}
f(x) + \varepsilon
&\leq\mathbb{E}_x\{f(X(T_{V_n}))\} \\
&= \mathbb{E}_x\{\mathbb{E}_{X({T_{V_n}})}\{e^{-\alpha T_{A}}u(X_{T_A});T_A<\infty\}\}\\
&=\mathbb{E}_x\{\mathbb{E}_x\{e^{-\alpha T_A\circ \theta_{T_{V_n}}}u(X({T_{V_n}+T_A\circ \theta_{T_{V_n}}})) ;T_A\circ \theta_{T_{V_n}}<\infty \; | \; \mathcal{F}_{T_{V_n}}\}\} \\
&=\mathbb{E}_x\{\mathbb{E}_x\{e^{-\alpha T_A\circ \theta_{T_{V_n}}}u(X({T_{V_n}+T_A\circ \theta_{T_{V_n}}})) ;T_A\circ \theta_{T_{V_n}}<\infty \; | \; \mathcal{F}_{T_{V_n}}\}\} \\
&=\mathbb{E}_x\{e^{-\alpha T_A\circ \theta_{T_{V_n}}}u(X({T_{V_n}+T_A\circ \theta_{T_{V_n}}}));T_A\circ \theta_{T_{V_n}}<\infty \}.
\end{align*}
Also, note that for every $x\in E$,
\begin{equation}\label{eq:T_nto0}
  \mathop{\lim}\limits_{n}T_{V_n}=T_V=0, \quad  \mathbb{P}_{x}\mbox{-a.s.},
\end{equation}
and also that $T_{V_n}+T_A\circ \theta_{T_{V_n}}\searrow_n T_A$ $\mathbb{P}_x$-a.s., see e.g. \cite[p. 65]{Sh88} or \cite[(10.3)]{BlGe68}.

Now, if $u$ is finely continuous, by \Cref{thm 4.6} it follows that the trajectories $t \mapsto u(X_t)$ are right continuous $\mathbb{P}_{x}$-a.s for all $x\in E$.
Thus, by dominated convergence, $\lim\limits_n\mathbb{E}_x\{f(X(T_{V_n}))\}=f(x)$, which is a contradiction with the inequality from above, hence (ii) is proved.

Let us now prove (i). 
To this end, let $x\in ({E\setminus A})^{\circ_f}$, so that by \Cref{rem:finely open prob} we have $\mathbb{P}_x(T_A>0)=1$.
Using again that $T_A$ is a terminal time (\cite[p. 65]{Sh88}) and \eqref{eq:T_nto0}, we get 
\begin{equation*}
    \lim\limits_nu(X({T_{V_n}+T_A\circ \theta_{T_{V_n}}}))=u(X(T_A)), \quad \mathbb{P}_{x}\mbox{-a.s.}
\end{equation*}
Now, by the same argument as in the proof of (ii) we arrive immediately at a contradiction.
\end{proof}

Let $(E,\mathcal{B})$ be endowed with a topology $\tau$ such that $\sigma(\tau)=\mathcal{B}$.
Recall that a right process $X$ on $E$ is a {\it Hunt process} if for $\mu\in \mathcal{P}(E)$, $X$ has left $\tau$-limits in $E$ $\mathbb{P}_\mu$-a.s., and for every every sequence of $\mathcal{F}(t)$-stopping times $(T_n)_{n\geq 1}$ such that $T_n \nearrow_n T$ we have
\begin{equation*} 
 \tau-\lim_{n}X({T_n}) = X(T) \mbox{ on } [T<\infty] \quad \mathbb{P}_\mu\mbox{-a.s.}
\end{equation*}

\begin{prop}\label{prop:B-boundary}
Let $A\in \mathcal{B}$, $u:E\rightarrow \mathbb{R}$ be $\mathcal{B}$-measurable, $\alpha\geq 0$, and $x\in E$ such that
\begin{equation}\label{eq:resolutive}
    \mathbb{E}_x\left\{e^{-\alpha T_A}|u|(X(T_A))\right\}<\infty.
\end{equation}
Further, consider the process
\begin{equation*}
    M=(M(t))_{t\geq 0}:=\left(e^{-\alpha (t\wedge T_A)}B_\alpha^Au(X(t\wedge T_A))\right)_{t\geq 0}.
\end{equation*}
The following assertions hold:
\begin{enumerate}
    \item[(i)] $M$ is an $\left(\mathcal{F}(t\wedge T_A)\right)$-martingale under $\mathbb{P}_x$ which is closed by $e^{-\alpha T_A}u(X(T_A))$.
    \item[(ii)] If $E\setminus A$ is finely open and $x\in E\setminus A$ then $M$ is c\`adl\`ag $\mathbb{P}_x$-a.s.
    \item[(iii)] If $u$ is finely continuous, then $M$ is c\`adl\`ag $\mathbb{P}_x$-a.s.
    \item[(iv)] If $X$ is a Hunt process with respect to a given topology $\tau$ on $E$ such that $\sigma(\tau)=\mathcal{B}$, then for every increasing sequence of $\left(\mathcal{F}(t)\right)$-stopping times $(\tau_k)_{k\geq 1}$ such that $\lim\limits_k \tau_k=T_A$ $\mathbb{P}_x$-a.s. we have
    \begin{equation*}
        \lim\limits_{k\to \infty} M(\tau_k)=e^{-\alpha  T_A}u(X( T_A))\quad \mathbb{P}_x\mbox{-a.s. and in } L^1(\mathbb{P}_x). 
    \end{equation*}
    \item[(v)] If $X$ is a Hunt process, in addition to \eqref{eq:resolutive} we have that either $u$ is finely continuous or $E\setminus A$ is finely open and $x\in E\setminus A$, whilst $T_A$ is predictable under $\mathbb{P}_x$, then
    \begin{equation}\label{eq:boundary convergence}
        \lim\limits_{t\nearrow T_A} M(t)=e^{-\alpha  T_A}B_\alpha^Au(X( T_A))\quad \mathbb{P}_x\mbox{-a.s.}
    \end{equation}
\end{enumerate}
\end{prop}
\begin{proof}
(i). First of all, recall that by \Cref{rem:strong Markov}, or by \cite[Corollary (8.6)]{BlGe68}, for all $x\in E$, $Y\in b\mathcal{F}$, and $\tau$ and $\left(\mathcal{F}(t)\right)$-stopping time, we have
\begin{equation}\label{eq:strong Markov resolutive}
    \mathbb{E}_x\left\{Y\circ \theta(\tau) | \mathcal{F}(\tau) \right\}
    =E^{X(\tau)}\left\{Y\right\}\quad \mathbb{P}_x\mbox{-a.s.} 
\end{equation}
As it easily follows by dominated convergence, \eqref{eq:strong Markov resolutive} also holds for $x\in E$ and $Y\in \mathcal{F}$ for which
\begin{equation*}
    \mathbb{E}_x\left\{|Y|\circ \theta(\tau)\right\}<\infty.
\end{equation*}
Now, let us notice that since \eqref{eq:resolutive} is in force and using the fact that $T_A$ is a terminal stopping time, we get
\begin{align*}
\mathbb{E}_x\left\{\left[e^{-\alpha T_A}|u|(X(T_A))\right]\circ \theta(t\wedge T_A)\right\}
&=\mathbb{E}_x\left\{e^{-\alpha (T_A\circ \theta(t\wedge T_A))}|u|(X(t\wedge T_A+T_A\circ \theta(t\wedge T_A)))\circ \theta(t\wedge T_A)\right\}\\
&=\mathbb{E}_x\left\{e^{-\alpha (T_A-t\wedge T_A)}|u|(X(T_A)\right\}\\
&\leq e^{\alpha t}\mathbb{E}_x\left\{e^{-\alpha T_A}|u|(X(T_A)\right\}<\infty.
\end{align*}
Consequently, by choosing $Y:=e^{-\alpha T_A}u(X(T_A))$ and $\tau:=t\wedge T_A$ we can apply \eqref{eq:strong Markov resolutive} to deduce that
\begin{align*}
     \mathbb{E}_x\left\{\left[e^{-\alpha T_A}u(X(T_A))\right]\circ \theta(t\wedge T_A) | \mathcal{F}(t\wedge T_A) \right\}
    &=E^{X(t\wedge T_A)}\left\{e^{-\alpha T_A}u(X(T_A))\right\}\\
    &=B_\alpha^Au(X(t\wedge T_A))\quad \mathbb{P}_x\mbox{-a.s.} 
\end{align*}
But, using again that $T_A$ is a terminal time and by similar computations as above, the left hand side of the above inequality also satisfies
\begin{equation*}
\mathbb{E}_x\left\{\left[e^{-\alpha T_A}u(X(T_A))\right]\circ \theta(t\wedge T_A) | \mathcal{F}(t\wedge T_A) \right\}
=e^{\alpha \left(t\wedge T_A\right)}\mathbb{E}_x\left\{e^{-\alpha T_A}u(X(T_A)) | \mathcal{F}(t\wedge T_A) \right\}.
\end{equation*}
Thus,
\begin{equation}\label{eq:M is closed}
    M(t)
    =e^{-\alpha (t\wedge T_A)}B_\alpha^Au(X(t\wedge T_A))
    =\mathbb{E}_x\left\{e^{-\alpha T_A}u(X(T_A)) | \mathcal{F}(t\wedge T_A) \right\},\quad t\geq 0,
\end{equation}
which proves that $M$ is a closed (hence uniformly integrable) $\left(\mathcal{F}(t\wedge T_A)\right)$-martingale under $\mathbb{P}_x$.

\medskip
(ii)\&(iii). Clearly, it is enough to consider the case $u\in \mathcal{B}$ satisfies $u\geq 0$.
Moreover, if we consider $u_n:=u\wedge n, n\geq 1$, by the first part of the proof we get that the processes
\begin{equation*}
    M_n=\left(M_n(t)\right)_{t\geq 0}:=\left(e^{-\alpha (t\wedge T_A)}B_\alpha^Au_n(X(t\wedge T_A))\right)_{t\geq 0}, \quad n\geq 1
\end{equation*}
are $\left(\mathcal{F}(t\wedge T_A)\right)$-martingales under $\mathbb{P}_x$.
Moreover, under the hypotheses in (ii) or (iii), by \Cref{lem:B_is_finely_feller} and \Cref{thm 4.6}-\Cref{rem:finely continuous on D}, they are all right-continuous, hence c\`adl\`ag $\mathbb{P}_x$-a.s.
Since $\lim\limits_nu_n=u$ increasingly, we get that by monotone convergence that $B_\alpha^A u_n\nearrow_n B_\alpha^A u$ and thus
\begin{equation*}
    M_n(t)\nearrow_n M(t)\quad \mbox{for every } t\geq 0 \quad\mathbb{P}_x\mbox{-a.s.} 
\end{equation*}
The fact that $M$ is c\`adl\`ag $\mathbb{P}_x$-a.s. now follows from a standard result of convergence for increasing families of c\`adl\`ag supermartingales, see e.g. \cite[Chapter VI, Theorem 18]{DeMe78} or \cite[Part II, Chapter IV, Section 4]{Do84}.

\medskip
(iv). As in the proof of (i), by the strong Markov property we deduce that $\left(M(\tau_k)\right)_{k\geq 1}$ is a (discrete time) $\left(\mathcal{F}(\tau_k)\right)$-martingale closed by $e^{-\alpha T_A}u(X(T_A))$ under $\mathbb{P}_x$.
Since $X$ is a Hunt process and $\tau_k\nearrow_kT_A$ $\mathbb{P}_x$-a.s., by arguing as in \cite[(4.1)-(4.2)]{BlGe68}, we get that $\mathbb{P}_x$-a.s.
\begin{equation*}
    X(T_A)=\lim\limits_k X(\tau_k)1_{[T_A<\infty]}+\Delta1_{[T_A=\infty]},\quad [T_A<\infty]=\mathop{\cup}\limits_{k\geq 1}\mathop{\cap}\limits_{n\geq k} [\tau_n\leq k],
\end{equation*}
and thus 
\begin{equation*}
    X(T_A)\in \sigma\left(\bigcup\limits_{k\geq 1} \mathcal{F}(\tau_k)\right),
\end{equation*}
so now (iv) follows by L\'evy's convergence theorem.

\medskip
(v). Let the assumptions in (v) be satisfied.
Then by (i), (ii), and (iii), we have that $M$ is a c\`adl\`ag $\left(\mathcal{F}(t\wedge T_A)\right)$-martingale under $\mathbb{P}_x$ which is closed by $M(\infty):=e^{-\alpha T_A}u(X(T_A))$. 
Let $(\tau_k)_{k\geq 1}$ be an increasing sequence of finite $\left(\mathcal{F}(t)\right)$-stopping times such that $\tau_k<T_A, k\geq 1$ on $[T_A>0]$ and $\lim\limits\tau_k=T_A$ $\mathbb{P}_x$-a.s.
To prove (iv) it is clearly enough to show
\begin{equation*}
    \lim\limits_k\sup\limits_{t\in[\tau_k,\infty)}\left|M(t)-M(\infty)\right|=0 \quad \mathbb{P}_x\mbox{-a.s.}
\end{equation*}
Note that the above limit is decreasing with respect to $k$, so the required a.s. convergence is in fact equivalent with the convergence in probability under $\mathbb{P}_x$. 
Thus, by introducing the processes
\begin{equation*}
    M_k(t):=M(\tau_k+t)-M(\tau_k),\quad t\geq 0, \quad k\geq 1,
\end{equation*}
the previous convergence boils down to showing
\begin{equation*}
    \lim\limits_k\sup\limits_{t\geq 0}\left|M_k(t)-M_k(\infty)\right|=0 \quad \mbox{in probability with respect to }\mathbb{P}_x.
\end{equation*}
By (iv) we have that $\lim\limits_kM_k(\infty)=0$ $\mathbb{P}_x$-a.s. and in $L^1(\mathbb{P}_x)$, hence the above convergence, thus it remains to show that
\begin{equation*}
    \lim\limits_k\sup\limits_{t\geq 0}\left|M_k(t)\right|=0 \quad \mbox{in probability with respect to }\mathbb{P}_x.
\end{equation*}
Now, let us note that by Doob's stopping theorem we have that for every $k\geq 1$ the process $(M_k(t))_{t\geq 0}$ is a c\`adl\`ag $\left(\mathcal{F}((\tau_k+t)\wedge T_A)\right)$-martingale under $\mathbb{P}_x$ which is closed by $M_k(\infty)$.
Consequently, by Doob's maximal inequality for (continuous-time) right-continuous martingales, and using again (iv), we get
\begin{equation*}
    \mathbb{P}_x\left(\sup\limits_{t\geq 0}\left|M_k(t)\right|\geq \varepsilon\right)
    \leq\frac{1}{\varepsilon}\mathbb{E}_x\left(\left|M_k(\infty)\right|\right)\mathop{\rightarrow}\limits_k 0,
\end{equation*}
which finishes the proof.
\end{proof}

\paragraph{Existence of a right process on a larger space.}

Condition $\mathbf{(H_{\mathcal{C}})}$, although necessary, is not sufficient to guarantee the existence of a right process with the given resolvent $\mathcal{U}$, even if the resolvent is strong Feller; see \cite{BeCiRo24}. 
However, under $\mathbf{(H_{\mathcal{C}})}$, there is always a larger space on which an associated right process exists.

We denote by $Exc(\mathcal{U}_\beta)$ the set of all $\mathcal{U}_\beta$-excessive measures:
$\xi \in Exc(\mathcal{U}_\beta)$ if and only if $\xi$ is a $\sigma$-finite measure on $E$ and $\xi \circ \alpha U_{\beta+\alpha} \leq \xi$ for all $\alpha >0$.

\begin{defi}[The saturation of $(E,\mathcal{U})$ w.r.t. $\mathcal{U}_\beta$] \label{defi:saturation}
Let $\beta >0$.
\begin{enumerate}[(i)]
\item The {\it energy functional} associated with $\mathcal{U}_\beta$ is $L^{\beta}: Exc(\mathcal{U}_\beta)\times \mathcal{E}(\mathcal{U}_\beta) \rightarrow \overline{\mathbb{R}}_+$ given by
\begin{align*}
L^{\beta}(\xi,v):&=\sup\{\mu(v) \; : \; \mu \mbox{ is a } \sigma\mbox{- finite measure, } \mu \circ U_\beta \leq \xi\}\\
&=\sup\{\xi(f) \; : \; f\geq 0 \mbox{ is a } \mathcal{B}\mbox{-measurable function such that } U_\beta f \leq v\}.
\end{align*}
\item The {\it saturation} of $E$ (with respect to $\mathcal{U}_\beta$) is the set $E_1$ of all extreme points of the set $\{\xi \in Exc(\mathcal{U}_\beta)\; : \; L^{\beta}(\xi,1)=1\}$.
\item The map $E\ni x \mapsto \delta_x \circ U_\beta \in Exc(\mathcal{U}_\beta)$ is a measurable embedding of $E$ into $E_1$ and 
every $\mathcal{U}_\beta$-excessive function $v$ has an extension $v_1$ to $E_1$, defined as $v_1(\xi):=L^{\beta}(\xi,v)$.
The set $E_1$ is endowed with the $\sigma$-algebra $\mathcal{B}_1$ generated by the family $\{v_1: \; v\in \mathcal{E}(\mathcal{U}_\beta)\}$.
In addition, as in \cite[Sections 1.1 and 1.2]{BeBoRo06}, there exists a unique resolvent of kernels $\mathcal{U}^{1}=(U^{1}_\alpha)_{\alpha>0}$ on $(E_1, \mathcal{B}_1)$ 
which is an {\it extension} of $\mathcal{U}$ in the sense that
\begin{enumerate}
    \item[(iii.1)] $U^1{1_{E_1\setminus E}}=0$ on $E_1$,
    \item[(iii.2)] $(U^1f)|_E=U(f|_E)$ for all $f \in b\mathcal{B}_1$.
\end{enumerate}
More precisely, it is given by
\begin{align} \label{eq 4.1}
U^{1}_\alpha f(\xi)&=L^{\beta}(\xi, U_\alpha (f|_E))\nonumber\\
&=L^{\beta}(\xi,U_\beta(f|_E+(\beta-\alpha)U_\alpha))\\
&=\xi(f+(\beta-\alpha)U_\alpha f) 
\mbox{ for all } f \in bp\mathcal{B}_1, \xi\in E_1, \alpha >0.
\end{align}
\end{enumerate}
\end{defi}

\begin{rem}\label{rem:saturation}
Note that $(E_1,\mathcal{B}_1)$ is a Lusin measurable space, the map $x \mapsto\delta_x \circ U_\beta$ identifies $E$ with a subset of $E_1$, $E\in \mathcal{B}_1$,  and $\mathcal{B}=\mathcal{B}_1|_E$. 
Furthermore, $E$ is dense in $E_1$ with respect to the fine topology on $E_1$ associated with $\mathcal{U}^1$. 
Hence, if $\alpha>0$ and $v$ is a $\mathcal{U}_\alpha$-excessive function then its $\mathcal{U}^1_\alpha$-excessive extension $v_1$ given by \Cref{defi:saturation} is the unique extension of $v$ from $E$ to $E_1$ by fine continuity.
\end{rem}

The existence of a right process on a larger space, more precisely on $E_1$, is given by the following result, for which we refer to \cite[(2.3)]{BeRo11a}, in \cite[Sections 1.7 and 1.8]{BeBo04a},   \cite[Theorem 1.3]{BeBoRo06}, and \cite[Section 3]{BeBoRo06a}; see also
\cite{St89}.

\begin{thm} \label{thm:saturation}
There is always a right process $X^1$ on the saturation $(E_1,\mathcal{B}_1)$, associated with $\mathcal{U}^{1}$. Moreover, the following assertions are equivalent:
\begin{enumerate}[(i)]
\item There exists a right process on $E$ associated with $\mathcal{U}$.
\item The set $E_1\setminus E$ is polar (w.r.t. $\mathcal{U}^1$).
\end{enumerate}
\end{thm}

\subsection{Choquet capacity}
Let $X=(\Omega, \mathcal{F}, \mathcal{F}(t) , X(t), \theta(t) , \mathbb{P}_x), x\in E$ be a right process on a Lusin measurable space $(E,\mathcal{B})$, and $\mathcal{T}$ be a natural topology on $E$.
It is well known that for every $\nu\in \mathcal{P}(E)$, $\alpha>0$ and $f\in b\mathcal{B}, f>0$, the mapping
\begin{equation}\label{eq:capacity}
     \mathbb{R}^d\supset A\longmapsto {\sf Cap}^\alpha_{\nu}(A):= \inf\left\{\nu\left(R^G_\alpha(U_\alpha f)\right) : A\subset G, \; G\in \mathcal{T}\right\},
\end{equation}
is a Choquet capacity on $(E,\mathcal{T})$. 
Moreover, if $X$ has left limits in $(E,\mathcal{T})$ $\mathbb{P}_\nu$-a.s., then the capacity ${\sf Cap}^\alpha_{\nu}$ is tight; cf. 
\cite{LyRo92}, \cite{BeBo05}, and \cite{BeRo11b}. 

For convenience, let us briefly recall the notion of capacity (for details see e.g.   \cite[Ch. III]{DeMe78} and \cite[Appendix A.1]{BeBo04a}):
\begin{defi}[Choquet capacity] \label{defi:capacity}
    Let $\mathcal{K}$ be the family of al $\mathcal{T}$-compact subsets of $E$.
A map 
\begin{equation*}
    {\sf Cap}: \mathcal{P}(E) \rightarrow \overline{\mathbb{R}}_+, \quad \mathcal{P}(E):=\{A : A\subset E\},
\end{equation*}
is called {\rm Choquet capacity} on $E$ provided that:
\begin{enumerate}
    \item[(i)] ${\sf Cap}(K) < \infty$ for every $K\in \mathcal{K}$;
    \item[(ii)] ${\sf Cap}(A_1)\leq 
{\sf Cap} (A_2)$  if $A_1\subset A_2\in \mathcal{P}(E)$;
    \item[(iii)] If $(A_n)_{n\geq 1} \subset \mathcal{P}(E)$ is such that $A_n\nearrow_n A \in \mathcal{P}(E)$, then
${\sf Cap} (A_n)\nearrow {\sf Cap} (A)$;
    \item[(iv)] If $(K_n)_{n\geq 1} \subset \mathcal{K}$ is such that $K_n\searrow_n A$, then ${\sf Cap} (K_n)\searrow {\sf Cap} (A)$.
\end{enumerate}
Also, recall that a Choquet capacity {\sf Cap} is {\rm tight} provided that there exists an increasing sequence $(K_n)_n\subset \mathcal{K}$ such that
\begin{equation*}
   \inf\limits_{n} {\sf Cap}(E\setminus K_n)=0.
\end{equation*}
\end{defi}

\subsection{Basic operations on right processes}
\subsubsection{Trivial modification and restriction of the resolvent}
Let $M\in\cb$ be such that $U_\alpha(1_M)=0$ on
$E\setminus M$ for one (and therefore for all) $\alpha>0$.
Following \cite{BeBoRo06a} and 
\cite{BeTr11} we can define the trivial modification of $\cu$ on $M$ and the restriction of $\cu$ to $E\setminus M$ as  follows.

\begin{defi}[Trivial modification]\label{defi:trivial modification}
For all $\alpha>0$ we define the kernel
\begin{equation}\label{trivial-mod}
  U'_\alpha f
  = 1_{E\setminus M} U_\alpha f
    + \frac{1}{\alpha} 1_M f\;,
  \quad f\in p\cb.
\end{equation}
The family $\cu'=(U'_\alpha)_{\alpha>0}$ is also a
sub-Markovian resolvent of kernels on $(E,\cb)$ 
and  there exists a cone $\mathcal{C}$ such that $\mathbf{(H_{\mathcal{C}})}$ is fulfilled.  
$\cu'$ is called the \emph{trivial modification} of the resolvent $\cu$
on $M$. 
\end{defi}


\begin{defi}[Restriction]\label{defi:restriction}
The family $\cu|_{E\setminus M}=$
$( U_\alpha|_{E\setminus M} )_{\alpha>0}$
is a sub-Markovian resolvent of kernels on
$( E\setminus M, \cb|_{E\setminus M})$,
called the \emph{restriction} of $\cu$ to $E\setminus M$.
Moreover, $\cu|_{E\setminus M}$
satisfies condition 
$\mathbf{(H_{\mathcal{C}})}$ 
on the measurable space
$( E\setminus M, \cb|_{E\setminus M})$.
\end{defi}

A function $u \in p\cb|_{E\setminus M}$ is
$\cu_{\beta}|_{E\setminus M}$-excessive  if and only if there exists a function
$\bar{u}$ which is $\cu_{\beta}$-excessive, such that $u=\bar{u}|_{E\setminus M}$.
Consequently, the fine topology on $E\setminus M$ with respect to 
the restriction of $\cu$ to $E\setminus M$  is the trace on $E\setminus M$ of
the fine topology on $E$ with respect to $\cu$. 
A function $v\in p\cb$ will be $\cu'_\beta$-excessive if and only
if $v|_{E\setminus M}$ is $\cu_\beta|_{E\setminus M}$-excessive.

\begin{rem}\label{rem:raycones} 
If $M\in\cb$ is a set such that  $U_\alpha(1_M)=0$ on $E\setminus M$
then the following assertions hold.\\
$(i)$ Let $\crr$ be a Ray cone associated with  $\cu_\beta$ and
$\cu'$ be the trivial modification of $\cu$ on $M$.
Then there exists a Ray cone $\crr'$ with respect to $\cu'_\beta$ such that
$\crr\subset \crr'$.\\
$(ii)$ If ${\crr}^\circ$ is  a Ray cone associated with  $\cu_\beta|_{E\setminus M}$
then there exists a Ray cone $\crr$ with respect to $\cu_\beta$ such that
${\crr}|_{E \setminus M}={\crr}^\circ$.
\end{rem}

\subsubsection{Restriction, trivial extension, and trivial modification of the process}\label{ss:restriction_process}  

\noindent
{\bf Strongly supermedian functions.} Recall first that a positive numerical function $f$ on $E$  is called {\it nearly Borel}  
provided that for any finite measure $\lambda$ on $E$ there exist two positive 
numerical Borelian functions $g$ and  $h$ on $E$ such that $g \leq f \leq h$ and the
set $[g < h]$ is $\lambda$-polar
and $\lambda$-negligible.
It is known that any $U_\alpha$-excessive
function
on $E$, $\alpha>0$, is nearly Borel.
Let $\mathcal{B}^n$ denote the $\sigma$-algebra of all nearly Borel sets.

A positive numerical function $f$ on $E$ is called 
{\it strongly supermedian} ({\it with respect
to $\mathcal{U}_\alpha$}) if it is nearly Borel and for any two finite measures $\mu$ and  $\nu$ on $E$ we have
$$
\mu\circ U_\alpha \leq 
\nu\circ U_\alpha \Longrightarrow 
\mu(f)\leq \nu(f).
$$

We denote by $\mathcal{S}_\alpha$ the set of all strongly supermedian functions on $E$ ({\it with respect
to $\mathcal{U}_\alpha$}).
Further, we present several basic properties of the strongly supermedian functions, cf. e.g.
\cite{BeBo99}, \cite{BeBo04a}, \cite{BeBo04b},  
and 
\cite{BeDeLu15}.

$\mathcal{S}_\alpha$ is a convex cone. 
By \cite{BeBo99}, assertion (1) of the Remark at page 369, it follows that 
a function $f\in p\mathcal{B}^n$ is strongly supermedian (with respect
to $\mathcal{U}_\alpha$) if and only if
$R^M_\alpha f\leq f$ for any $M\in\cb$.
Let $(v_n)_n$ be  a sequence in
$\mathcal{S}_\alpha$, then $\inf_n v_n\in \mathcal{S}_\alpha$. 
If in addition $(v_n)_n$ is increasing then $\sup_n v_n\in \mathcal{S}_\alpha$.
For all $v\in \mathcal{S}_\alpha$ we have
$v=\inf\{ u:  u\in \mathcal{E(\mathcal{U}_\alpha}), u\geq v\}$.
In particular, every strongly supermedian function is finely upper semicontinuous.

If $u\in \mathcal{E}(\mathcal{U}_\alpha)$ and $A \in \mathcal{B}$, then $R^A_\alpha u$, the $\alpha$-order reduced function of $u$ on $A$,  is a strongly supermedian function, particularly, it is nearly Borel measurable. 
More generally, the following assertions are equivalent for a 
positive  nearly Borel measurable function $v$ on $E$: 
\begin{enumerate}
    \item[(i)] The function $v$ is strongly supermedian with respect to $\cu_\alpha$.
    \item[(ii)] There exists a family $\cg$ of  $\cu_\alpha$-excessive functions such that $v=\inf \cg$.
    \item[(iii)] We have $v=\inf \{ u\in \ce(\cu_\alpha):\, u{\geqslant}v \}.$
\end{enumerate}

Following \cite{BeRo11a}, we set
\begin{equation}\label{eq: absorbing}
\ca(\cu):=\left\{A\in{\cb} : {R}^{E \setminus{A}}_{\alpha}1=0 \mbox{ on } A \mbox{ for some } \alpha>0\right\}    
\end{equation}
An element $A\in \ca(\cu)$ is called an {\it absorbing set}. 
An absorbing set is finely open and in what follows we shall recall several key facts concerning the restriction of a right process to an absorbing set.

\begin{rem} \label{rem8.25}
By \cite[A.1.2]{BeDeLu15} the following assertions are equivalent for a set $A\in \cb$.
\begin{enumerate}
    \item[(i)] The set $A$ belongs to $\ca(\cu)$.
    \item[(ii)] We have $\mathbb{P}_x$-a.s. $D_{E\setminus A}= \infty$ for every $x\in A$.
    \item[(iii)] There exists a strongly supermedian function $v$ with respect to  $\cu_\alpha$ such that $A=[v=0]$.
\end{enumerate}
\end{rem}

\begin{rem}\label{rem:exc_absorbing}
If $v\in \mathcal{S}_\alpha\cap p\cb$, then
\begin{equation*}
    [v<\infty],\;[v=0] \in \ca(\cu).
\end{equation*}
The fact that the set $[v=0]$ belongs 
to $\ca(\cu)$ follows from the above considerations. 
To show that $[v<\infty]\in \ca(\cu)$, 
we complete the arguments from \cite{BeRo11a} and A.1.4 from \cite{BeDeLu15}, where the function $v$ was supposed
to be $\cu_\alpha$-excessive. 
If $n\geq 1$ then we have
$1\leq \frac{1}{n} v$ on $[v= \infty]$. 
Let $u\in \ce(\cu_\alpha)$, $u\geq v$.
Then 
$R^{[v=\infty]}_\alpha 1\leq \frac{1}{n} u$ and consequently 
\begin{equation*}
    R^{[v=\infty]}_\alpha 1\leq v/n
    =\inf \{ u\in \mathcal{E(\mathcal{U}_\alpha}):  u\geq v\}/n.
\end{equation*} 
We conclude that
$R^{[v=\infty]}_\alpha 1=0$ on $[v <\infty]$, 
hence $[v <\infty]\in \ca(\cu)$.
\end{rem}

\begin{rem}\label{rem:absorbing_intersection}
    \begin{enumerate}
        \item[(i)] If $(A_n)_{n\geq 1}\subset \mathcal{A}(\mathcal{U})$, then $\cap_n A_n\in \mathcal{A}(\mathcal{U})$.
        \item[(ii)] A set $A\in \mathcal{B}$ is $\mu$-inessential (see \Cref{defi:smallsets}), if and only if $A\in \mathcal{A}(\mathcal{U})$ and $\mu(A)=0$.
        In particular, a countable intersection of $\mu$-inessential sets is again $\mu$-inessential.
    \end{enumerate}
\end{rem}

If $A\in \ca(\cu)$ then $U_{\beta}(1_{E\setminus{A}})=0$ on $A$
and  therefore we may consider the {restriction}  
$\cu|_A$
{of}  $\cu$ {\rm to} $A$.
In
particular, the resolvent 
$\cu|_A=(U_{\alpha}|_A)_{\alpha>0}$
satisfies condition 
$\mathbf{(H_{\mathcal{C}})}$ 
on the measurable space $(A,\cb|_{A})$ and we already observed that the fine topology on $A$ with respect to 
the restriction of $\cu$ to $A$  is the trace on $A$ of
the fine topology on $E$ with respect to $\cu$. 
If $\mathcal{R}$ is a Ray cone associated with
$\cu_{\beta}$ then $\mathcal{R}|_{A}$
is a Ray cone associated with $\cu_{\beta}|_A$.






\begin{prop}[cf. {\cite[Corollary 3.2]{BeRo11a}}]\label{prop:restriction}
If $\cu$ is the resolvent of  a
right process with state space $E$  and  $A \in \ca(\cu)$,  then the
restriction of $\cu$ to $A$ is the resolvent of a right process
with state space $A$.
\end{prop}


\begin{defi}[Restriction of the process]\label{defi:restr-process}
If $X$ is a right process with resolvent $\mathcal{U}$, and $A\in \mathcal{A}(\mathcal{U})$, then one can construct a right process  $\widetilde{X}$ associated to the restriction of $\mathcal{U}$ to $A$ as in
\Cref{prop:restriction}, by setting
\begin{align*}
\widetilde{\Omega}
&:= \{ \omega \in
\Omega : \, {X}_t(\omega)\in A, \mbox{ for all }0\leq t < \zeta (\omega) \}\\
\widetilde{\mathbb{P}}^x
&:= \mathbb{P}_x|_{\widetilde{\Omega}}, x\in A, 
\quad \widetilde{X}_t(\omega):= {X}_t(\omega), \omega\in
\widetilde{\Omega}.
\end{align*}
The right process $\widetilde{X}$ defined above is called the {\rm restriction of $X$ to $A$}; for more details, see \cite{Sh88}.
\end{defi}

\begin{prop}\label{prop:continuous_modification}
    Let $\left(X, \mathbb{P}_x,x\in E\right)$ be a right process associated to the resolvent $\mathcal{U}$ on the Lusin measurable space $(E,\mathcal{B})$. 
    Further, let $m$ be a $\sigma$-finite measure on $E$, and $\tau$ be a natural topology on $E$.
    Furthermore, suppose that
    \begin{equation*}
        \mathbb{P}_m\left([0,\infty)\ni t\mapsto X(t)\in E \mbox{ is } \tau\mbox{-continuous} \right)=1.
    \end{equation*}
    Then there exists a set $E_0\subset E$ such that $E\setminus E_0$ is $m$-inessential and
    \begin{equation}\label{eq:continuous_modification}
        \mathbb{P}_x\left([0,\infty)\ni t\mapsto X(t)\in E_0 \mbox{ is } \tau\mbox{-continuous} \right)=1 \quad \mbox{for all } x\in E_0.
    \end{equation}
    Consequently, one can consider the restriction $\widetilde{X}$ of $X$ from $E$ to $E_0$, so that $\widetilde{X}$ is a right process with a.s. $\tau$-continuous paths in $E_0$; in particular, $\tau$ is a natural topology on $E_0$.
    Furthermore, the above assertion remains true if instead of "$\tau$-continuous" we use "$\tau$-right-continuous" or "$\tau$-c\`adl\`ag".
\end{prop}
\begin{proof}
    The proof of \eqref{eq:continuous_modification} follows by the same arguments as in \cite[Proposition 5.1]{BeRo11a}, so we skip it.   
    The fact that $\tau$ is a natural topology on $E_0$ follows immediately from \Cref{thm 4.6}, (i).
\end{proof}

\begin{rem}
    We would like to stress that in \Cref{prop:continuous_modification}, the assumption that $\tau$ is a natural topology is crucial. 
    Indeed, the proof of \cite[Proposition 5.1]{BeRo11a}, realies heavily on the fact that $X$ is apriori a.s. right-continuous at $t=0$ with respect to $\tau$, and this is automatically satisfied if $\tau$ is a natural topology, by \Cref{thm 4.6}. 
    In the next result, which is needed in the main body of the paper, we drop the assumption that $\tau$ is a natural topology; however, we need to assume (i) below, instead.
\end{rem}

\begin{prop}\label{prop:continuous_modification_special}
    Let $\left(X, \mathbb{P}_x,x\in E\right)$ be a right process associated to the resolvent $\mathcal{U}$ on the Lusin measurable space $(E,\mathcal{B})$. 
    Further, let $\tau$ be a Lusin topology on $E$ whose Borel $\sigma$-algerba is precisely $\mathcal{B}$.
    Furthermore, suppose that
    \begin{enumerate}
        \item[(i)] For every $x\in E$ we have 
        \begin{equation*}
            \mathbb{P}_x\left((0,\infty)\ni t\mapsto X(t)\in E \mbox{ is } \tau\mbox{-continuous} \right)=1.
        \end{equation*}
        \item[(ii)] $m$ is a $\sigma$-finite measure on $E$ such that
        \begin{equation*}
        \mathbb{P}_m\left([0,\infty)\ni t\mapsto X(t)\in E \mbox{ is } \tau\mbox{-continuous} \right)=1.
        \end{equation*}
    \end{enumerate}
    Further, set
    \begin{equation*}
        E_0:=\left\{x\in E : \mathbb{P}_x\left([0,\infty)\ni t\mapsto X(t)\in E \mbox{ is } \tau\mbox{-continuous} \right)=1\right\}.
    \end{equation*}
    Then $E\setminus E_0$ is $m$-inessential.
    
    Consequently, one can consider the restriction $\widetilde{X}$ of $X$ from $E$ to $E_0$, so that $\widetilde{X}$ is a right process with a.s. $\tau$-continuous paths in $E_0$; in particular, $\tau$ is a natural topology on $E_0$.
    Furthermore, the above assertion remains true if instead of "$\tau$-continuous" we use "$\tau$-right-continuous" or "$\tau$-c\`adl\`ag".
\end{prop}
\begin{proof}

Let 
$$
\Omega_o:=\{\omega \in \Omega :(0,\infty)\ni t\mapsto X(t)\in E \mbox{ is } \tau\mbox{-continuous}\}.
$$
By assumption $(i)$ we have $\mathbb{P}(\Omega_o)=1.$
Let
$$
\Lambda:=\{\omega \in\Omega : [0,\infty)\ni t\mapsto X(t, \omega)\in E \mbox{ is not } \tau\mbox{-continuous}\}
$$ 
and
$$
\Lambda_o:=
\{\omega \in\Omega : [0,\infty)\ni t\mapsto X(t, \omega)\in E \mbox{ is not } \tau\mbox{-continuous in zero}\}
$$
Clearly we have 
$\Omega_o\cap\Lambda=
\Omega_o\cap\Lambda_o$ and
\begin{align*}
\Lambda_o^c
&=\{\omega \in \Omega: X(t,\omega) \rightarrow X(0,\omega) \mbox{ when }t \to 0 \}\\ 
&=
\bigcap_{\varepsilon\in \mathbb{Q}_+} \bigcup_{t\in\mathbb{Q}_+}
\{w\in \Omega: \sup_{0<s\leq t} | X(s, \omega)- X(0, \omega)|\leq \varepsilon\}\\
&=\bigcap_{\varepsilon\in \mathbb{Q}_+} \bigcup_{t\in\mathbb{Q}_+}\bigcap_n\{w\in \Omega: \sup_{\frac{1}{n}<s\leq t} | X(s, \omega)- X(0, \omega)|\leq \varepsilon \}.    
\end{align*}
Further, let
\begin{equation*}
    \Lambda_1:=
\bigcap_{\varepsilon\in \mathbb{Q}_+} \bigcup_{t\in\mathbb{Q}_+}\bigcap_n
\{ w\in \Omega: \sup_{\frac{1}{n}<s\leq t, s\in \mathbb{Q}} | X(s, \omega)- X(0, \omega)|\leq \varepsilon \},
\end{equation*}
and $\mathcal{F}^0:= \sigma(X(t): t \geq 0)$.
Then $\Lambda_1\in \mathcal{F}^0$ and
$\Omega_o\cap \Lambda_o^c= 
\Omega_o\cap \Lambda_1$ since 
the trajectories from $\Omega_o$ are $\tau$-continuous on $(0, \infty)$.

Define the function $V$ on $E$ as
$$
V(x):=\mathbb{P}_x(\Lambda), \ x \in E.
$$
Let us show that $V$ is a Borel function on $E$.
Indeed, recall that by \cite[Theorem (3.6)]{BlGe68} we have that
$E\ni x \mapsto \mathbb{P}_x(\Gamma)$ is a Borel function on $E$ for every $\Gamma$ in $\mathcal{F}^0.$
Because $\mathbb{P}_x(\Omega_o)=1$ we get that
$V(x)=\mathbb{P}_x (\Lambda_o)=1-\mathbb{P}_x(\Omega_o\cap\Lambda_o^c)=
1- \mathbb{P}_x(\Lambda_1).$

We show that $V$ is a strongly supermedian function w.r.t. $\mathcal{U}_\alpha$
for some $\alpha>0$. 
It is sufficient to prove that $R^K_\alpha V\leq V$ for any Ray compact subset $K$ of $E$.
Using the probabilistic characterization of the reduced function due to G.A. Hunt, see (\ref{Hunt-th}), we have for every $x\in E$: 
\begin{align*}
R^K_\alpha V(x)
&=\mathbb{E}_x (e^{-\alpha D_K} V(X(D_K)))
=\mathbb{E}_x (e^{-\alpha D_K} \mathbb{P}_{X(D_K)}(\Lambda))\\
&=
\mathbb{E}_x (e^{-\alpha D_K} 1_\Lambda \circ \theta(D_K))
=
\mathbb{E}_x \left\{e^{-\alpha D_K} 1_{[\theta(D_K)]^{-1}(\Lambda)}\right\}
\leq 
\mathbb{P}_x(\Lambda)\\
&=V(x),
\end{align*}
where the last inequality holds because $\theta_T^{-1}(\Lambda)\subset \Lambda$. 
By Remark \ref{rem8.25} it follows that the set 
$E_0=
\{x \in E:  \mathbb{P}_x(\Lambda)=0\}$ belongs to $\ca(\cu)$.
By assumption $(ii)$ we have $m(V)=0$ and therefore 
$m(E\setminus E_0)=m([V>0])=0$.
We conclude that $E\setminus E_0$ is $m$-inessential, completing the proof.
\end{proof}

Let $M\in \cb$, $F:= E\setminus M$
and assume that 
$\cu=(U_\alpha)_{\alpha>0}$ is a sub-Markovian resolvent of kernels on 
$(F, \cb|_F)$,
the 
resolvent of a right process with state space $F$.
For all $\alpha>0$ define the kernel
\[
  U'_\alpha f
  = 1_F U_\alpha f
    + \frac{1}{\alpha} 1_M f, \;
 f\in p\cb.
\]
Then, as in Definition \ref{defi:trivial modification},
the  family $\cu'=(U'_\alpha)_{\alpha>0}$ is a
sub-Markovian resolvent of kernels on $(E,\cb)$. 
We can expose now the trivial extension of a right process to a larger set.

\begin{prop} \label{prop-trivial}
Assume that $\cu=(U_\alpha)_{\alpha>0}$ 
is the 
resolvent on $F$ of a right process $X$ with state space $F$.
Then $\cu'=(U'_\alpha)_{\alpha>0}$ 
is the resolvent of a right process $X'=(\Omega', \cf', X'_t, \mathbb{P}'^x , x\in E)$ with state space $E$ and 
each point of $M$ is a trap for $X'$, that is, 
$
\mathbb{P}'^x[X'_t=x \, \mbox{ for all }\,  t\geq 0]=1\,  \mbox{ for all }\, x\in M.$
Moreover, $X$ is the restriction of $X'$ from $E$ to $F$ in the sense of \Cref{defi:restr-process}.
\end{prop}
For the construction of the process $X'$, the extension of $X$ from $F$ to $E$, we refer to \cite[p. 118]{MaRo92}.

\begin{defi}[Trivial extension of the process]\label{defi:trivial}
The process $X'$ from Proposition 
\ref{prop-trivial} is called the {\rm trivial extension} of the process $X$ from $F$ to $E$.
\end{defi}

\begin{prop}\label{prop-trivial-process}
Let $X$ be a right process with state space $E$ and resolvent family 
$\cu=(U_\alpha)_{\alpha>0}$ and  
 let $M\in \cb$ be such that
$R^M_\beta 1=0$ on $E\setminus M$.
Consider the resolvent $\cu'=(U'_\alpha)_{\alpha>0}$ on $E$ defined by (\ref{trivial-mod}).
Then $\cu'=(U'_\alpha)_{\alpha>0}$ 
is the resolvent of a right process $X'$  with state space $E$ and 
each point of $M$ is a trap for $X'$.
\end{prop}

\begin{proof}
Because 
$E\setminus M\in \ca(\cu)$, we can 
apply Proposition \ref{prop:restriction}, to the  restrict of $X$ to $E\setminus M$.
The requested process $X'$ on $E$ is obtained now applying 
Proposition \ref{prop-trivial} to this restriction.
\end{proof}

\begin{defi}[Trivial modification of the process]\label{defi:trivial-process}
The process $X'$ from Proposition 
\ref{prop-trivial-process} is called the {\rm trivial modification of the process} $X$ on $M$.
\end{defi}

\subsubsection{The killed process}
In this part we assume that $X$ is a right process on $(E,\mathcal{B})$ with transition semigroup $(P_t)_{t\geq 0}$ and resolvent $\mathcal{U}=(U_\alpha)_{\alpha>0}$.

\begin{prop}\label{Borel-measurab}
Let $F$ be a  closed set in a natural topology.
Then $R_\alpha^F$ and $B^F_\alpha$ are Borel kernels on $E$, that is, if 
$g\in p\cb$ then 
$R_\alpha^F g$ and 
$B_\alpha^F g$
are Borel functions on $E$.
\end{prop}

\begin{proof}
Let $\cf^0:=\sigma(X(t), t\geq 0)$.
The main observation is that 
$\Omega\ni\omega \mapsto D_F(\omega)$
is $\cf^0$-measurable, which follows arguing as in $(6.16)$ from \cite{BlGe68}, page 35.
Indeed, if $s\geq 0$ then we have
\begin{align*}
[D_F>s]
&=\{\omega \in \Omega:\mbox{there exists } \varepsilon \in \mathbb{Q}_+^*
\mbox{ such that } X(t)(\omega)\in E\setminus F \mbox{ for all } t\in[s, s+\varepsilon]\}\\
&=\bigcup_{\varepsilon \in \mathbb{Q}_+^*}\bigcap_{t\in[s, s+\varepsilon]\cap\mathbb{Q}} X(t)^{-1}(E\setminus F),   
\end{align*} 
where for the second equality we  used the right continuity of the paths of $X$ (in the natural topology) and the fact that $E\setminus F$ is an open set.
Since clearly $X(t)^{-1}(E\setminus F)\in \cf^0$, we conclude that also $[D_F>s]\in \cf^0$.
The process  $X$ is progressively measurable 
and consequently $(t, \omega)\mapsto g(X(t)(\omega))$
is 
$\cb(\mathbb{R}_+)\otimes \cf^0$-measurable.
It follows that
\begin{equation*}
    \Omega\ni\omega\mapsto Y(\omega):= e^{-\alpha D_F(\omega)} g (X(D_F(\omega))(\omega)) \mbox{ is } \cf^0\mbox{-measurable}.
\end{equation*}
Hence $R^F_\alpha g(x)=
\mathbb{E}_x (Y)$, with $Y\in b\cf^0$.
On the other hand, by
\cite[Theorem (3.6)]{BlGe68}, 
the function
$E\ni x\mapsto \mathbb{E}_x (Y)$
is $\cb$-measurable for all $Y\in b\cf^0$, since the transition semigroup  consists of Borel kernels on $E$.
Consequently, $R^F_\alpha g\in p\cb$.
If $g\in b\ce(\cu_\alpha)$ then
$B^F_\alpha g=\widehat{R^F_\alpha g}\in 
\ce(\cu_\alpha)$, so, it is also a  Borel function on $E$.
By a monotone class argument it results that $B^F_\alpha g\in p\cb$ for all $g\in p\cb$.
\end{proof}

Let $D\in \cb$ be a finely open subset of $E$.
For every $\alpha > 0$ we denote by $U^D_\alpha$
the kernel on  $(E, \cb^n)$  given by 
\begin{equation*}
    U^D_\alpha f := U_\alpha f - B_\alpha^{E\setminus D}U_\alpha f, \quad f \in bp\cb^n.
\end{equation*}
Using (\ref{Hunt-th}) one can see that
\begin{equation*}
U^D_\alpha f(x)
=\mathbb{E}_x \left\{\int_0^{T_D}e^{-\alpha t} f(X(t))\;dt\right\} \mbox{ for } f \in bp\cb^n, x\in E, \mbox{ and } \alpha>0.
\end{equation*}
Then the family 
$\mathcal{U}^D=(U^D_\alpha)_{\alpha>0}$
is a sub-Markovian resolvent of kernels on $(E, \cb^n)$,
$U^D_\alpha\leq U_\alpha$, and 
$U^D_\alpha(1_{E\setminus D})=0$ for all $\alpha>0$; cf. e.g. \cite{BeBo96} and \cite{BeBo04a}. 
In particular, we may consider the restriction of $\mathcal{U}^D$ to $D$, it is a sub-Markovian resolvent of kernels 
on  $(D, \cb^n)$, also denoted by $\mathcal{U}^D$.

Let
\begin{equation*}
    \widetilde{D}:= \left\{ x\in E: \exists\; u\in \ce(\cu_\alpha)  \mbox{ such that } B_\alpha^{E\setminus D} u(x) < u(x)\right\}. 
\end{equation*}
Then by \cite{BeBo96} and 
Section 3.6 in \cite{BeBo04a} the set $\widetilde{D}$ is finely open,
$\widetilde{D}\in \cb^n$, and
$D\subset \widetilde{D}$.

\begin{prop} \label{killed-process}
Let $D\in \cb$ be a finely open subset of $E$. 
Then the following assertions hold.
\begin{enumerate}
    \item[(i)] There exists $\sigma$-algebra $\widetilde{\cb}$  on $\widetilde{D}$ 
    such that 
    $(\widetilde{D}, \widetilde{\cb})$ 
    is a Radon measurable space,  
    $\cb|_{\widetilde{D}}
    \subset \widetilde{\cb}
    \subset \cb^n|_{\widetilde{D}}$,
    and the family 
    $\mathcal{U}^D=(\mathcal{U}^D_\alpha)_{\alpha>0}$ 
    is a sub-Markovian resolvent of kernels on $(\widetilde{D}, \widetilde{\cb})$ 
    such that
    $\ce(\cu^D_\alpha)$ generates $\widetilde{\cb}$.
    \item[(ii)]  Let $X^D$ denote the process $X$ killed at time $T_{E\setminus D}$, that is,
    \begin{equation*}
        X^D(t)(\omega):=
        \begin{cases}
            
    {X}(t)(\omega), \quad &t<T_{E\setminus D}(\omega)\\ 
    \Delta,\quad & 
    t\geq T_{E\setminus D}(\omega)
        \end{cases},
        \quad \omega\in\Omega.
    \end{equation*}
    Then $X^D$ is a right process with state space $\widetilde{D}$ and life time $T_{E\setminus D}$, 
    and $\cu^D$ is its associated resolvent.
    In addition, the set $\widetilde{D}\setminus D$ is polar with respect to $X^D$.
    \item[(iii)] If $D$ is an open set in a natural topology then $\cu^D$ is a resolvent of kernels on $(D, \cb|_D)$ and 
    $X^D$ is a right process with state space $D$ and resolvent $\cu^D$.
\end{enumerate}
\end{prop}

\begin{proof}
Assertion $(i)$ and the fact that the set $\widetilde{D}\setminus D$ is polar follow from Proposition 3.6.7 from \cite{BeBo04a}.

Assertion $(ii)$ is a consequence of 
Corollary 12.24 from \cite{Sh88}, since $T_{E\setminus D}$ is an exact terminal time.

To prove assertion $(iii)$, observe first that using Proposition \ref{Borel-measurab}, $U^D_\alpha$ is a kernel on $(D, \cb|_D)$ for each  $\alpha>0$.
It follows that $X^D$ is a right process with state space $D$.
\end{proof}

\begin{rem}
\begin{enumerate}
    \item[(i)] Because the set $\widetilde{D}\setminus D$ is polar (according to assertion $(ii)$ of Proposition \ref{killed-process}), we can always restrict the process $\overline{X}$ from $\widetilde{D}$ to $D$, not only in the case when $D$ is open in a natural topology, as in assertion $(iii)$ of the above proposition. 
    So, we can regard $X^D$ as a right process on $D$, but in a wider sense, having the transition semigroup only universally $\cb|_D$-measurable.
    \item[(ii)] The following properties are equivalent for a subset $D$ of $E$.
    \begin{enumerate}
        \item[(ii.1)] $D$ is open in a natural topology.
        \item[(ii.2)] $D$ is open in a Ray topology.
        \item[(ii.3)] $D$ is open in the topology generated by a countable family of $\alpha$-excessive functions for some $\alpha>0$.
    \end{enumerate}
\end{enumerate}
\end{rem}

\begin{defi}[Killing]\label{defi:killing}
We say that the process $X^D$ from Proposition \ref{killed-process} is obtained by
{\rm killing $X$ at the first entry time in $E\setminus D$}.
\end{defi}

\begin{rem}\label{rem:finely continuous on D}
Let $f:E\rightarrow \mathbb{R}$ be universally $\mathcal{B}$-measurable and $D\in \mathcal{B}$ be finely open.
Then, it follows from \Cref{killed-process} and \Cref{thm 4.6} that $f$ is finely continuous on $D$ if and only if $(f(X(t)))_{0\leq t<T_{E\setminus D}}$ has $\mathbb{P}^{x}$-a.s. right continuous paths for all $x\in D$.
\end{rem}

\subsubsection{The natural extension}

In this paragraph we recall one of the key tools (see \Cref{thm:natural} below) that we employ in the proof of \Cref{thm:main1} for solving the open problem OP2.
\begin{defi}[Natural extension]\label{defi:natural}
Let $\left(E,\mathcal{B}\right)$ and $\left(\overline{E},\overline{\mathcal{B}}\right)$ be two Lusin measurable spaces such that $E\in \overline{\mathcal{B}}$ and $\mathcal{B}=\overline{\mathcal{B}}|_{\mathcal{B}}$.
\begin{enumerate}
    \item[(i)] Let $\mathcal{U}$ be a sub-Markovian resolvent of kernels on $E$.
    A second sub-Markovian resolvent of kernels $\overline{\mathcal{U}}:=(\overline{U}_\alpha)_{\alpha>0}$ on $\left(\overline{E}, \overline{B}\right)$  is called an extension of $\mathcal{U}$ if:
     \begin{enumerate}
     \item[(i.1)] $\overline{U}_\alpha (1_{\overline{E} \setminus E})=0$.
     \item[(i.2)] $(\overline{U}_\alpha f)|_E=U_\alpha (f|_E) $ (on $E$) for all $\alpha>0$ and $f \in b\overline{\mathcal{B}}$.
     \end{enumerate} 
     In this case, for shortness, we write  \begin{equation}\label{eq:not_resolvent_extension}
         \left(E,\mathcal{U}\right)\subset \left(\overline{E},\overline{\mathcal{U}}\right)
     \end{equation}
     \item[(ii)] Given a right process $X$ on $E$, we say that a second Markov process $\overline{X}=(\overline{\Omega}, \overline{\mathcal{F}}, \overline{\mathcal{F}}_t, \overline{X}(t), \overline{\theta}(t), \overline{\mathbb{P}}^x)$, with state space $\overline{E}$,  is a {\it natural extension} of $X$ if the following conditions are fulfilled:
    \begin{enumerate}
    \item[(ii.1)] $\overline{X}$ is a right process.
    \item[(ii.2)] The processes $(({X}(t))_{t \geq 0}, \mathbb{P}_{x})$ and $((\overline{X}(t))_{t\geq 0}, \overline{\mathbb{P}}^{x})$ are equal in distribution for all $x\in E$;
    \item[(ii.3)] The set $\overline{E}\setminus E$ is polar with respect to (the resolvent of) $\overline{X}$.
    \end{enumerate}
    In this case, again for shortness, we write
     \begin{equation}\label{eq:not_process_extension}
         \left(E,X\right)\subset \left(\overline{E},\overline{X}\right).
     \end{equation}
\end{enumerate}
\end{defi}

\begin{prop} \label{prop 2.3}
If $\overline{X}$ is a natural extension of $X$, then its resolvent denoted by $\overline{\mathcal{U}}$ is an extension of $\mathcal{U}$.
\end{prop}

\begin{thm}[{cf \cite[Theorem 1.10]{BeCiRo20}}] \label{thm:natural}
Let $\overline{\mathcal{U}}$ be an extension of $\mathcal{U}$.  Then there exists a natural extension $\overline{X}$ of $X$, with resolvent $\overline{\mathcal{U}}$, if and only if $\mathbf{(H_{\mathcal{C}})}$ is satisfied.
\end{thm}

\begin{rem}\label{rem:maximal}
As already stated in \cite{BeCiRo24a}, the following assertion can be easily deduced from \Cref{thm:natural}:
Let $X$ be a right process on $E$. The right process $X^1$ on the saturation $E^1$ given by \Cref{thm:saturation} is maximal, in the sense that $X^1$ is a natural extension of any natural extension $\widetilde{X}$ of $X$.
In particular, for any such natural extension $\widetilde{X}$ on $\widetilde{E}$, we have that $\widetilde{E}\setminus E$ is polar.
\end{rem}

\begin{lem}[{cf \cite[Lemma 2.1]{BeBoRo06}}] \label{lem:kernel_trick}
Let $K$ be a kernel and $\nu$ be a $\sigma$-finite measure on a measurable space $(E,\mathcal{B})$, such that if $B\in \mathcal{B}$ and $\nu(B)=0$, then $K1_B = 0$ $\nu$-a.e. 
If $E_0\in \mathcal{B}$ satisfies $\nu(E\setminus E_0)$, then there exits $F\in \mathcal{B}, F\subset E_0$, such that $\nu(E\setminus F)=0$ and $K1_{E\setminus F}=0$ on $F$.
\end{lem}

\subsection{On martingale additive functionals}\label{ss:MAF}
Let $X=(\Omega, \mathcal{F}, \mathcal{F}(t) , X(t), \theta(t) , \mathbb{P}_x, t\geq 0, x\in E)$ be a right process on a Lusin measurable space $(E,\mathcal{B})$, with transition $(P_t)_{t\geq 0}$ and resolvent $\mathcal{U}=(U_\alpha)_{\alpha>0}$.
Further, let $\mu\in \mathcal{P}(E)$ and $f,g:E\rightarrow \mathbb{R}$ be $\mathcal{B}$-measurable such that
\begin{equation}\label{eq:f g mu}
    f \mbox{ is bounded and finely continuous} \quad \mbox{and} \quad \mu\left(U_{\alpha_0}|g|\right)<\infty \quad \mbox{for some } \alpha_0>0.
\end{equation}
Consequently, the process $M=(M(t))_{t\geq 0}$ given by
\begin{equation}\label{eq: process M}
    M(t):=f(X(t))-f(X(0))-\int_0^tg(X(r))dt, \quad t\geq 0,
\end{equation}
is well defined $\mathbb{P}_\mu$-a.e. and from $L^1(\mathbb{P}_\mu)$.
Furthermore, it is easily seen that $M$ is a right-continuous {\it additive functional} under $\mathbb{P}_\mu$, namely
\begin{equation*}
     [0,\infty)\ni r\mapsto M(r) \mbox{ is right-continuous and } M(t+s)=M(s)+M(t)\circ\theta(s)\quad \mathbb{P}_\mu\mbox{-a.s.}, \quad t,s\geq 0.
\end{equation*}
As a consequence, we deduce:
\begin{lem}\label{lem:martingale 1}
    The process $M$ given by \eqref{eq: process M} is a right-continuous $(\mathcal{F}_t)$-martingale under $\mathbb{P}_\mu$ if and only if $\mathbb{E}_\mu\left\{M(t)\right\}=0$ for every $t> 0$.
\end{lem}

Another straightforward fact is the following:
\begin{lem}\label{lem:martingale 2}
If the process $M$ given by \eqref{eq: process M} is a well-defined and right-continuous $(\mathcal{F}_t)$-martingale under $\mathbb{P}_x$ for $x\in E$ $\mu$-a.e., then $M$ is a well-defined and right-continuous $(\mathcal{F}_t)$-martingale under $\mathbb{P}_\nu$ for every $\nu \in \mathcal{P}$ for which there exists a constant $c>0$ such that $\nu\leq c\mu$.
\end{lem}

The following fine converse of \Cref{lem:martingale 2} is useful:
\begin{prop}\label{prop:martingale_absorbing}
   Let $X$ be a right process on a Lusin measurable space $(E,\mathcal{B})$, with transition $(P_t)_{t\geq 0}$ and resolvent $\mathcal{U}=(U_\alpha)_{\alpha>0}$.
   Let $\mu\in \mathcal{P}(E)$ such that $\mu\circ U_\alpha<<\mu$ for one (hence all) $\alpha>0$, $f,g:E\rightarrow \mathbb{R}$ be $\mathcal{B}$-measurable such that \eqref{eq:f g mu} holds, and moreover, assume that the process $M$ given by \eqref{eq: process M} is an $(\mathcal{F}(t))$-martingale under $\mathbb{P}_\nu$ for all $\nu\in \mathcal{P}(E)$ for which there exists a constant $c>0$ such that $\nu\leq c \mu$.
   Then there exists a set $E'\in \mathcal{B}$ such that $E\setminus E'$ is $\mu$-inessential and $M$ is a well-defined right-continuous $(\mathcal{F}_t)$-martingale under $\mathbb{P}_x$ for every $x\in E'$.
\end{prop}
\begin{proof}
Let us consider 
$
    E_0:=\left[U_{\alpha_0} |g|<\infty\right]
$,
so that, by \Cref{rem:exc_absorbing}, $E_0\in \mathcal{A}(\mathcal{U})$.
Since $\mu(E\setminus E_0)=0$ we deduce that $E\setminus E_0$ is  $\mu$-inessential.

Consequently, we can restrict the process $X$ (see \Cref{defi:restr-process}), the measure $\mu$, as well as the functions $f$ and $g$, from $E$ to $E_0$.
Note that by \eqref{eq:f g mu} we have that $M(t)$ is well defined $\mathbb{P}_x$-a.s. for every $x\in E_0$ and $t\geq 0$.
Moreover, 
\begin{equation*}
    M(t)\in L^1(\mathbb{P}_x)\cap L^1(\mathbb{P}_\mu), \quad t\geq 0,x\in E_0,
\end{equation*}
and it is easy to see that $M$ is a right-continuous {\it additive functional} on $E_0$, namely
\begin{equation*}
    [0,\infty)\ni r\mapsto M(r) \mbox{ is right-continuous and } M(t+s)=M(s)+M(t)\circ\theta(s)\quad \mathbb{P}_x\mbox{-a.s.},
\end{equation*}
for all $t,s\geq 0, x\in E_0$.
As a consequence, by \Cref{lem:martingale 1} we deduce: {\it If $x\in E_0$, then the process $M$ given by \eqref{eq: process M} is a right-continuous $(\mathcal{F}_t)$-martingale under $\mathbb{P}_x$ if and only if $\mathbb{E}_x\left\{M(t)\right\}=0$ for every $t\geq 0$.}

Note that
$\mathbb{E}_\nu\left\{M(t)\right\}=0$ for every $t\geq 0$ and $\nu$ as in the statement.
Thus,
\begin{equation*}
    \mathbb{E}_x\left\{M(t)\right\}=0, \quad x\in E_0\; \mu\mbox{-a.e. }, \quad t\geq 0.
\end{equation*}
Furthermore, since for $x\in E_0$ we have that $M$ is right-continuous $\mathbb{P}_x$-a.s., we deduce on the one hand that
\begin{equation}\label{eq:E(M)=0 mu-a.e.}
    \mathbb{E}_x\left\{M(t)\right\}=0, \quad t\geq 0, \quad x\in E_0\; \mu\mbox{-a.e. }
\end{equation}
We also have
\begin{equation*}
    \mathbb{E}_x\{M(t)\}=P_tf(x)-f(x)-\int_0^tP_rg(x)\;dr,\quad t\geq 0, x\in E_0.
\end{equation*}
Noticing that 
\begin{align*}
    e^{-\alpha t}\int_0^tP_rg(x)\;dr
    &= e^{-(\alpha-\alpha_0) t}\int_0^te^{-\alpha_0t}P_rg(x)\;dr
    \leq  e^{-(\alpha-\alpha_0) t}\int_0^te^{-\alpha_0r}P_rg(x)\;dr\\
    &\leq e^{-(\alpha-\alpha_0) t} U_{\alpha_0}g(x) \mathop{\longrightarrow}\limits_{t\to \infty} 0 \quad \mbox{ for } \alpha>\alpha_0,
\end{align*}
by taking the Laplace transform and using integration by parts we obtain, on the other hand, that for $x\in E_0$
\begin{equation*}
    \mathbb{E}_x\left\{M(t)\right\}=0, \quad t\geq 0 \Longleftrightarrow \alpha U_\alpha f(x)-f(x)-U_\alpha g(x)=0, \quad \alpha_0<\alpha\in \mathbb{Q}_+.
\end{equation*}
Note that since $f$ is finely continuous and using also \Cref{rem:U_exc}, we deduce that 
\begin{equation*}
    u_\alpha:=\alpha U_\alpha f-f-U_\alpha g \quad \mbox{is finely continuous on }  E_0, \quad \alpha>\alpha_0.
\end{equation*}
From \eqref{eq:E(M)=0 mu-a.e.}, the fact that $u_\alpha$ is finely continuous, and \Cref{rem 2.2} we deduce
\begin{equation*}
    u_\alpha=0 \;\mu\mbox{-a.e. on }E_0, \quad \mbox{hence} \quad E_\alpha:=\{x\in E_0 : u_\alpha(x)=0\} \mbox{ is } \mu\mbox{-polar}.
\end{equation*}
Consequently, by the same \Cref{rem 2.2} we deduce that there exists a $\mu$-inessential set $E'_\alpha$ such that $E'_\alpha\supset E_\alpha$, $\alpha>\alpha_0$.
Finally, we consider
\begin{equation*}
E':=\mathop{\bigcap}\limits_{\alpha_0<\alpha\in \mathbb{Q}} E'_\alpha.
\end{equation*}
We can now conclude since by \Cref{rem:absorbing_intersection} we get that $E'$ is $\mu$-inessential, and for every $x\in E'$ we have that $\mathbb{E}_x\left\{ M(t)\right\}=0, t\geq 0$, hence $M$ is an $(\mathcal{F}_t)$-martingale under $\mathbb{P}_x$.
\end{proof}

\subsection{Adding jumps by perturbation with kernels}\label{ss:adding jumps}

Let $X$ be a right process on a Lusin measurable space $(E,\mathcal{B})$ which for simplity is considered to be conservative, i.e. it has a.s. infinite lifetime. 
Let $(P_t)_{t\geq 0}$ denote its transition function, with resolvent $\mathcal{U}=(U_\alpha)_{\alpha>0}$.
In the sequel  $(\mathcal{L}, \mathcal{D}_b(\mathcal{L}))$ will be the generator (on $b\cb$) 
of $X$ in the following sense: 
\begin{equation}\label{eq:D(L)_weak}
    \mathcal{D}_b(\mathcal{L}):= U_\alpha (b\cb)\quad \mbox{and}\quad \mathcal{L}f:= \alpha f -g, \; f=U_\alpha g, \;g\in b\cb.
\end{equation}

The following connection between $(\mathcal{L}, \mathcal{D}_b(\mathcal{L}))$ and the martingale problem associated to $X$ is known.  
We present its proof for completeness.

\begin{prop}\label{prop:connection}
The following assertions are equivalent for
$f,g\in b\mathcal{B}$.
\begin{enumerate}
    \item[(i)] $f\in \mathcal{D}_b(\mathcal{L})$ and $\mathcal{L}f=g$.
    \item[(ii)] 
    The process
    $
        \left(M_f(t):=f(X(t))-f(X(0))-\int_0^t g(X(s))\;ds\right)_{t\geq 0}
    $
    is a c\`adl\`ag $(\mathcal{F}(t))$-martingale under $\mathbb{P}_x$ for every $x\in E$.
\end{enumerate}
\end{prop}

\begin{proof}
Assume (i). 
Since it is easy to see that
\begin{equation*}
P_tf=f+\int_0^t   P_sg\;ds, \quad t\geq 0, 
\end{equation*}
it follows that $\mathbb{E}_x\left\{f(X(t))-f(X(0))-\int_0^t g(X(s))\right\}=0$ for every $t\geq 0$.
Assertion (ii) now follows by \Cref{lem:martingale 1}.

Assume now that (ii) holds. 
Then we clearly have
\begin{equation*}
    \mathbb{E}_x\left\{ M_f(t)\right\}=0,\;  t \geq 0, 
\end{equation*}
hence
\begin{equation*}
    \int_0^\infty e^{-\alpha t}P_tf\;dt=\frac{1}{\alpha}f+\int_0^\infty e^{-\alpha t}\int_0^t P_sg\;ds\;dt, \, \alpha> 0,
\end{equation*}
which means that
\begin{equation*}
    f=U_\alpha \left(\alpha f-g\right)\in \mathcal{D}_b(\mathcal{L}), \quad \mbox{and } \quad \mathcal{L}f=\alpha f-(\alpha-g)=g.
\end{equation*}
\end{proof}

\begin{rem} 
Assertion (ii) in Proposition \ref{prop:connection}
is precisely the property  of the function $f\in b\cb$ to belong to the domain of the extended generator considered in \cite{HiYo12}; see also Remark 3.3 from \cite{BeBezzCi24}. 
Consequently, Proposition \ref{prop:connection} shows that
$\mathcal{D}_b(\mathcal{L})$ coincides with the domain from \cite{HiYo12} of the extended generator acting on bounded functions.

\end{rem}

Let $K$ be a bounded kernel on $(E, \cb)$,   
 $(P'_t)_{t\geqslant 0}$ be the transition function of the process $X$ killed with the multiplicative functional induced by $K1$
(cf. Ch. III from \cite{BlGe68}),   
 and let $(\mathcal{L}-K1, \mathcal{D}_b(\mathcal{L}-K1))$
be its generator. 
Applying Lemma 1.4 from \cite{BeSt93} on $b\cb$
it follows that
\begin{equation*}
    \mathcal{D}_b(\mathcal{L}-K1)= \mathcal{D}_b(\mathcal{L}) \quad \mbox{and} \quad (\mathcal{L}-K1)u=\mathcal{L}u-K1 u \quad \mbox{for every}\quad u\in\mathcal{D}_b(\mathcal{L}).
\end{equation*}

The next proposition is a version of Proposition 4.5 from 
\cite{BeLu16}; 
see also Proposition 3.4 from \cite{BeLuVr20}.

\begin{prop}\label{prop:perturbation}
Let $K$ be a bounded kernel on $(E, \cb)$.
Then the following assertions hold.

\begin{enumerate}
    \item[(i)] There exists a right process $X^K=(\Omega^K, \mathcal{F}^K, \mathcal{F}^K_t , X^K(t), \theta^K(t) , \mathbb{P}^{K}_{x})$ with state space $E$, having the generator
    \begin{equation*}
        (\mathcal{L}-K1+ K, \mathcal{D}_b(\mathcal{L}-K1+K)) \quad \mbox{and} \quad \mathcal{D}_b(\mathcal{L}-K1+K)= \mathcal{D}_b(\mathcal{L}).
    \end{equation*}
    Moreover, the fine topology induced by (the resolvent of) $X^K$ coincides with that induced by $X$.
    \item[(ii)] 
    The process $X^K$ solves the martingale problem for $(\mathcal{L}-K1+ K, \mathcal{D}_b(\mathcal{L})))$ under $\mathbb{P}^{K}_x$. 
    More precisely, for every $f\in \mathcal{D}_b(\mathcal{L})$ and $x\in E$ the process
    \begin{equation}\label{eq:martingale_weak}
        \left(f(X^K(t))-  f(X^K (0)) -\int_0^t (\mathcal{L}f-K1 f + Kf)(X^K (s)) ds\right)_{t\geqslant 0}
    \end{equation}
    is a c\`adl\`ag $(\mathcal{F}^K_t)$-martingale under $\mathbb{P}^{K}_{x}$.
\end{enumerate} 
\end{prop}

\begin{proof}
$(i)$ We apply  \cite[Proposition 4.5]{BeLu16} for $c:=K1$ and for the 
sub-Markovian kernel $\frac{1}{c} K$,
where the measure $\frac{1}{c} K(x,dy)$ is zero provided that $c(x)=0$. 
It follows that for any $f \in bp\mathcal{B}$ the equation
$$
r_t(x)=P'_t f (x)+\int_{0}^t P'_{t-u}(K{r_u})(x) du,\; t\geqslant  0,\; x\in E,
$$
has a unique solution $Q_t f  \in bp \mathcal{B}$,  the function  
$(t,x)\longmapsto Q_t f (x)$ is measurable, 
the family $(Q_t)_{t\geqslant  0}$ is a  semigroup of sub-Markovian kernels on $(E, \mathcal{B})$
and it is the transition function of a right process 
$X^K=(\Omega^K,\mathcal{F}^K, \mathcal{F}^K_t, X^K(t), \mathbb{P}^K_x)$
with state space $E$.
Let $\cu^K=(U^K_\alpha)_{\alpha>0}$ on $(E,\cb)$ be 
the resolvent of kernels induced by $(Q_t)_{t\geqslant  0}$, 
$U^K_\alpha =\int_0^\infty e^{-\alpha t} Q_t dt$, $\alpha>0$. 
Then 
\begin{equation} \label{Stoica}
U^K_\beta =U'_\beta+ U'_\beta  KU^K_\beta \quad \mbox{ for all } \beta >0,
\end{equation}
where $\cu'=(U'_\alpha)_{\alpha >0}$ is the resolvent of kernels induced by $(P'_t)_{t\geqslant 0}$.
We  have
$\ce(\cu^K_\beta)\subset  \ce(\cu'_\beta)$ and
$[b\ce(\cu^K_\beta)]=  [b\ce(\cu'_\beta)]=[b\ce(\cu_\beta)]$.
In particular, the fine topology induced by $X^K$ and the fine topology of $X$ coincide. 

Let $(\mathcal{L}^K, \mathcal{D}_b(\mathcal{L}^K))$ be the generator of $X^K$. 
By \cite[Lemma 1.4]{BeSt93} applied for the Banach space $b\cb$, and using also $\eqref{Stoica}$, it follows that
$\mathcal{D}_b(\mathcal{L}^K)=\mathcal{D}_b(\mathcal{L})$ and 
$\mathcal{L}^K= \mathcal{L}-K1 + K$.
Hence 
$(\mathcal{L}-K1+ K, \mathcal{D}_b(\mathcal{L}-K1+K))$
is the generator of $X^K$ 
and we have 
$\mathcal{D}_b(\mathcal{L}-K1+ K)= \mathcal{D}_b(\mathcal{L})$ as claimed.

The proof of assertion $(ii)$ is standard and we omit it; see e.g. \Cref{ss:MAF}. 
The right continuity of the martingale follows from \Cref{thm 4.6} since $f\in \mathcal{D}_b(\mathcal{L})$, hence $f$ is finely continuous; by the general theory, the martingale is in fact c\`adl\`ag.
\end{proof}

\paragraph{Example: adding jumps to a path-continuous Markov process.} 
The perturbation with a kernel of the generator was used in \cite{Bass79} and
\cite{BeSt93} in order to modify the jumps in the time evolution of a Markov process. 
However, the basic reference here is the pioneering 
articles \cite{IkNaWa68} and \cite{Me75}; see also \cite{KrMa25} for a very recent related work.

\begin{prop}[cf. e.g. \cite{Me75}, \cite{BeSt93}]\label{prop:adding jumps diffusion}
Assume that $(E,\tau)$ is locally compact with countable basis and $\mathcal{B}$ is the Borel $\sigma$-algebra.
Let $X$ be the right process on $E$ fixed above which is in addition assumed to have $\tau$-continuous paths.
Further, let $K$ be a bounded kernel on $(E, \cb)$. 
Then the right process $X^K$ provided by \Cref{prop:perturbation} is a Hunt process on $(E,\tau)$.
Moreover, for every $f\in pb\mathcal{B}\otimes\mathcal{B}, f(y,y)=0, y\in E$, we have
\begin{equation}\label{eq:jumps}
\mathbb{E}_x
\left\{\sum_{s\leqslant t}
f(X^K({s-}), X^K(s))\right\}
=
\mathbb{E}_x
\left\{\int_0^t \int_E 
f(X^K(s), y) K(X^K(s), dy) ds\right\}, \quad t\geq 0.
\end{equation}
\end{prop}

\begin{proof}
Let $X'=(\Omega',\mathcal{F}', \mathcal{F}'_t, X'(t), \mathbb{P}'^x)$
be the right process with state space $E$ and lifetime $\zeta'$, obtained by killing $X$  with the multiplicative functional induced by $K1$. 
We consider 
(using the terminology from \cite{Me75} and 
\cite{BeSt93}) 
the (simple) Markov process 
$X^K$ with $\tau$-c\`adl\`ag paths 
obtained  by resurrecting $X'$ with the kernel $R$ defined as
$$
R(\omega', dy)=
\frac{1_{[\zeta'< \infty]}(\omega')}{K1(X'({\zeta'-})(\omega'))}
K(X'({\zeta'-})(\omega'), dy) + 
1_{[\zeta'=\infty]}(\omega')\delta_\Delta (dy), \quad \omega' \in \Omega'.
$$
The proof now finishes by noticing that $X^K$ corresponds to the process constructed in \cite[Theorem 2.5]{IkNaWa68}, hence it is in fact a Hunt process. 

\end{proof}

\begin{rem}\label{rem:bounded_coefficients}
Note that $f\in D(\mathcal{L})$ used as test function for the martingale problem \eqref{eq:martingale_weak} automatically satisfies the restriction $\mathcal{L}f\in b\mathcal{B}$. 
This, of course, fits very well with generators $L$ in differential form that have locally bounded coefficients, which is the case of \Cref{lem:b locally bounded} and \Cref{prop: b locally bounded} in \Cref{ss:jumps}, or the case of the linearized operator ${\sf L}_t^u$ given by \eqref{eq:a^ub^u} in \Cref{ss:example_nonlinear}. 
However, in \Cref{s:reg_superposition_linear} we are interested to cover more general coefficients,
so in what follows we aim to present an $L^1$-approach that allows to solve the perturbed martingale problem \eqref{eq:martingale_weak} for discontinuous diffusion coefficients and drift coefficients that are merely in $L^2_{\sf loc}$.
\end{rem}

\paragraph{The $L^1$-approach.}
In this paragraph we assume that there exists a $\sigma$-finite measure $\mu$ on $E$ such that
\begin{equation}\label{eq:L^1(mu)}
    (P_t)_{t\geq 0} \mbox{ can be extended to a } C_0\mbox{-semigroup on } L^1(\mu),
\end{equation}
where recall that $(P_t)_{t\geq 0}$ is the transition function of the right process $X$ fixed in the beginning of this part. 
\begin{rem}
Notice that, by \cite[Proposition 2.1]{BeCiRo18} it follows that $(\ref{eq:L^1(mu)})$ holds provided that
$\mu$ is an excessive measure w.r.t. $X$,
i.e., $\mu\circ P_t\leqslant \mu$ for all $t>0$.
The converse also holds: If  a sub-Markovian $C_0$-resolvent of contractions on $L^1(\mu)$ is given, then it is associated to a right process, however, one has to enlarge $E$ by a zero set; cf. Theorem 2.2 from \cite{BeBoRo06}.
One can show that under certain additional conditions the above right process lives on $E$;  cf. \cite{BeBoRo06a}.    
\end{rem}

We further denote by $\left(L_1, D(L_1)\right)$ the generator induced by $(P_t)_{t\geq 0}$ on $L^1(\mu)$.

Let $K$ be a bounded kernel on $(E,\mathcal{B})$ as above, and consider the following assumption on $K$:

\medskip
\noindent{$\bf (H_K^\mu)$} $K$ induces a bounded operator on $L^1(\mu)$, namely
\begin{equation*}
    \exists\; c\in (0,\infty): \quad \int_E |Kf(x)| \;\mu(dx)\leq c \int_E |f(x)| \;\mu(dx), \quad f\in L^1(\mu).
\end{equation*}

\begin{rem}
Let $v\in bp\mathcal{B}$.
Then the multiplication by $v$ is trivially a bounded operator on $L^p(\mu)$. 
Moreover, if $K$ satisfies $\bf (H_K^\mu)$, then $vK$ trivially satisfies $\bf (H_{vK}^\mu)$ for any $v\in pb\mathcal{B}$, in particular for $v=K1$.
\end{rem}

\begin{prop}\label{prop:L1_approach}
    Let $X$ be the right process fixed above with its transition function satisfying \eqref{eq:L^1(mu)}, and $K$ be a bounded kernel on $(E,\mathcal{B})$ which satisfies $\bf (H_K^\mu)$, and set $c:=K1$. 
    Let $X^K$ be the right process provided by \Cref{prop:perturbation}.
    Then the following assertions hold.
    \begin{enumerate}
        \item[(i)] The transition function of $X^K$ extends to a $C_0$-semigroup on $L^1(\mu)$, whose generator is 
        \begin{equation*}
            L_1-c+K, \quad D(L_1-c+K)=D(L_1).
        \end{equation*}
        \item[(ii)] Let $f\in D(L_1)\cap b\mathcal{B}$ be finely continuous.
        Then there exists a set $\tilde{E}\in\mathcal{B}$ such that $E\setminus \tilde{E}$ is $\mu$-inessential and for every $x\in \tilde{E}$
        \begin{equation*}
        \mathbb{E}^K_{x}\left\{\int_0^{t} \left|L_1f-cf+Kf\right|(X^K(r))\; dr
               \right\}<\infty, \quad t\geq 0,
        \end{equation*}
        and the process
        \begin{equation}\label{eq:martingale_mu}
        \left(f(X^K(t))-  f(X^K(0)) -\int_0^t (L_1f-cf + Kf)(X^K(s)) ds\right)_{t\geqslant 0}
    \end{equation}
    is a c\`adl\`ag $(\mathcal{F}^K_t)$-martingale under $\mathbb{P}^{K}_{x}$.
    \end{enumerate}
\end{prop}
    \begin{proof}
        (i). First of all, we can apply \cite[Chapter III, 1.3]{EnNa00} two times in a row, first for the perturbation $L_1-c$ and then for the perturbation $L_1-c+K$, to ensure that $ \left(L_1-c+K, \;D(L_1)\right)$ generates a $C_0$-semigroup on $L^1(\mu)$, which we denote by $(S_t)_{t\geq 0}$.
        Furthermore, by \cite[Chapter III, 1.7]{EnNa00} we have
        \begin{equation*}
            S_t=P_t^c+\int_0^t P_{t-s}^ccK S_s\;ds \quad \mbox{as operators on } L^1(\mu),\quad t\geq 0.
        \end{equation*}
        It is then straightforward to deduce that 
        \begin{equation*}
            S_tf=Q_tf \quad \mu\mbox{-a.e.}, \quad f\in L^1(\mu)\cap b\mathcal{B}, \quad t\geq 0,
        \end{equation*}
        where recall that $(Q_t)_{t\geq 0}$ is the transition function of $X^K$.
        Thus, assertion (i) is proved.

\medskip
(ii). Let $f\in L^1(\mu)\cap b\mathcal{B}$ be finely continuous and consider
\begin{equation*}
    M(t):=f(X^K(t))-f(X^K(0))-\int_0^t (L_1-c+K)u(X^K(r))dr, \quad t\geq 0,
\end{equation*}
which is well defined $\mathbb{P}_\mu$-a.e. and from $L^1(\mathbb{P}^{K}_{\mu})$.
Since $(Q_t)_{t\geq 0}$ is a $C_0$-semigroup on $L^1(\mu)$ with generator $L_1-c+K$, we get that
$\mathbb{E}^{K}_{\nu} \left\{M(t)\right\}=0$, $t\geq 0$ for every $\nu\in \mathcal{P}(E)$ which satisfy $\nu\leq \delta \mu$ for some constant $\delta\in (0,\infty)$.
Now, assertion (ii) is entailed by \Cref{prop:martingale_absorbing}.
\end{proof}

\medskip

\paragraph{The Lyapunov approach.} 
In this paragraph we let $X$ be the right process fixed above with transition function $(P_t)_{t\geq 0}$ and resolvent $\mathcal{U}:=(U_\alpha)_{\alpha>0}$,
and assume
\begin{equation}\label{eq:V}
    \exists\; V:E\rightarrow [0,\infty)\; \mathcal{B}\mbox{-measurable} \quad \mbox{and} \quad \omega>0 \mbox{ such that} \quad P_tV\leq e^{\omega t} V,\;t\geq 0.
\end{equation}
That is, using the terminology from \Cref{defi 4.1}, $V$ is $\mathcal{U}_\omega$-supermedian.

Furthermore, we consider
\begin{equation*}
    \mathcal{B}_{\lesssim V}:=\left\{f\in \mathcal{B} : \mbox{there exists } \delta\in (0,\infty) \mbox{ such that } |f|\leq \delta V\right\},
\end{equation*}
and set 
\begin{equation}\label{eq:D(L)_V}
    \mathcal{D}_V(L):= U_{\alpha} (\mathcal{B}_{\lesssim V}), \alpha>\omega, \quad \mbox{and}\quad \mathcal{L}f:= \alpha f -g, \; f=U_\alpha g, \;g\in \mathcal{B}_{\lesssim V}.
\end{equation}
\begin{rem}
Note that $\mathcal{D}_V(\mathcal{L})$ is independent of $\alpha>\omega$.  
\end{rem}

Let us endow $\mathcal{B}_{\lesssim V}$ with the following weighted supremum norm
\begin{equation*}
    |f|_{V}:=\sup_{x\in E}\frac{|f(x)|}{1+V(x)}, \quad f\in \mathcal{B}_{\lesssim V}.
\end{equation*}
It is easy to check that $(\mathcal{B}_{\lesssim V}, |\cdot|_V)$ is a Banach space.

Further, we consider the following condition on $K$:

\medskip
\noindent{$\bf (H_K^V)$} $K$ is a bounded kernel on $(E,\mathcal{B})$ and 
\begin{equation*}
    \exists\; \delta\in(0,\infty):\quad KV\leq \delta V.
\end{equation*}

\begin{prop}
    Let $X$ be the right process fixed above with its transition function satisfying \eqref{eq:L^1(mu)}, $K$ be a bounded kernel on $(E,\mathcal{B})$ which satisfies $\bf (H_K^V)$, and set $c:=K1$. 
    Let $X^K$ be the right process provided by \Cref{prop:perturbation}, with transition function $(Q_t)_{t\geq 0}$ and resolvent $\mathcal{U}^K:=(U_{\alpha}^K)_{\alpha>0}$.
    Then the following assertions hold.
    \begin{enumerate}
        \item[(i)] We have 
        $
            \alpha U^K_{\alpha+\delta+\omega}V\leq V, \quad \alpha> 0,
        $
        and if we set
        \begin{equation*}
           \mathcal{D}_V(\mathcal{L}-c+K):=U^K_{\alpha}(\mathcal{B}_{\lesssim V}), \alpha>\delta+\omega, \quad (\mathcal{L}-c+K) f:= \alpha f+g, \quad f=U^K_{\alpha}g,\;g\in\mathcal{B}_{\lesssim V},
        \end{equation*}
        then  
        \begin{equation*}
             \mathcal{D}_V(\mathcal{L}-c+K)= \mathcal{D}_V(\mathcal{L}).
        \end{equation*}
        \item[(ii)] For every $f\in \mathcal{D}_V(\mathcal{L})$ and every $x\in E$, the process
    \begin{equation*}
        \left(f(X^K(t))-  f(X^K(0)) -\int_0^t (\mathcal{L}f-cf + Kf)(X^K(s)) ds\right)_{t\geqslant 0}
    \end{equation*}
    is a c\`adl\`ag $(\mathcal{F}^K_t)$-martingale under $\mathbb{P}^K_x$.
    \end{enumerate}
\end{prop}
\begin{proof}
(i). Note that by \eqref{eq:V} we deduce
\begin{equation*}
    \alpha U^c_{\alpha+\omega}V\leq \alpha U_{\alpha+\omega}V\leq V, \quad \alpha>0
\end{equation*}
Also, iteration \eqref{Stoica} we have 
\begin{equation*}
U^K_\alpha 
=  U^c_\alpha\left(  \sum_{n\geq 0}(KU^c_\alpha)^n\right) \quad \mbox{ for all } \alpha >0.
\end{equation*}
Moreover,
\begin{equation*}
    K U_{\alpha+\omega}^c V\leq \frac{\delta}{\alpha} V, \quad \alpha>0.
\end{equation*}
By setting $V_n:=V\wedge n, \quad n\geq 1$, we have 
\begin{align*}
U^K_{\alpha+\delta+\omega} V 
&=\sup_n U^K_{\alpha+\delta+\omega} V_n
\leq 
\sup_n U^c_{\alpha+\delta+\omega}\left(  \sum_{l\geq 0}(KU^c_{\alpha+\delta+\omega})^lV_n\right)\\
&\leq
U^c_{\alpha+\delta+\omega}\left(  \sum_{l\geq 0}(KU^c_{\alpha+\delta+\omega})^lV\right)
\leq
U^c_{\alpha+\delta+\omega}\left(V  \sum_{l\geq 0}\left(\frac{\delta}{\alpha+\delta}\right)^l\right)\\
&=
\frac{1}{\alpha}V
\quad \mbox{ for all } \alpha >0.
\end{align*}
In order to prove the rest of assertion (i) we simply apply \cite[Lemma 1.4]{BeSt93} for the resolvents $(U_{\alpha+\delta+\omega})_{\alpha>0}$, $(U^K_{\alpha+\delta+\omega})_{\alpha>0}$, and the Banach space $\left(\mathcal{B}_{\lesssim V}, |\cdot|_V\right)$.

\medskip
(ii). This is now standard, see e.g. \Cref{ss:MAF}, noticing that the integrability of the martingale is ensured by (i).
\end{proof}

\paragraph{Acknowledgements.} 
{\rm 
We gratefully acknowledge financial support by the Deutsche Forschungsgemeinschaft (DFG, German Research Foundation) – Project-ID 317210226 – SFB 1283. 
I.C. was partially supported by a grant of the Ministry of Research, Innovation and Digitization, CNCS-UEFISCDI, project number PN-IV-P2-2.1-T-TE-2023-0219, within PNCDI IV.
}

\end{document}